\documentclass[reqno, 11pt]{amsart} 
\usepackage{amsfonts, amsmath, amssymb, amsthm}
\usepackage{mathrsfs}
\usepackage[margin=2.71cm, heightrounded]{geometry}
\usepackage{nccmath}
\usepackage[dvipsnames]{xcolor}
\usepackage[hidelinks]{hyperref}

\makeatletter 
\def\@tocline#1#2#3#4#5#6#7{\relax
  \ifnum #1>\c@tocdepth 
  \else
    \par \addpenalty\@secpenalty\addvspace{#2}%
    \begingroup \hyphenpenalty\@M
    \@ifempty{#4}{%
      \@tempdima\csname r@tocindent\number#1\endcsname\relax
    }{%
      \@tempdima#4\relax
    }%
    \parindent\z@ \leftskip#3\relax \advance\leftskip\@tempdima\relax
    \rightskip\@pnumwidth plus4em \parfillskip-\@pnumwidth
    #5\leavevmode\hskip-\@tempdima
      \ifcase #1
       \or\or \hskip 1em \or \hskip 2em \else \hskip 3em \fi%
      #6\nobreak\relax
    \hfill\hbox to\@pnumwidth{\@tocpagenum{#7}}\par
    \nobreak
    \endgroup
  \fi}
\makeatother

\newtheorem{thm}{Theorem}[section]
\newtheorem{lemma}[thm]{Lemma}
\newtheorem{prop}[thm]{Proposition}
\newtheorem{cor}[thm]{Corollary}

\theoremstyle{definition}
\newtheorem{rmk}[thm]{Remark}
\newtheorem{defn}[thm]{Definition}
\newtheorem{conj}[thm]{Conjecture}

\newtheorem{claim}[thm]{Claim}

\newcommand{\ga}{\gamma}
\newcommand{\ep}{\epsilon}
\newcommand{\ka}{\kappa}
\newcommand{\vep}{\varepsilon}
\newcommand{\vsi}{\varsigma}
\newcommand{\pa}{\partial}

\newcommand{\dx}{{\delta,\xi}}
\newcommand{\edx}{{\ep\delta,\ep\xi}}

\newcommand{\N}{\mathbb{N}}
\newcommand{\R}{\mathbb{R}}
\renewcommand{\S}{\mathbb{S}}

\newcommand{\mbp}{\mathbf{P}}

\newcommand{\mca}{\mathcal{A}}
\newcommand{\mcd}{\mathcal{D}}
\newcommand{\mce}{\mathcal{E}}
\newcommand{\mcf}{\mathcal{F}}
\newcommand{\mch}{\mathcal{H}}
\newcommand{\mci}{\mathcal{I}}
\newcommand{\mcl}{\mathcal{L}}
\newcommand{\mcp}{\mathcal{P}}
\newcommand{\mcr}{\mathcal{R}}
\newcommand{\mcv}{\mathcal{V}}
\newcommand{\mcw}{\mathcal{W}}

\newcommand{\mfc}{\mathfrak{c}}
\newcommand{\mfg}{\mathfrak{g}}
\newcommand{\mfh}{\mathfrak{h}}
\newcommand{\mfl}{\mathfrak{L}}

\newcommand{\msfA}{\mathsf{A}}
\newcommand{\msfF}{\mathsf{F}}
\newcommand{\msfa}{\mathsf{a}}
\newcommand{\msfc}{\mathsf{c}}
\newcommand{\msfd}{\mathsf{d}}
\newcommand{\msfk}{\mathsf{k}}
\newcommand{\msfm}{\mathsf{m}}
\newcommand{\msfp}{\mathsf{p}}
\newcommand{\msfq}{\mathsf{q}}
\newcommand{\msfpq}{\mathsf{p}\mathsf{q}}
\newcommand{\msd}{\mathscr{D}}

\newcommand{\oc}{\overline{C}}
\newcommand{\ox}{\overline{X}}
\newcommand{\omcr}{\overline{\mcr}}
\newcommand{\opsi}{\overline{\Psi}}

\newcommand{\bc}{\bar{c}}
\newcommand{\bsi}{\bar{\sigma}}

\newcommand{\tig}{\tilde{g}}
\newcommand{\tih}{\tilde{h}}
\newcommand{\tmfc}{\tilde{\mathfrak{c}}}
\newcommand{\KN}{\mathbin{\bigcirc\mspace{-15mu}\wedge\mspace{3mu}}}

\newcommand{\wte}{\widetilde{E}}
\newcommand{\wth}{\widetilde{H}}
\newcommand{\wtm}{\widetilde{M}}
\newcommand{\wtp}{\widetilde{P}}
\newcommand{\wtv}{\widetilde{V}}
\newcommand{\wtPsi}{\widetilde{\Psi}}

\renewcommand{\(}{\left(}
\renewcommand{\)}{\right)}
\newcommand{\la}{\left\langle}
\newcommand{\ra}{\right\rangle}
\newcommand{\bla}{\big\langle}
\newcommand{\bra}{\big\rangle}

\newcommand{\RN}[1]{%
\textup{\uppercase\expandafter{\romannumeral#1}}%
}

\def\ricci{\mathrm{Ric}}
\def\rm{\mathrm{Rm}}
\def\diver{\mathrm{div}}
\def\proj{\mathrm{Proj}}
\def\tr{\mathrm{tr}}
\def\norm#1#2{\|#2\|_{#1}}

\allowdisplaybreaks
\numberwithin{equation}{section}

\makeatletter
\@namedef{subjclassname@2020}{%
\textup{2020} Mathematics Subject Classification}
\makeatother

\begin{document}
\title[Compactness and non-compactness of the constant $Q$-curvature problems]{Compactness and non-compactness theorems \\ of the fourth-order and sixth-order \\ constant $Q$-curvature problems}

\author{Liuwei Gong}
\address[Liuwei Gong]{Department of Mathematics, Chinese University of Hong Kong, Shatin, NT, Hong Kong}
\email{lwgong@math.cuhk.edu.hk}

\author{Seunghyeok Kim}
\address[Seunghyeok Kim]{Department of Mathematics and Research Institute for Natural Sciences, College of Natural Sciences, Hanyang University, 222 Wangsimni-ro Seongdong-gu, Seoul 04763, Republic of Korea}
\email{shkim0401@hanyang.ac.kr shkim0401@gmail.com}

\author{Juncheng Wei}
\address[Juncheng Wei]{Department of Mathematics, Chinese University of Hong Kong, Shatin, NT, Hong Kong}
\email{wei@math.cuhk.edu.hk}

\begin{abstract}
We provide a complete resolution to the question of compactness for the full solution sets of the fourth-order and sixth-order constant $Q$-curvature problems on smooth closed Riemannian manifolds
not conformally diffeomorphic to the standard unit $n$-sphere, provided the associated conformally covariant differential operator has a positive Green's function.

Firstly, we prove that the solution set of the fourth-order constant $Q$-curvature problem is $C^4$-compact in dimensions $5 \le n \le 24$.
For $n \ge 25$, an example of an $L^{\infty}$-unbounded sequence of solutions has been known for over a decade (Wei and Zhao \cite{WZ}).
Additionally, the compactness result for $5 \le n \le 9$ was established by Li and Xiong \cite{LX}.

Secondly, we demonstrate that the solution set of the sixth-order constant $Q$-curvature problem is $C^6$-compact in dimensions $7 \le n \le 26$, whereas a blow-up example exists for $n \ge 27$.

Our main observation is that the linearized equations associated with both $Q$-curvature problems can be transformed into overdetermined linear systems, which admit nontrivial solutions due to unexpected algebraic structures of the Paneitz operator and the sixth-order GJMS operator.
This key insight not only plays a crucial role in deducing the compactness result for high-dimensional manifolds,
but also reveals an elegant hierarchical pattern with respect to the order of the conformally covariant operators, suggesting the possibility of a unified theory of the compactness of the constant $Q$-curvature problems of all admissible even integer orders.
\end{abstract}

\date{\today}
\subjclass[2020]{Primary: 53C18, Secondary: 35B44, 35J30, 35R01}
\keywords{Fourth-order and sixth-order constant $Q$-curvature problems, compactness, linearized equation, Weyl vanishing theorem, Pohozaev quadratic form}
\maketitle

\tableofcontents

\section{Introduction}
Let $(M,g)$ be a smooth closed Riemannian manifold. Also, let $P_g^{(2\msfk)}$ and $Q_g^{(2\msfk)}$ be the GJMS operator and the associated $Q$-curvature on $(M,g)$ of order $2\msfk$, respectively.
The main purpose of this paper is presenting a complete resolution to the question of compactness and non-compactness for the entire solution sets of the fourth-order ($\msfk=2$) and sixth-order ($\msfk=3$) constant $Q$-curvature problems on $(M,g)$ provided the positivity of the Green's function $G_g^{(2\msfk)}$ of the operator $(-1)^{\msfk}P_g^{(2\msfk)}$.

\subsection{Compactness of the constant $Q^{(4)}$-curvature problem}
Let $(M,g)$ be a smooth closed Riemannian manifold of dimension $n \ge 3$.
Also, let $\ricci_g$ be the Ricci curvature tensor on $(M,g)$, $R_g$ the scalar curvature, $A_g = \frac{1}{n-2}(\ricci_g - \frac{1}{2(n-1)}R_gg)$ the Schouten tensor,
$\sigma_k(A_g)$ the $k$-th symmetric function of the eigenvalues of $A_g$, and $\Delta_g = \text{div}_g \nabla_g$ the Laplace-Beltrami operator with nonpositive eigenvalues.

Then, the Branson's $Q$-curvature $Q_g^{(4)}$ and the Paneitz operator $P_g^{(4)}$ are defined as
\[Q_g^{(4)} = -\Delta_g\sigma_1(A_g) + 4\sigma_2(A_g) + \frac{n-4}{2} \sigma_1(A_g)^2\]
and
\[P_g^{(4)}u = \Delta_g^2u + \text{div}_g\left\{(4A_g - (n-2)\sigma_1(A_g)g)(\nabla_g u,\cdot)\right\} + \frac{n-4}{2} Q_g^{(4)}u \quad \text{for } u \in C^4(M),\]
respectively, where the superscripts $4$ on $Q_g^{(4)}$ and $P_g^{(4)}$ stand for their order.

The constant $Q^{(4)}$-curvature problem refers to the fourth-order elliptic equation
\begin{equation}\label{eq:main4-4d}
P_g^{(4)}u + 2Q_g^{(4)} = 2\lambda e^{4u}, \ u > 0 \quad \text{on } M
\end{equation}
for $n = 4$, and
\begin{equation}\label{eq:main40}
P_g^{(4)}u = \lambda u^{\frac{n+4}{n-4}}, \ u > 0 \quad \text{on } M
\end{equation}
for $n = 3$ or $n \ge 5$, where $\lambda \in \R$ is a constant whose sign is determined by the conformal structure of $(M,g)$.

If $n = 4$, the existence theory of \eqref{eq:main4-4d} was developed by Chang and Yang \cite{CY}, Djadli and Malchiodi \cite{DM}, and Li, Li, and Liu \cite{LLL}, among others.
For the existence result for \eqref{eq:main40} in $n = 3$, refer to Hang and Yang \cite{HY3} and references therein.

Since our main focus in this paper is on the case $n \ge 5$, we provide a more detailed description on the development of the existence theory under this setting:
For $n \ge 5$, the first existence result on \eqref{eq:main40} with $\lambda > 0$ was achieved by Qing and Raske \cite{QR}.
By appealing to the work of Schoen and Yau \cite{SY} to lift the metric to a domain in the sphere via the developing map,
they proved that \eqref{eq:main40} admits a solution if $(M,g)$ is locally conformally flat, its Yamabe invariant $Y_g$ is positive, and its Poincar\'e exponent is less than $\frac{n-4}{2}$.
Later, Gursky and Malchiodi \cite{GM} established the existence result under the conditions that $R_g \ge 0$ and $Q_g^{(4)} \ge 0$ on $M$, and $Q_g^{(4)} > 0$ somewhere, which in particular imply $\lambda > 0$.
They achieved this through the study of the maximum principle for $P_g^{(4)}$ and a sequential convergence of a non-local flow as $t \to \infty$.
The condition $R_g \ge 0$ was relaxed by Hang and Yang \cite{HY}, who substituted it with the requirement that $Y_g > 0$.
They deduced this by solving a maximization problem associated with a nonlinear integral equation equivalent to \eqref{eq:main40}.
Refer to also \cite{GHL}, where Gursky, Hang, and Lin proved that the existence of a conformal metric $g$ with $R_g,\, Q_g^{(4)} > 0$ on $M$ is equivalent to the positivity of both $Y_g$ and $P_g^{(4)}$ provided $n \ge 6$.
Recently, several results concerning the multiplicity of solutions to equation \eqref{eq:main40} have appeared under specific settings; see \cite{BPS, APR, JP}.
This leads us to investigate the properties of the full solution set.

As the first main result of this paper, we examine the $C^4(M)$-compactness of the entire solution set of \eqref{eq:main40} with $\lambda > 0$.
Establishing this compactness property is important in its own right and also enables us to calculate the Leray-Schauder degree for $C^4(M)$-bounded subsets of the solution set or to develop Morse theory, as accomplished in the Yamabe case \cite{KMS}.
In \cite{HRW}, Hebey, Robert and Wen deduced a compactness result for equations similar to but not close to \eqref{eq:main40}.
The first compactness results on \eqref{eq:main40} were independently obtained by Qing and Raske \cite{QR}, and by Hebey and Robert \cite{HRo}, each under different settings that included the locally conformally flat condition.
Later, assuming $5 \le n \le 9$, $\text{Ker}P_g^{(4)} = \{0\}$, the positivity of the Green's function $G_g^{(4)}$ of $P_g^{(4)}$, and the validity of the positive mass theorem for $P_g^{(4)}$,
Y.Y. Li and Xiong \cite{LX} proved that the solution set is $C^4(M)$-compact; we also refer to the work of G. Li \cite{Li} that studied when $5 \le n \le 7$.
On the other hand, Wei and Zhao \cite{WZ} constructed a smooth Riemannian metric $g$ on the unit $n$-sphere $\S^n$ with $n \ge 25$, which is not conformally diffeomorphic to the standard metric on $\S^n$,
such that the solution set of \eqref{eq:main40} is $L^{\infty}(\S^n)$-unbounded (and so $C^4(\S^n)$-noncompact).
The case $10 \le n \le 24$ has remained an open problem until now.
Here, we establish the compactness result for these dimensions, thereby providing the conclusive answer to the question of the $C^4(M)$-compactness of the full solution set of the $Q^{(4)}$-curvature problem \eqref{eq:main40} on $(M,g)$ with $\lambda > 0$ and $Y_g > 0$.

A simple proof of the $C^4$-compactness of the solution set of \eqref{eq:main40} with $\lambda < 0$ was given by Y.Y. Li and Xiong \cite[Theorem 1.3]{LX}.

In what follows, we write $\lambda = \mfc_4(n) := (n-4)(n-2)n(n+2) > 0$ so that \eqref{eq:main40} is rewritten as
\begin{equation}\label{eq:main4}
P_g^{(4)}u = \mfc_4(n)u^{\frac{n+4}{n-4}}, \ u > 0 \quad \text{on } M.
\end{equation}
Our theorem for the constant $Q^{(4)}$-curvature problem \eqref{eq:main4} is stated as follows.
\begin{thm}\label{thm:main4}
Let $(M,g)$ be a smooth closed Riemannian manifold of dimension $5\le n\le 24$ and not conformally diffeomorphic to the standard unit $n$-sphere $\S^n$. Assume either
\begin{itemize}
\item[(i)] $\textup{Ker}P_g^{(4)} = \{0\}$, the positivity of the Green's function $G_g^{(4)}$ of $P_g^{(4)}$, and the validity of the positive mass theorem for $P_g^{(4)}$; or
\item[(ii)] $Q_g^{(4)} \ge 0$ on $M$, $Q_g^{(4)} > 0$ somewhere, and $Y_g > 0$.
\end{itemize}
Then there exists a constant $C > 0$ depending only on $n$ and $(M,g)$ such that any solution $u \in C^4(M)$ of equation (\ref{eq:main4}) satisfies
$$
\norm{C^4(M)}{u}+\norm{C^4(M)}{u^{-1}}\le C.
$$
\end{thm}

As we will discuss in Subsection \ref{subsec:org}, the proof of Theorem \ref{thm:main4} requires a new analytic and geometric perspective for constructing a precise approximation of the blowing-up solutions, particularly in solving the linearized equation \eqref{linearized eqn k,s} associated to \eqref{eq:main4} explicitly.

\begin{rmk}\label{rmk:main4}
Four remarks concerning Theorem \ref{thm:main4} are in order.

\noindent 1. \textbf{Threshold dimension in relation to the Yamabe problem.} Remarkably, the threshold dimension $24$ for the compactness of the $Q^{(4)}$-curvature problem coincides with that of the Yamabe problem.
We remind that the $C^2(M)$-compactness of the Yamabe problem for $n \le 24$ was established through the successive works of Li and Zhu \cite{LZhu}, Druet \cite{Dr}, Li and Zhang \cite{LZha, LZha2}, Marques \cite{Ma}, and Khuri, Marques, and Schoen \cite{KMS}.
Furthermore, Brendle \cite{Br2} and Brendle and Marques \cite{BM} exhibited a smooth metric $g$ on $\S^n$ for which the Yamabe problem on $(\S^n,g)$ admits an $L^{\infty}(\S^n)$-unbounded sequence of solutions for $n \ge 25$.
Recently, Gong and Li \cite{GL} constructed another smooth metric $g'$ on $\S^n$ yielding an $L^{\frac{2n}{n-2}}(\S^n)$-unbounded sequence of solutions for $n \ge 25$.
Besides, Premoselli and V\'etois \cite{PV} employed techniques similar to \cite{KMS} to study sign-changing solutions of the Yamabe equation for $3 \le n \le 10$, establishing a dichotomy with the case $n \ge 11$.

\noindent 2. \textbf{Green's function.} The geometric assumptions in the aforementioned existence results \cite{QR, GM, HY} imply $\text{Ker}P_g^{(4)} = \{0\}$ and the positivity of $G_g^{(4)}$.
By \cite{HY2}, if $Y_g > 0$, then the existence of positive $Q^{(4)}$-curvature in the conformal class $[g]$ is equivalent to $\text{Ker}P_g^{(4)} = \{0\}$ together with the positivity of $G_g^{(4)}$.

\noindent 3. \textbf{Positive mass theorem.} We say that the positive mass theorem for $P_g^{(4)}$ holds if the constant-order term $A$ in the expansion \eqref{expansion of Green's function} of $G_g^{(4)}$ is positive.
In Proposition \ref{prop:Green}, we will derive this positivity from condition (ii) (first introduced by Hang and Yang \cite{HY}) by applying the positive mass theorem of Avalos, Laurain, and Lira \cite{ALL}.
Subsection \ref{subsec:pmt} contains the precise statement of their theorem and the definition of the fourth-order mass.
The condition $Y_g>0$ is needed solely to ensure the validity of the positive mass theorem.
An interesting question is whether this condition can be removed or replaced with alternative conditions, such as spin conditions.

\noindent 4. \textbf{Subcritical exponents.} Let $\ep > 0$ be an arbitrarily chosen, sufficiently small number. A slight modification of the proof of Theorem \ref{thm:main4} shows that the set $\{u \in C^4(M) \mid P_g^{(4)}u = \mfc_4(n)u^p,\, u > 0 \text{ on } M,\, 1+\ep \le p \le \frac{n+4}{n-4}\}$ remains $C^4(M)$-compact under the assumptions of Theorem \ref{thm:main4}. See \cite[Theorem 1.1]{LX}. \hfill $\diamond$
\end{rmk}

\subsection{Compactness and non-compactness of constant $Q^{(6)}$-curvature problem}
Let $\msfk = 1,2,\ldots,\frac{n}{2}$ when $n$ is even, and let $\msfk \in \N$ when $n$ is odd.
In their seminal work \cite{GJMS}, Graham, Jenne, Mason, and Sparling built a family of differential operators of orders $2\msfk$ on a smooth closed Riemannian manifold $M$, satisfying the conformal covariance property
\begin{equation}\label{eq:ccp}
P_{u^{\frac{4}{n-2\msfk}}g}^{(2\msfk)}(\cdot) = u^{-\frac{n+2\msfk}{n-2\msfk}}P_g^{(2\msfk)}(u\cdot) \quad \text{for } 0 < u \in C^{\infty}(M),
\end{equation}
by means of the Fefferman-Graham ambient metric \cite{FG}.
These operators, now known as the GJMS operators, represent a natural generalization of the conformal Laplacian of order $2$ and the Paneitz operator of order $4$.
In contrast, Gover and Hirachi \cite{GH} later showed that if $n \ge 4$ is even and $\msfk > \frac{n}{2}$, no conformally covariant differential operator of leading part $\Delta^{\msfk}$ exists.

Given a GJMS operator $P_g^{(2\msfk)}$ and its associate $Q^{(2\msfk)}$-curvature $Q_g^{(2\msfk)} := (-1)^{\msfk} \frac{2}{n-2\msfk} P_g^{(2\msfk)}(1)$, one can formulate the constant $Q^{(2\msfk)}$-curvature problem, extending the Yamabe problem and the constant $Q^{(4)}$-curvature problem.
By employing the explicit formula for the GJMS operator presented by W\"unsch \cite{Wu} for $\msfk = 3$ and Juhl \cite{J} for all $\msfk$,
Chen and Hou \cite{CH} studied the constant $Q^{(6)}$-curvature problem on Einstein manifolds, and Mazumdar and V\'etois \cite{MV} examined the constant $Q^{(2\msfk)}$-curvature problem on general Riemannian manifolds for all $\msfk \in \N \cap (0,\frac{n}{2})$.
We also refer to a recent paper of Robert \cite{Ro}, where a bubbling analysis for subcritical perturbations of the constant $Q^{(2\msfk)}$-curvature problem was carried out.

In light of Remark \ref{rmk:main4}.1, we may pose the following natural questions:
\begin{itemize}
\item[-] What is the threshold dimension for the compactness of the constant $Q^{(6)}$-curvature problem or its higher-order analogues? Is it again 24?
\item[-] To answer the above question, can our strategy, including the observation regarding the explicit solvability of the associated linearized problems in the fourth-order case, be extended to the higher-order setting?
\end{itemize}
As the second and third main results of this paper, we investigate the $C^6(M)$-compactness and non-compactness of the entire solution set of the constant $Q^{(6)}$-curvature problem
\begin{equation}\label{eq:main6}
-P_g^{(6)}u = \mfc_6(n) u^{\frac{n+6}{n-6}}, \ u > 0 \quad \text{on } M,
\end{equation}
where $P_g^{(6)}$ is the sixth-order GJMS operator whose explicit form is given in \eqref{eq:P6} and $\mfc_6(n) := (n-6)(n-4)(n-2)n(n+2)(n+4)$.
\begin{thm}\label{thm:main6}
Let $(M,g)$ be a smooth closed Riemannian manifold of dimension $7\le n\le 26$ and not conformally diffeomorphic to the standard unit $n$-sphere $\S^n$.
Assume that $\textup{Ker}P_g^{(6)} = \{0\}$, the positivity of the Green's function $G_g^{(6)}$ of $-P_g^{(6)}$, and the validity of the positive mass theorem for $-P_g^{(6)}$.
Then there exists a constant $C > 0$ depending only on $n$ and $(M,g)$ such that any solution $u \in C^6(M)$ of equation (\ref{eq:main6}) satisfies
$$
\norm{C^6(M)}{u}+\norm{C^6(M)}{u^{-1}}\le C.
$$
\end{thm}
\begin{thm}\label{thm:main6n}
Assume that $n \ge 27$. There exist a smooth Riemannian metric $g$ on $M = \S^n$ and a sequence of smooth solutions $\{u_a\}_{a \in \N}$ of \eqref{eq:main6}
such that $g$ is not conformally flat, $\textup{Ker}P_g^{(6)} = \{0\}$, the Green's function $G_g^{(6)}$ of $-P_g^{(6)}$ is positive, and $\|u_a\|_{L^{\infty}(\S^n)} \to \infty$ as $a \to \infty$.
\end{thm}

\begin{rmk}\label{rmk:main6}
Seven remarks concerning Theorems \ref{thm:main6} and \ref{thm:main6n} are in order.

\noindent 1. \textbf{Threshold dimension.} Our strategy to analyze the fourth-order case still works well for the sixth-order case, yielding that the threshold dimension for the sixth-order problem is $26$, not $24$,
which is indeed surprising. Particularly, the answers to the above questions are, respectively, ``no" and ``yes".
Refer also to Conjecture \ref{conj:dim} below.

\noindent 2. \textbf{Correction term.} We will observe a simple yet elegant pattern for the correction terms, that is, the solutions of the linearized equations of the Yamabe, $Q^{(4)}$-, and $Q^{(6)}$-curvature problems.
The coefficients of the correction terms satisfy the same equations in \eqref{eq:Gammanew} across all three problems; see Remarks \ref{rmk:lin sol rel} and \ref{rmk:lin sol rel Q6}.
This fact gives a useful recursive pattern to study the first variation of the constant higher-order $Q^{(2\msfk)}$-curvature problems, and suggests the possibility of a unified theory of the compactness of the $Q^{(2\msfk)}$-curvature problems of all orders $2\msfk$.

\noindent 3. \textbf{Total $Q^{(2\msfk)}$-curvature.} The compactness of the constant $Q^{(2\msfk)}$-curvature problem is closely related the second variation of total $Q^{(2\msfk)}$-curvature, which has been extensively studied in the literature; see, for example, the work of Matsumoto \cite{Mats}.
The relation between these two problems deserves further explorations.

\noindent 4. \textbf{Composition structure.} We also notice that the second variation of the energy functional has a composition structure regarding how many numbers of the linear operator $\mcl_k$ in \eqref{eq:mclk} are composed. See Remarks \ref{rmk:composed decomp} and \ref{rmk:k+m relation}.
Previously, a related composition structure is known for Einstein metrics (see e.g. \cite{FG, Go}) and some Einstein products metrics (see \cite{CM}).

\noindent 5. \textbf{Bach tensor.} The Bach tensor is known as the ambient obstruction tensor of Fefferman and Graham \cite{FG} for $n=4$ and a constant multiple of the first extended obstruction tensor of Graham \cite{Gr} for $n\ne 4$.
One of the main difficulties in proving Theorem \ref{thm:main6} is that we need to study the second-order expansion of the Bach tensor, which involves fourth-order derivatives of the metric, whereas we only expand the Ricci curvature up to second-order to prove Theorem \ref{thm:main4}.
It is new in the literature on the sixth-order GJMS operator, while the first-order expansion of the Bach tensor has been studied before; see e.g. \cite{CH, LY}.

\noindent 6. \textbf{Quadratic form.} We point out that the quadratic forms relevant to both the compactness and non-compactness problems are identical. See Corollary \ref{delta direction}. This highlights a natural connection between these two problems.

\noindent 7. \textbf{Linear polynomial.} As a byproduct of Theorem \ref{thm:main6n}, we find that choosing a linear polynomial to construct a background metric, as in \cite{Br2, WZ},
allows us to produce a blowup sequence solving equation \eqref{eq:main6} when $n\geq 52$. The dimension restriction coincides with that in the Yamabe and $Q^{(4)}$-curvature problems \cite{Br2, WZ}. \hfill $\diamond$
\end{rmk}

To prove Theorem \ref{thm:main6n}, we will first construct a metric $\tig$ on $\R^n$ using a certain polynomial of degree $2d_0+2$ as in \cite{Br2, BM, WZ} (see \eqref{eq:Hij}), and then obtain a metric $g$ by lifting $\tig$ to $\S^n$ via the stereographic projection.
In this construction, we are required to integrate functions of the form $\frac{|x|^{2(2d_0+2)}}{(1+|x|)^{2((n-2\msfk)+\msfk)}}$ over $\R^n$, which is finite if and only if $n > 2(2d_0+2)+2\msfk$.
Brendle and Marques \cite{BM} chose $d_0 = 3$ to identify the optimal blow-up dimension for the Yamabe problem ($\msfk = 1$), while Wei and Zhao \cite{WZ} had to select $d_0 = 4$ for their corresponding result on the $Q^{(4)}$-curvature problem ($\msfk = 2$).
From this viewpoint, it is natural to choose $d_0 \ge 4$ for higher-order $Q^{(2\msfk)}$-curvature problems. Indeed, we will use $d_0 = 4$ to deduce Theorem \ref{thm:main6n}.
Note that if $d_0 = 4$, the integral is finite if and only if $n > 20+2\msfk$. Thus, the result in \cite{WZ} and ours yield the best possible lower bound on the dimension ($n > 24$ or $26$, respectively) that can be obtained through this approach.

The above observation leads us to raise the following conjecture:
\begin{conj}\label{conj:dim}
Let $(M^n,g)$ be a smooth closed Riemannian manifold not conformally diffeomorphic to $\S^n$. Assume that $4 \le \msfk < \frac{n}{2}$.
\begin{itemize}
\item[(i)] If $1+2\msfk \le n \le 20+2\msfk$, then the full solution set of the constant $Q^{(2\msfk)}$-curvature problem is $C^{2\msfk}(M)$-compact, provided that $\textup{Ker}P_g^{(2\msfk)} = \{0\}$,
    the positivity of the Green's function $G_g^{(2\msfk)}$ of the operator $(-1)^{\msfk}P_g^{(2\msfk)}$, and the validity of the positive mass theorem for $(-1)^{\msfk}P_g^{(2\msfk)}$.
\item[(ii)] If $n > 20+2\msfk$, then there exist a smooth Riemannian metric $g$ on $M = \S^n$ and a sequence of smooth solutions $\{u_a\}_{a \in \N}$ of the constant $Q^{(2\msfk)}$-curvature problem
    such that $g$ is not conformally flat, $\textup{Ker}P_g^{(2\msfk)} = \{0\}$, the Green's function $G_g^{(2\msfk)}$ of $-P_g^{(2\msfk)}$ is positive, and $\|u_a\|_{L^{\infty}(\S^n)} \to \infty$ as $a \to \infty$.
\end{itemize}
\end{conj}
\noindent It is an arduous task to verify Conjecture \ref{conj:dim} using the explicit formula for GJMS operators for $\msfk$ whose values greater than $3$, due to their drastic increase in complexity.
In the explicit formula, the GJMS operator of order $\msfk$ is written as a sum of $2^{\msfk-1}$ terms, each being is a composition of second-order differential operators.
However, it would be interesting to see whether the recursive formula for GJMS operators, as given in \cite{J, FG2}, can be applied to this problem. We hope to address it in a future work.

\medskip
We conclude this subsection with a brief discussion of the compactness issue for certain variants of the classical Yamabe problem, which are of independent interest:

Escobar \cite{Es} introduced a version of the boundary Yamabe problem on a smooth compact manifold $(\ox^n,g)$ with boundary $M^{n-1}$, generalizing of the Riemann mapping theorem.
In \cite{Al}, Almaraz constructed an $L^{\infty}(\ox)$-unbounded example for $n \ge 25$.
This finding suggests that the threshold dimension for the $C^2(\ox)$-compactness of the boundary Yamabe problem is likely $24$, as for the Yamabe and constant $Q^{(4)}$-curvature problems.
Compactness results have so far been fully established only up to dimension $5$ \cite{AdQW, KMW}, primarily due to serious challenges in the linear theory.
Hence, a natural question is whether our approach can contribute to tackling the problem.
We note that an immediate technical issue arises to achieve the compactness result, because it is necessary to establish vanishing theorems for not only the Weyl tensor but also the second fundamental form.

The boundary Yamabe problem on $(\ox^{n+1},g)$ can be viewed as a fractional Yamabe problem involving the half-Laplacian $(-\Delta)^{1/2}$ on $M^n$.
Since the fractional Yamabe problem is defined on $M$, we redefine $n$ to represent the boundary dimension. In particular, Almaraz's non-compactness example in \cite{Al} now holds for $n \ge 24$.
In view of the non-compactness example of Brendle and Marques \cite{BM} for the Yamabe problem existing for $n \ge 25$,
the critical dimension for the fractional Yamabe problems involving $(-\Delta)^{\ga}$ must shift at some $\ga^* \in (\frac{1}{2},1)$. In \cite{KMW2}, Kim, Musso, and Wei identified $\ga^* \simeq 0.940197$.
The investigation of $C^2(M)$-compactness for the fractional Yamabe problem is still in its infancy, with only the works \cite{QR2, KMW3} currently addressing this topic.

Finally, as fully nonlinear analogues, the $\sigma_k$-Yamabe problems were proposed by Viaclovsky \cite{Vi} for $k\in\{1,\dots,n\}$, and the classical Yamabe problem is the case $k=1$.
One standard route to existence is to first establish the compactness of the solution set and then conclude by a degree/continuity argument.
Along this route, Chang, Gursky, and Yang \cite{CGY} treated the case $k=2$ in dimension $4$;
A. Li and Y.Y. Li \cite{LL} proved existence and compactness on locally conformally flat manifolds;
Gursky and Viaclovsky \cite{GV} obtained existence and compactness in the regime $k>\frac{n}{2}$;
Trudinger and Wang \cite{TW} studied the intermediate case $k=\frac{n}{2}$;
and Li and Nguyen \cite{LN2} examined when a lower Ricci curvature bound holds (covering $k<\frac{n}{2}$ as well).
Note that these citations are not exhaustive; see the papers above for further developments.
Despite substantial advances, the compactness and non-compactness of the solution set for $2\le k<\frac{n}{2}$ remain open in general.

\subsection{Weyl vanishing theorem}
A crucial step to proving Theorems \ref{thm:main4} and \ref{thm:main6} is the establishment of the following Weyl vanishing theorem for the constant $Q^{(4)}$-curvature problem \eqref{eq:main4} and the constant $Q^{(6)}$-curvature problem \eqref{eq:main6}.
\begin{thm}\label{thm:Weyl}
Let $g$ be a smooth Riemannian metric in the unit $n$-ball $B_1$. Assume either
\begin{itemize}
\item[(i)] $8\le n\le 24$ and $\{u_a\}_{a \in \N} \subset C^4(B_1)$ is a sequence of solutions of (\ref{eq:main4}) with $M = B_1$; or
\item[(ii)] $10\le n\le 26$ and $\{u_a\}_{a \in \N} \subset C^6(B_1)$ is a sequence of solutions of (\ref{eq:main6}) with $M = B_1$.
\end{itemize}
Suppose that for each $\ep>0$, there is a constant $C=C(\ep)>0$ such that $\sup_{B_1\setminus B_{\ep}}u_a\le C(\ep)$ and $\lim_{a\to\infty} \sup_{B_1}u_a=\infty$. Then the Weyl tensor $W_g$ satisfies
$$
|W_g|(x)\le C|x|^l \quad \text{for } x \in B_1
$$
for some $l \in \N$ satisfying $l>\frac{n-8}{2}$ when (i) holds, and $l>\frac{n-10}{2}$ when (ii) holds.
\end{thm}
\begin{rmk}
Thanks to the previous result, the $C^4(M)$-compactness theorem for \eqref{eq:main4} holds for $8\le n\le 24$ without demanding the positive mass theorem for $P_g^{(4)}$ provided
\[\sum_{m=0}^{\lfloor \frac{n-8}{2} \rfloor} \left|\nabla_g^m W_g\right| > 0 \quad \text{on } M,\]
where $\lfloor \frac{n-8}{2} \rfloor$ is the greatest integer that does not exceed $\frac{n-8}{2}$.
The analogous $C^6(M)$-compactness theorem also holds for \eqref{eq:main6}. \hfill $\diamond$
\end{rmk}

\subsection{Organizations of the paper and novelties}\label{subsec:org}
To prove the compactness results, Theorems \ref{thm:main4} and \ref{thm:main6}, we adapt Schoen's strategy for the compactness of the Yamabe problem depicted in \cite{Sc, Sc2} and further developed in \cite{LZhu, Dr, LZha, LZha2, Ma, KMS}.

In Section \ref{sec:pre}, we present some relevant background, including the bubbles, the local Pohozaev(-Pucci-Serrin) identities, and the positive mass theorem for the Paneitz operator and the GJMS operators.

\medskip
Sections \ref{sec:curv}--\ref{sec:tech} are devoted to $C^4$-compactness theorem of the constant $Q^{(4)}$-curvature problem \eqref{eq:main4} for $5 \le n \le 24$ (the proof of Theorem \ref{thm:main4}).
Throughout the proof, we assume $n \ge 8$ unless otherwise stated. The proof of Theorem \ref{thm:main4} for $5 \le n \le 7$ does not require the Weyl vanishing theorem as can be seen in \cite{Li, LX}:

In Section \ref{sec:curv}, we derive the first- and second-order expansions of the Ricci curvature tensor, the scalar curvature, the $Q^{(4)}$-curvature, and the Paneitz operator as in e.g. \cite{AM, BMal, Br}.
For further expansions of quantities such as the Schouten tensor, the Cotton tensor, the Bach tensor, and so forth, we refer the reader to Subsection \ref{subsec:curv6}.

In Section \ref{sec:corr}, we introduce inhomogeneous linear problems \eqref{linearized eqn k,s} and explicitly solve them in rational form,
which are indispensable for constructing sharp approximations of blowing-up sequences of solutions in a neighborhood of each blowup point.
We emphasize that this step constitutes one of our main contributions, as such problems generally do not admit explicit rational solutions unless the inhomogeneous term exhibits special properties (see Remark \ref{rmk:lin}).
Notably, the algebraic structure of the Paneitz operator ensures that the inhomogeneous term is indeed special, facilitating the existence of these explicit rational solutions.

In Section \ref{sec:blowup}, we build sharp approximations of blowing-up sequences and estimate their pointwise error.

In Section \ref{sec:Weyl}, we establish the Weyl vanishing theorem (Theorem \ref{thm:Weyl2}, see also Theorem \ref{thm:Weyl}) by defining and analyzing the Pohozaev quadratic form.
The analysis requires a full understanding of the eigenspaces of the operator $\mcl_k$ defined in \eqref{eq:mclk},
which highlights an additional distinction between the Yamabe and $Q^{(4)}$-curvature problems. Refer to the paragraphs following Corollary \ref{I_1 coro}.

In Sections \ref{sec:lsr} and \ref{sec:comp}, we prove the non-negativity of a local Pohozaev term and the positive mass theorem for the Paneitz operator under Condition (ii) of Theorem \ref{thm:main4}, and then complete the proof of Theorems \ref{thm:main4} and \ref{thm:Weyl} under the assumption (i).

In Section \ref{sec:tech}, we verify the positivity of the Pohozaev quadratic form in three mutually exclusive cases, which is crucial in the proof of the Weyl vanishing theorem (Theorem \ref{thm:Weyl2}, or more precisely, Proposition \ref{Positive definiteness}).
We note that verifying the positivity of the matrices $(m^{D,s}_{qq'})$, $(m^{W,s}_{qq'})$, and $(m^{H,s}_{qq'})$ defined in Lemmas \ref{lemma of case 1}--\ref{lemma of case 3} is too tedious to perform manually.
Therefore, we carried out these calculations with the assistance of the computer software Mathematica.

\medskip
In Section \ref{sec:comp6}, we establish Theorem \ref{thm:main6}, namely, the $C^6$-compactness theorem of the constant $Q^{(6)}$-curvature problem \eqref{eq:main6} for $7 \le n \le 26$.
The overall scheme of the proof is the same as that of the proof of Theorem \ref{thm:main4}. However, as discussed in Remark \ref{rmk:main6}, we must overcome new technical difficulties and develop additional insights, even in comparison to the fourth-order case.

In Section \ref{sec:noncomp6}, we build an example of $L^{\infty}$-unbounded solutions to \eqref{eq:main6} for $n \ge 27$, proving Theorem \ref{thm:main6n}.
Although our approach is essentially the same as \cite{Br2,BM}, we make extensive use of the computations obtained for the compactness result, thereby highlighting the close relationship between the problems of compactness and non-compactness.

\medskip
For the reader's convenience, we included the Mathematica code for Sections \ref{sec:tech}, \ref{sec:comp6}, and \ref{sec:noncomp6} as the ancillary files with our arXiv submission.

\medskip
We also provide five appendix sections:

In Appendix \ref{sec:useful}, we provide elementary yet useful tools for the proof of main theorems.

In Appendix \ref{sec:cnc}, we provide several curvature identities in conformal normal coordinates. One of the features of this paper is the heavy usage of conformal normal coordinates.

In Appendix \ref{sec:sphere Q6}, we collect integrals of geometric quantities induced by the $Q^{(6)}$-curvature over the unit sphere.

In Appendix \ref{sec:eigenspaces}, we provide properties of the eigenvectors of the operator $\mcl_k$ given in \eqref{eq:mclk}.

Finally, in Appendix \ref{sec:geoexp}, we elaborate on a geometric interpretation of the explicit solvability phenomenon for inhomogeneous linear problems via the image of the conformal Killing operator. Refer also to Remarks \ref{rmk:lin sol rel} and \ref{rmk:lin sol rel Q6}.

\medskip
In proving Theorems \ref{thm:main4} and \ref{thm:main6}, we will omit several proofs concerning the blowup analysis or quantitative estimates, because the necessary arguments, with some modifications, can be derived from \cite{KMS, JLX, LX}.
We remark that the authors of \cite{LX} laid the analytic foundations for studying the compactness of the constant $Q^{(4)}$-curvature problem, resolving the technical difficulties stemming from the fourth-order structure of equation \eqref{eq:main4}.
Their work was motivated by their earlier collaboration with Jin \cite{JLX} on the fractional Nirenberg problem of order $\msfk \in (0,\frac{n}{2})$.
Together, these works essentially address the technical challenges associated with the sixth-order structure \eqref{eq:main6} as well.

\subsection{Notations}
We list the notations used in the introduction and the rest of the paper.

\medskip \noindent For all the theorems,
\begin{itemize}
\item[-] $\R^n$ is the $n$-dimensional Euclidean space $\R^n$ and $g_{\R^n}$ is its standard metric.
\item[-] $\S^{n-1}$ is the $(n-1)$-dimensional unit sphere in $\R^n$, $g_{\S^{n-1}}$ is its standard metric, and $|\S^{n-1}|$ is the canonical surface area of $(\S^{n-1},g_{\S^{n-1}})$.
\item[-] Given $\sigma, \sigma_1, \sigma_2 \in (M,g)$ and $R>0$, $B^g_R(\sigma)$ is the open geodesic ball of radius $R$ centered at $\sigma$ and $d_g(\sigma_1,\sigma_2)$ is the geodesic distance between $\sigma_1$ and $\sigma_2$.
    If $(M,g) = (\R^n,g_{\R^n})$, we write $B^n_R(\sigma) = B^g_R(\sigma)$ and $B_R = B^n(0,R)$.
\item[-] For $i,j \in \N \cup \{0\}$ such that $i-2j<-1$, we set
\begin{equation}\label{eq:mci}
\mci^i_j := \int_0^{\infty}\frac{r^i}{(1+r^2)^{j}}dr.
\end{equation}
\item[-] $dv_g$ and $dS_g$ denote the volume form and the surface measure, respectively. Furthermore, $dS$ represents the surface measure on Euclidean space.
\item[-] For $\msfk \in \N$, $\dot{H}^{\msfk}(\R^n)$ denotes the homogeneous Sobolev space of order $\msfk$, which is the closure of $C_c^{\infty}(\R^n)$ under the norm $\|u\|_{\dot{H}^{\msfk}(\R^n)} := \|(-\Delta)^{\frac{\msfk}{2}}u\|_{L^2(\R^n)}$.
\item[-] For $\msfk \in \N$, $H^{\msfk}(\S^n)$ denotes the inhomogeneous Sobolev space of order $\msfk$, which is the closure of $C_c^{\infty}(\R^n)$ under the norm $\|u\|_{H^{\msfk}(\S^n)} := \sum_{m=0}^{\msfk}\|\nabla^mu\|_{L^2(\S^n)}^2$.
\item[-] For a function $f\in C^l(B_1)$ and $m\in \R$, $f(x)=O^{(l)}(|x|^{m})$ denotes that
$$
|\nabla^i f|(x)=O(|x|^{m-i}), \quad i=0,\ldots, l.
$$
\item[-] $\mcp_k$ denotes the space of homogeneous polynomial on $\R^n$ of degree $k$.
\item[-] $\mch_k$ denotes the space of homogeneous harmonic polynomial on $\R^n$ of degree $k$, that is,
$$
\mch_k := \{P\in \mcp_k\mid\Delta P=0\}.
$$
\item[-] $\delta_{ij}$ is the Kronecker delta, that is, $\delta_{ij} = 1$ if $i = j$ and $0$ if $i \ne j$.
\item[-] Repeated indices follow the summation convention. A comma denotes partial differentiation in local coordinates, and a semicolon denotes covariant differentiation. Also, the empty sum is taken to be zero.
\item[-] For a symmetric $2$-tensor $T$, we write $\tr\, T:=T_{ii}$, $\delta_i T:=T_{ij,j}$, and $\delta^2 T:=T_{ij,ij}$. Also, define $\tr_gT:=g^{ij}T_{ij}$, $(\diver_gT)_i:=g^{jl}T_{ij;l}$, and $\diver_g\left\{T(\nabla_g u,\cdot)\right\}:=(g^{ik}g^{jl}T_{ij}u_{,l})_{;k}$.
\item[-] For two symmetric $2$-tensors $S$ and $T$, we write
\begin{align*}
S\cdot T&:=S_{ij}T_{ij} &&\text{(dot product)}, \\
(S\KN T)_{ijkl}&:=S_{il}T_{jk}+S_{jk}T_{il}-S_{ik}T_{jl}-S_{jl}T_{ik} &&\text{(Kulkarni-Nomizu product)}.
\end{align*}
\item[-] Let $m \in \N$. For an $m$-tensor $T$ with components $T_\alpha$, where $\alpha$ is a multi-index satisfying $|\alpha| = m$, we write $T_\alpha T_\alpha = (T_\alpha)^2$.
\item[-] Unless otherwise specified, $C > 0$ denotes a universal constant, which may vary from line to line or even in the same line.
\end{itemize}

\medskip \noindent For Theorem \ref{thm:main4},
\begin{itemize}
\item[-] For a function $u \in C^4(M)$, let $\mce_g(u) = P_g^{(4)}u - \Delta^2_gu$.
\item[-] Constants: $\mfc_4(n) := (n-4)(n-2)n(n+2)$ and $\tmfc_4(n) := (n-2)n(n+2)(n+4)$.
\item[-] Constant: $\vsi_n := 2(n-4)(n-2)|\S^{n-1}|$.
\item[-] Constants: $\alpha_n^1 := \frac{4+(n-2)^2}{2(n-2)(n-1)}$, $\alpha_n^2 :=
    \frac{n-6}{2(n-1)}$, and $\beta_n := \frac{n-4}{4(n-1)}$.
\item[-] Constants: $K := n-6$, $d := \lfloor\frac{n-4}{2}\rfloor$, and $\theta_k=1$ if $k=\frac{n-4}{2}$ while $\theta_k=0$ otherwise.
\item[-] Functions: $w(r)=(1+r^2)^{-\frac{n-4}{2}}$ and $Z(r)=rw'(r)+\frac{n-4}{2}w(r)$ for $r \in [0,\infty)$.
\end{itemize}

\medskip \noindent For Theorems \ref{thm:main6} and \ref{thm:main6n},
\begin{itemize}
\item[-] For a function $u \in C^6(M)$, let $\mce_g(u) = P_g^{(6)}u - \Delta^3_gu$.
\item[-] Constants: $\mfc_6(n) := (n-6)(n-4)(n-2)n(n+2)(n+4)$ and $\tmfc_6(n) := (n-4)(n-2)n(n+2)(n+4)(n+6)$.
\item[-] Constant: $\vsi_n := 8(n-6)(n-4)(n-2)|\S^{n-1}|$.
\item[-] Constants: $\alpha_n^1:=\frac{3 n^4- 33 n^3 + 100 n^2- 68 n + 144}{4(n-4)(n-2)}$, $\alpha_n^2:=\frac{3 n^3 - 40 n^2+ 140 n -48}{2(n-4)(n-2)}$, \\
\hspace*{53pt} $\beta_n:=\frac{3 n^3 - 25 n^2+ 10 n +152}{4(n-4)(n-1)}$, and $\ga_n:=2^{-7}\pi^{-n/2}\Gamma(\frac{n-6}{2})$.
\item[-] Constants: $K := n-8$, $d := \lfloor\frac{n-6}{2}\rfloor$, and $\theta_k=1$ if $k=\frac{n-6}{2}$ while $\theta_k=0$ otherwise.
\item[-] Functions: $w(r)=(1+r^2)^{-\frac{n-6}{2}}$ and $Z(r)=rw'(r)+\frac{n-6}{2}w(r)$ for $r \in [0,\infty)$.
\end{itemize}

\section{Preliminaries}\label{sec:pre}
\subsection{Bubbles and non-denegeracy}\label{subsec:bubble}
Fix any $\msfk \in \N$ and $n > 2\msfk$. Given parameters $\delta > 0$ and $\xi \in \R^n$, we define a bubble by
\begin{equation}\label{eq:bubble}
w_{\dx}(x) = \(\frac{\delta}{\delta^2+|x-\xi|^2}\)^{\frac{n-2\msfk}{2}} \quad \text{for } x \in \R^n.
\end{equation}
According to Li \cite{YYLi} and Chen, Li, and Ou \cite{CLO}, the set $\{w_{\dx}: \delta > 0, \xi \in \R^n\}$ constitutes the set of all solutions to
\[(-\Delta)^{\msfk} u = \mfc_{2\msfk}(n) u^{\frac{n+2\msfk}{n-2\msfk}}, \ u > 0 \quad \text{in } \R^n, \quad u \in L^{\frac{2n}{n-2\msfk}}_{\text{loc}}(\R^n),\]
where $\mfc_{2\msfk}(n) := \Pi_{m=-\msfk}^{\msfk-1}(n+2m)$.

Given $\xi = (\xi^1,\ldots,\xi^n) \in \R^n$, we define
\begin{equation}\label{eq:bubbleZ}
Z^0_{\dx} = -\frac{\pa w_{\dx}}{\pa \delta} \quad \text{and} \quad Z^i_{\dx} = \frac{\pa w_{\dx}}{\pa \xi^i} \quad \text{for } i = 1,\ldots,n.
\end{equation}
Bartsch, Weth, and Willem \cite[Theorem 2.1]{BWW} proved that the solution space of the linear problem
\begin{equation}\label{eq:bubbleZeq}
L(\Psi) := (-\Delta)^{\msfk} \Psi - \tmfc_{2\msfk}(n)w_{1,0}^{\frac{4\msfk}{n-2\msfk}}\Psi = 0 \quad \text{in } \R^n, \quad \Psi \in \dot{H}^{\msfk}(\R^n)
\end{equation}
is spanned by $Z^0_{1,0},\, Z^1_{1,0}, \ldots,\, Z^n_{1,0}$. Here, $\tmfc_{2\msfk}(n):=\frac{n-2\msfk}{n+2\msfk}\mfc_{2\msfk}(n)$.
Lu and Wei \cite[Theorem 2.2]{LW} obtained this non-degeneracy in the case $\msfk = 2$, and Li and Xiong \cite[Lemma 5.1]{LX} extended it to all $\msfk \in (0,\frac{n}{2})$, by replacing the condition $\Psi \in \dot{H}^{\msfk}(\R^n)$ with $\Psi \in L^{\infty}(\R^n)$.

For simplicity, we will write $w = w_{1,0}$ and $Z^i = Z^i_{1,0}$ for $i = 0,\ldots,n$. If there is no ambiguity, we will simply write $Z=Z^0$.

\subsection{Pohozaev-Pucci-Serrin identities}
Let $u:B_R \to \R$ be a positive $C^4$ function satisfying $P_g^{(4)}u = \mfc_4(n)u^{\frac{n+4}{n-4}}$ in $B_R$. For $0<r<R$, we define
\begin{equation}\label{eq:Poho}
\mbp^{(4)}(r,u):=\int_{\pa B_r} \left[-\frac{r}{2}(\Delta u)^2-\(r\pa_ru+\frac{n-4}{2}u\)\pa_r(\Delta u)+\Delta u\pa_r\(r\pa_ru+\frac{n-4}{2}u\)\right] dS,
\end{equation}
where $\pa_r := \frac{x_i}{r}\pa_{x_i}$. Then
\begin{equation}\label{eq:Poho2}
\mbp^{(4)}(r,u)=\int_{B_r}\(r\pa_ru+\frac{n-4}{2}u\) (\mce_g+\Delta^2_g-\Delta^2)(u)\, dx-\frac{\mfc_4(n)(n-4)r}{2n}\int_{\pa B_r}u^{\frac{2n}{n-4}}dS.
\end{equation}

Similarly, let $u:B_R \to \R$ be a positive $C^6$ function satisfying $-P_g^{(6)}u=\mfc_6(n)u^{\frac{n+6}{n-6}}$ in $B_R$. For $0<r<R$, we define
\begin{equation}\label{eq:Poho Q6}
\begin{aligned}
\mbp^{(6)}(r,u):=\int_{\pa B_r}&\left[-\frac{r}{2}|\nabla\Delta u|^2+\(r\pa_ru+\frac{n-6}{2}u\)\pa_r(\Delta^2 u)\right.\\
&\ \left.-\Delta^2 u\pa_r\(r\pa_ru+\frac{n-6}{2}u\)+\pa_r\Delta u\Delta\(r\pa_ru+\frac{n-6}{2}u\)\right]dS.
\end{aligned}
\end{equation}
Then
\begin{equation}\label{eq:Poho2 Q6}
\mbp^{(6)}(r,u)=-\int_{B_r}\(r\pa_ru+\frac{n-6}{2}u\) (\mce_g+\Delta^3_g-\Delta^3)(u)\, dx-\frac{\mfc_6(n)(n-6)r}{2n}\int_{\pa B_r}u^{\frac{2n}{n-6}}dS.
\end{equation}

The proofs of \eqref{eq:Poho2} and \eqref{eq:Poho2 Q6} can be found in \cite[Proposition 13.1]{Ro}, where Robert refers to them as the Pohozaev-Pucci-Serrin identities.

\subsection{Positive mass theorem}\label{subsec:pmt}
Recall that an asymptotically flat manifold $(M^n,\hat{g})$ (with one end and decay rate $\tau>0$) is defined by $M=M_0\cup M_{\infty}$,
where $M_0$ is compact and $M_{\infty}$ is diffeomorphic to $\R^n\setminus \overline{B_1}$ such that $g_{ij}(y)-\delta_{ij}=O^{(4)}(|y|^{-\tau})$ as $|y|\to \infty$. Moreover, if $\tau>\frac{n-4}{2}$ and $Q_g^{(4)} \in L^1(M,\hat{g})$, we can define a fourth-order mass
\begin{equation}\label{eq:hom}
m(\hat{g}):=\lim_{r\to\infty} \int_{\pa B_r}(\hat{g}_{ii,jjk}-\hat{g}_{ij,ijk})(y)\frac{y_k}{|y|}dS.
\end{equation}
In \cite[Theorem A]{ALL}, Avalos, Laurain, and Lira proved the following result.
\begin{thm}\label{PMT2}
Assume that $(M^n,\hat{g})$ is an asymptotically flat manifold with decay rate $\tau>\frac{n-4}{2}$ satisfying
\begin{itemize}
\item[(i)] $n\ge 5$;
\item[(ii)] $Y_{\hat{g}}>0$, $Q_{\hat{g}}\in L^1(M,\hat{g})$, and $Q_{\hat{g}}\geq 0$.
\end{itemize}
Then the fourth-order mass $m(\hat{g})$ is nonnegative. Moreover, $m(\hat{g})=0$ if and only if $(M,\hat{g})$ is isometric to the Euclidean space $\R^n$.
\end{thm}
\noindent Prior to \cite{ALL}, the positive mass theorem for manifolds of dimensions $5 \le n \le 7$ and for locally conformally flat manifolds was proved in \cite{HR, GM, HY2}.

\medskip
In general, according to Michel \cite{Mi}, the $(2\msfk)$th-order mass for $2 \le \msfk \in \N$ can be defined by
\[m_\msfk(\hat{g}) := \int_M (-\Delta_{\hat{g}})^{\msfk-1}R_{\hat{g}} dv_{\hat{g}}.\]
However, to the best of the authors' knowledge, no $(2\msfk)$-order version of Theorem \ref{PMT2} involving the mass $m_\msfk$ has yet been established.

We remark that there are various definitions of higher-order masses and corresponding positive mass theorems: Refer to the work of Ge, Wang, and Wu \cite{GWW} for the Gauss-Bonnet-Chern mass and that of Li and Nguyen \cite{LN} for the mass associated with the $\sigma_k$ Yamabe problem.

\subsection{Properties of blowup sequences}
Throughout the subsection, we assume that $\msfk = 2, 3$ and $n > 2\msfk$.

\medskip
We assume that $\{g_a\}_{a \in \N} \subset [g]$ is a sequence of metrics on a smooth closed Riemannian manifold $M$ converging to a metric $g_{\infty} \in [g]$ in $C^l(M)$ as $a \to \infty$ for any $l \in \N$, and $\{u_a\}_{a \in \N} \subset C^{2\msfk}(M)$ is a sequence of solutions of
\begin{equation}\label{eq:maina}
(-1)^{\msfk}P_{g_a}^{(2\msfk)}u_a = \mfc_{2\msfk}(n)u_a^{\frac{n+2\msfk}{n-2\msfk}}, \ u_a > 0 \quad \text{on } M.
\end{equation}

\begin{defn}[Blowup points]
\

\noindent 1. A point $\bsi \in M$ is called a blowup point for $\{u_a\}$ if $u_a(\sigma_a) \to \infty$ as $a \to \infty$ for some sequence $\{\sigma_a\} \subset M$ such that $\sigma_a \to \bsi$.

\noindent 2. A blowup point $\bsi \in M$ is called isolated if there exist constants $R,\, C>0$ such that
$$
u_a(\sigma) \le Cd_{g_a}(\sigma,\sigma_a)^{-\frac{n-2\msfk}{2}} \quad \text{for all } y \in B^{g_a}_R(\sigma_a).
$$

\noindent 3. An isolated blowup point $\bsi \in M$ is called simple if the map $r \mapsto r^{\frac{n-2\msfk}{2}}\bar{u}_a(r)$ has only one critical point in the interval $(0,R)$, where
\[\bar{u}_a(r) := |\pa B_r^{g_a}(\sigma_a)|^{-1}_{g_a}\int_{\pa B_r^{g_a}(\sigma_a)} u_a dS_{g_a}. \tag*{$\diamond$}\]
\end{defn}
\noindent To specify the points $\{\sigma_a\}$, we will use the expression that $\sigma_a \to \bsi \in M$ is a blowup point for $\{u_a\}$.

\medskip
Let $G_g^{(2\msfk)}$ be the Green's function of the operator $(-1)^{\msfk}P_g^{(2\msfk)}$.
A standard parametrix construction of the Green's function shows that
\begin{equation}\label{eq:Green}
|G_{g_a}^{(2\msfk)}(\sigma,\tau)| \le Cd_{g_a}(\sigma,\tau)^{2\msfk-n} \quad \text{for } \sigma, \tau \in M.
\end{equation}
Fix any $\sigma_0 \in M$ and $\delta > 0$ small. By applying the Green's representation formula for $u_a$, one can rewrite \eqref{eq:maina} as
\[u_a(\sigma) = \mfc_{2\msfk}(n) \int_{B^{g_a}_{\delta}(\sigma_0)} G_{g_a}^{(2\msfk)}(\sigma,\tau) u_a^{\frac{n+2\msfk}{n-2\msfk}}(\tau) dv_{g_a} + \mfh_a(\sigma) \quad \text{for } \sigma \in B^{g_a}_{\delta}(\sigma_0).\]
Here, $\{\mfh_a\}_{a \in \N}$ is a sequence of smooth positive functions on $B^{g_a}_{\delta}(\sigma_0)$ such that for any $r \in (0,\frac{\delta}{4})$,
\begin{equation}\label{eq:mfha}
\sup_{B^{g_a}_r(\sigma)} \mfh_a \le C\inf_{B^{g_a}_r(\sigma)} \mfh_a \quad \text{and} \quad \sum_{m=1}^{2\msfk+1} r^m |\nabla^m\mfh_a(\sigma)| \le C\|\mfh_a\|_{L^{\infty}(B^{g_a}_r(\sigma))}, \quad \sigma \in B^{g_a}_{\delta/2}(\sigma_0),
\end{equation}
where the constant $C > 0$ depends only on $n$ and $(M,g_{\infty})$.
Building upon these observations, the proofs of the following propositions were provided in \cite[Section 4]{LX} for $\msfk = 2$. The proof for the case $\msfk = 3$ is analogous; see \cite{JLX}.
\begin{prop}
Let $\sigma_a \to \bsi \in M$ be an isolated blowup point for a sequence $\{u_a\}$ of solutions to \eqref{eq:maina}, and $\ep_a = u_a(\sigma_a)^{-\frac{2}{n-2\msfk}} > 0$.
Suppose that $\{R_\ell\}_{\ell \in \N}$ and $\{\tau_\ell\}_{\ell \in \N}$ are arbitrary sequences of positive numbers such that $R_\ell \to \infty$ and $\tau_\ell \to 0$ as $\ell \to \infty$.
Then $\{u_a\}$ has a subsequence $\{u_{a_\ell}\}_{\ell \in \N}$ such that
\begin{equation}\label{eq:conv}
\left\|\ep_{a_\ell}^{\frac{n-2\msfk}{2}} u_{a_\ell}\(\ep_{a_\ell} \cdot\) - w\right\|_{C^{2\msfk-1}(B_{R_{a_\ell}})} \le \tau_\ell
\end{equation}
in $g_{a_\ell}$-normal coordinates centered at $\sigma_{a_\ell}$, and $R_\ell \ep_{a_\ell} \to 0$ as $\ell \to \infty$. Here, $w = w_{1,0}$.
\end{prop}
\noindent Thus, we can select $\{R_\ell\}_{\ell \in \N}$ and $\{u_{a_\ell}\}_{\ell \in \N}$ satisfying \eqref{eq:conv} and $R_\ell \ep_{a_\ell} \to 0$.
To simplify notations, we will write $\{u_a\}_{a \in \N}$ instead of $\{u_{a_\ell}\}_{\ell \in \N}$, and so on.
\begin{prop}\label{prop:isosim decay}
If $\sigma_a \to \bsi \in M$ is an isolated simple blowup point for a sequence $\{u_a\}$ of solutions to \eqref{eq:maina}, then there exist constants $\msfd_0 \in (0,1)$ and $C>0$ such that
\begin{equation}\label{eq:isosim decay}
u_a(\sigma_a)u_a(\sigma) \le Cd_{g_a}(\sigma,\sigma_a)^{2\msfk-n} \quad \text{for all } 0<d_{g_a}(\sigma,\sigma_a)<\msfd_0.
\end{equation}
Moreover, $u_a(\sigma_a)u_a \to \mfg := \vsi_nG_{g_{\infty}}(\cdot,\bsi)+\mfh$ in $C^{2\msfk-1}_{\mathrm{loc}}(B^{g_{\infty}}_{\msfd_0}(\bsi) \setminus \{\bsi\})$,
where $\vsi_n=2(n-4)(n-2)|\S^{n-1}|$ for $\msfk = 2$ and $\vsi_n=8(n-6)(n-4)(n-2)|\S^{n-1}|$ for $\msfk = 3$, $G_{g_{\infty}}$ is the Green's function of the operator $P_{g_{\infty}}^{(4)}$ or $-P_{g_{\infty}}^{(6)}$,
and $\mfh \in C^{2\msfk+1}(B^{g_{\infty}}_{\msfd_0}(\bsi))$ is a nonnegative function satisfying \eqref{eq:mfha}.
\end{prop}

\section{Expansions of Curvatures}\label{sec:curv}
Let $\sigma_0$ be an arbitrarily fixed point on $M$. By a conformal change, one can assume that
\[\det g(x) = 1 \quad \text{for } x \in B^n(0,\delta),\]
where $x = (x_1,\ldots,x_n)$ are normal coordinates centered at $\sigma_0$ and $\delta > 0$ is a small number determined by $(M,g)$.

Adopting the idea of Brendle \cite{Br} (see also Khuri, Marques, and Schoen \cite{KMS}), we consider the metric $g=\exp(h)$, where $h$ is a symmetric $2$-tensor whose norm $|h|$ is small. Then we have the following expansion:
$$
g^{ij}=\delta_{ij}-h_{ij}+\frac{1}{2}h_{ik}h_{kj}+O(|h|^3).
$$

Because of $\det g(x) = 1$ and Gauss's lemma, $h$ has two properties:
\begin{equation}\label{conformal normal}
h_{ii}(x)=0 \quad \text{and} \quad x_ih_{ij}(x)=0.
\end{equation}
A direct and useful corollary is
\begin{equation}\label{conformal normal coro}
x_ih_{ij,k}(x)=\pa_k(x_ih_{ij}(x))-\delta_{ki}h_{ij}=-h_{jk},
\end{equation}
where $\pa_k := \pa_{x_k}$.

\medskip
The proofs of Lemma \ref{Ricci curva}--Lemma \ref{Q curva} are straightforward, relying only on $\tr \,h=0$ and the definition of the curvatures. We skip them.
\begin{lemma}\label{Ricci curva}
Let $g=\exp(h)$, $\tr \,h=0$, and $\ricci_g$ be the Ricci curvature tensor on $(M,g)$. Then
$$
\ricci_g=\Dot{\ricci}[h]+\Ddot{\ricci}[h,h]+O(|h|^2|\pa^2h|+|h||\pa h|^2),
$$
where
\begin{align}
2(\Dot{\ricci}[h])_{ij} &:= h_{jk,ik}+h_{ik,jk}-h_{ij,kk}; \label{eq:Ric1} \\
4(\Ddot{\ricci}[h,h])_{ij} &:= h_{lk,l}(2h_{ij,k}-h_{kj,i}-h_{ki,j}) \nonumber \\
&\ +h_{lk,i}h_{kj,l}+h_{lk,j}h_{ki,l}-h_{kl,i}h_{kl,j}-2h_{ik,l}h_{jl,k} \nonumber \\
&\ +h_{kl}(2h_{ij,kl}-h_{ki,lj}-h_{kj,il}) \nonumber \\
&\ +h_{ki}h_{kl,lj}+h_{kj}h_{kl,li}-h_{ki}h_{kj,ll}-h_{kj}h_{ki,ll}. \nonumber
\end{align}
\end{lemma}
\noindent Note that each term in $\Ddot{\ricci}[h,h]$ involves two factors of $h$. For instance, the term $h_{lk,l}(2h_{ij,k}-h_{kj,i}-h_{ki,j})$ includes one factor of $h$ in $h_{lk,l}$ and another in $2h_{ij,k}-h_{kj,i}-h_{ki,j}$.
Accordingly, the bracket $[h,h]$ in $\Ddot{\ricci}[h,h]$ indicates these two distinct appearance of $h$. One can define
\begin{equation}\label{eq:polar}
\Ddot{\ricci}[h,h'] = \frac{1}{2}\(\Ddot{\ricci}[h+h',h+h']-\Ddot{\ricci}[h,h]-\Ddot{\ricci}[h',h']\)
\end{equation}
for two symmetric $2$-tensors $h,\, h'$ so that $\Ddot{\ricci}[h,h']$ is bilinear and symmetric in $h$ and $h'$. If no ambiguity arises, we will simply write $\Dot{\ricci} = \Dot{\ricci}[h]$ and $\Ddot{\ricci} = \Ddot{\ricci}[h,h]$. The same rule will apply in what follows.

\begin{lemma}
Let $g=\exp(h)$, $\tr \,h=0$, and $R_g$ be the scalar curvature on $(M,g)$. Then
$$
R_{g}=\Dot{R}[h]+\Ddot{R}[h,h]+O(|h|^2|\pa^2h|+|h||\pa h|^2),
$$
where
\begin{align}
\Dot{R}[h] &:=h_{ij,ij}; \label{eq:R1} \\
\Ddot{R}[h,h] &:= -\pa_{i}(h_{ij}h_{kj,k}) + \frac{1}{2}h_{ij,i}h_{kj,k} - \frac{1}{4}(h_{jk,i})^2. \label{eq:R2}
\end{align}
\end{lemma}

\begin{cor}
Under the same assumption, we have
\begin{align*}
\Dot{R}&=\Dot{\ricci}_{ii};\\
\Ddot{R}&=\Ddot{\ricci}_{ii}-h_{ij}\Dot{\ricci}_{ij}.
\end{align*}
\end{cor}

\begin{lemma}\label{Q curva}
Let $g=\exp(h)$, $\tr \,h=0$, and $Q_g^{(4)}$ be the $Q^{(4)}$-curvature on $(M,g)$. Then
$$
Q_g^{(4)}=\Dot{Q}[h]+\Ddot{Q}[h,h]+O(|h|^2|\pa^4h|+|h||\pa h||\pa^3h|+ |h||\pa^2h|^2 + |\pa h|^2|\pa^2h|),
$$
where
\begin{align*}
\Dot{Q}[h] &:=-\frac{1}{2(n-1)}\Delta \Dot{R}[h];\\
\Ddot{Q}[h,h] &:=-\frac{1}{2(n-1)}(\Delta \Ddot{R}[h,h]-\pa_{i}(h_{ij}\pa_j\Dot{R}[h]))\\
&\ +\frac{n^3-4n^2+16n-16}{8(n-2)^2(n-1)^2} (\Dot{R}[h])^2 - \frac{2}{(n-2)^2} (\Dot{\ricci}_{ij}[h])^2.
\end{align*}
\end{lemma}

\begin{lemma}\label{mce}
Let $g=\exp(h)$, $\tr \,h=0$, and $P_g^{(4)}$ be the Paneitz operator on $(M,g)$. Then
\begin{align*}
\mce_g := P_g^{(4)}-\Delta_g^2 = L_1[h]+L_2[h,h] &+ O([|h|^2|\pa^2h|+|h||\pa h|^2]\pa^2) \\
&+ O([|h|^2|\pa^3h|+|h||\pa h||\pa^2h|+|\pa h|^3]\pa) \\
&+ O(|h|^2|\pa^4h|+|h||\pa h||\pa^3h|+|h||\pa^2h|^2+|\pa h|^2|\pa^2h|),
\end{align*}
where $L_1[h]$ and $L_2[h,h]$ are the second-order differential operators defined by
\begin{align}
L_1[h] &:=\frac{4}{n-2}(\Dot{\ricci}[h])_{ij}\pa^2_{ij} - \alpha_n^1\Dot{R}[h]\Delta - \alpha_n^2\pa_k\Dot{R}[h]\pa_k + \frac{n-4}{2}\Dot{Q}[h]; \label{eq:L1} \\
L_2[h,h] &:= \left[\frac{4}{n-2}(\Ddot{\ricci}[h,h])_{ij} - \frac{8}{n-2}(\Dot{\ricci}[h])_{ik}h_{kj} + \alpha_n^1\Dot{R}[h]h_{ij}\right]\pa^2_{ij} - \alpha_n^1\Ddot{R}[h,h]\Delta \nonumber \\
&\ + \left[-\frac{2}{n-2}(\Dot{\ricci}[h])_{ij}(h_{kj,i}+h_{ki,j}-h_{ij,k}) + \alpha_n^1\Dot{R}[h]h_{ik,i} \right. \label{eq:L2} \\
&\hspace{135pt} - \alpha_n^2(\pa_k\Ddot{R}[h,h]-\pa_i\Dot{R}[h]h_{ik})\biggr]\pa_k + \frac{n-4}{2}\Ddot{Q}[h,h]. \nonumber
\end{align}
Here, $\alpha_n^1 = \frac{4+(n-2)^2}{2(n-2)(n-1)}$ and $\alpha_n^2 = \frac{n-6}{2(n-1)}$.
\end{lemma}
\begin{proof}
By \cite[Lemma 2.8]{GM}, we know that
\[\mce_g = \frac{4}{n-2}g^{ik}g^{jl}(\ricci_g)_{ij}\nabla_k\nabla_l - \alpha_n^1R_g\Delta_g - \alpha_n^2g^{ij}\pa_iR_g\pa_j + \frac{n-4}{2}Q_g^{(4)}.\]
Then, owing to the expansion of $g$, $\text{tr}(h)=0$, and the definitions of the Hessian tensor and $\Delta_g$, we have
\begin{align*}
\nabla_k\nabla_l&=\pa^2_{kl}-\frac{1}{2}(h_{ki,l}+h_{li,k}-h_{kl,i})\pa_i+O(|h|^2\pa^2+|h||\pa h|\pa+|\pa^2h|+|\pa h|^2),\\
\Delta_g&=\Delta-(h_{ij}\pa^2_{ij}+h_{ij,j}\pa_i)+O(|h|^2\pa^2+|h||\pa h|\pa+|\pa^2h|+|\pa h|^2).
\end{align*}
Combining these relations with Lemmas \ref{Ricci curva}--\ref{Q curva}, we obtain the conclusion.
\end{proof}

If $u$ is radial, we can further simplify $L_1[h]u$ and $L_2[h,h]u$ using Gauss's lemma and its corollary \eqref{conformal normal}--\eqref{conformal normal coro}.
\begin{lemma}\label{spherical expansion}
Let $u$ be radial, namely, $u(x)=u(r)$ where $r = |x|$. Then,
\begin{align}
L_1[h]u &= -\alpha_n^1\Dot{R}[h]u'' + \left[\(\frac{2}{n-2}-\frac{n-2}{2}\)\Dot{R}[h] - \alpha_n^2x_i\pa_i\Dot{R}[h]\right]\frac{u'}{r} - \beta_n\Delta\Dot{R}[h]u; \label{eq:L1r} \\
L_2[h,h]u &\begin{medsize}
\displaystyle = -\alpha_n^1\Ddot{R}[h,h]u'' - \frac{(x_ih_{kl,i})^2}{n-2}\(\frac{u''}{r^2}-\frac{u'}{r^3}\)
\end{medsize} \nonumber \\
&\begin{medsize}
\displaystyle \ + \left[\(\frac{2}{n-2}-\frac{n-2}{2}\)\Ddot{R}[h,h] - \alpha_n^2x_i\pa_i\Ddot{R}[h,h] + \frac{2}{n-2}x_ih_{kl,i}\Dot{\ricci}_{kl}[h]\right]\frac{u'}{r}
\end{medsize} \label{eq:L2r} \\
&\begin{medsize}
\displaystyle \ - \beta_n\left[\Delta \Ddot{R}[h,h]-\pa_{i}(h_{ij}\pa_j\Dot{R}[h]) - \frac{n^3-4n^2+16n-16}{4(n-2)^2(n-1)}(\Dot{R}[h])^2 + \frac{4(n-1)}{(n-2)^2} (\Dot{\ricci}_{ij}[h])^2\right]u
\end{medsize}, \nonumber
\end{align}
and
\begin{align}
\mce_gu = L_1[h]u+L_2[h,h]u &+ O([|h|^2|\pa^2h|+|h||\pa h|^2]|u''|) \nonumber \\
&+ O([|h|^2|\pa^3h|+|h||\pa h||\pa^2h|+|\pa h|^3 + r^{-1}(|h|^2|\pa^2h|+|h||\pa h|^2)]|u'|) \nonumber \\
&+ O([|h|^2|\pa^4h|+|h||\pa h||\pa^3h|+|h||\pa^2h|^2+|\pa h|^2|\pa^2h|]|u|). \label{eq:mcer}
\end{align}
Here, $\alpha_n^1 = \frac{4+(n-2)^2}{2(n-2)(n-1)}$, $\alpha_n^2 = \frac{n-6}{2(n-1)}$, $\beta_n = \frac{n-4}{4(n-1)}$, and $u'=\pa_ru$.
\end{lemma}
\begin{proof}
We shall first prove two identities:
\begin{align}
h_{ij,kl}x_ix_j&=2h_{kl}; \label{conformal normal coro2}\\
h_{ik,jl}x_ix_j&=h_{kl}-x_ih_{kl,i}\label{conformal normal coro3}.
\end{align}
Indeed, by (\ref{conformal normal}) and (\ref{conformal normal coro}), it holds that
\begin{align*}
h_{ij,kl}x_ix_j&=\pa_l(h_{ij,k}x_ix_j)-h_{lj,k}x_j-h_{il,k}x_i\\
&=\pa^2_{lk}(h_{ij}x_ix_j)-\pa_l(h_{kj}x_j)-\pa_l(h_{ik}x_i)+2h_{kl}\\
&=2h_{kl}.
\end{align*}
and
\begin{align*}
h_{ik,jl}x_ix_j&=\pa_l(h_{ik,j}x_ix_j)-h_{lk,j}x_j-h_{ik,l}x_i\\
&=\pa^2_{lj}(h_{ik}x_ix_j)-\pa_l(h_{jk}x_j)-\pa_l(nh_{ik}x_i)+h_{kl}-x_ih_{kl,i}\\
&=h_{kl}-x_ih_{kl,i}.
\end{align*}

To evaluate $L_1[h]$, we only have to compute $\Dot{\ricci}_{ij}x_ix_j$, since $\Dot{\ricci}_{ii}=\Dot{R}$. As a matter of fact, by (\ref{conformal normal}), (\ref{conformal normal coro2}), and (\ref{conformal normal coro3}), we have
$$
2\Dot{\ricci}_{ij}x_ix_j=(h_{jk,ik}+h_{ik,jk}-h_{ij,kk})x_ix_j=0.
$$

To evaluate $L_2[h,h]$, we need to compute $\Ddot{\ricci}_{ij}x_ix_j$, since $\Ddot{\ricci}_{ii}=\Ddot{R}+h_{ij}\Dot{\ricci}_{ij}$. As the following claim shows, it has a simple expression even though $\Ddot{\ricci}$ itself is complicated.
\begin{claim}\label{ricci_2 radial}
$$
4\Ddot{\ricci}_{ij}x_ix_j=-(x_k\pa_kh_{ij})^2.
$$
\end{claim}
\noindent \textsc{Proof of Claim \ref{ricci_2 radial}.} By (\ref{conformal normal}) and (\ref{conformal normal coro}), we know that
\begin{equation}\label{eq:DdotRic}
4\Ddot{\ricci}_{ij}x_ix_j = -2x_ih_{lk,i}h_{kl}-2(h_{kl})^2-(x_ih_{lk,i})^2 +h_{kl}(2h_{ij,kl}x_ix_j-2h_{ki,lj}x_ix_j).
\end{equation}
Plugging \eqref{conformal normal coro2}--\eqref{conformal normal coro3} into \eqref{eq:DdotRic}, we deduce Claim \ref{ricci_2 radial}.

Since $u$ is radial, we know that $\pa_ku=x_k\frac{u'}{r}$, $\pa^2_{ij}u=\frac{x_ix_j}{r^2}u''+\(\delta_{ij}-\frac{x_ix_j}{r^2}\)\frac{u'}{r}$ and $\Delta u=u''+(n-1)\frac{u'}{r}$. Combining all the computations made here with Lemma \ref{mce}, we finish the proof.
\end{proof}

\section{Correction Terms}\label{sec:corr}
Let $h$ be a symmetric $2$-tensor on $B^n(0,\delta)$ whose entry is a smooth function and that satisfies \eqref{conformal normal}. Given $K=n-6$, $k=2,\ldots,K$, and a multi-index $\alpha$, we define
\begin{equation}\label{eq:H}
H^{(k)}_{ij}(x) = \sum_{|\alpha|=k} \frac{h_{ij,\alpha}(0)}{\alpha!} x_{\alpha}, \quad H_{ij}(x) = \sum_{2 \le |\alpha| \le K} \frac{h_{ij,\alpha}(0)}{\alpha!} x_{\alpha} = \sum_{k=2}^{K} H^{(k)}_{ij}(x),
\end{equation}
and
\begin{equation}\label{eq:Hnorm}
|H^{(k)}| = \Bigg(\sum_{|\alpha|=k} |h_{ij,\alpha}(0)|^2\Bigg)^{\frac{1}{2}},
\end{equation}
where $0 \in \R^n$ is identified with $\sigma_0 \in M$. Then, $h_{ij}(x) = H_{ij}(x) + O(|x|^{K+1})$, $H_{ij}(x) = H_{ji}(x)$, $H_{ij}(x)x_j = 0$, and $\text{tr}(H(x)) = 0$.
\medskip
Although the following result is well-known (see \cite[Section 4]{KMS}), we include a proof for the reader's convenience.
\begin{lemma}\label{orthogonal}
For $l=1,\ldots,n$,
$$
\int_{\S^{n-1}}H_{ij,ij}=0 \quad \text{and} \quad \int_{\S^{n-1}}x_lH_{ij,ij}=0.
$$
\end{lemma}
\begin{proof}
Because of equations (\ref{conformal normal}) and (\ref{conformal normal coro}), we have
$$
x_iH^{(k)}_{ij,j}=0 \quad \text{for any } k=2,\ldots,K.
$$
Then, according to Lemma \ref{IBP}, we have
$$
\int_{\S^{n-1}}H^{(k)}_{ij,ij} = \int_{\S^{n-1}}\pa_jH^{(k)}_{ij,i} = (n+k-2)\int_{\S^{n-1}}x_iH^{(k)}_{ij,j}=0,
$$
and for any $l=1,\ldots,n$,
\[\int_{\S^{n-1}}x_lH^{(k)}_{ij,ij} = \int_{\S^{n-1}}\pa_j\(x_lH^{(k)}_{ij,i}-H^{(k)}_{jl}\) = (n+k-1)\int_{\S^{n-1}}x_j\(x_lH^{(k)}_{ij,i}-H^{(k)}_{jl}\)=0. \qedhere\]
\end{proof}

Since $H^{(k+2)}_{ij,ij}$ is a homogeneous polynomial on $\R^n$ of degree $k$, i.e., $H^{(k+2)}_{ij,ij} \in \mcp_k$, we can write it by using the spherical harmonic decomposition:
\begin{equation}\label{eq:Hshd}
H^{(k+2)}_{ij,ij}(x) = \sum_{s=0}^{\lfloor\frac{k-2}{2}\rfloor}r^{2s}p^{(k-2s)}(x),
\end{equation}
where $r = |x|$ and $p^{(k-2s)}\in \mch_{k-2s}$. By Lemma \ref{orthogonal} and Lemma \ref{IBP}, $p^{(0)}$ and $p^{(1)}$ vanish.

\begin{rmk}
Let $\beta^*_1 := \frac{2(n-1)}{n-2}-\frac{(n-2)(n-1)}{2} + 6-n$, $\beta^*_2 := -\frac{n-2}{2}-\frac{2}{n-2}$, and $\beta_n = \frac{n-4}{4(n-1)}$. For smooth functions $u,\, v$ on $\R^n$, where $u$ is radial, we define
\[\mfl_1[v]u := \frac{v}{n-1} \left[\beta^*_2u''+\beta^*_1\frac{u'}{r}-\frac{(n-6)(k-2)}{2}\frac{u'}{r}\right] - \beta_n\Delta v u.\]
Then, $\mfl_1[v]u$ is linear in both $u$ and $v$. By Lemma \ref{spherical expansion}, \eqref{eq:R1} and Lemma \ref{Euler's homo thm}, we have
\begin{equation}\label{eq:L_1mfl_1}
L_1[H^{(k+2)}]u = \mfl_1[H^{(k+2)}_{ij,ij}]u.
\end{equation}
When $k=2$, it holds that $H^{(4)}_{ij,ij}=p^{(2)}\in \mch_{2}$ because $p^{(0)} = 0$. In this case, our computation agrees with \cite[Lemma 2.1]{LX}. \hfill $\diamond$
\end{rmk}

Given the normalized bubble $w = w_{1,0} = (1+|\cdot|^2)^{-(n-4)/2}$ (see \eqref{eq:bubble} with $\msfk=2$), we decompose $L_1[H^{(k+2)}]w$ into
\[L_1[H^{(k+2)}]w = \mfl_1\Bigg[\sum_{s=0}^{\lfloor\frac{k-2}{2}\rfloor}r^{2s}p^{(k-2s)}\Bigg]w = \sum_{s=0}^{\lfloor\frac{k-2}{2}\rfloor} \mfl_1[r^{2s}p^{(k-2s)}]w.\]
Then,
\begin{equation}\label{def of mfl_1(k,s)}
\begin{aligned}
&\ \mfl_1[r^{2s}p^{(k-2s)}]w \\
&= \frac{n-4}{2(n-1)} \left[\frac{(k+2)(n-6)r^{2s}}{(1+r^2)^{\frac{n-2}{2}}} + \frac{(n^2-4n+8)r^{2s}}{(1+r^2)^{\frac{n}{2}}} - \frac{s(2k-2s+n-2)r^{2s-2}}{(1+r^2)^{\frac{n-4}{2}}}\right] p^{(k-2s)}.
\end{aligned}
\end{equation}

In this section, we will find an explicit solution of the linearized equation
\begin{equation}\label{linearized eqn k,s}
\Delta^2\Psi^{k,s}-\tmfc_4(n)w^{\frac{8}{n-4}}\Psi^{k,s} = -\mfl_1[r^{2s}p^{(k-2s)}]w
\end{equation}
for each $k = 2, \ldots, K-2$ and $s = 0, \ldots, \lfloor\frac{k-2}{2}\rfloor$.

Considering the form of the right-hand side of equation (\ref{linearized eqn k,s}), we seek a solution $\Psi^{k,s}$ expressed as a linear combination of
$$
\left\{r^{2j}(1+r^2)^{-\frac{n-2}{2}}p^{(k-2s)}\mid j=0,\ldots,s+2,\, \Delta p^{(k-2s)}=0\right\}.
$$
However, when we apply the linearized operator $\Delta^2-\tmfc_4(n)(1+r^2)^{-4}$ to $r^{2j}(1+r^2)^{-\frac{n-2}{2}}p^{(k-2s)}$, the computations become lengthy. So we decide to use another equivalent linear combinations of
\begin{equation}\label{eq:nbasis}
\left\{(1+r^2)^{-\frac{n-2j}{2}}p^{(k-2s)}\mid j=1,\ldots,s+3,\, \Delta p^{(k-2s)}=0\right\}.
\end{equation}

Let us express the term $\mfl_1[r^{2s}p^{(k-2s)}]w$ from the right-hand side of (\ref{linearized eqn k,s}) and one more term $\mfl_1[r^{2s}p^{(k-2s)}]Z$ in terms of the new basis given in \eqref{eq:nbasis}.
The latter term will arise when we analyze the Pohozaev quadratic form as shown in the proof of Lemma \ref{lemma of case 3}.
\begin{defn}[Right-hand side of the linearized equation]\label{linearized eqn RHS}
\

\noindent We define two $(s+4)\times 1$ column vectors $\Vec{b}=(b_i)$ and $\Vec{b}'=(b'_i)$, where the elements $b_i=b_i(n,k,s) \in \R$ and $b'_i=b'_i(n,k,s) \in \R$ are defined by the relation
\begin{align*}
\mfl_1[r^{2s}p^{(k-2s)}](1+r^2)^{-\frac{n-4}{2}} &:= -\frac{n-4}{4(n-1)}\sum_{i=1}^{s+4}b_i(1+r^2)^{-\frac{n+6-2i}{2}}p^{(k-2s)};\\
\mfl_1[r^{2s}p^{(k-2s)}](1+r^2)^{-\frac{n-2}{2}} &:= -\frac{1}{4(n-1)}\sum_{i=1}^{s+4}b'_i(1+r^2)^{-\frac{n+8-2i}{2}}p^{(k-2s)}. \tag*{$\diamond$}
\end{align*}
\end{defn}

In view of equation (\ref{def of mfl_1(k,s)}), we have $b_1=0$, $b_2=0$,
\begin{align*}
b_3&=(-1)^{s+1}2(n^2-4n+8);\\
b_4&=(-1)^{s}(2s(n^2-4n+8)-2(k+2)(n-6));\\
b_{s+4}&=2s(2k-2s+n-2)-2(k+2)(n-6),
\end{align*}
and when $s\ge 2$, for $i=5,\ldots, s+3$,
\begin{align*}
b_{i}&=(-1)^{s+4-i} \left[2(n^2-4n+8)\binom{s}{i-3}+2s(2k-2s+n-2)\binom{s-1}{i-5}\right.\\
&\hspace{195pt} \left.-2(k+2)(n-6)\binom{s}{i-4}\right].
\end{align*}
Similarly to (\ref{def of mfl_1(k,s)}), we can calculate
\begin{multline*}
\mfl_1[r^{2s}p^{(k-2s)}](1+r^2)^{-\frac{n-2}{2}}\\
\begin{medsize}
\displaystyle = \frac{1}{2(n-1)} \left[\frac{(k(n-6)(n-2)-8(n-1))r^{2s}}{(1+r^2)^{\frac{n}{2}}} + \frac{n(n^2-4n+8)r^{2s}}{(1+r^2)^{\frac{n+2}{2}}} - \frac{(n-4)s(2k-2s+n-2)r^{2s-2}}{(1+r^2)^{\frac{n-2}{2}}}\right]p^{(k-2s)}.
\end{medsize}
\end{multline*}
Then, we have $b'_1=0$, $b'_2=0$,
\begin{align*}
b'_3&=(-1)^{s+1}2n(n^2-4n+8);\\
b'_4&=(-1)^{s}[2sn(n^2-4n+8)-2(k(n-6)(n-2)-8(n-1))];\\
b'_{s+4}&=2(n-4)s(2k-2s+n-2)-2(k(n-6)(n-2)-8(n-1)),
\end{align*}
and when $s\ge 2$, for $i=5,\ldots, s+3$,
\begin{align*}
b'_{i}&= (-1)^{s+4-i} \left[2n(n^2-4n+8)\binom{s}{i-3}+2(n-4)s(2k-2s+n-2)\binom{s-1}{i-5}\right.\\
&\hspace{165pt} \left.-2(k(n-6)(n-2)-8(n-1))\binom{s}{i-4}\right].
\end{align*}

We are now ready to solve the linearized equation (\ref{linearized eqn k,s}).
\begin{prop}\label{prop:Psi}
Given $n\ge 10$, $k=2,\ldots,K-2$, and $s=0,\ldots, \lfloor\frac{k-2}{2}\rfloor$, we assume $\Delta p^{(k-2s)}=0$.
Then equation (\ref{linearized eqn k,s}) has a solution of the form
\begin{equation}\label{eq:Psiks}
\Psi^{k,s}=\frac{n-4}{4(n-1)}\sum_{j=1}^{s+3}\Gamma_j(1+r^2)^{-\frac{n-2j}{2}}p^{(k-2s)},
\end{equation}
where the precise value of $\Gamma_j=\Gamma_j(n,k,s) \in \R$ is given as follows: For $j=3,\ldots,s+2$,
\begin{equation}\label{eq:Gammanew}
\begin{aligned}
(k-2s-1)\Gamma_1&=2\Gamma_2;\\
(k-2s)\Gamma_2&=\frac{4(n-1)}{n-2}\Gamma_3;\\
(k-2s+j-2)\Gamma_j&=\frac{2j(n-j+1)}{(n-2j+2)}\Gamma_{j+1}-\frac{(-1)^{s+3-j}}{2(n-2j+2)}\binom{s}{j-3};\\
(k-s+1)\Gamma_{s+3}&=-\frac{1}{2(n-2s-4)}.
\end{aligned}
\end{equation}
\end{prop}
\begin{proof}
For $j=1,\ldots,s+3$, let $a=\frac{n-2j}{2}$ and $b=k-2s$. By Lemma \ref{lemma:newbasis}, we find
\begin{align*}
&\ (\Delta^2-\tmfc_4(n)(1+r^2)^{-4})((1+r^2)^{-\frac{n-2j}{2}}p^{(k-2s)})\\
&=-8(j-1)(n+2-j)(n^2-2(j-2)n+2j(j-3))(1+r^2)^{-\frac{n+8-2j}{2}}p^{(k-2s)}\\
&\ +4(n-2j)(n+2-2j)(n+4-2j)(k-2s+j-2)(1+r^2)^{-\frac{n+6-2j}{2}}p^{(k-2s)}\\
&\ +4(n-2j)(n+2-2j)(k-2s+j-1)(k-2s+j-2)(1+r^2)^{-\frac{n+4-2j}{2}}p^{(k-2s)}.
\end{align*}
Then we define an $(s+4)\times (s+3)$ matrix $A=(a_{i,j})$:
\[(\Delta^2-\tmfc_4(n)(1+r^2)^{-4})((1+r^2)^{-\frac{n-2j}{2}}p^{(k-2s)}) = \sum_{i=1}^{s+4}a_{i,j}(1+r^2)^{-\frac{n+6-2i}{2}}p^{(k-2s)}.\]
Note that the matrix $A$ has non-zero entries only on the main diagonal, the superdiagonal (i.e., upper secondary diagonal), and the subdiagonal (i.e., lower secondary diagonal). We have
\begin{align*}
a_{1,1}&=4(n-2)n(n+2)(k-2s-1);\\
a_{2,1}&=4(n-2)n(k-2s)(k-2s-1),
\end{align*}
and for $j=2,\ldots,s+3$,
\begin{align*}
a_{j-1,j}&=-8(j-1)(n+2-j)(n^2-2(j-2)n+2j(j-3));\\
a_{j,j}&=4(n-2j)(n+2-2j)(n+4-2j)(k-2s+j-2);\\
a_{j+1,j}&=4(n-2j)(n+2-2j)(k-2s+j-1)(k-2s+j-2).
\end{align*}
To prove this proposition, we just need to solve the following overdetermined linear system:
\begin{equation}\label{eq:AGb}
A\Gamma=\Vec{b},
\end{equation}
where $\Gamma=(\Gamma_j)$ is an $(s+3)\times 1$ column vector. By direct computation, we obtain
\begin{equation}
\begin{aligned}\label{eq:Gammacs3s2}
\Gamma_{s+3}&=a_{s+4,s+3}^{-1}b_{s+4};\\
\Gamma_{s+2}&=a_{s+3,s+2}^{-1}(b_{s+3}-a_{s+3,s+3}\Gamma_{s+3}),
\end{aligned}
\end{equation}
and for $j=1,\ldots,s+1$, the following recurrence relations
\begin{equation}\label{eq:Gammacj}
\Gamma_{j}=a_{j+1,j}^{-1}(b_{j+1}-a_{j+1,j+1}\Gamma_{j+1}-a_{j+1,j+2}\Gamma_{j+2}).
\end{equation}
System \eqref{eq:AGb} has a solution if the condition
\begin{equation}\label{eq:cancel}
a_{1,1}\Gamma_1+a_{1,2}\Gamma_2=0
\end{equation}
holds.

Let us verify \eqref{eq:cancel}. By plugging the values of $A$ and $\Vec{b}$ into \eqref{eq:Gammacs3s2} and \eqref{eq:Gammacj}, we obtain
\begin{equation}\label{eq:Gammas3s21}
\begin{aligned}
\Gamma_{s+3} &= -\frac{(k+2)(n-6)-s(2k-2s+n-2)}{2(n-2s-6)(n-2s-4)(k-s+2)(k-s+1)}; \\
\Gamma_{s+2} &= -\frac{n^2-4n+8+s(s-1)(2k-2s+n-2)-s(k+2)(n-6)}{2(n-2s-4)(n-2s-2)(k-s+1)(k-s)} - \frac{n-2s-6}{k-s}\Gamma_{s+3}; \\
\Gamma_{1} &= -\frac{n-4}{k-2s-1}\Gamma_{2}+\frac{4(n-1)}{(k-2s)(k-2s-1)}\Gamma_{3}. \end{aligned}
\end{equation}
When $s\ge 1$, we obtain
\begin{equation}\label{eq:Gamma2}
\Gamma_{2}=\frac{(-1)^{s+1}(n^2-4n+8)}{2(n-4)(n-2)(k-2s+1)(k-2s)}-\frac{n-6}{k-2s}\Gamma_3 + \frac{6(n^2-4n+8)}{(n-4)(k-2s+1)(k-2s)}\Gamma_{4}.
\end{equation}
When $s\ge 2$, we obtain
\begin{equation}\label{eq:Gamma3}
\begin{aligned}
\Gamma_{3}&=\frac{(-1)^{s}(s(n^2-4n+8)-(k+2)(n-6))}{2(n-6)(n-4)(k-2s+2)(k-2s+1)} \\
&\ -\frac{n-8}{k-2s+1}\Gamma_{4} + \frac{8(n-3)(n^2-6n+20)}{(n-6)(n-4)(k-2s+2)(k-2s+1)}\Gamma_{5}.
\end{aligned}
\end{equation}
When $s\ge 3$, for $j=4,\ldots,s+1$, we obtain the following recurrence relations
\begin{align}
\begin{medsize}
\displaystyle \Gamma_j
\end{medsize}
&\begin{medsize}
\displaystyle = (-1)^{s+3-j} \frac{(s-j+3)(n^2-4n+8)+(j-2)(j-3)(2k-2s+n-2)-(j-2)(k+2)(n-6)}{2(j-2)(n-2j)(n+2-2j)(k-2s+j-1)(k-2s+j-2)} \binom{s}{j-3}
\end{medsize} \nonumber \\
&\begin{medsize}
\displaystyle \ - \frac{n-2j-2}{k-2s+j-2}\Gamma_{j+1} + \frac{2(j+1)(n-j)(n^2-2jn+2(j+2)(j-1))}{(n-2j)(n+2-2j)(k-2s+j-1)(k-2s+j-2)}\Gamma_{j+2}. \label{eq:Gammaj}
\end{medsize}
\end{align}
In the end, using \eqref{eq:Gammas3s21}--\eqref{eq:Gammaj} and a computer software such as Mathematica, we see that the following cancellation appears:
\begin{equation}\label{eq:Gamma1new}
(k-2s-1)\Gamma_1-2\Gamma_2=0,
\end{equation}
which is exactly \eqref{eq:cancel}. Consequently, system \eqref{eq:AGb} has a solution.

\medskip
In the rest of the proof, we will prove \eqref{eq:Gammanew}. The first equation in \eqref{eq:Gammanew} is nothing but \eqref{eq:cancel}.
Plugging \eqref{eq:Gamma1new} back into the third equation in \eqref{eq:Gammas3s21}, we obtain the second equation in \eqref{eq:Gammanew}:
\begin{equation}\label{eq:Gamma2new}
(k-2s)\Gamma_2=\frac{4(n-1)}{n-2}\Gamma_3.
\end{equation}

Next, we will show the third equation in \eqref{eq:Gammanew} for $j=3,\ldots,s+1$:
\begin{equation}\label{eq:Gammajnew}
(k-2s+j-2)\Gamma_j=\frac{2j(n-j+1)}{n-2j+2}\Gamma_{j+1}-\frac{(-1)^{s+3-j}}{2(n-2j+2)}\binom{s}{j-3}.
\end{equation}
We proceed by induction. To begin with, by putting \eqref{eq:Gamma2new} back into \eqref{eq:Gamma2}, we can obtain \eqref{eq:Gammajnew} with $j=3$.
Similarly, plugging \eqref{eq:Gammajnew} with $j=3$ back into \eqref{eq:Gamma3}, we can obtain \eqref{eq:Gammajnew} with $j=4$.
Assuming the validity of \eqref{eq:Gammajnew} for a given $j\in \{4,\ldots,s\}$, we combine it with \eqref{eq:Gammaj} to see
\begin{align}
&\begin{medsize}
\displaystyle 2j(n-j+1)\Gamma_{j+1}+ (n-2j+2)(n-2j-2)\Gamma_{j+1}
\end{medsize}\nonumber \\
&\begin{medsize}
\displaystyle = (-1)^{s+3-j} \frac{(s-j+3)(n^2-4n+8)+(j-2)(j-3)(2k-2s+n-2)-(j-2)(k+2)(n-6)}{2(j-2)(n-2j)(k-2s+j-1)} \binom{s}{j-3}
\end{medsize} \nonumber \\
&\begin{medsize}
\displaystyle \ +\frac{(-1)^{s+3-j}}{2}\binom{s}{j-3} + \frac{2(j+1)(n-j)(n^2-2jn+2(j+2)(j-1))}{(n-2j)(k-2s+j-1)}\Gamma_{j+2}.
\end{medsize}\nonumber
\end{align}
Then we have
\begin{align*}
&\ (n-2j)(k-2s+j-1)\Gamma_{j+1}-2(j+1)(n-j)\Gamma_{j+2}\\
&=(-1)^{s+3-j} \frac{(s-j+3)}{2(j-2)} \binom{s}{j-3}=- \frac{(-1)^{s+2-j}}{2} \binom{s}{j-2},
\end{align*}
which is exactly \eqref{eq:Gammajnew} with $j+1$.

In the end, by the first two equations in \eqref{eq:Gammas3s21}, it holds that
\begin{align*}
(k-s)\Gamma_{s+2}&=\frac{2(s+2)(n-s-1)}{(n-2s-2)}\Gamma_{s+3}+\frac{1}{2(n-2s-2)}\binom{s}{s-1};\\
(k-s+1)\Gamma_{s+3}&=-\frac{1}{2(n-2s-4)},
\end{align*}
so the third equation for $j=s+2$ and the fourth equation in \eqref{eq:Gammanew} is true.
\end{proof}

\begin{rmk}\label{rmk:lin}
It is a surprising fact that equation \eqref{linearized eqn k,s} possesses a solution of the form
\begin{equation}\label{eq:Psiform}
\Psi(x)=F(x)(1+r^2)^{-\frac{n-2}{2}},
\end{equation}
where $F(x)$ is a polynomial. To explain why, we assume the form of $\Psi$ in \eqref{eq:Psiform} and define
$$
T_Q(F) := (1+r^2)^{\frac{n+4}{2}}\(\Delta^2-\tmfc_4(n)w^{\frac{8}{n-4}}\)\Psi.
$$
Then,
\begin{align*}
T_Q(F) &= (1+r^2)^3\Delta^2F-4(n-2)(1+r^2)^2(\Delta F+x_i\pa_i\Delta F)\\
&\ -2(n-2)n(1+r^2)(\Delta F-2x_ix_j\pa_i\pa_jF)+4(n-2)n(n+2)(x_i\pa_i F-F).
\end{align*}
We observe that $\deg T_Q(F)=\deg F+2$.

In the context of the Yamabe equation, the map corresponding to $T_Q$ is
\begin{align*}
T_Y(F) &:= (1+r^2)^{\frac{n+2}{2}}\(\Delta+n(n+2)(1+r^2)^{-2}\) (F(x)(1+r^2)^{-\frac{n}{2}}) \\
&= (1+r^2)\Delta F - 2nx \cdot \nabla F + 2nF,
\end{align*}
which appeared in the proof of \cite[Proposition 4.1]{KMS}, and so $\deg T_Y(F)=\deg F$.
This leads that the analogous matrix to $A$ in the proof of Proposition \ref{prop:Psi} for the Yamabe case is a \textbf{square} matrix with non-vanishing entries only on the diagonal and superdiagonal, making the proof of \cite[Proposition 4.1]{KMS} relatively simple.

In contrast, if $T_Q(F)$ equals a general polynomial, there may be no polynomial solution $F$ with degree 2 less.
Recall that the linear system \eqref{eq:AGb} is overdetermined, so it generally has no solution for a typical vector $\Vec{b}$.
Remarkably, the algebraic structure of the Paneitz operator yields a specific vector $\Vec{b}$ for which \eqref{eq:AGb} admits a solution.

In Appendix \ref{sec:geoexp} and Remark \ref{rmk:lin sol rel Q6}, we will offer a geometric intuition for the explicit solvability of \eqref{linearized eqn k,s}. \hfill $\diamond$
\end{rmk}

\begin{rmk}\label{rmk:lin sol rel}
Despite the difference between the Yamabe problem and $Q^{(4)}$-curvature problem depicted in the previous remark, we can find a clean relation \eqref{relation of Psi} between the solutions of the linearized equations of these two problems:

Let $H^{(k+2)}_{ij,ij}=\sum_{s=0}^{\lfloor\frac{k-2}{2}\rfloor}r^{2s}p^{(k-2s)}$ with $\Delta p^{(k-2s)}=0$ as in \eqref{eq:Hshd}. The linearized equation of the Yamabe problem is
\begin{align*}
(\Delta+n(n+2)(1+r^2)^{-2})\Psi^{2,k,s}&=\frac{n-2}{4(n-1)}r^{2s}p^{(k-2s)}(1+r^2)^{-\frac{n-2}{2}}\\
&=\frac{n-2}{4(n-1)}\sum_{i=1}^{s+3}b_i(1+r^2)^{-\frac{n+4-2i}{2}}p^{(k-2s)}.
\end{align*}
Here, $b_i=b_i(n,k,s)$ is defined by $b_1=0$, $b_2=0$, and
$$
b_{i}=(-1)^{s+3-i} \binom{s}{i-3} \quad \text{for } i=3,\ldots, s+3.
$$
Then, the solution $\Psi^{2,k,s}$ (here, the superscript $2$ indicates the order of the Yamabe problem) takes the form
$$
\Psi^{2,k,s}=\frac{n-2}{4(n-1)}\sum_{j=1}^{s+3}\Gamma^{(2)}_j(1+r^2)^{-\frac{n+2-2j}{2}}p^{(k-2s)}
$$
for some numbers $\Gamma^{(2)}_j \in \R$. Let $\Gamma^{(4)}_j := \Gamma_j$ be the constants given by \eqref{eq:Gammanew}. Surprisingly, the equalities
\begin{align*}
&\ (\Delta+n(n+2)(1+r^2)^{-2})((1+r^2)^{-\frac{n+2-2j}{2}}p^{(k-2s)})\\
&=(n(n+2)-(n+2-2j)(n+4-2j))(1+r^2)^{-\frac{n+6-2j}{2}}p^{(k-2s)}\\
&\ -2(n+2-2j)(k-2s+j-2)(1+r^2)^{-\frac{n+4-2j}{2}}p^{(k-2s)}
\end{align*}
for $j=1,\ldots,s+3$, following from Lemma \ref{lemma:newbasis}, reveal that
$$
\Gamma^{(2)}_j=\Gamma^{(4)}_j \quad \text{for } j=1,\dots, s+3.
$$
In other words, we have
\begin{equation}\label{relation of Psi}
\frac{1}{n-2}\Psi^{2,k,s}=\frac{1}{n-4}(1+r^2)^{-1}\Psi^{4,k,s},
\end{equation}
where $\Psi^{4,k,s} := \Psi^{k,s}$ in \eqref{eq:Psiks}.

Refer also to Remark \ref{rmk:lin sol rel Q6}, where the same observation applies to the sixth-order problem as well. \hfill $\diamond$
\end{rmk}

Given \eqref{eq:Hshd} and \eqref{linearized eqn k,s}, we define
\begin{equation}\label{eq:PsiH}
\Psi[H^{(k+2)}] := \sum_{s=0}^{\lfloor\frac{k-2}{2}\rfloor}\Psi^{k,s}.
\end{equation}
From Proposition \ref{prop:Psi}, we immediately deduce
\[\Delta^2\Psi[H^{(k+2)}] - \tmfc_4(n)w^{\frac{8}{n-4}}\Psi[H^{(k+2)}] = - L_1[H^{(k+2)}]w = -\mfl_1[H^{(k+2)}_{ij,ij}]w \quad \text{in } \R^n,
\]
and the following estimate.
\begin{cor}\label{cor:Psidecay}
For $n\ge 10$ and $k=2,\ldots,K-2$, there exists a constant $C=C(n,k)>0$ such that
$$
\left|\pa_{\beta}\Psi[H^{(k+2)}]\right|(x) \le C|H^{(k+2)}|(1+|x|)^{k+6-n-|\beta|}
$$
for a multi-index $\beta$ with $|\beta|=0,\ldots,4$. Here, $|H^{(k+2)}|$ is the quantity defined in \eqref{eq:Hnorm}.
\end{cor}

\section{Refined Blowup Analysis}\label{sec:blowup}
Throughout Sections \ref{sec:blowup} and \ref{sec:Weyl}, we assume that $\sigma_a \to \bsi \in M$ is an isolated simple blowup point for a sequence $\{u_a\}_{a \in \N}$ of solutions to \eqref{eq:maina} with $\msfk=2$.
We work in $g_a$-normal coordinates $x$ centered at $\sigma_a$, chosen such that $\det g_a = 1$.

\medskip
We write $u_a(x) = u_a(\exp^{g_a}_{\sigma_a}(x))$ by slight abuse of notation and $r = |x| = d_{g_a}(\exp^{g_a}_{\sigma_a}(x),\sigma_a)$, where $\exp^{g_a}$ is the exponential map on $(M,g_a)$. We set a rescaling factor $\ep_a = u_a(0)^{-\frac{2}{n-4}} > 0$
so that $\ep_a \to 0$ as $a \to \infty$, $\tig_a(y) = g_a(\ep_a y)$, $P_a^{(4)} = P_{\tig_a}^{(4)}$, $\mce_a = \mce_{\tig_a}$, and the natural normalization of $u_a$ by
\[U_a(y) = \ep_a^{\frac{n-4}{2}} u_a(\ep_a y).\]

We note that $K=n-6 \ge d=\lfloor\frac{n-4}{2}\rfloor$ for all $n \ge 8$. Let $g_a = \exp(h_a)$ and $\tig_a=\exp(\tih_a)$ so that $\tih_a(y)=h_a(\ep_a y)$. We define the $2$-tensor $H_a$ by \eqref{eq:H}, where $h$ is replaced by $h_a$, and write
\[\wth_a^{(k)}(y) := H_a^{(k)}(\ep_a y) = \ep_a^k H_a^{(k)}(y) \quad \text{for } k=2,\ldots, K\]
and $\wth_a := \sum_{k=2}^{K} \wth_a^{(k)}$. For $n = 8$ or $9$, we set $\wtPsi_a := 0$. For $n \ge 10$, we set
\[\wtPsi_a := \sum_{k=4}^{K} \Psi[\wth_a^{(k)}],\]
where $\Psi[\wth_a^{(k)}]$ is the function in \eqref{eq:PsiH} with $H^{(k+2)}$ replaced by $\wth_a^{(k)}$. Then,
\[\Delta^2\wtPsi_a - \tmfc_4(n)w^{\frac{8}{n-4}}\wtPsi_a = - \sum_{k=4}^{K} L_1[\wth_a^{(k)}]w \quad \text{in } \R^n\]
and $\wtPsi_a$ is an explicit rational function on $\R^n$ satisfying
\begin{equation}\label{eq:wtPsidec}
\left|\pa_{\beta}\wtPsi_a\right|(y)\le C\sum_{k=4}^{K}\ep_a^k|H_a^{(k)}|(1+|y|)^{k+4-n-|\beta|}
\end{equation}
for a multi-index $\beta$ with $|\beta|=0,\ldots,4$ (see Corollary \ref{cor:Psidecay}).

The following result improves estimate \eqref{eq:isosim decay} in Proposition \ref{prop:isosim decay}.
\begin{prop}\label{prop:refest}
Let $n\ge 8$ and $\sigma_a \to \bsi \in M$ be an isolated simple blowup point for a sequence $\{u_a\}$ of solutions to \eqref{eq:maina} with $\msfk=2$. Then there exist $\msfd_0 \in (0,1)$ and $C>0$ such that
$$
\left|\pa_{\beta}(U_a-w-\wtPsi_a)\right|(y) \le C\sum_{k=2}^{d-1}\ep_a^{2k}|H_a^{(k)}|^2(1+|y|)^{2k+4-n-|\beta|} +C\ep_a^{n-5}(1+|y|)^{-1-|\beta|}$$
for all $a \in \N$, multi-indices $\beta$ with $|\beta| = 0,1,\ldots,4$, and $|y|\le \msfd_0\ep_a^{-1}$.
\end{prop}
\noindent In order to deduce this proposition, one may suitably modify the argument in \cite[Section 5]{LX} for general manifolds of dimension $5 \le n \le 9$, so we omit the details.
Interested readers may consult \cite{KMS} to see how the arguments in \cite{Ma}, originally formulated for the low-dimensional Yamabe problem ($3 \le n \le 7$), can be extended to the higher-dimensional range $3 \le n \le 24$.

\begin{cor}\label{cor:scale back}
Scaling back to the $x$-coordinates, we set $w_a(x):=\ep_a^{-\frac{n-4}{2}}w(\ep_a^{-1}x)$ and $\Psi_a(x)=\ep_a^{-\frac{n-4}{2}}\wtPsi(\ep_a^{-1}x)$ so that
$$
\Delta^2\Psi_a - \tmfc_4(n)w_a^{\frac{8}{n-4}}\Psi_a = - \sum_{k=4}^{K} L_1[H_a^{(k)}]w_a \quad \text{in } \R^n.
$$
Under the conditions of Proposition \ref{prop:refest}, there exist $\msfd_0 \in (0,1)$ and $C>0$ such that
$$
\left|\pa_{\beta}(u_a-w_a-\Psi_a)\right|(x) \le C\ep_a^{\frac{n-4}{2}+|\beta|}\left[\sum_{k=2}^{d-1}|H_a^{(k)}|^2(\ep_a+|x|)^{2k+4-n-|\beta|} +(\ep_a+|x|)^{-1-|\beta|}\right]
$$
for all $a \in \N$, multi-indices $\beta$ with $|\beta| = 0,1,\ldots,4$, and $|x|\le \msfd_0$.
\end{cor}

\section{Weyl Vanishing Theorem}\label{sec:Weyl}
In this section, we prove the Weyl vanishing theorem, which is necessary to apply the positive mass theorem for the proof of Theorem \ref{thm:main4}.
We set $\theta_k=1$ if $k=\frac{n-4}{2}$ and $\theta_k=0$ otherwise.
\begin{thm}\label{thm:Weyl2}
Let $8 \le n \le 24$ and $\sigma_a \to \bsi \in M$ be an isolated simple blowup point for a sequence $\{u_a\}$ of solutions to \eqref{eq:maina} with $\msfk=2$. Then for all $a \in \N$ and $k = 0,\ldots,\lfloor \frac{n-8}{2} \rfloor$,
$$
|\nabla^k_{g_a}W_{g_a}|^2(\sigma_a) \le C\ep_a^{n-8-2k}|\log\ep_a|^{-\theta_{k+2}}.
$$
Consequently,
\begin{equation}\label{eq:Weylv}
\nabla_g^k W_g(\bsi) = 0.
\end{equation}
\end{thm}
\noindent The corresponding result for $n = 8, 9$ was deduced in \cite[Proposition 6.1]{LX}.
Besides, as shown in \cite{HV}, \eqref{eq:Weylv} holds for $k = 0,\ldots,d-2=\lfloor \frac{n-8}{2} \rfloor$ if and only if $\nabla_g^k H(\bsi) = 0$ for $k = 2,\ldots,d$, provided $\det g = 1$ near $\bsi$.
The proof of Theorem \ref{thm:Weyl2} is long and will be conducted in Subsections \ref{subsec:Weyl21}--\ref{subsec:Weyl23}, Section \ref{sec:tech}, and Appendix \ref{sec:eigenspaces}.

\subsection{Pohozaev quadratic form}\label{subsec:Weyl21}
We remind that $w(y)=(1+|y|^2)^{-\frac{n-4}{2}}$ and $Z(y)=\frac{n-4}{2}(1-|y|^2)(1+|y|^2)^{-\frac{n-2}{2}}$ for $y \in \R^n$. Because they are radial and $\det \tig_a = 1$, it holds that
\begin{equation}\label{eq:Delta w}
\Delta_{\tig_a}w = \Delta w \quad \text{and} \quad \Delta_{\tig_a}Z = \Delta Z.
\end{equation}
Let $\mce_a = \mce_{\tig_a} = P_{\tig_a}-\Delta_{\tig_a}^2$ be the differential operator analyzed in Lemma \ref{mce}, $\mce_a' := \mce_a+\Delta^2_{\tig_a}-\Delta^2$, and $\mcr_a := U_a-w-\wtPsi_a$.
By employing the local Pohozaev identity \eqref{eq:Poho}--\eqref{eq:Poho2} with the selection $(r,u) = (\msfd_0\ep_a^{-1},U_a)$, \eqref{eq:wtPsidec}, Proposition \ref{prop:refest}, and \eqref{eq:Delta w}, we discover
\begin{align}
O(\ep_a^{n-4}) &\ge \int_{B^n(0,\msfd_0\ep_a^{-1})} \(y \cdot \nabla U_a + \frac{n-4}{2}U_a\) \mce_a'(U_a) dy \nonumber \\
&= \mci_{1,\ep_a,a} + \mci_{2,\ep_a,a} + \mci_{3,\ep_a,a} + O(\ep_a^{n-4}) + \int_{B^n(0,\msfd_0\ep_a^{-1})} \(y \cdot \nabla\wtPsi_a + \frac{n-4}{2}\wtPsi_a\) \mce_a'(\wtPsi_a) dy \nonumber \\
&\ + \left[\int_{B^n(0,\msfd_0\ep_a^{-1})} Z\mce_a(\mcr_a) dy + \int_{B^n(0,\msfd_0\ep_a^{-1})} \(y \cdot \nabla\wtPsi_a + \frac{n-4}{2}\wtPsi_a\) \mce_a'(\mcr_a) dy \right. \label{eq:Pohoineq1} \\
&\hspace{125pt} \left.+ \int_{B^n(0,\msfd_0\ep_a^{-1})} \(y \cdot \nabla\mcr_a + \frac{n-4}{2}\mcr_a\) \mce_a'(w+\wtPsi_a+\mcr_a) dy\right]. \nonumber
\end{align}
Here, $\msfd_0 \in (0,1)$ is the small number from Proposition \ref{prop:refest}, and for $\ep > 0$ small,
\begin{align*}
\mci_{1,\ep,a} &:= \int_{B^n(0,\msfd_0\ep^{-1})} Z\mce_a(w) dy; \\
\mci_{2,\ep,a} &:= \int_{B^n(0,\msfd_0\ep^{-1})} \(y \cdot \nabla \wtPsi_a + \frac{n-4}{2}\wtPsi_a\) \mce_a(w) dy; \\
\mci_{3,\ep,a} &:= \int_{B^n(0,\msfd_0\ep^{-1})} Z\mce_a(\wtPsi_a) dy.
\end{align*}
Using the estimate
\begin{align*}
&\ (|\mce_a|+|\Delta^2_{\tig_a}-\Delta^2|)(u) \\
&= \sum_{m=0}^4 O\(\left[\sum_{k=2}^{n-4} \ep_a^k|H_a^{(k)}||y|^{k-m} + O(\ep_a^{n-3}|y|^{n-3-m}) \right]|\nabla^{4-m}_{g_a}u|\)\\
&= O\(\ep_a^2|y|^2|\nabla^4_{g_a}u| + \ep_a^2|y||\nabla^3_{g_a}u| + \sum_{m=0}^2\ep_a^{2+m}|\nabla^{2-m}_{g_a}u|\) \quad \text{for } u \in C^4(B^n(0,\msfd_0\ep_a^{-1})),
\end{align*}
we further reduce \eqref{eq:Pohoineq1} to
\begin{equation}\label{eq:Pohoineq2}
\begin{aligned}
O(\ep_a^{n-4}) &\ge \mci_{1,\ep_a,a} + \mci_{2,\ep_a,a} + \mci_{3,\ep_a,a} + O(\ep_a^{n-4}) \\
&\ + O\(\sum_{k=2}^{\lfloor\frac{n-4}{3}\rfloor} \ep_a^{3k}|\log\ep_a| |H_a^{(k)}|^3\) + O\(\ep_a^2|\log\ep_a| \sum_{k=2}^{d-1} \ep_a^{2k}|H_a^{(k)}|^2\).
\end{aligned}
\end{equation}

To control the terms $\mci_{1,\ep_a,a}$, $\mci_{2,\ep_a,a}$ and $\mci_{3,\ep_a,a}$, we first observe from \eqref{eq:R1} and Lemmas \ref{spherical expansion}, \ref{orthogonal} and \ref{IBP} that
$$
\int_{B^n(0,\msfd_0\ep_a^{-1})} ZL_1[\wth_a]w=0 \quad \text{for all } a \in \N.
$$
We then introduce several symmetric bilinear forms defined on symmetric $2$-tensors whose entries are smooth functions.
\begin{defn}[Pohozaev quadratic form]
\

\noindent
Given any symmetric $2$-tensors $h,\, h'$, we define $L_2[h,h']$ by using \eqref{eq:L2} and the polarization identity as in \eqref{eq:polar}. For $\ep > 0$ small, let
\begin{equation}\label{eq:I_1}
I_{1,\ep}[h,h'] := \int_{B^n(0,\msfd_0\ep^{-1})} ZL_2[h,h']w.
\end{equation}
Whenever $\Psi[h]$ and $\Psi[h']$ are well-defined via \eqref{eq:PsiH}, we also set
\begin{equation}\label{eq:I_2}
I_{2,\ep}[h,h'] := \frac{1}{2}\int_{B^n(0,\msfd_0\ep^{-1})} \left[\(y_i\pa_i+\frac{n-4}{2}\)\Psi[h] L_1[h']w + \(y_i\pa_i+\frac{n-4}{2}\)\Psi[h'] L_1[h]w\right];
\end{equation}
\begin{equation}\label{eq:I_3}
I_{3,\ep}[h,h'] := \frac{1}{2}\int_{B^n(0,\msfd_0\ep^{-1})} \(\Psi[h] L_1[h']Z+\Psi[h'] L_1[h]Z\),
\end{equation}
and
\begin{equation}\label{eq:I}
I_{\ep}[h,h'] := I_{1,\ep}[h,h'] + I_{2,\ep}[h,h'] + I_{3,\ep}[h,h'].
\end{equation}
\hfill $\diamond$
\end{defn}

Similarly to \cite[Lemmas A.2--A.3]{KMS}, we can invoke \eqref{eq:mcer} to derive the following estimate.
\begin{lemma}\label{lemma:diffmciI}
Let $d=\lfloor\frac{n-4}{2}\rfloor$. Given any small $\eta \in (0,1)$, there exists $C = C(n,M,g) > 0$ such that
\begin{equation}\label{eq:diffmciI}
\left|\mci_{i,\ep_a,a} - \sum_{k,m=2}^d I_{i,\ep_a}\big[\wth^{(k)}_a,\wth^{(m)}_a\big]\right| \le C\eta \sum_{k=2}^d \ep_a^{2k}|\log\ep_a|^{\theta_k}|H_a^{(k)}|^2 + C(\msfd_0\eta^{-1}+\msfd_0^{10-n})\ep_a^{n-4}
\end{equation}
for all $a \in \N$ and $i = 1,2,3$.
\end{lemma}
\begin{proof}
We will derive \eqref{eq:diffmciI} for $i=3$. Handling the cases $i=1,2$ is easier.

From \eqref{eq:wtPsidec} and Lemma \ref{mce}, we observe
\begin{align*}
\mci_{3,\ep_a,a} &= \int_{B^n(0,\msfd_0\ep_a^{-1})} Z\mce_a(\wtPsi_a) dy \\
&= \int_{B^n(0,\msfd_0\ep_a^{-1})} ZL_1[\tih_a]\wtPsi_a dy + \sum_{k=2}^d o(\ep_a^{2k}|\log\ep_a|^{\theta_k})|H_a^{(k)}|^2 + O(\msfd_0^5\ep_a^{n-4}).
\end{align*}
Since \eqref{eq:L1} is rewritten as
\[L_1[h] = \pa_j\left[\(\frac{4}{n-2}\Dot{\ricci}_{ij}[h]-\alpha_n^1\Dot{R}[h]\delta_{ij}\)\pa_i\right] + \frac{n-4}{2}\Dot{Q}[h],\]
using $\Psi[\wth_a^{(2)}] = \Psi[\wth_a^{(3)}] = 0$, integrating by parts, and proceeding with the calculations, we obtain
\begin{align*}
\mci_{3,\ep_a,a} &= \sum_{k=2}^K\sum_{m=2}^d \left[\int_{B^n(0,\msfd_0\ep_a^{-1})} \Psi[\wth_a^{(k)}] L_1[\wth_a^{(m)}]Z dy \right.\\
&\hspace{55pt} + \int_{\pa B^n(0,\msfd_0\ep_a^{-1})} \frac{y_j}{|y|}Z\pa_i\Psi[\wth_a^{(k)}] \(\frac{4}{n-2}\Dot{\ricci}_{ij}[\wth_a^{(m)}]-\alpha_n^1\Dot{R}[\wth_a^{(m)}]\delta_{ij}\) dS \\
&\hspace{55pt} \left. - \int_{\pa B^n(0,\msfd_0\ep_a^{-1})} \frac{y_j}{|y|}\pa_iZ\Psi[\wth_a^{(k)}] \(\frac{4}{n-2}\Dot{\ricci}_{ij}[\wth_a^{(m)}]-\alpha_n^1\Dot{R}[\wth_a^{(m)}]\delta_{ij}\) dS\right]\\
&\ + O(\eta) \sum_{k=2}^d \ep_a^{2k}|\log\ep_a|^{\theta_k}|H_a^{(k)}|^2 + O(\msfd_0\eta^{-1}\ep_a^{n-4})
\end{align*}
for any $\eta \in (0,1)$ small. By \eqref{eq:wtPsidec}, the former boundary integral can be evaluated as
\begin{align*}
&\ \int_{\pa B^n(0,\msfd_0\ep_a^{-1})}\frac{y_j}{|y|}Z\pa_i\Psi[\wth_a^{(k)}] \(\frac{4}{n-2}\Dot{\ricci}_{ij}[\wth_a^{(m)}]-\alpha_n^1\Dot{R}[\wth_a^{(m)}]\delta_{ij}\) dS \\
&\leq C(\msfd_0\ep_a^{-1})^{n-1}(\msfd_0\ep_a^{-1})^{4-n} [\ep_a^{k}|H_a^{(k)}|(\msfd_0\ep_a^{-1})^{k+3-n}] [\ep_a^{m}|H_a^{(m)}|(\msfd_0\ep_a^{-1})^{m-2}]\\
&= C\msfd_0^{k+m+4-n}\ep_a^{n-4}|H_a^{(k)}||H_a^{(m)}|,
\end{align*}
and the latter can be treated similarly. It follows that
\begin{align*}
\mci_{3,\ep_a,a} &= \sum_{k=2}^K\sum_{m=2}^d \int_{B^n(0,\msfd_0\ep_a^{-1})} \Psi[\wth_a^{(k)}] L_1[\wth_a^{(m)}]Z dy \\
&\ + O(\eta) \sum_{k=2}^d \ep_a^{2k}|\log\ep_a|^{\theta_k}|H_a^{(k)}|^2 + O((\msfd_0\eta^{-1}+\msfd_0^{10-n})\ep_a^{n-4}).
\end{align*}
On the other hand,
\begin{multline*}
\left|\sum_{k=2}^K\sum_{m=2}^d \int_{B^n(0,\msfd_0\ep_a^{-1})} \Psi[\wth_a^{(k)}] L_1[\wth_a^{(m)}]Z dy - \sum_{k,m=2}^d I_{3,\ep_a}\big[\wth^{(k)}_a,\wth^{(m)}_a\big]\right| \\
\le C\eta \sum_{k=2}^d \ep_a^{2k}|\log\ep_a|^{\theta_k}|H_a^{(k)}|^2 + C\msfd_0\eta^{-1}\ep_a^{n-4}.
\end{multline*}
Therefore, inequality \eqref{eq:diffmciI} with $i=3$ holds.
\end{proof}

\medskip
We introduce the Pohozaev quadratic form for the $Q^{(4)}$-curvature problem:
\begin{equation}\label{eq:Pohoquad}
I_{\ep_a}\big[\wth_a,\wth_a\big] = \sum_{k,m=2}^d I_{\ep_a}\big[\wth_a^{(k)},\wth_a^{(m)}\big].
\end{equation}
In the next subsection, we will establish the following key estimate.
\begin{prop}\label{Positive definiteness}
For any $8 \le n \le 24$, there exists a constant $C=C(n)>0$ such that
$$
I_{\ep_a}\big[\wth_a,\wth_a\big] \ge C^{-1} \sum_{k=2}^d \ep_a^{2k}|\log\ep_a|^{\theta_k}|H_a^{(k)}|^2 + O(\ep_a^{n-4}) \quad \text{for all } a \in \N.
$$
\end{prop}
\noindent For brevity, we will often drop the subscript $a$, writing e.g. $I_{1,\ep}[\wth^{(k)},\wth^{(m)}] = I_{1,\ep_a}[\wth_a^{(k)},\wth_a^{(m)}]$, and assume that both $k$ and $m$ range from $2$ to $d$ unless stated otherwise.

\subsection{Proof of Proposition \ref{Positive definiteness}}\label{subsec:proppd}
Given two matrices $\bar{H}$ and $\bar{H}'$ of polynomials on $\R^n$, we set the inner product of $\bar{H}$ and $\bar{H}'$ and the norm of $\bar{H}$ as
\begin{equation}\label{eq:inner}
\la \bar{H},\bar{H}' \ra := \int_{\S^{n-1}} \bar{H}_{ij}\bar{H}'_{ij} \quad \text{and} \quad \|\bar{H}\| = \sqrt{\la \bar{H},\bar{H} \ra},
\end{equation}
respectively. Also, we define the divergence $\delta\bar{H}$ and the double divergence $\delta^2\bar{H}$ of the matrix $\bar{H}$ as
\[\delta_j\bar{H} := \bar{H}_{ij,i} \quad \text{and} \quad \delta^2\bar{H} := \bar{H}_{ij,ij},\]
and their inner products by
\begin{equation}\label{eq:inner2}
\la \delta\bar{H},\delta\bar{H}' \ra := \int_{\S^{n-1}} \delta_i\bar{H}\delta_i\bar{H}'
\quad \text{and} \quad \la \delta^2\bar{H},\delta^2\bar{H}' \ra := \int_{\S^{n-1}} \delta^2\bar{H}\delta^2\bar{H}'.
\end{equation}
For $k \in \N \cup \{0\}$, let $\mcv_k$ be a subspace of symmetric matrices whose elements are homogeneous polynomials on $\R^n$ of degree $k$ such that
\[\bar{H}_{ii} = 0 \quad \text{and} \quad x_i\bar{H}_{ij} =0\]
for each $\bar{H} \in \mcv_k$ and $x \in \R^n$.

By \eqref{conformal normal} and \eqref{eq:H}, we have that $H^{(k)}, \wth^{(k)} \in \mcv_k$ for $k = 2,\ldots, K$. As a preliminary step, we evaluate $I_{1,\ep}[\wth^{(k)},\wth^{(m)}]$ in polar coordinates.
\begin{prop}\label{I_1}
For $k, m = 2, \ldots, d$, let $\Ddot{R}^{(k,m)} := \Ddot{R}[H^{(k)},H^{(m)}]$, $\Dot{\ricci}^{(k)} := \Dot{\ricci}[H^{(k)}]$, and
\begin{align*}
J_1[H^{(k)},H^{(m)}] &:= \frac{(n-4)^2}{8(n-3)(n-2)(n-1)} (\mci_{n-3}^{n+k+m-3})^{1-\theta_{\frac{k+m}{2}}}\\
&\ \times\left[-c_1(n,k,m)\int_{\S^{n-1}}\Ddot{R}^{(k,m)} - c_2(n,k,m)\bla H^{(k)},H^{(m)}\bra \right. \\
&\hspace{20pt}\left. + c_3(n,k,m)\bla\Dot{\ricci}^{(k)},\Dot{\ricci}^{(m)}\bra - c_4(n,k,m)\bla\delta^2H^{(k)},\delta^2H^{(m)}\bra \right],
\end{align*}
where $\theta_{\frac{k+m}{2}}=1$ if and only if $k+m=n-4$, and
\begin{align*}
c_1(n,k,m)&:= (k+m)(n^3-(k+m+2)n^2+(6(k+m)-4)n-4(k+m)+8);\\
c_2(n,k,m)&:= 2(n-1)(k+m)(n+k+m-2)km;\\
c_3(n,k,m)&:= \frac{8(n-3)(n-1)(k+m)}{(n-2)(n+k+m-4)};\\
c_4(n,k,m)&:= \frac{(n-3)(n^3-4n^2+16n-16)(k+m)}{2(n-2)(n-1)(n+k+m-4)}.
\end{align*}
Then, for $k, m = 2, \ldots, d$, it holds that
\begin{equation}\label{eq:I_1J_1}
I_{1,\ep}\big[\wth^{(k)},\wth^{(m)}\big] = \ep^{k+m}|\log\ep|^{\theta_{\frac{k+m}{2}}} J_1[H^{(k)},H^{(m)}] + O(\ep^{n-4}),
\end{equation}
where $\mci_{n-2}^{n+k+m-3}$ is the quantity defined by \eqref{eq:mci}.
\end{prop}
\begin{proof}
Applying \eqref{eq:L2r} and \eqref{eq:ricci h2}, we observe
\begin{align}
&\ \ep^{-(k+m)} I_{1,\ep}\big[\wth^{(k)},\wth^{(m)}\big] \nonumber \\
&= \int_{\S^{n-1}} \left[-\alpha_n^1\Ddot{R}^{(k,m)}-\frac{km}{n-2}H^{(k)}_{ij}H^{(m)}_{ij}\right] dS \int_0^{\msfd_0\ep^{-1}}r^{k+m+n-1}Z\(\frac{w''}{r^2}-\frac{w'}{r^3}\) dr \nonumber \\
&\ + \int_{\S^{n-1}} \left[\frac{2}{n-2}-\frac{n-2}{2}-\alpha_n^1-(k+m-2)\alpha_n^2-\frac{2(k+m)}{n-2}\right]\Ddot{R}^{(k,m)} dS \int_0^{\msfd_0\ep^{-1}}r^{k+m+n-3}Z\frac{w'}{r} dr \nonumber \\
&\ - \frac{(n+k+m-2)km}{n-2} \int_{\S^{n-1}} H^{(k)}_{ij}H^{(m)}_{ij} dS \int_0^{\msfd_0\ep^{-1}}r^{k+m+n-3}Z\frac{w'}{r} dr \label{eq:I_1exp}\\
&\ - \beta_n \int_{\S^{n-1}} \left[\Delta \Ddot{R}^{(k,m)}
-\frac{n^3-4n^2+16n-16}{4(n-2)^2(n-1)} \delta^2H^{(k)}\delta^2H^{(m)}\right] dS \int_{0}^{\msfd_0\ep^{-1}}r^{k+m+n-5}Zw dr \nonumber \\
&\ - \frac{n-4}{(n-2)^2}\int_{\S^{n-1}} \Dot{\ricci}^{(k)}_{ij}\Dot{\ricci}^{(m)}_{ij} dS \int_{0}^{\msfd_0\ep^{-1}}r^{k+m+n-5}Zw dr. \nonumber
\end{align}
Then, \eqref{eq:I_1J_1} follows from Lemma \ref{IBP}, Lemma \ref{lemma:mcinl}, and Corollary \ref{cor:mcinl}.
\end{proof}

\begin{rmk}
In the Yamabe case, it holds that
\[I_{1,\ep}\big[\wth^{(k)},\wth^{(m)}\big] = -\frac{n-2}{8(n-1)}(k+m)\, \mathcal{I}_{n-2}^{n+k+m-3} \ep^{k+m} \int_{\S^{n-1}}\Ddot{R}^{(k,m)} + O(\ep^{n-2})\]
for $k,m = 2,\ldots,\lfloor \frac{n-2}{2} \rfloor$ such that $k+m < n-2$. \hfill $\diamond$
\end{rmk}

We observe from the definition of $\Dot{\ricci}[\bar{H}]$ in \eqref{eq:Ric1} that
\begin{align*}
(\Dot{\ricci}[\bar{H}])_{ii}&=\delta^2\bar{H};\\
x_i(\Dot{\ricci}[\bar{H}])_{ij}&=\frac{k}{2}\delta_j\bar{H};\\
x_ix_j(\Dot{\ricci}[\bar{H}])_{ij}&=0
\end{align*}
for $\bar{H} \in \mcv_k$. Thus, we can apply Lemma \ref{projection lemma} to define the linear operator $\mcl_k: \mcv_k\to \mcv_k$:
\begin{equation}\label{eq:mclk}
\mcl_k \bar{H} := \proj\left[|x|^2\Dot{\ricci}[\bar{H}]\right] \quad \text{for } \bar{H} \in \mcv_k.
\end{equation}
We call $\mcl_k$ as the ``stability" operator because it essentially comes from the second variation formula for the total scalar curvature functional (see \cite{Sc3}).

By \eqref{eq:projH} and \eqref{eq:Ric1},
\begin{align*}
(\mcl_k \bar{H})_{ij} &= \frac{|x|^2}{2}(\bar{H}_{jl,il}+\bar{H}_{il,jl}-\bar{H}_{ij,ll}) - \frac{k}{2}(x_j\delta_i\bar{H}+x_i\delta_j\bar{H}) \\
&\ + \frac{1}{n-1}(x_ix_j-|x|^2\delta_{ij})\delta^2\bar{H}.
\end{align*}
From this, one can check that $\mcl_k: \mcv_k\to \mcv_k$ is symmetric with respect to the inner product $\la \cdot,\cdot \ra$ (the proof is given in \cite[Page 182]{KMS}), so the spectral theorem is applicable. Using $\mcl_k$, Proposition \ref{I_1} is rephrased as follows.
\begin{cor}\label{I_1 coro}
For $k, m = 2, \ldots, d$, it holds that
\begin{align*}
J_1[H^{(k)},H^{(m)}] &= \frac{(n-4)^2}{8(n-3)(n-2)(n-1)} (\mci_{n-3}^{n+k+m-3})^{1-\theta_{\frac{k+m}{2}}}\\
&\ \times\left[\frac{1}{4}c_1(n,k,m) \(\bla\mcl_kH^{(k)},H^{(m)}\bra + \bla H^{(k)},\mcl_mH^{(m)}\bra\)\right. \\
&\hspace{20pt} + c_3(n,k,m) \bla\mcl_kH^{(k)},\mcl_mH^{(m)}\bra \\
&\hspace{20pt} + \left\{\frac{1}{8}(n+k+m-2)(k+m)c_1(n,k,m)-c_2(n,k,m)\right\} \bla H^{(k)},H^{(m)}\bra \\
&\hspace{20pt} + \frac{km}{2}c_3(n,k,m) \bla\delta H^{(k)},\delta H^{(m)}\bra \\
&\hspace{20pt} + \left. \left\{\frac{1}{n-1}c_3(n,k,m)-c_4(n,k,m)\right\} \bla\delta^2H^{(k)},\delta^2H^{(m)}\bra\right].
\end{align*}
\end{cor}
\begin{proof}
It is a simple consequence of Proposition \ref{I_1} and Lemma \ref{ricci ricci lemma}.
\end{proof}

Before proving Proposition \ref{Positive definiteness}, we consider a decomposition of the space $\mcv_k$ arising from the eigenspace decomposition of $\mcl_k$ for each $k \ge 2$:
\begin{equation}\label{eq:mcv decom}
\mcv_k=\mcv_k/\mcw_k\oplus\mcw_k/\mcd_k\oplus\mcd_k,
\end{equation}
where
\begin{equation}\label{eq:mcv decom2}
\begin{aligned}
\mcw_k &:= \{W\in \mcv_k\mid \delta^2W=0\};\\
\mcd_k &:= \{D\in \mcv_k\mid \delta D=0\};\\
\mcv_k/\mcw_k &:= \{\hat{H}\in \mcv_k\mid \langle \hat{H},W \rangle=0 \text{ for all } W\in \mcw_k\};\\
\mcw_k/\mcd_k &:= \{\hat{W}\in \mcw_k\mid \langle \hat{W},D \rangle=0 \text{ for all } D\in \mcd_k\}.
\end{aligned}
\end{equation}
We remark that this decomposition, introduced in \cite{KMS}, shares many similarities with the decomposition mentioned in \cite{FM}.

From Corollary \ref{I_1 coro}, we see that a key distinction between the Pohozaev quadratic form for the $Q^{(4)}$-curvature problem and the Yamabe problem is the presence of the following three terms:
\begin{equation}\label{eq:extra}
\bla\mcl_kH^{(k)},\mcl_mH^{(m)}\bra, \quad \bla\delta H^{(k)},\delta H^{(m)}\bra \quad \text{and} \quad \bla\delta^2H^{(k)},\delta^2H^{(m)}\bra.
\end{equation}
In \cite{KMS}, the eigenvectors of $\mcl_k$ in $\mcv_k/\mcw_k$ and $\mcd_k$ were found by using the spherical harmonics. Readers can refer to Lemmas \ref{H hat}--\ref{D hat} for their explicit expressions.
In contrast, the eigenvectors of $\mcl_k$ in $\mcw_k/\mcd_k$ were not written out in \cite{KMS}. As we shall see, they all share the same eigenvalue $-\frac{(n+k-2)k}{2}$, and only that eigenvalue information was necessary there.
For the $Q^{(4)}$-curvature problem, the occurrence of the last two integrals in \eqref{eq:extra} within the expansion of $I_{1,\ep}[\wth^{(k)},\wth^{(m)}]$ forces us to determine explicit expressions for the eigenvectors in $\mcw_k/\mcd_k$.
This is the content of Lemma \ref{lemma:W_hat0}, which will be reinforced in Lemma \ref{W hat} below.

Let $V = (V_1,\ldots,V_n)$ be a smooth vector field on $\R^n$ and $\delta V := V_{i,i}$ be its divergence. We define the conformal Killing operator by
\begin{equation}\label{eq:cko}
(\msd V)_{ij} := \pa_i V_j+\pa_j V_i-\frac{2}{n}\delta V\delta_{ij}.
\end{equation}

\begin{lemma}\label{lemma:W_hat0}
For $k \ge 2$ and $q = 1,\ldots,\lfloor\frac{k-1}{2}\rfloor$, we write $s = k-2q$. We choose a smooth vector field on $\R^n$,
$$
V^{(s+1)} \in \{V = (V_1,\ldots,V_n) \mid V_i \in \mch_{s+1} \text{ for } i = 1,\ldots,n,\, \delta V=0,\, x_iV_i=0\},
$$
and set $\hat{W} := \proj[|x|^{2q}\msd V^{(s+1)}] \in \mcv_k$. Then $\hat{W}\in \mcw_k$. In addition, it holds that
\begin{equation}\label{eq:hatWq10}
\mcl_k\hat{W}=-\frac{(n+k-2)k}{2}\hat{W}.
\end{equation}
\end{lemma}
\begin{proof}
We infer from Lemmas \ref{Euler's homo thm} and \ref{projection lemma} that
\[\hat{W}_{ij}=|x|^{2q} (\msd V^{(s+1)})_{ij}-s|x|^{2q-2} (x_jV^{(s+1)}_{i}+x_iV^{(s+1)}_{j}).\]
A straightforward computation yield
\begin{align*}
\delta_j\hat{W} &= -(n+s)s|x|^{2q-2}V_j^{(s+1)}, \\
\Delta\hat{W}_{ij} &= \(2q(2q+2s+n-2)-2s\)|x|^{2q-2} (\msd V^{(s+1)})_{ij} \\
&\ -s(2q-2)(2q+2s+n)|x|^{2q-4}(V_ix_j+V_jx_i).
\end{align*}
It follows that $\delta^2\hat{W} = 0$ so $\hat{W} \in \mcw_k$.

Also, by combining these identities with the definition of $\mcl_k$ in \eqref{eq:mclk}, we derive
\begin{align*}
(\mcl_k\hat{W})_{ij}
&= \frac{|x|^2}{2} (\pa_i\delta_j\hat{W}+\pa_j\delta_i\hat{W}-\Delta\hat{W}_{ij}) - \frac{2q+s}{2} (x_j\delta_i\hat{W}+x_i\delta_j\hat{W})\\
&= \frac{|x|^{2q}}{2} \left[-(n+s)s(\pa_i V_j+\pa_j V_i)-(2q(2q+2s+n-2)-2s)(\msd V)_{ij}\right] \\
&\ + \frac{s|x|^{2q-2}}{2} [-(n+s)(2q-2)+(2q-2)(2q+2s+n)+(2q+s)(n+s)] (V_ix_j+V_jx_i) \\
&= -\frac{(n+2q+s-2)(2q+s)}{2}\hat{W}_{ij} =-\frac{(n+k-2)k}{2}\hat{W}_{ij},
\end{align*}
which reads \eqref{eq:hatWq10}.
\end{proof}
\begin{rmk}\label{rmk:W_hat0}
\

\noindent 1. Note that
$$
-\frac{(n+k-2)k}{2}\neq -q(n-2q+2k-2) \quad \text{for } q=0,\ldots,\lfloor\tfrac{k-2}{2}\rfloor.
$$
In light of \eqref{eq:hatWq10} and \eqref{eq:hatDq1}, it follows that $\hat{W} = \proj[|x|^{2q}\msd V^{(s+1)}] \in \mcw_k/\mcd_k \subset \mcw_k$.

\noindent 2. In Lemma \ref{W hat}, we verify that every element of $\mcw_k/\mcd_k$ can be expressed as a linear combination of matrices $\hat{W}$'s. \hfill $\diamond$
\end{rmk}

The following corollary comes from Lemmas \ref{H hat}--\ref{W hat}.
\begin{cor}
Given any $\bar{H} \in \mcv_k$ for $k \ge 2$, there exist $\hat{H} \in \mcv_k/\mcw_k$, $\hat{W} \in \mcw_k/\mcd_k$, and $\hat{D} \in \mcd_k$ such that
$$
\bar{H}=\hat{H}+\hat{W}+\hat{D}.
$$
In fact, one can pick $\hat{H}_1,\ldots,\hat{H}_{\lfloor\frac{k-2}{2}\rfloor} \in \mcv_k/\mcw_k$ satisfying \eqref{eq:hatHq1}--\eqref{eq:hatHq2},
$\hat{W}_1,\ldots,\hat{W}_{\lfloor\frac{k-1}{2}\rfloor} \in \mcw_k/\mcd_k$ satisfying \eqref{eq:hatWq1}--\eqref{eq:hatWq2}, and $\hat{D}_0,\ldots,\hat{D}_{\lfloor\frac{k-2}{2}\rfloor} \in \mcd_k$ satisfying \eqref{eq:hatDq1}--\eqref{eq:hatDq2} such that
$$
\bar{H} = \sum_{q=1}^{\lfloor\frac{k-2}{2}\rfloor}\hat{H}_q + \sum_{q=1}^{\lfloor\frac{k-1}{2}\rfloor}\hat{W}_q + \sum_{q=0}^{\lfloor\frac{k-2}{2}\rfloor}\hat{D}_q.
$$
In particular, $\mcv_2 = \mcd_2$ and $\mcv_3 = \mcw_3$.
\end{cor}

We are now ready to prove Proposition \ref{Positive definiteness}. Due to its technical nature, we will outline the main structure of the proof here and postpone several detailed computations to Section \ref{sec:tech} and Appendix \ref{sec:eigenspaces}.
\begin{proof}[Proof of Proposition \ref{Positive definiteness}]
Given $k = 2,\ldots,d$, we write
\[H^{(k)} = \hat{H}^{(k)} + \hat{W}^{(k)} + \hat{D}^{(k)},\]
where $\hat{H}^{(k)} \in \mcv_k/\mcw_k$, $\hat{W}^{(k)} \in \mcw_k/\mcd_k$ and $\hat{D}^{(k)} \in \mcd_k$. Then $\hat{H}^{(2)}=\hat{H}^{(3)}=0$ and $\hat{W}^{(2)}=0$.
We also set
\[\check{H}^{(k)}(y) := \hat{H}^{(k)}(\ep y), \quad \check{W}^{(k)}(y) := \hat{W}^{(k)}(\ep y), \quad \text{and} \quad \check{D}^{(k)}(y) := \hat{D}^{(k)}(\ep y)\]
so that $\wth^{(k)} = \check{H}^{(k)} + \check{W}^{(k)} + \check{D}^{(k)}$. As shown in Lemma \ref{lemma:Hortho}, it holds that
\begin{equation}\label{eq:I_1decom}
I_{1,\ep}\big[\wth^{(k)},\wth^{(m)}\big] = I_{1,\ep}\big[\check{H}^{(k)},\check{H}^{(m)}\big] + I_{1,\ep}\big[\check{W}^{(k)},\check{W}^{(m)}\big] + I_{1,\ep}\big[\check{D}^{(k)},\check{D}^{(m)}\big].
\end{equation}
By \eqref{eq:R1}, we have that $\dot{R}[\check{W}^{(k)}] = \delta^2\check{W}^{(k)} = 0$ and $\dot{R}[\check{D}^{(k)}] = \delta(\delta \check{D}^{(k)}) = 0$. This together with \eqref{eq:L1r} leads to
\begin{equation}\label{eq:van}
L_1[\check{W}^{(k)}]u=L_1[\check{D}^{(k)}]u=0 \quad \text{for } u \text{ radial} \quad \text{and} \quad \Psi[\check{W}^{(k)}]=\Psi[\check{D}^{(k)}]=0.
\end{equation}
From \eqref{eq:Pohoquad}, \eqref{eq:I}, \eqref{eq:I_1decom}, \eqref{eq:I_2}--\eqref{eq:I_3}, and \eqref{eq:van}, we readily observe
\begin{equation}\label{eq:Idecom}
I_{\ep}\big[\wth,\wth\big] = \sum_{k,m=2}^d I_{1,\ep}\big[\check{D}^{(k)},\check{D}^{(m)}\big] + \sum_{k,m=3}^d I_{1,\ep}\big[\check{W}^{(k)},\check{W}^{(m)}\big] + \sum_{k,m=4}^d I_{\ep}\big[\check{H}^{(k)},\check{H}^{(m)}\big].
\end{equation}

For the remainder of the proof, we analyze each term on the right-hand side of \eqref{eq:Idecom}.

\medskip \noindent \textbf{Case 1.} Let us study the term $\sum_{k,m=2}^d I_{1,\ep}[\check{D}^{(k)},\check{D}^{(m)}]$.
Thanks to Lemma \ref{D hat}, we can write $\hat{D}^{(k)} = \sum_{q=0}^{\lfloor\frac{k-2}{2}\rfloor}|x|^{2q} M^{(k-2q)}$. Setting $s=k-2q$ for $q = 0,\ldots,\lfloor\frac{k-2}{2}\rfloor$ and
\begin{equation}\label{eq:MqsEDs}
\begin{cases}
\displaystyle M_q^{(s)} := M^{(k-2q)}, &\displaystyle E^D_s := \sum_{q=0}^{\lfloor \frac{d-s}{2} \rfloor}|x|^{2q}M_q^{(s)}, \\
\displaystyle \wtm_q^{(s)}(y) := M_q^{(s)}(x), &\displaystyle \wte^D_s(y) := E^D_s(x) \quad \text{for } x = \ep y \in \R^n,
\end{cases}
\end{equation}
we obtain
$$
\sum_{k=2}^d \check{D}^{(k)} = \sum_{k=2}^d\sum_{q=0}^{\lfloor\frac{k-2}{2}\rfloor}|x|^{2q}\wtm^{(k-2q)} = \sum_{s=2}^d\sum_{q=0}^{\lfloor \frac{d-s}{2} \rfloor}|x|^{2q}\wtm_q^{(s)} = \sum_{s=2}^d \wte^D_s.
$$
As pointed out in Lemma \ref{lemma:I_1ortho}, we have that $I_{1,\ep}[\wte^D_s,\wte^D_{s'}] = 0$ for $s \ne s'$, which combined with \eqref{eq:I_1J_1} implies that
\begin{align*}
\sum_{k,m=2}^d I_{1,\ep}\big[\check{D}^{(k)},\check{D}^{(m)}\big] &= \sum_{s=2}^d I_{1,\ep}\big[\wte^D_s,\wte^D_s\big] = \sum_{s=2}^d \sum_{q,q'=0}^{\lfloor \frac{d-s}{2} \rfloor} I_{1,\ep}\left[|x|^{2q}\wtm_q^{(s)},|x|^{2q'}\wtm_{q'}^{(s)}\right] \\
&= \sum_{s=2}^d \sum_{q,q'=0}^{\lfloor \frac{d-s}{2} \rfloor} |\log\ep|^{\theta_{q+q'+s}} J_1\left[\ep^{2q+s}|x|^{2q}M_q^{(s)},\ep^{2q'+s}|x|^{2q'}M_{q'}^{(s)}\right] + O(\ep^{n-4}).
\end{align*}
Lemma \ref{lemma of case 1} tells us that if $8 \le n \le 24$, then there exists a constant $C=C(n)>0$ such that
\begin{equation}\label{eq:Idecom1}
\begin{aligned}
\sum_{k,m=2}^d I_{1,\ep}\big[\check{D}^{(k)},\check{D}^{(m)}\big] &\ge C^{-1} \sum_{s=2}^d \sum_{q=0}^{\lfloor \frac{d-s}{2} \rfloor} |\log\ep|^{\theta_{2q+s}} \left\|\ep^{2q+s}|x|^{2q}M_q^{(s)}\right\|^2 + O(\ep^{n-4}) \\
&= C^{-1} \sum_{k=2}^d \ep^{2k}|\log\ep|^{\theta_k} \big\|\hat{D}^{(k)}\big\|^2 + O(\ep^{n-4}).
\end{aligned}
\end{equation}

\medskip \noindent \textbf{Case 2.} Let us study the term $\sum_{k,m=3}^d I_{1,\ep}[\check{W}^{(k)},\check{W}^{(m)}]$.
Thanks to Lemma \ref{W hat}, we can write $\hat{W}^{(k)} = \sum_{q=1}^{\lfloor\frac{k-1}{2}\rfloor}\proj[|x|^{2q}\msd V^{(k-2q+1)}]$. Setting $s=k-2q$ for $q=1,\ldots,\lfloor\frac{k-1}{2}\rfloor$ and
\begin{equation}\label{eq:VqsEWs}
\begin{cases}
\displaystyle V_q^{(s+1)} := V^{(k-2q+1)}, &\displaystyle E^W_s := \sum_{q=1}^{\lfloor \frac{d-s}{2} \rfloor} \proj\left[|x|^{2q}\msd V_q^{(s+1)}\right], \\
\displaystyle \wtv_q^{(s+1)}(y) := V_q^{(s+1)}(x), &\displaystyle \wte^W_s(y) := E^W_s(x) \quad \text{for } x = \ep y \in \R^n,
\end{cases}
\end{equation}
we obtain
$$
\sum_{k=3}^d \check{W}^{(k)}=\sum_{k=3}^d\sum_{q=1}^{\lfloor\frac{k-1}{2}\rfloor} \proj\left[|x|^{2q}\msd \wtv^{(k-2q+1)}\right]
=\sum_{s=1}^{d-2}\sum_{q=1}^{\lfloor \frac{d-s}{2} \rfloor} \proj\left[|x|^{2q}\msd \wtv_q^{(s+1)}\right] = \sum_{s=1}^{d-2} \wte^W_s.
$$
By Lemma \ref{lemma:I_1ortho} and \eqref{eq:I_1J_1}, we deduce
\begin{align*}
\sum_{k,m=3}^d I_{1,\ep}\big[\check{W}^{(k)},\check{W}^{(m)}\big]
&= \sum_{s=1}^{d-2} \sum_{q,q'=1}^{\lfloor \frac{d-s}{2} \rfloor} |\log\ep|^{\theta_{q+q'+s}} J_1\left[\proj\left[|x|^{2q}\msd V_q^{(s+1)}\right], \proj\left[|x|^{2q'}\msd V_{q'}^{(s+1)}\right]\right] \\
&\ + O(\ep^{n-4}).
\end{align*}
If $n = 8$ or $9$, then $d=2$ so that $\sum_{k,m=3}^d I_{1,\ep}[\check{W}^{(k)},\check{W}^{(m)}]=0$. If $10 \le n \le 28$, then Lemma \ref{lemma of case 2} tells us that there exists a constant $C=C(n)>0$ such that
\begin{equation}\label{eq:Idecom2}
\sum_{k,m=3}^d I_{1,\ep}\big[\check{W}^{(k)},\check{W}^{(m)}\big] \ge C^{-1} \sum_{k=3}^d \ep^{2k}|\log\ep|^{\theta_k} \big\|\hat{W}^{(k)}\big\|^2 + O(\ep^{n-4}).
\end{equation}

\medskip \noindent \textbf{Case 3.} Let us study the term $\sum_{k,m=4}^d I_{\ep}[\check{H}^{(k)},\check{H}^{(m)}]$.
Thanks to Lemma \ref{H hat}, we can write $\hat{H}^{(k)} = \sum_{q=1}^{\lfloor\frac{k-2}{2}\rfloor}\proj[|x|^{2q+2}\nabla^2P^{(k-2q)}]$.
Setting $s=k-2q$ for $q=1,\ldots,\lfloor\frac{k-2}{2}\rfloor$ and
\begin{equation}\label{eq:PqsEHs}
\left\{\begin{array}{lll}
\!\!\! \displaystyle P_q^{(s)} := P^{(k-2q)}, & \displaystyle \hat{H}_q^{(2q+s)} := \proj\left[|x|^{2q+2}\nabla^2P_q^{(s)}\right],
& \displaystyle E^H_s := \sum_{q=1}^{\lfloor \frac{d-s}{2} \rfloor}\hat{H}_q^{(2q+s)}, \\
\!\!\! \displaystyle \wtp_q^{(s)} := \wtp^{(k-2q)}, & \displaystyle \check{H}_q^{(2q+s)} := \proj\left[|x|^{2q+2}\nabla^2\wtp_q^{(s)}\right],
& \displaystyle \wte^H_s := \sum_{q=1}^{\lfloor \frac{d-s}{2} \rfloor}\check{H}_q^{(2q+s)},
\end{array} \right.
\end{equation}
we obtain
$$
\sum_{k=4}^d \check{H}^{(k)} = \sum_{k=4}^d\sum_{q=1}^{\lfloor\frac{k-2}{2}\rfloor} \proj\left[|x|^{2q+2}\nabla^2\wtp^{(k-2q)}\right]
= \sum_{s=2}^{d-2}\sum_{q=1}^{\lfloor \frac{d-s}{2} \rfloor}\check{H}_q^{(2q+s)} = \sum_{s=2}^{d-2} \wte^H_s.
$$
We note that $\Psi[\wte^H_s] = \sum_{q=1}^{\lfloor \frac{d-s}{2} \rfloor}\Psi[\check{H}_q^{(2q+s)}]$ is well-defined, so are $I_{2,\ep}[\wte^H_s,\wte^H_{s'}]$ and $I_{3,\ep}[\wte^H_s,\wte^H_{s'}]$. By Lemma \ref{lemma:I_1ortho}, we deduce
\[\sum_{k,m=4}^d I_{i,\ep}\big[\check{H}^{(k)},\check{H}^{(m)}\big] = \sum_{s=2}^{d-2} \sum_{q,q'=1}^{\lfloor \frac{d-s}{2} \rfloor} I_{i,\ep}\left[\check{H}_q^{(2q+s)},\check{H}_{q'}^{(2q'+s)}\right] \quad \text{for } i=1,2,3.\]
If $8 \le n \le 11$, then $d=2$ or $3$ so that $\sum_{k,m=4}^d I_{i,\ep}[\check{H}^{(k)},\check{H}^{(m)}] = 0$ for $i=1,2,3$. If $12 \le n \le 32$, then Lemma \ref{lemma of case 3} tells us that there exists a constant $C=C(n)>0$ such that
\begin{equation}\label{eq:Idecom3}
\sum_{k,m=4}^d I_{\ep}\big[\check{H}^{(k)},\check{H}^{(m)}\big] \ge C^{-1} \sum_{k=4}^d \ep^{2k}|\log\ep|^{\theta_k} \big\|\hat{H}^{(k)}\big\|^2 + O(\ep^{n-4}).
\end{equation}

\medskip
As a consequence, by putting \eqref{eq:Idecom1}, \eqref{eq:Idecom2} and \eqref{eq:Idecom3} into \eqref{eq:Idecom}, and employing the identity
$\|H^{(k)}\|^2 = \|\hat{H}^{(k)}\|^2 + \|\hat{W}^{(k)}\|^2 + \|\hat{D}^{(k)}\|^2$ for $k = 2,\ldots,d$ as well as the equivalence between two norms $|H^{(k)}|$ and $\|H^{(k)}\|$, we conclude the proof of Proposition \ref{Positive definiteness}.
\end{proof}

\subsection{Proof of Theorem \ref{thm:Weyl2}}\label{subsec:Weyl23}
By virtue of \eqref{eq:Pohoineq2}, Lemma \ref{lemma:diffmciI} with $\eta \in (0,1)$ small, and Proposition \ref{Positive definiteness}, there exists a constant $C=C(n)>0$ such that
\[O(\ep_a^{n-4}) \ge C^{-1} \sum_{k=2}^d \ep_a^{2k}|\log\ep_a|^{\theta_k}|H_a^{(k)}|^2.\]
Therefore,
\begin{equation}\label{eq:Hak dec}
|H_a^{(k)}|^2 \le C\ep_a^{n-4-2k}|\log\ep_a|^{-\theta_k} \quad \text{for } k=2,\ldots,d,
\end{equation}
which implies
\[|\nabla^k_{g_a}W_{g_a}|^2(\sigma_a) \le C\ep_a^{n-8-2k}|\log\ep_a|^{-\theta_{k+2}} \quad \text{for } k=0,\ldots,d-2.\]
This concludes the proof of Theorem \ref{thm:Weyl2}.

\subsection{A corollary of Theorem \ref{thm:Weyl2}}
By Theorem \ref{thm:Weyl2} (or \eqref{eq:Hak dec}) and Corollary \ref{cor:scale back}, a direct corollary follows.

\begin{cor}\label{cor:scale back improved}
Under the conditions of Theorem \ref{thm:Weyl2}, there exist $\msfd_0 \in (0,1)$ and $C>0$ such that
$$
\left|\pa_{\beta}(u_a-w_a-\Psi_a)\right|(x) \le C\ep_a^{\frac{n-4}{2}+|\beta|}(\ep_a+|x|)^{-1-|\beta|}
$$
for all $a \in \N$, multi-indices $\beta$ with $|\beta| = 0,1,\ldots,4$, and $|x|\le \msfd_0$.
\end{cor}

\section{Local Sign Restriction}\label{sec:lsr}
A slight modification of the proof of Theorem \ref{thm:Weyl2} yields the following local sign restriction of the function $\mbp^{(4)}$ in \eqref{eq:Poho}. We skip the proof.

\begin{prop}\label{prop:lsr}
Let $8 \le n \le 24$ and $\sigma_a \to \bsi \in M$ be an isolated simple blowup point for a sequence $\{u_a\}$ of solutions to \eqref{eq:maina} with $\msfk=2$, where $\det g_a = 1$ near $\sigma_a \in M$.
Assume that $u_a(\sigma_a)u_a \to \mfg$ in $C^3_{\mathrm{loc}}(B^{g_{\infty}}_{\msfd_0}(\bsi) \setminus \{\bsi\})$, where $\msfd_0 \in (0,1)$ is the number in Proposition \ref{prop:refest} and $\mfg$ is the function in Proposition \ref{prop:isosim decay}. Then
$$
\liminf_{r\to 0}\mbp^{(4)}(r,\mfg)\ge 0.
$$
\end{prop}
We also recall a well-known result, which can be found in e.g. \cite[Proposition 7.1]{LX}.
\begin{prop}\label{global prop}
For any given $\ep, R>0$, there exists a constant $C=C(M,g,\ep,R)>0$ such that if a solution $u$ of equation \eqref{eq:main4} satisfies $\max_{\sigma\in M} u(\sigma) \ge C$, then there exist local maximum points $\sigma^1,\ldots,\sigma^N \in M$, where $N \in \N$ depends on $u$, such that
\begin{itemize}
\item[(i)] $\{B_{r_i}(\sigma^i)\}_{1\le i\le N}$ are disjoint, where $r_i=Ru(\sigma^i)^{-\frac{2}{n-4}}$;
\item[(ii)] For each $i=1,\ldots,N$, it holds that $\norm{C^4(B_R)}{u(\sigma^i)^{-1} u(u(\sigma^i)^{-\frac{2}{n-4}}\cdot)-w}<\ep$;
\item[(iii)] $u(\sigma) \le Cd_g(\sigma,\{\sigma^1,\ldots,\sigma^N\})^{-\frac{n-4}{2}}$ for any $\sigma\in M \setminus \{\sigma^1,\ldots,\sigma^N\}$.
\end{itemize}
\end{prop}

Finally, the following corollaries can be achieved as in \cite[Proposition 6.2, Proposition 7.2]{LX}.
\begin{cor}\label{cor:iso sim}
Let $8 \le n \le 24$ and $\sigma_a \to \bsi \in M$ be an isolated blowup point for a sequence $\{u_a\}$ of solutions to \eqref{eq:maina} with $\msfk=2$, where $\det g_a = 1$ near $\sigma_a \in M$. Then $\sigma_a \to \bsi \in M$ is isolated simple.
\end{cor}

\begin{cor}\label{cor:iso}
Let $8 \le n\le 24$. Let $\ep,\, R,\, C=C(M,g,\ep,R),\, u$ and $\{\sigma^1,\ldots,\sigma^N\}$ be as in Proposition \ref{global prop}.
Suppose that $\ep > 0$ is sufficiently small and $R>0$ is sufficiently large. Then, there exists $\oc=\oc(M,g,\ep, R) > 0$ such that if $\max_{\sigma\in M} u(\sigma)\ge C$, then $d_g(\sigma^i,\sigma^j)\ge \oc$ for any $i\neq j\in \{1,\ldots, N\}$.
\end{cor}

\section{Compactness of the Constant $Q^{(4)}$-curvature Problem}\label{sec:comp}
Because $\mathrm{Ker}P_g^{(4)} = \{0\}$, the Green's function $G_g^{(4)}$ of the Paneitz operator $P_g^{(4)}$ has the expansion
\begin{equation}\label{expansion of Green's function 1}
G_g^{(4)}(x,0)=\vsi_n^{-1}|x|^{4-n}\left[1+\sum_{k=4}^{n-4}\psi^{(k)}(x)\right]+A+B\log|x|+O^{(4)}(|x|)
\end{equation}
in normal coordinates $x$ centered at a point $\bsi \in M$, provided $\det g = 1$ near $\bsi$. Here, $\vsi_n=2(n-4)(n-2)|\S^{n-1}|$, $A, B \in \R$, $\psi^{(k)}\in \mcp_k$, and $\int_{\S^{n-1}}\psi^{(n-4)}=0$. If $n$ is odd, we also have $B=0$.
For its proof, refer to \cite[Subsection 2.4]{HY}.

\begin{lemma}\label{lemma:expansion of Green's function}
Let $8\leq n\leq 24$ and $\sigma_a \to \bsi \in M$ be an isolated simple blowup point for a sequence $\{u_a\}$ of solutions to \eqref{eq:maina} with $\msfk=2$, where $\det g_a = 1$ near $\sigma_a \in M$.
Assume that $u_a(\sigma_a)u_a \to \vsi_nG_{g_{\infty}}^{(4)}(\cdot,\bsi)+\mfh$ in $C^3_{\mathrm{loc}}(B^{g_{\infty}}_{\msfd_0}(\bsi) \setminus \{\bsi\})$,
where $\msfd_0 \in (0,1)$ is the number in Proposition \ref{prop:refest}, $G_{g_{\infty}}^{(4)}$ is the Green's function of the Paneitz operator $P_{g_{\infty}}^{(4)}$,
and $\mfh \in C^5(B^{g_{\infty}}_{\msfd_0}(\bsi))$ is a nonnegative function from Proposition \ref{prop:isosim decay}. Then it holds that
\begin{equation}\label{expansion of Green's function}
\begin{aligned}
G_{g_{\infty}}^{(4)}(x,0)&=\vsi_n^{-1}|x|^{4-n}\left[1+\sum_{k=d+1}^{n-4}\psi^{(k)}(x)\right]+A+O^{(4)}(|x|),\\
\text{where} &\quad \int_{\S^{n-1}}\psi^{(k)}=0 \quad \text{and} \quad \int_{\S^{n-1}}x_i\psi^{(k)}=0.
\end{aligned}
\end{equation}
\end{lemma}
\begin{proof}
First, we prove that for $k=4,\ldots,d$, $\psi^{(k)}=0$. Let $\Psi_a^{(\geq d+1)}$ be the solution to equation
$$
\Delta^2\Psi_a^{(\geq d+1)} - \tmfc_4(n)w_a^{\frac{8}{n-4}}\Psi_a^{(\geq d+1)} = - \sum_{k=d+1}^{K} L_1[H_a^{(k)}]w_a \quad \text{in } \R^n.
$$
From \eqref{eq:wtPsidec}, Theorem \ref{thm:Weyl2} and Corollary \ref{cor:scale back improved}, we obtain that
$$
\lim_{a \to 0} \ep_a^{-\frac{n-4}{2}}\left|u_a-w_a-\Psi_a^{(\geq d+1)}\right|(x) \le C|x|^{-1}
$$
for $0 < |x| \le \msfd_0$. Since
\[\ep_a^{-\frac{n-4}{2}}u_a \to \vsi_nG_{g_{\infty}}^{(4)}(\cdot,0)+\mfh \quad \text{and} \quad \ep_a^{-\frac{n-4}{2}}w_a \to |\cdot|^{4-n} \quad \text{in } C^3_{\mathrm{loc}}(B_{\msfd_0} \setminus \{0\})\]
and
\[\ep_a^{-\frac{n-4}{2}}\left|\wtPsi_a^{(\geq d+1)}\right|(x) \le C\sum_{k=d+1}^{K}(\ep_a+|x|)^{4-n+k} \to C\sum_{k=d+1}^{K}|x|^{4-n+k},\]
the function $\psi^{(k)}$ in \eqref{expansion of Green's function 1} vanishes for $k=4,\ldots, d$.

Next, for $k=d+1,\ldots, n-3$, we prove that $\int_{\S^{n-1}}\psi^{(k)} = B =0$. On the one hand, due to Theorem \ref{thm:Weyl2} and Lemmas \ref{spherical expansion} and \ref{orthogonal},
\begin{align*}
\int_{\pa B_r}(P_{g_{\infty}}^{(4)}-\Delta^2)(|x|^{4-n}) &= O(r^{2d+1})\leq Cr^{n-4};\\
\int_{\pa B_r}(P_{g_{\infty}}^{(4)}-\Delta^2)(|x|^{4-n}\psi^{(k)}) &= O(r^{d+k})\leq Cr^{n-4}.
\end{align*}
On the other hand, it holds that $\int_{\pa B_r}\Delta^2(|x|^{4-n})=0$, $\int_{\pa B_r}\Delta^2(|x|^{4-n}\psi^{(n-4)})=0$, and
\begin{align*}
\int_{\pa B_r}\Delta^2(|x|^{4-n}\psi^{(k)}) &= 4(n-4)(n-2)k(k+1) r^{k-1} \int_{\S^{n-1}}\psi^{(k)};\\
\int_{\pa B_r}\Delta^2(B\log |x|)&= -2(n-4)(n-2)|\S^{n-1}| B r^{n-5}.
\end{align*}
Because $G_{g_{\infty}}^{(4)}$ is the Green's function of $P_{g_{\infty}}^{(4)}$, we also have that for any $r \in (0,\msfd_0)$,
$$
\int_{\pa B_r}(P_{g_{\infty}}^{(4)}-\Delta^2)G_{g_{\infty}}^{(4)}(\cdot,0)dS = -\int_{\pa B_r}\Delta^2G_{g_{\infty}}^{(4)}(\cdot,0)dS.
$$
It follows that $\int_{\S^{n-1}}\psi^{(k)} = O(r)$ and $B = O(r)$. Taking $r \to 0$, we conclude that $\int_{\S^{n-1}}\psi^{(k)} = B = 0$.

Finally, in a similar manner, we can also derive that $\int_{\S^{n-1}}x_i\psi^{(k)} = 0$ for $k=d+1,\ldots, n-4$.
\end{proof}
\begin{rmk}
It is worth noting that in the proof, we have
$$
\ep_a^{-\frac{n-4}{2}}\Psi_a^{(\geq d+1)}\to\sum_{k=d+1}^K\psi^{(k)} \text{ in } C^3_{\mathrm{loc}}(B_{\msfd_0} \setminus \{0\}).
$$
By explicitly solving of the linearized equations according to Proposition \ref{prop:Psi}, we could provide an alternative method for computing the expansion of the Green's function. \hfill $\diamond$
\end{rmk}

The following lemma generalizes \cite[Proposition 6]{ALL}. The analogous result for the Yamabe problem is well-known; see, for example, \cite[Proposition 19]{Br}.
\begin{prop}\label{prop:Green}
Let $n\geq 5$ and $(M^n,g)$ be a closed manifold such that $\ker P_g^{(4)} = 0$. Suppose that the Weyl tensor vanishes up to an order greater than or equal to $d-2$ at a point $\bsi \in M$, and that the Green's function $G_g^{(4)}$ of $P_g^{(4)}$ satisfies the expansion \eqref{expansion of Green's function} in a small neighborhood of $\bsi$.
Then, $(M\setminus\{\bsi\}, \hat{g}:=(G_g^{(4)})^{\frac{4}{n-4}}g)$ is an asymptotically flat manifold with decay rate $\tau >\frac{n-4}{2}$ such that $Q_{\hat{g}}=0$.
Moreover, the fourth-order mass of $\hat{g}$ defined in \eqref{eq:hom} is a positive multiple of the constant $A$ in \eqref{expansion of Green's function}:
$$
m(\hat{g})=\frac{4(n-1)}{n-4}A.
$$
\end{prop}
\begin{proof}
In conformal normal coordinates centered at $\bsi$, we have $\tr\, g=n$ and $x_ig_{ij}(x)=x_j$.

We introduce an inversion variable given by $y=\vsi_n^{-\frac{2}{n-4}}\frac{x}{|x|^2}$, whose transition relation is $\pa_{y_i}=\vsi_n^{\frac{2}{n-4}}(|x|^2\delta_{ij}-2x_ix_j)\pa_{x_j}$. Then we have
\begin{align*}
\hat{g}_{ij}(y)&=\vsi_n^{\frac{4}{n-4}}|x|^4G(x)^{\frac{4}{n-4}}g_{ij}(x)\\
&=\left[1+\sum_{k=d+1}^{n-4}\psi^{(k)}(x)+\vsi_nA|x|^{n-4}+O^{(4)}(|x|^{n-3})\right]^{\frac{4}{n-4}}g_{ij}(x)\\
&:=(F(x))^{\frac{4}{n-4}}g_{ij}(x),
\end{align*}
where $g_{ij}(x)=\delta_{ij}+O(|x|^{d+1})$ by the Weyl vanishing condition. Hence $\hat{g}_{ij}(y)=\delta_{ij}+O(|y|^{-(d+1)})$. Since $d+1=\lfloor \frac{n-4}{2}\rfloor+1>\frac{n-4}{2}$, we deduce that $(M\setminus\{\bsi\}, \hat{g})$ is asymptotic flat.

Next, we evaluate the fourth-order mass:
\begin{equation}\label{eq:hom2}
\begin{aligned}
m(\hat{g})&=\lim_{\vep \to 0} \int_{\{|y|=\vsi_n^{-\frac{2}{n-4}}\vep^{-1}\}} \pa_{y_k}(\hat{g}_{ii,jj}-\hat{g}_{ij,ij})\frac{y_k}{|y|}dS\\
&=\lim_{\vep \to 0} \vsi_n^{-\frac{2(n-1)}{n-4}} \int_{\pa B_{\vep}}\pa_{y_k}(\hat{g}_{ii,jj}-\hat{g}_{ij,ij})\frac{x_k}{|x|}|x|^{2-2n}dS\\
&=-\lim_{\vep \to 0} \vsi_n^{-\frac{2(n-2)}{n-4}} \int_{\pa B_{\vep}}x_k\pa_{x_k}(\hat{g}_{ii,jj}-\hat{g}_{ij,ij})|x|^{3-2n}dS.
\end{aligned}
\end{equation}
It holds that
\begin{align*}
\vsi_n^{-\frac{2}{n-4}}\hat{g}_{ii,j} &= n(|x|^2\delta_{ij}-2x_ix_j)\pa_{x_i}(F^{\frac{4}{n-4}});\\
\vsi_n^{-\frac{4}{n-4}}\hat{g}_{ii,jj} &= n|x|^4\Delta(F^{\frac{4}{n-4}}) -2(n-2)n|x|^2x_i\pa_{x_i}(F^{\frac{4}{n-4}});\\
\vsi_n^{-\frac{2}{n-4}}\hat{g}_{ij,i} &= (|x|^2g_{ij}-2x_ix_j)\pa_{x_i}(F^{\frac{4}{n-4}}) +|x|^2g_{ij,i}F^{\frac{4}{n-4}};\\
\vsi_n^{-\frac{4}{n-4}}\hat{g}_{ij,ij} &= |x|^4g_{ij}\pa^2_{x_ix_j}(F^{\frac{4}{n-4}}) -2(n-2)|x|^2x_i\pa_{x_i}(F^{\frac{4}{n-4}})\\
&\ +2|x|^4g_{ij,i}\pa_{x_j}(F^{\frac{4}{n-4}})+|x|^4g_{ij,ij}F^{\frac{4}{n-4}}.
\end{align*}
Plugging the above computation of $\hat{g}_{ii,jj}-\hat{g}_{ij,ij}$ into \eqref{eq:hom2}, we obtain
\begin{align*}
m(\hat{g})&=2(n-2)(n-1)\vsi_n^{-2}\lim_{\vep \to 0} \int_{\pa B_{\vep}} (x_k\pa_{x_k}+2)\circ(x_l\pa_{x_l})(F^{\frac{4}{n-4}})|x|^{5-2n}dS\\
&\ -\vsi_n^{-2}\lim_{\vep \to 0} \int_{\pa B_{\vep}} (x_k\pa_{x_k}+4)\left[n\Delta(F^{\frac{4}{n-4}})-\pa^2_{x_ix_j}(g_{ij}F^{\frac{4}{n-4}})\right]|x|^{7-2n}dS.
\end{align*}
We now make use of the following assertion, whose derivation is deferred to the end of the proof.
\begin{claim}\label{claim:coarea}
$$
\int_{\pa B_{\vep}} (x_k\pa_{x_k}+4)\circ\pa^2_{x_ix_j}\((g_{ij}-\delta_{ij})F^{\frac{4}{n-4}}\)dS=0.
$$
\end{claim}
With Claim \ref{claim:coarea}, we obtain that
\begin{align*}
&\begin{medsize}
\displaystyle \ (n-1)^{-1}\vsi_n^2m(\hat{g})
\end{medsize} \nonumber \\
&\begin{medsize}
\displaystyle =\lim_{\vep \to 0} \int_{\pa B_{\vep}} \left[2(n-2)(x_k\pa_{x_k}+2)\circ(x_l\pa_{x_l})(F^{\frac{4}{n-4}})\vep^{5-2n} -(x_k\pa_{x_k}+4)\circ\Delta(F^{\frac{4}{n-4}})\vep^{7-2n}\right]dS
\end{medsize} \nonumber \\
&\begin{medsize}
\displaystyle =\lim_{\vep \to 0} \int_{\pa B_{\vep}} \left[\frac{8(n-2)}{n-4}(x_k\pa_{x_k}+2)(x_l\pa_{x_l}F)\vep^{5-2n} -\frac{4}{n-4}(x_k\pa_{x_k}+4)(\Delta F)\vep^{7-2n}\right.
\end{medsize} \nonumber \\
&\begin{medsize}
\displaystyle \hspace{185pt}\left.+\frac{4(n-8)}{(n-4)^2}(x_k\pa_{x_k}+4)((F_{,i})^2)\vep^{7-2n}\right]dS.
\end{medsize}
\end{align*}
Because $(F_{,i})^2=O(|x|^{2d})$ and $2d\geq n-5>n-6$, the limit of the last term vanishes. So,
\begin{align*}
&\ (n-1)^{-1}\vsi_n^2m(\hat{g})\\
&=\frac{8(n-2)}{n-4}\lim_{\vep\to 0}\vep^{5-2n}\int_{\pa B_{\vep}}(x_k\pa_{x_k}+2)\circ(x_l\pa_{x_l}) \left[\sum_{k=d+1}^{n-4}\psi^{(k)}(x)+\vsi_n A|x|^{n-4}\right]dS\\
&\ -\frac{4}{n-4}\lim_{\vep\to 0}\vep^{7-2n}\int_{\pa B_{\vep}}(x_k\pa_{x_k}+4)\circ\Delta \left[\sum_{k=d+1}^{n-4}\psi^{(k)}(x)+\vsi_n A|x|^{n-4}\right]dS.
\end{align*}
Using Lemma \ref{IBP} and $\int_{\S^{n-1}}\psi^{(k)}=0$ for $k=d+1,\ldots,n-4$, we arrive at
\begin{align*}
m(\hat{g})&=8(n-2)(n-1)|\S^{n-1}|\vsi_n^{-1}A.
\end{align*}

\noindent \medskip \textsc{Proof of Claim \ref{claim:coarea}.} Define
$$
I(\vep):=\int_{B^n(0,\vep)}(x_k\pa_{x_k}+4)\circ\pa^2_{x_ix_j}((g_{ij}-\delta_{ij})F^{\frac{4}{n-4}})dx.
$$
Since
\begin{align*}
&\ (x_k\pa_{x_k}+4)\circ\pa^2_{x_ix_j}((g_{ij}-\delta_{ij})F^{\frac{4}{n-4}})\\
&= \pa^2_{x_ix_j}\circ(x_k\pa_{x_k}+2)((g_{ij}-\delta_{ij})F^{\frac{4}{n-4}})\\
&= \pa^2_{x_ix_j}\left[(g_{ij}-\delta_{ij})(x_k\pa_{x_k}+2)(F^{\frac{4}{n-4}})\right] + \pa^2_{x_ix_j}\(x_kg_{ij,k}F^{\frac{4}{n-4}}\),
\end{align*}
we can employ the divergence theorem to find
\[I(\vep)=\frac{1}{\vep} \int_{\pa B_{\vep}} x_i\pa_{x_j}\left[(g_{ij}-\delta_{ij})(x_k\pa_{x_k}+2)(F^{\frac{4}{n-4}})\right] +x_i\pa_{x_j}\(x_kg_{ij,k}F^{\frac{4}{n-4}}\)dS.\]
Using $x_i(g_{ij}-\delta_{ij})=0$, $\tr(g-\delta)=0$, and $x_ix_kg_{ij,k}=x_k(\delta_{kj}-g_{kj})=0$, we obtain that $I(\vep)\equiv0$. Then we can apply $I'(\vep)\equiv0$ and the coarea formula to deduce Claim \ref{claim:coarea}.
\end{proof}

We now complete the proof of our main result.
\begin{proof}[Proof of Theorem \ref{thm:Weyl} under assumption (i)]
It is a simple consequence of Theorem \ref{thm:Weyl2} and Corollaries \ref{cor:iso sim}--\ref{cor:iso}.
\end{proof}

\begin{proof}[Proof of Theorem \ref{thm:main4}]
By elliptic regularity theory, it suffices to prove the uniform $L^{\infty}(M)$-boundedness of solutions of \eqref{eq:main4}.
Suppose by contradiction that there exist a sequence $\{g_a\}_{a \in \N} \subset [g]$ of metrics on $M$, a sequence $\{u_a\}_{a \in \N} \subset C^4(M)$, and a sequence $\{\sigma_a\}_{a \in \N} \subset M$ such that
\eqref{eq:maina} holds with $\msfk=2$, $\sigma_a \to \bsi \in M$ is a blowup point for $\{u_a\}$, $\det g_a = 1$ near $\sigma_a \in M$, and $g_a$ converges to a metric $g_{\infty} \in [g]$ in $C^l(M)$ as $a \to \infty$ for any $l \in \N$.
Corollaries \ref{cor:iso sim}--\ref{cor:iso} imply that $\sigma_a \to \bsi \in M$ is an isolated simple blowup point for $\{u_a\}$.

Owing to Proposition \ref{prop:isosim decay}, $u_a(\sigma_a)u_a \to \mfg := \vsi_nG_{g_{\infty}}^{(4)}(\cdot,\bsi)+\mfh$ in $C^3_{\mathrm{loc}}(B^{g_{\infty}}_{\msfd_0}(\bsi) \setminus \{\bsi\})$,
where $G_{g_{\infty}}^{(4)}$ is the Green's function of the Paneitz operator $P_{g_{\infty}}^{(4)}$ and $\mfh \in C^5(B^{g_{\infty}}_{\msfd_0}(\bsi))$ is a nonnegative function satisfying \eqref{eq:mfha}. Combining this fact with Proposition \ref{prop:lsr}, we reach
\begin{equation}\label{lsr in proof}
\liminf_{r\to 0}\mbp^{(4)}(r,G_{g_{\infty}}^{(4)})=\vsi_n^{-2}\liminf_{r\to 0}\mbp^{(4)}(r,\mfg)\ge 0.
\end{equation}

By virtue of Lemma \ref{lemma:expansion of Green's function}, $G_{g_{\infty}}^{(4)}$ has the expansion \eqref{expansion of Green's function}.
If Condition (i) holds, then the positivity of the constant $A$ of $G_{g_{\infty}}^{(4)}$ in \eqref{expansion of Green's function} follows immediately.
Under Condition (ii), we can deduce $A>0$ from Proposition \ref{prop:Green}, Theorem \ref{PMT2} and the condition that $M$ is not conformally diffeomorphic to the standard sphere $\S^n$. However, a direct computation using \eqref{expansion of Green's function} tells us
\[\lim_{r\to 0}\mbp^{(4)}(r,G_{g_{\infty}}^{(4)})=-\frac{n-4}{2}A<0,\]
which contradicts \eqref{lsr in proof}. This finishes the proof of Theorem \ref{thm:main4}.
\end{proof}

\section{Evaluation of the Pohozaev Quadratic Form}\label{sec:tech}
We examine the positivity of the Pohozaev quadratic form in three mutually exclusive cases, which is crucial in the proof of Proposition \ref{Positive definiteness}.
We set $\theta_k=1$ if $k=\frac{n-4}{2}$ and $\theta_k=0$ otherwise.
\begin{lemma}\label{lemma of case 1}
Assume that $8\le n\le 24$ and $s=2,\ldots, d$. Let $M_q^{(s)}$ be the matrix defined in \eqref{eq:MqsEDs} and $\ep > 0$ small. Then there exists a constant $C=C(n,s) > 0$ such that
\begin{multline*}
\sum_{q,q'=0}^{\lfloor \frac{d-s}{2} \rfloor} |\log\ep|^{\theta_{q+q'+s}} J_1\left[\ep^{2q+s}|x|^{2q}M_q^{(s)},\ep^{2q'+s}|x|^{2q'}M_{q'}^{(s)}\right] \\
\ge C^{-1} \sum_{q=0}^{\lfloor \frac{d-s}{2} \rfloor} \ep^{2(2q+s)}|\log\ep|^{\theta_{2q+s}} \left\||x|^{2q}M_q^{(s)}\right\|^2.
\end{multline*}
\end{lemma}
\begin{proof}
Fixing any $s = 2,\ldots, d$, let $k=2q+s$, $m=2q'+s$, $\lambda_q= -q(n+2q+2s-2)$, and $\lambda_{q'}= -q'(n+2q'+2s-2)$. By applying Corollary \ref{I_1 coro} and \eqref{eq:hatDq1}--\eqref{eq:hatDq2}, we compute
\begin{align*}
&\ J_1\left[|x|^{2q}M_q^{(s)},|x|^{2q'}M_{q'}^{(s)}\right] \\
&=\frac{(n-4)^2}{8(n-3)(n-2)(n-1)} (\mci_{n-3}^{n+k+m-3})^{1-\theta_{\frac{k+m}{2}}} \left[\frac{8(n-3)(n-1)(k+m)}{(n-2)(n+k+m-4)}\lambda_{q}\lambda_{q'}\right.\\
&\hspace{50pt} + \left.\frac{1}{8}c_1(n,k,m)\(2(\lambda_{q}+\lambda_{q'})+(n+k+m-2)(k+m)\) - c_2(n,k,m)\right] \bla M_q^{(s)},M_{q'}^{(s)} \bra\\
&:= (m^{D,s}_{qq'})' \bla M_q^{(s)},M_{q'}^{(s)} \bra,
\end{align*}
where
\begin{equation}\label{eq:c1c2nkm}
\begin{aligned}
c_1(n,k,m)&=(k+m)(n^3-(k+m+2)n^2+(6(k+m)-4)n-4(k+m)+8);\\
c_2(n,k,m)&=2(n-1)(k+m)(n+k+m-2)km.
\end{aligned}
\end{equation}
Then we set
\begin{equation}\label{eq:mDs}
m^{D,s}_{qq'}:=N_0^{\theta_{\frac{k+m}{2}}}(m^{D,s}_{qq'})'
\end{equation}
where $N_0 \in \N$ is taken to be large enough; for example, $N_0 = 10^{10}$ suffices.

Using Mathematica, we observe that matrices $(m^{D,s}_{qq'})$ are positive-definite for all $s=2,\ldots, d$ when $n\le 24$. In addition, $(m^{D,2}_{qq'})$ has a negative eigenvalue when $n\ge 25$.

Let $B'$ be a Gram matrix defined as $B'_{qq'} = \langle \ep^kM_q^{(s)},\ep^mM_{q'}^{(s)} \rangle$, which is always positive semi-definite.
Then, we set a matrix $B$ by $B_{qq'} = B'_{qq'}$ if $k+m<n-4$ and $B_{qq'} = |\log\ep|N_0^{-1}B'_{qq'}$ if $k+m=n-4$, which is also positive semi-definite for $\ep$ so small that $\ep \le e^{-N_0}$.
At this stage, we apply Lemma \ref{lemma:tr} with $l = \lfloor \frac{d-s}{2} \rfloor+1$, $A = m^{D,s}$, and the aforementioned $B$. It follows that \begin{align*}
\sum_{q,q'=0}^{\lfloor \frac{d-s}{2} \rfloor} |\log\ep|^{\theta_{\frac{k+m}{2}}} J_1\left[\ep^k|x|^{2q}M_q^{(s)},\ep^m|x|^{2q'}M_{q'}^{(s)}\right] &= \tr(AB) \ge C^{-1}\tr(B) \\
&= C^{-1} \sum_{q=0}^{\lfloor \frac{d-s}{2} \rfloor} |\log\ep|^{\theta_k} \left\|\ep^k|x|^{2q}M_q^{(s)}\right\|^2. \qedhere
\end{align*}
\end{proof}

\begin{lemma}\label{lemma of case 2}
Assume that $10 \le n \le 28$ and $s=1,\ldots, d-2$. Let $V_q^{(s+1)}$ be the vector field defined in \eqref{eq:MqsEDs} and $\ep > 0$ small. Then there exists a constant $C=C(n,s) > 0$ such that
\begin{multline*}
\sum_{q,q'=1}^{\lfloor \frac{d-s}{2} \rfloor} |\log\ep|^{\theta_{q+q'+s}} J_1\left[\proj\left[|x|^{2q}\msd V_q^{(s+1)}\right], \proj\left[|x|^{2q'}\msd V_{q'}^{(s+1)}\right]\right] \\
\ge C^{-1} \sum_{q=1}^{\lfloor \frac{d-s}{2} \rfloor} \ep^{2(2q+s)}|\log\ep|^{\theta_{2q+s}} \left\|\proj\left[|x|^{2q}\msd V_q^{(s+1)}\right]\right\|^2.
\end{multline*}
\end{lemma}
\begin{proof}
Fixing any $s = 1,\ldots, d-2$, let $k=2q+s$, $m=2q'+s$, $\lambda_q= -\frac{1}{2}(n+s+2q-2)(s+2q)$, and $\lambda_{q'}= -\frac{1}{2}(n+s+2q'-2)(s+2q')$.
By applying Corollary \ref{I_1 coro}, \eqref{eq:hatWq1}--\eqref{eq:hatWq2}, and Lemma \ref{lemma:addition info of eigenvector}, we compute
\begin{align*}
&\begin{medsize}
\displaystyle \ J_1\left[\proj\left[|x|^{2q}\msd V_q^{(s+1)}\right],\proj\left[|x|^{2q'}\msd V_{q'}^{(s+1)}\right]\right]
\end{medsize} \\
&\begin{medsize}
\displaystyle =\frac{(n-4)^2}{8(n-3)(n-2)(n-1)} (\mci_{n-3}^{n+k+m-3})^{1-\theta_{\frac{k+m}{2}}} \left[\frac{4(n-3)(n-1)(k+m)}{(n-2)(n+k+m-4)}\(4s(n+s)\lambda_q\lambda_{q'} + kms^2(n+s)^2\)\right.
\end{medsize} \\
&\begin{medsize}
\displaystyle \ + \left.2s(n+s) \left\{\frac{1}{8}c_1(n,k,m)\(2(\lambda_q+\lambda_{q'})+(n+k+m-2)(k+m)\)-c_2(n,k,m)\right\}\right] \bla V_q^{(s+1)},V_{q'}^{(s+1)} \bra
\end{medsize} \\
&\begin{medsize}
\displaystyle := (m^{W,s}_{qq'})' \bla V_q^{(s+1)},V_{q'}^{(s+1)} \bra,
\end{medsize}
\end{align*}
where $c_1(n,k,m)$ and $c_2(n,k,m)$ are the numbers in \eqref{eq:c1c2nkm}. Then we set $m^{W,s}_{qq'}$ by exploiting \eqref{eq:mDs} in which all the superscripts $D$ are replaced with $W$.

Using Mathematica, we observe that matrices $(m^{W,s}_{qq'})$ are positive-definite for all $s=1,\ldots, d-2$ when $n\le 28$. In addition, $(m^{W,1}_{qq'})$ has a negative eigenvalue when $n\ge 29$.

Following the rest of the proof of Lemma \ref{lemma of case 1}, we complete the proof.
\end{proof}

\begin{lemma}\label{lemma of case 3}
Assume that $12 \le n \le 32$ and $s=2,\ldots, d-2$. Let $\check{H}_q^{(2q+s)}$ be the matrix defined in \eqref{eq:PqsEHs} and $\ep > 0$ small. Then there exists a constant $C=C(n,s) > 0$ such that
\[\sum_{q,q'=1}^{\lfloor \frac{d-s}{2} \rfloor} I_{\ep}\left[\check{H}_q^{(2q+s)},\check{H}_{q'}^{(2q'+s)}\right] \ge C^{-1} \sum_{q=1}^{\lfloor \frac{d-s}{2} \rfloor} \ep^{2(2q+s)}|\log\ep|^{\theta_{2q+s}} \big\|\hat{H}_q^{(2q+s)}\big\|^2.\]
\end{lemma}
\begin{proof}
We continue to use the notations introduced in \eqref{eq:PqsEHs}. Fixing any $s=2,\ldots,d-2$, let $k=2q+s$ and $m=2q'+s$.
We recall from Lemma \ref{H hat} that each $\hat{H}_q^{(2q+s)} = \proj[|x|^{2q+2}\nabla^2P_q^{(s)}]$ is an eigenvector of the operator $\mcl_k$ corresponding to the eigenvalue
\[A_{2q+s,q}=(s-1)\left[2-\frac{n-2}{n-1}(n+s-1)\right]-(q+1)(n+2q+2s-4).\]
We write $\lambda_q= A_{2q+s,q}$, $\lambda_{q'}= A_{2q'+s,q'}$, and $\ka_s:=\frac{n-2}{n-1}s(s-1)(n+s-1)(n+s-2)$.

By applying Corollary \ref{I_1 coro} and \eqref{eq:hatHq1}--\eqref{eq:hatHq2} and Lemma \ref{lemma:addition info of eigenvector}, we find
\begin{equation}\label{eq:J_1 H}
\begin{aligned}
&\ J_1 \left[\hat{H}_q^{(k)},\hat{H}_{q'}^{(m)}\right]\\
&= \frac{(n-4)^2}{8(n-3)(n-2)(n-1)} (\mci_{n-3}^{n+k+m-3})^{1-\theta_{\frac{k+m}{2}}}\ka_s \left[\frac{8(n-3)(n-1)(k+m)\lambda_q\lambda_{q'}}{(n-2)(n+k+m-4)}\right.\\
&\hspace{30pt} +\frac{4(n-3)(k+m)(s-1)(n+s-1)km}{n+k+m-4} - \frac{(n-4)(n-3)n^2(k+m)\ka_s}{2(n-2)(n-1)(n+k+m-4)}\\
&\hspace{30pt} +\left.\frac{1}{8}c_1(n,k,m)\(2(\lambda_q+\lambda_{q'})+(n+k+m-2)(k+m)\) - c_2(n,k,m)\right] \bla P^{(s)}_q,P^{(s)}_{q'} \bra \\
&:= (m^{H,s,1}_{qq'})' \bla P^{(s)}_q,P^{(s)}_{q'} \bra,
\end{aligned}
\end{equation}
where $c_1(n,k,m)$ and $c_2(n,k,m)$ are the numbers in \eqref{eq:c1c2nkm}. Then we set $m^{H,s,1}_{qq'}$ as in \eqref{eq:mDs}.

For the $I_{2,\ep}[\check{H}^{(k)}_{q},\check{H}^{(m)}_{q'}]$ term, we remind from \eqref{eq:L_1mfl_1}, \eqref{eq:ddhatH}, and $s=(k-2)-2(q-1)$ that
\begin{equation}\label{eq:L_1Hw}
\begin{aligned}
L_1[\hat{H}^{(k)}_q]w &= \mfl_1[\delta^2\hat{H}^{(k)}_q]w = \ka_s \mfl_1[|x|^{2(q-1)}P_q^{(s)}]w \\
&= -\frac{(n-4)\ka_s}{4(n-1)} \sum_{i=1}^{q+3}b_i(n,k-2,q-1)(1+r^2)^{-\frac{n+6-2i}{2}}P_q^{(s)};\\
\Psi[\hat{H}^{(k)}_q] &= \frac{(n-4)\ka_s}{4(n-1)} \sum_{j=1}^{q+2}\Gamma_j(n,k-2,q-1)(1+r^2)^{-\frac{n-2j}{2}}P_q^{(s)},
\end{aligned}
\end{equation}
where $r=|x|$, and $b_i(n,k-2,q-1)$ and $\Gamma_j(n,k-2,q-1)$ are the numbers appearing in Definition \ref{linearized eqn RHS} and Proposition \ref{prop:Psi}, respectively. We also observe
\begin{multline*}
\(x_i\pa_i+\frac{n-4}{2}\)\((1+r^2)^{-\frac{n-2j}{2}}P_q^{(s)}\) \\
=(n-2j)(1+r^2)^{-\frac{n+2-2j}{2}}P_q^{(s)} - \frac{1}{2}(n+4-4j-2s)(1+r^2)^{-\frac{n-2j}{2}}P_q^{(s)}.
\end{multline*}
Making use of polar coordinates, we determine
\begin{align*}
&\ \int_{B^n(0,\msfd_0\ep^{-1})} \(y_i\pa_i+\frac{n-4}{2}\)\Psi[\check{H}^{(k)}_{q}] L_1[\check{H}^{(m)}_{q'}]w dy \\
&= \ep^{k+m}\left[\frac{(n-4)\ka_s}{4(n-1)}\right]^2 \sum_{i=1}^{q'+3}\sum_{j=1}^{q+2} b_i\Gamma_j \left[-(n-2j) \int_0^{\msfd_0\ep^{-1}} \frac{r^{n-1+2s}}{(1+r^2)^{n+4-i-j}} dr \right. \\
&\hspace{140pt} \left. + \frac{1}{2} (n+4-4j-2s)\int_0^{\msfd_0\ep^{-1}} \frac{r^{n-1+2s}}{(1+r^2)^{n+3-i-j}} dr\right] \bla P^{(s)}_q,P^{(s)}_{q'} \bra \\
&= \ep^{k+m}|\log\ep|^{\theta_{\frac{k+m}{2}}} (m^{H,s,2}_{qq'})' \bla P^{(s)}_q,P^{(s)}_{q'} \bra + O(\ep^{n-4}),
\end{align*}
where $b_i = b_i(n,m-2,q'-1)$, $\Gamma_j = \Gamma_j(n,k-2,q-1)$, and
\[(m^{H,s,2}_{qq'})' := \begin{cases}
\begin{medsize}
\displaystyle -\frac{1}{2}\left[\frac{(n-4)\ka_s}{4(n-1)}\right]^2 \sum_{i=1}^{q'+3}\sum_{j=1}^{q+2}b_i\Gamma_j \left[\frac{(n-2s-2i-2j+6)(n-2j)}{n+3-i-j}-(n+4-4j-2s)\right] \mci_{n+3-i-j}^{n-1+2s}
\end{medsize} \\
\hfill \begin{medsize}
\displaystyle \text{if } k+m=2(q+q'+s) < n-4,
\end{medsize} \\
\begin{medsize}
\displaystyle 0
\end{medsize}
\hfill \begin{medsize}
\displaystyle \text{if } k+m=2(q+q'+s) = n-4.
\end{medsize}
\end{cases}\]
We define
\begin{equation}\label{eq:mHs23}
m^{H,s,2}_{qq'} := \frac{1}{2} N_0^{\theta_\frac{k+m}{2}}\left[(m^{H,s,2}_{qq'})' + (m^{H,s,2}_{q'q})'\right]
\end{equation}
where $N_0 \in \N$ is taken to be large enough; for example, $N_0 = 10^{10}$ suffices. Clearly, $m^{H,s,2}_{qq'}=0$ when $k+m = n-4$. Then, by \eqref{eq:I_2},
\begin{equation}\label{eq:I_2H}
I_{2,\ep}\left[\check{H}^{(k)}_{q},\check{H}^{(m)}_{q'}\right] = \ep^{k+m} \left[\frac{|\log\ep|}{N_0}\right]^{\theta_{\frac{k+m}{2}}} m^{H,s,2}_{qq'} \bla P^{(s)}_q,P^{(s)}_{q'} \bra + O(\ep^{n-4}).
\end{equation}

For the $I_{3,\ep}[\check{H}^{(k)}_{q},\check{H}^{(m)}_{q'}]$ term, we need to evaluate
\begin{equation}\label{eq:L_1HZ0}
L_1[\hat{H}_q^{(k)}]Z=-\frac{n-4}{2}L_1[\hat{H}_q^{(k)}] (1+r^2)^{-\frac{n-4}{2}}+(n-4)L_1[\hat{H}_q^{(k)}](1+r^2)^{-\frac{n-2}{2}}.
\end{equation}
We have \eqref{eq:L_1Hw} and
\begin{equation}\label{eq:L_1HZ01}
L_1[\hat{H}^{(k)}_q](1+r^2)^{-\frac{n-2}{2}} =- \frac{\ka_s}{4(n-1)}\sum_{i=1}^{q+3} b'_i(n,k-2,q-1)(1+r^2)^{-\frac{n+8-2i}{2}}P_q^{(s)},
\end{equation}
where $b'_i(n,k-2,q-1)$ is the number appearing in Definition \ref{linearized eqn RHS}. It follows that
\begin{align*}
&\ \int_{B^n(0,\msfd_0\ep^{-1})} \Psi[\check{H}^{(k)}_{q}]L_1[\check{H}^{(m)}_{q'}]Z dy \\
&= \ep^{k+m} \left[\frac{(n-4)\ka_s}{4(n-1)}\right]^2 \sum_{i=1}^{q'+3}\sum_{j=1}^{q+2} \Gamma_j \left[-b'_i \int_0^{\msfd_0\ep^{-1}} \frac{r^{n-1+2s}}{(1+r^2)^{n+4-i-j}} dr \right. \\
&\hspace{160pt} \left. + \frac{n-4}{2}b_i \int_0^{\msfd_0\ep^{-1}} \frac{r^{n-1+2s}}{(1+r^2)^{n+3-i-j}} dr\right] \bla P^{(s)}_q,P^{(s)}_{q'} \bra\\
&= \ep^{k+m}|\log\ep|^{\theta_{\frac{k+m}{2}}} (m^{H,s,3}_{qq'})' \bla P^{(s)}_q,P^{(s)}_{q'} \bra + O(\ep^{n-4}),
\end{align*}
where $b_i = b_i(n,m-2,q'-1)$, $b'_i = b'_i(n,m-2,q'-1)$, $\Gamma_j = \Gamma_j(n,k-2,q-1)$, and
\[(m^{H,s,3}_{qq'})' := \begin{cases}
\begin{medsize}
\displaystyle -\frac{1}{2} \left[\frac{(n-4)\ka_s}{4(n-1)}\right]^2 \sum_{i=1}^{q'+3}\sum_{j=1}^{q+2}\Gamma_j \left[\frac{n-2s-2i-2j+6}{n+3-i-j}b'_i-(n-4)b_i\right] \mci_{n+3-i-j}^{n-1+2s}
\end{medsize} \\
\hfill \begin{medsize}
\displaystyle \text{if } k+m=2(q+q'+s) < n-4,
\end{medsize} \\
\begin{medsize}
\displaystyle \frac{n-4}{2} \left[\frac{(n-4)\ka_s}{4(n-1)}\right]^2 \Gamma_{q+2}b_{q'+3}
\end{medsize}
\hfill \begin{medsize}
\displaystyle \text{if } k+m=2(q+q'+s) = n-4.
\end{medsize}
\end{cases}\]

If we define $m^{H,s,3}_{qq'}$ as in \eqref{eq:mHs23}, then by \eqref{eq:I_3},
\begin{equation}\label{eq:I_3H}
I_{3,\ep}\left[\check{H}^{(k)}_{q},\check{H}^{(m)}_{q'}\right] = \ep^{k+m} \left[\frac{|\log\ep|}{N_0}\right]^{\theta_{\frac{k+m}{2}}} m^{H,s,3}_{qq'} \bla P^{(s)}_q,P^{(s)}_{q'} \bra + O(\ep^{n-4}).
\end{equation}

Adding up \eqref{eq:I_1J_1}, \eqref{eq:J_1 H}, \eqref{eq:I_2H}, and \eqref{eq:I_3H}, we arrive at
$$
(I_{1,\ep}+I_{2,\ep}+I_{3,\ep})\left[\check{H}_q^{(k)},\check{H}_{q'}^{(m)}\right] = \ep^{k+m} \left[\frac{|\log\ep|}{N_0}\right]^{\theta_{\frac{k+m}{2}}} m^{H,s}_{qq'} \bla P^{(s)}_q,P^{(s)}_{q'} \bra + O(\ep^{n-4}),
$$
where $m^{H,s}_{qq'} := m^{H,s,1}_{qq'} + m^{H,s,2}_{qq'} + m^{H,s,3}_{qq'}$. Using Mathematica, we observe that matrices $(m^{H,s}_{qq'})$ are positive-definite for all $s=2,\ldots, d-2$ when $n\le 32$. In addition, $(m^{H,2}_{qq'})$ has a negative eigenvalue when $n\ge 33$.

Following the rest of the proof of Lemma \ref{lemma of case 1}, we complete the proof.
\end{proof}
\begin{rmk}
In the Yamabe case, the Pohozaev quadratic form on the subspace $\oplus_{k=2}^d \mcd_k$ fails to be positive for $n \ge 25$, as shown in \cite[Proposition A.8]{KMS}.
In the proof of Lemma \ref{lemma of case 1}, we saw that the same phenomenon happens for the $Q^{(4)}$-curvature problem \eqref{eq:main4}.
Wei and Zhao \cite{WZ} implicitly used this fact to construct blowup examples for \eqref{eq:main4} provided $n \ge 25$. \hfill $\diamond$
\end{rmk}

\section{Compactness of the Constant $Q^{(6)}$-curvature Problem}\label{sec:comp6}
Let $(M,g)$ be a smooth compact Riemannian manifold of dimension $n \ge 3$ and $n\neq 4$.
Let $\rm_g$, $\ricci_g$, $A_g$, and $W_g:=\rm_g-A_g\KN g$ be, the Riemann curvature tensor on $(M,g)$, the Ricci curvature tensor, the Schouten tensor, and the Weyl tensor, respectively.
Besides, let $C_g$ and $B_g$ be the Cotton tensor and the Bach tensor defined as
\begin{align*}
(C_g)_{ijk}&:= (A_g)_{ij;k}-(A_g)_{ik;j};\\
(B_g)_{ij}&:=(C_g)_{ijk;l}g^{kl}+(W_g)_{kijl}(A_g)^{kl},
\end{align*}
respectively, where $(A_g)^{kl}:=(A_g)_{ij}g^{ki}g^{lj}$ and $g^{ij}$ is the inverse of $g_{ij}$.

Let
\begin{align}
T_2&:=(n-2)\sigma_1(A_g)g-8A_g; \label{eq:T2} \\
T_4&:=(n-6)\Delta_g\sigma_1(A_g)g-\frac{16}{n-4}B_g \nonumber \\
&\ -\frac{3n^2-12n-4}{4}\sigma_1(A_g)^2g+4(n-4)|A_g|^2_g g+8(n-2)\sigma_1(A_g)A_g-48A_g^2; \label{eq:T4} \\
v_6&:=-\frac{1}{8}\sigma_3(A_g)-\frac{1}{24(n-4)}(A_g,B_g)_g, \nonumber
\end{align}
where $|A_g|^2_g:=(A_g,A_g)_g=(A_g)^{ij}(A_g)_{ij}$, and $(A_g^2)_{ij}:=(A_g)_{il}g^{kl}(A_g)_{kj}$. Also, $\sigma_k(A_g)$ denotes the $k$-th symmetric function of the eigenvalues of $A_g$ and $\Delta_g = \diver_g \nabla_g$ is the Laplace-Beltrami operator.

By Juhl's formula \cite[Theorem 10.2]{J}, the sixth-order $Q$-curvature $Q_g^{(6)}$ and the sixth-order GJMS operator $P_g^{(6)}$ are then defined as follows:
\begin{equation}\label{eq:Q6}
\begin{aligned}
Q_g^{(6)} &:= \Delta_g^2\sigma_1(A_g)-\frac{n+2}{2}\Delta_g(\sigma_1(A_g)^2)+4\Delta_g|A_g|^2_g+8\diver_g\left\{A_g(\nabla_g \sigma_1(A_g),\cdot)\right\}\\
&\ -3!2^6v_6-\frac{n-6}{2}\sigma_1(A_g)\Delta_g\sigma_1(A_g) -4(n-6)\sigma_1(A_g)|A_g|^2_g +\frac{(n-6)(n+6)}{4}\sigma_1(A_g)^3,
\end{aligned}
\end{equation}
and
\begin{equation}\label{eq:P6}
\begin{aligned}
P_g^{(6)}u &:= \Delta_g^3u+\frac{n-2}{2}\Delta_g(\sigma_1(A_g)\Delta_gu) - \frac{n-6}{2} Q_g^{(6)}u \\
&\ -\Delta_g\diver_g\left\{T_2(\nabla_g u,\cdot)\right\}-\diver_g\left\{T_2(\nabla_g \Delta_gu,\cdot)\right\}-\diver_g\left\{T_4(\nabla_g u,\cdot)\right\}.
\end{aligned}
\end{equation}
It follows immediately from \eqref{eq:P6} that $P_g^{(6)}$ is self-adjoint and the constant $Q^{(6)}$-curvature problem \eqref{eq:main6} has a variational structure: A solution to \eqref{eq:main6} is a positive critical point of the energy functional
\begin{equation}\label{eq:mcf}
\mcf_g(u) := \frac{1}{2} \int_M u (-P_g^{(6)}u) dv_g - \frac{n-6}{2n}\mfc_6(n) \int_M |u|^{\frac{2n}{n-6}} dv_g \quad \text{for } u \in \dot{H}^3(\R^n).
\end{equation}
Also, as pointed out in \cite{CH}, there holds that
\begin{equation}\label{eq:mcf2}
\begin{aligned}
&\ -\int_{\R^n} \(uP_g^{(6)}u + |\nabla_g\Delta_g u|_g^2\) dv_g\\
&=\int_{\R^n} \left[-2T_2(\nabla_g\Delta_gu,\nabla_g u)-T_4(\nabla_g u,\nabla_g u)-\frac{n-2}{4(n-1)}R_g(\Delta_gu)^2+\frac{n-6}{2}Q^{(6)}_g u^2\right] dv_g.
\end{aligned}
\end{equation}
We will extensively use this structure in Section \ref{sec:noncomp6}.

\subsection{Expansions of curvatures}\label{subsec:curv6}
Adapting the ideas in Section \ref{sec:curv}, we carry out expansions of several geometric quantities such as the Schouten tensor, the Weyl tensor, the Cotton tensor, and the Bach tensor in conformal normal coordinates.
Then, by exploiting the explicit expression of $Q_g^{(6)}$ and $P_g^{(6)}$ in \eqref{eq:Q6}--\eqref{eq:P6}, we expand them.

\medskip
The proofs of Lemmas \ref{lemma:diver}--\ref{Schouten tensor}, Lemmas \ref{Weyl curva}--\ref{Bach tensor}, and Lemma \ref{T2T4} are straightforward. We skip them.
\begin{lemma}\label{lemma:diver}
It holds that
\begin{align}
(\diver_g A_g)_{i}&=(\tr_g A_g)_{,i}; \label{eq:diver schouten 0} \\
(W_g)_{ijkp;q}g^{pq}&=(n-3)(C_g)_{kji}; \label{eq:cor weyl} \\
(\diver_gB_g)_j&=(n-4)A_g^{ik}[(A_g)_{ki;j}-(A_g)_{kj;i}]. \label{eq:diver bach 0}
\end{align}
\end{lemma}

\begin{lemma}
Let $g=\exp(h)$, $\tr \,h=0$, $\Gamma_{ij}^k$ be a Christoffel symbol on $(M,g)$, and $\rm_g$ be the $(0,4)$-Riemann curvature tensor. Then
\begin{align*}
\Gamma_{ij}^k&= \Dot{\Gamma}_{ij}^k[h]+O(|h||\pa h|);\\
\rm_g&= \Dot{\rm}[h]+O(|h||\pa^2h|+|\pa h|^2),
\end{align*}
where
\begin{align*}
2\Dot{\Gamma}_{ij}^k[h] &:= h_{ik,j}+h_{jk,i}-h_{ij,k};\\
2\Dot{\rm}_{ijkl}[h] &:= 2\Dot{\Gamma}_{kj,i}^l[h]-2\Dot{\Gamma}_{ki,j}^l[h]=h_{lj,ki}+h_{ik,lj}-h_{jk,li}-h_{li,kj}.
\end{align*}
\end{lemma}

\begin{lemma}\label{Schouten tensor}
Let $g=\exp(h)$ and $\tr \,h=0$. Then
$$
A_{g}=\Dot{A}[h]+\Ddot{A}[h,h]+O(|h|^2|\pa^2h|+|h||\pa h|^2),
$$
where
\begin{align*}
\Dot{A}_{ij}[h] &:= \frac{1}{n-2}\Dot{\ricci}_{ij}[h]-\frac{1}{2(n-2)(n-1)}\Dot{R}[h]\delta_{ij};\\
\Ddot{A}_{ij}[h,h] &:= \frac{1}{n-2}\Ddot{\ricci}_{ij}[h,h] - \frac{1}{2(n-2)(n-1)}(\Ddot{R}[h,h]\delta_{ij}+\Dot{R}[h]h_{ij}).
\end{align*}
\end{lemma}

\begin{cor}\label{cor:diver schouten}
Under the same assumption, we have
\[\Dot{\Gamma}_{ii}^j[h]=\delta_jh,\quad \Dot{\Gamma}_{ik}^l[h]+\Dot{\Gamma}_{il}^k[h]=h_{kl,i},\quad \Dot{\rm}_{lijl}[h]=\Dot{\ricci}_{ij}[h].\]
From the definition of $A_g$, it readily follows that
\begin{align}
\Ddot{R}[h,h]&=2(n-1)(\tr\Ddot{A}[h,h]-\Dot{A}[h]\cdot h); \label{eq:DdotR} \\
\Dot{\ricci}_{ij}[h]&=(n-2)\Dot{A}_{ij}[h]+\tr\Dot{A}[h]\delta_{ij}. \nonumber
\end{align}
Moreover, there holds that
\begin{equation}\label{eq:diver schouten}
\begin{aligned}
\delta_i\Dot{A}[h] &= \tr\Dot{A}_{,i}[h];\\
\delta_i\Ddot{A}[h,h] &= (\tr\Ddot{A}[h,h]-\Dot{A}[h]\cdot h)_{,i} + h_{kl}\Dot{A}_{ik,l}[h]+\Dot{\Gamma}_{ll}^k[h]\Dot{A}_{ik}[h]+\Dot{\Gamma}_{ki}^l[h]\Dot{A}_{kl}[h].
\end{aligned}
\end{equation}
\end{cor}
\begin{proof}
We verify \eqref{eq:diver schouten} only. If we write
\[\diver_gA_g:=\Dot{\diver_gA}+\Ddot{\diver_gA}+O(|h|^2|\pa^3h|+|h||\pa h||\pa^2h|+ |\pa h|^3),\]
then
\begin{align*}
(\Dot{\diver_gA})_i &= \delta_i\Dot{A};\\
(\Ddot{\diver_gA})_i &= \delta_i\Ddot{A}-h_{jl}\Dot{A}_{ij,l}-\Dot{\Gamma}_{ji}^k\Dot{A}_{kj} -\Dot{\Gamma}_{jj}^k\Dot{A}_{ik}.
\end{align*}
Thus, \eqref{eq:diver schouten} holds as a consequence of \eqref{eq:diver schouten 0} and $\tr A_g = \frac{1}{2(n-1)}R$.
\end{proof}

\begin{lemma}\label{Weyl curva}
Let $g=\exp(h)$ and $\tr \,h=0$. Then
$$
W_g= \Dot{W}[h]+O(|h||\pa^2h|+|\pa h|^2),
$$
where $\delta$ denotes the Kronecker delta and
$$
\Dot{W}[h]:=\Dot{\rm}[h]-\Dot{A}[h]\KN\delta.
$$
Also, by \eqref{eq:cor weyl}, there holds that
$$
\Dot{W}_{ijkl,l}[h]=(n-3)(\Dot{A}_{kj,i}[h]-\Dot{A}_{ki,j}[h]).
$$
\end{lemma}

\begin{lemma}\label{Cotton tensor}
Let $g=\exp(h)$ and $\tr \,h=0$. Then
$$
C_{g}=\Dot{C}[h]+\Ddot{C}[h,h]+O(|h|^2|\pa^3h|+|h||\pa h||\pa^2h|+ |\pa h|^3),
$$
where
\begin{align*}
\Dot{C}_{ijk}[h] &:= \Dot{A}_{ij,k}[h]-\Dot{A}_{ik,j}[h];\\
\Ddot{C}_{ijk}[h,h] &:= \Ddot{A}_{ij,k}[h,h]-\Ddot{A}_{ik,j}[h,h] -\Dot{\Gamma}_{ik}^l[h]\Dot{A}_{lj}[h]+\Dot{\Gamma}_{ij}^l[h]\Dot{A}_{lk}[h].
\end{align*}
\end{lemma}

\begin{lemma}\label{Bach tensor}
Let $g=\exp(h)$ and $\tr \,h=0$. Then
$$
B_{g}=\Dot{B}[h]+\Ddot{B}[h,h]+O(|h|^2|\pa^4h|+|h||\pa h||\pa^3h|+ |h||\pa^2h|^2 + |\pa h|^2|\pa^2h|),
$$
where
\begin{align*}
\Dot{B}_{ij}[h] &:= \Delta \Dot{A}_{ij}[h]-\Dot{A}_{il,jl}[h];\\
\Ddot{B}_{ij}[h,h] &:= \Ddot{C}_{ijk,k}[h,h] + \Dot{W}_{kijl}[h]\Dot{A}_{kl}[h] -\Dot{\Gamma}_{il}^k[h]\Dot{C}_{kjl}[h]-\Dot{\Gamma}_{jl}^k[h]\Dot{C}_{ikl}[h]-\Dot{\Gamma}_{ll}^k[h]\Dot{C}_{ijk}[h] \\
&\ -h_{kl}(\Dot{A}_{ij,kl}[h]-\Dot{A}_{ik,jl}[h]).
\end{align*}
\end{lemma}

\begin{cor}
Under the same assumption, we have $\Dot{B}_{ij}[h]=\Dot{B}_{ji}[h]$ and $\Ddot{B}_{ij}[h,h]=\Ddot{B}_{ji}[h,h]$. Moreover,
\begin{align}
\Dot{B}_{ij}[h] &= \Delta\Dot{A}_{ij}[h]-\tr\Dot{A}_{,ij}[h] \label{eq:dot bach}\\
&= \frac{1}{n-2}\Delta \Dot{\ricci}_{ij}[h] - \frac{1}{2(n-2)(n-1)}\Delta\Dot{R}[h]\delta_{ij}-\frac{1}{2(n-1)}\Dot{R}_{,ij}[h]; \nonumber \\
\Ddot{B}_{ij}[h,h] &= \Delta\Ddot{A}_{ij}[h,h]-(\tr\Ddot{A}[h,h]-\Dot{A}[h]\cdot h)_{,ij} \nonumber \\
&\ - \Dot{A}[h]\cdot\Dot{A}[h]\delta_{ij}-2\tr\Dot{A}[h]\Dot{A}_{ij}[h]-(n-4)\Dot{A}_{il}[h]\Dot{A}_{lj}[h] \nonumber \\
&\ -(h_{kl}\Dot{A}_{ij,k}[h])_{,l}-h_{kl,ij}\Dot{A}_{kl}[h] +2\Dot{\Gamma}_{ij,k}^l[h]\Dot{A}_{kl}[h]\label{eq:ddot bach} \\
&\ -(\Dot{\Gamma}_{ik,k}^l[h]\Dot{A}_{lj}[h]+\Dot{\Gamma}_{jk,k}^l[h]\Dot{A}_{li}[h]) \nonumber \\
&\ +\Dot{\Gamma}_{ij}^l[h]\tr\Dot{A}_{,l}[h]-2(\Dot{\Gamma}_{ik}^l[h]\Dot{A}_{jl,k}[h] + \Dot{\Gamma}_{jk}^l[h]\Dot{A}_{il,k}[h]) \nonumber.
\end{align}
By the identity $\tr_g B_g=0$, we have
\begin{equation}\label{eq:trace bach}
\begin{aligned}
\tr \Dot{B}[h] &= 0;\\
\tr \Ddot{B}[h,h] &=h\cdot \Dot{B}[h]=h\cdot \Delta\Dot{A}[h]-h_{ij}\tr\Dot{A}_{,ij}[h].
\end{aligned}
\end{equation}
Also, by \eqref{eq:diver bach 0}, there holds that
\begin{align*}
\delta_i\Dot{B}[h] &= 0;\\
\delta_i\Ddot{B}[h,h] &= (n-4)\Dot{A}_{kl}[h] (\Dot{A}_{kl,i}[h]-\Dot{A}_{ki,l}[h]) + h_{kl}\Dot{B}_{ik,l}[h]+\Dot{\Gamma}_{ll}^k[h]\Dot{B}_{ik}[h]+\Dot{\Gamma}_{ki}^l[h]\Dot{B}_{kl}[h].
\end{align*}
\end{cor}
\begin{proof}
Equation \eqref{eq:dot bach} is a direct corollary of \eqref{eq:diver schouten}.

Let us prove \eqref{eq:ddot bach}. We see from Lemma \ref{Weyl curva} that
\begin{equation}\label{eq:ddot bach 1}
\Dot{W}_{kijl}\Dot{A}_{kl} = (\Dot{\Gamma}_{ji,k}^l-\Dot{\Gamma}_{jk,i}^l)\Dot{A}_{kl} -\Dot{A}\cdot\Dot{A}\delta_{ij}-\tr\Dot{A}\Dot{A}_{ij}+2\Dot{A}_{il}\Dot{A}_{lj}.
\end{equation}
Combining \eqref{eq:ddot bach 1} with Lemmas \ref{Cotton tensor} and \ref{Bach tensor}, we obtain that
\begin{equation}\label{eq:ddot bach 2}
\begin{aligned}
\Ddot{B}_{ij}&=\Delta\Ddot{A}_{ij}-\Ddot{A}_{il,jl}-h_{kl}(\Dot{A}_{ij,kl}-\Dot{A}_{ik,jl})\\
&\ -\Dot{A}\cdot\Dot{A}\delta_{ij}-\tr\Dot{A}\Dot{A}_{ij} +2\Dot{A}_{il}\Dot{A}_{lj}-\Dot{\Gamma}_{ik,k}^l\Dot{A}_{lj}+(2\Dot{\Gamma}_{ij,k}^l-\Dot{\Gamma}_{jk,i}^l)\Dot{A}_{kl}\\
&\ +\Dot{\Gamma}_{ij}^l\delta_l\Dot{A}-\Dot{\Gamma}_{ik}^l(2\Dot{A}_{jl,k} -\Dot{A}_{kl,j})+\Dot{A}_{il,k}(h_{jk,l}-h_{jl,k})+\delta_kh(\Dot{A}_{ik,j}-\Dot{A}_{ij,k}).
\end{aligned}
\end{equation}
Putting \eqref{eq:ddot bach 2} together with \eqref{eq:diver schouten}, we can derive \eqref{eq:ddot bach}.
\end{proof}

We now examine the expansions of the tensors $T_2$ and $T_4$ given in \eqref{eq:T2}--\eqref{eq:T4}, along with their traces and divergences.
\begin{lemma}\label{T2T4}
Let $g=\exp(h)$ and $\tr \,h=0$. Then
\begin{align*}
(\Dot{T_2})_{ij}[h] &:= -8\Dot{A}_{ij}[h]+(n-2)\tr\Dot{A}[h]\delta_{ij}\\
&= -\frac{8}{n-2}\Dot{\ricci}_{ij}[h] +\frac{n^2-4n+12}{2(n-2)(n-1)}\Dot{R}[h]\delta_{ij};\\
(\Ddot{T_2})_{ij}[h,h] &:= -8\Ddot{A}_{ij}[h,h]+(n-2)[(\tr\Ddot{A}[h,h]-\Dot{A}[h]\cdot h)\delta_{ij}+\tr\Dot{A}[h]h_{ij}]\\
&=-\frac{8}{n-2}\Ddot{\ricci}_{ij}[h,h] +\frac{n^2-4n+12}{2(n-2)(n-1)}(\Ddot{R}[h,h]\delta_{ij}+\Dot{R}[h]h_{ij})
\end{align*}
and
\begin{align*}
(\Dot{T_4})_{ij}[h] &:= -\frac{16}{n-4}\Dot{B}_{ij}[h]+(n-6)\Delta\tr\Dot{A}[h]\delta_{ij} \\
&= \begin{medsize}
\displaystyle -\frac{16}{(n-4)(n-2)}\Delta \Dot{\ricci}_{ij}[h] +\frac{n^3-12n^2+44n-32}{2(n-4)(n-2)(n-1)}\Delta\Dot{R}[h]\delta_{ij} +\frac{8}{(n-4)(n-1)}\pa^2_{ij}\Dot{R}[h];
\end{medsize} \\
(\Ddot{T_4})_{ij}[h,h] &:= -\frac{16}{n-4}\Ddot{B}_{ij}[h,h]+(n-6)\Delta\tr\Dot{A}[h]h_{ij} + 8(n-2)\tr\Dot{A}[h]\Dot{A}_{ij}[h]-48\Dot{A}_{il}[h]\Dot{A}_{lj}[h] \\
&\ +\left[(n-6)\left\{\Delta(\tr\Ddot{A}[h,h]-\Dot{A}[h]\cdot h)-(h_{kl}\tr\Dot{A}_{,l}[h])_{,k}\right\} -\frac{3n^2-12n-4}{4}(\tr\Dot{A}[h])^2 \right. \\
&\hspace{280pt} \left. +4(n-4)\Dot{A}[h]\cdot\Dot{A}[h]\right]\delta_{ij}.
\end{align*}
\end{lemma}

\begin{cor}\label{cor:t2t4}
Under the same assumption, we have
\begin{equation}\label{eq:T2T4 trace}
\begin{aligned}
\tr\, \Dot{T_2}[h] &= (n^2-2n-8)\tr\Dot{A}[h];\\
\tr\, \Ddot{T_2}[h,h] &= (n^2-2n-8)(\tr\Ddot{A}[h,h]-\Dot{A}[h]\cdot h)-8\Dot{A}[h]\cdot h;\\
\tr\, \Dot{T_4}[h] &= (n-6)n\Delta\tr\Dot{A}[h];\\
\tr\, \Ddot{T_4}[h,h] &= (n-6)n[\Delta(\tr\Ddot{A}[h,h]-\Dot{A}[h]\cdot h)-(h_{kl}\tr\Dot{A}_{,l}[h])_{,k}] -\frac{16}{n-4}(h\cdot \Delta\Dot{A}[h]-h_{ij}\tr\Dot{A}_{,ij}[h])\\
&\ -\frac{1}{4}(3n^3-12n^2-36n+64)(\tr\Dot{A}[h])^2+4(n^2-4n-12)\Dot{A}[h]\cdot\Dot{A}[h].
\end{aligned}
\end{equation}
Moreover, it holds that
\begin{equation}\label{eq:T2T4 diver}
\begin{aligned}
(\Dot{\diver_g T_2})_i[h] &= (n-10)\tr\Dot{A}_{,i}[h]; \\
(\Ddot{\diver_g T_2})_i[h,h] &= (n-10)(\tr\Ddot{A}[h,h]-\Dot{A}[h]\cdot h)_{,i}; \\
(\Dot{\diver_g T_4})_i[h] &=(n-6)\Delta\tr\Dot{A}_{,i}[h]; \\
(\Ddot{\diver_g T_4})_i[h,h] &= \begin{medsize}
\displaystyle -32\Dot{A}_{jk}[h]\Dot{A}_{ij,k}[h]+8(n-8)\Dot{A}_{ij}[h]\tr\Dot{A}_{,j}[h]
\end{medsize} \\
&\ \begin{medsize}
\displaystyle +\left[(n-6)\left\{\Delta (\tr\Ddot{A}[h,h]-\Dot{A}[h]\cdot h)-(h_{kj}\tr\Dot{A}_{,k}[h])_{,j}\right\} -\frac{3n^2-28n+28}{4}(\tr\Dot{A}[h])^2\right.
\end{medsize} \\
&\ \begin{medsize}
\displaystyle \hspace{240pt} \left.+4(n-6)\Dot{A}[h]\cdot\Dot{A}[h]\right]_{,i}.
\end{medsize}
\end{aligned}
\end{equation}
\end{cor}
\begin{proof}
The trace formulas in \eqref{eq:T2T4 trace} are derived directly from Lemma \ref{T2T4} and equation \eqref{eq:trace bach}.

By virtue of Lemma \ref{lemma:diver} and the identities $\delta_i\Dot{\ricci} = \frac{1}{2}\pa_i \Dot{R}$ and $\delta_i\Dot{T_2} = \frac{n-10}{2(n-1)}\pa_i \Dot{R}$, we observe
\begin{align*}
(\diver_gT_2)_{j}&=(n-10)(\tr_g A_g)_{,j};\\
(\diver_gT_4)_j&=-32A_g^{lk}(A_g)_{jl;k}+8(n-8)g_{kj}A_g^{lk}(\tr_g A_g)_{,l}\\
&\ +\left[(n-6)\Delta_g\tr_g A_g-\frac{3n^2-28n+28}{4}(\tr_g A_g)^2+4(n-6)|A_g|^2_g \right]_{,j}.
\end{align*}
Then, we deduce the divergence formulas in \eqref{eq:T2T4 diver} as in the proof of Corollary \ref{cor:diver schouten}.
\end{proof}

\begin{rmk}
Recall the definition of the Paneitz operator $P_g^{(4)}$:
\[P_g^{(4)}u = \Delta_g^2u + \diver_g\left\{(4A_g - (n-2)\sigma_1(A_g)g)(\nabla_g u,\cdot)\right\} + \frac{n-4}{2} Q_g^{(4)}u.\]
The divergence of the tensor $4A_g - (n-2)\sigma_1(A_g)g$ can be written as the gradient of a scalar function $-(n-6)\tr_gA_g$.
In contrast, the divergence of the tensor $T_4$ in the definition of the GJMS operator $P_g^{(6)}$ cannot be written as the gradient of a scalar function.
It is one of the structural differences between $P_g^{(4)}$ and $P_g^{(6)}$. \hfill $\diamond$
\end{rmk}

For a symmetric $2$-tensor $T$ and a smooth function $u$ on $M$, we have
\[\diver_g\left\{T(\nabla_g u,\cdot)\right\}=g^{ik}g^{jl}T_{ij}\nabla_k\nabla_lu+(\diver_gT)_{j}g^{jl}\pa_lu\]
and
\[\nabla_k\nabla_lu = \pa^2_{kl}u-\Dot{\Gamma}_{kl}^i\pa_iu+O(|h||\pa h||\nabla u|).\]
Using these formulas, we obtain the first two terms in the expansion of $\diver_g\left\{T(\nabla_g u,\cdot)\right\}$:
\begin{equation}\label{eq:diver T nabla u}
\begin{aligned}
\Dot{\diver_g}\left\{T(\nabla_g u,\cdot)\right\}[h] &= \Dot{T}_{ij}[h]\pa^2_{ij}u+\delta_i\Dot{T}[h]\pa_iu=\pa_j(\Dot{T}_{ij}[h]\pa_{i}u);\\
\Ddot{\diver_g}\left\{T(\nabla_g u,\cdot)\right\}[h,h] &= (\Ddot{T}_{ij}[h,h]-2\Dot{T}_{il}[h]h_{lj})\pa^2_{ij}u \\
&\ + [(\Ddot{\diver_gT})_i[h,h]-\delta_j\Dot{T}[h]h_{ij}-\Dot{T}_{jl}[h]\Dot{\Gamma}_{jl}^i[h]]\pa_iu.
\end{aligned}
\end{equation}
Especially, taking $T = g$, we see
\[\Delta_gu=\Delta u-\Dot{\diver_g}\left\{g(\nabla_g u,\cdot)\right\}[h]+O(|h|^2|\nabla^2u|+|h||\pa h||\nabla u|).\]

Employing \eqref{eq:Q6}--\eqref{eq:P6}, Lemma \ref{Bach tensor}, and the computations from the preceding paragraph, we derive the first- and second-order expansions of $Q_g^{(6)}$ and $P_g^{(6)}$. The proofs are omitted.
\begin{lemma}\label{Q6 curva}
Let $g=\exp(h)$ and $\tr \,h=0$. Then
$$
Q_g^{(6)}=\Dot{Q^{(6)}}[h]+\Ddot{Q^{(6)}}[h,h] + O\Bigg(\sum_{\substack{0 \le \alpha_1,\alpha_2,\alpha_3 \le 6 \\ \alpha_1+\alpha_2+\alpha_3=6}} |\pa^{\alpha_1}h||\pa^{\alpha_2}h||\pa^{\alpha_3}h|\Bigg),
$$
where
\begin{align*}
\Dot{Q^{(6)}}[h] &:= \Delta^2\tr \Dot{A}[h]=\frac{1}{2(n-1)}\Delta^2 \Dot{R}[h];\\
\Ddot{Q^{(6)}}[h,h] &:=\Delta^2 (\tr\Ddot{A}[h,h]-\Dot{A}[h]\cdot h)-[h_{ij}\Delta\tr\Dot{A}_{,i}[h]+\Delta (h_{ij}\tr\Dot{A}_{,i}[h])]_{,j}\\
&\ -\frac{3n-2}{2}\tr\Dot{A}[h]\Delta\tr\Dot{A}[h]-(n-6)\tr\Dot{A}_{,i}[h]\tr\Dot{A}_{,i}[h]\\
&\ +4\Delta (\Dot{A}[h]\cdot\Dot{A}[h])+\frac{16}{n-4}\Dot{A}_{ij}[h] (\Delta\Dot{A}_{ij}[h]-\tr\Dot{A}_{,ij}[h])+8\Dot{A}_{ij}[h]\tr\Dot{A}_{,ij}[h].
\end{align*}
\end{lemma}

\begin{lemma}\label{mce6}
Let $g=\exp(h)$ and $\tr \,h=0$. Then
\[\mce_g := P_g^{(6)}-\Delta_g^3=L_1[h]+L_2[h,h]+O\Bigg(\sum_{\beta=0}^4 \sum_{\substack{0 \le \alpha_1,\alpha_2,\alpha_3 \le 6-\beta \\ \alpha_1+\alpha_2+\alpha_3=6-\beta}} |\pa^{\alpha_1}h||\pa^{\alpha_2}h||\pa^{\alpha_3}h| \pa^{\beta}\Bigg),\]
where $L_1[h]$ and $L_2[h,h]$ are the fourth-order self-adjoint differential operators defined by
\begin{align}
L_1[h] &:= \frac{n-2}{2}\Delta(\tr\Dot{A}[h]\Delta) - \Delta\circ \Dot{\diver_g}\left\{T_2(\nabla_g,\cdot)\right\}[h] - \Dot{\diver_g}\left\{T_2(\nabla_g,\cdot)\right\}[h]\circ \Delta \label{eq:L1 Q6} \\
&\ -\Dot{\diver_g}\left\{T_4(\nabla_g,\cdot)\right\}[h]-\frac{n-6}{2}\Dot{Q^{(6)}}[h]; \nonumber \\
L_2[h,h] &:= \frac{n-2}{2}\left[\Delta((\tr\Ddot{A}[h,h]-\Dot{A}[h]\cdot h)\Delta) - (h_{ij}(\tr\Dot{A}[h]\Delta)_{,j})_{,i} - \Delta(\tr\Dot{A}[h](h_{ij}\pa_i)_{,j})\right] \nonumber \\
&\ - \Delta\circ \Ddot{\diver_g}\left\{T_2(\nabla_g,\cdot)\right\}[h,h] - \Ddot{\diver_g}\left\{T_2(\nabla_g,\cdot)\right\}[h,h]\circ \Delta \nonumber \\
&\ - \Ddot{\diver_g}\left\{T_4(\nabla_g,\cdot)\right\}[h,h] \label{eq:L2 Q6} \\
&\ + \Dot{\diver_g}\left\{g(\nabla_g,\cdot)\right\}[h]\circ \Dot{\diver_g}\left\{T_2(\nabla_g,\cdot)\right\}[h] + \Dot{\diver_g}\left\{T_2(\nabla_g,\cdot)\right\}[h]\circ \Dot{\diver_g}\left\{g(\nabla_g,\cdot)\right\}[h] \nonumber \\
&\ - \frac{n-6}{2}\Ddot{Q^{(6)}}[h,h]. \nonumber
\end{align}
\end{lemma}
\begin{rmk}
One may attempt to substitute the expansions of $T_2$, $T_4$, and $Q^{(6)}$ into the expressions of $L_1[h]$ and $L_2[h,h]$ to obtain a result analogous to Lemma \ref{mce}.
While we will conduct this procedure for $L_1[h]$ in Lemma \ref{lemma:Q6 L1 formula} and Corollary \ref{cor:Q6 spherical expansion}, adopting the same approach for $L_2[h,h]$ leads to overwhelming complexity.

To deduce an analogue of Proposition \ref{I_1}, we opt to simplify the integral $\int_{B^n(0,\msfd_0\ep^{-1})}ZL_2[h,h]w$ rather than $L_2[h,h]$ itself,
apply polar coordinates to the simplified expression, and then insert the expansions of $T_2$, $T_4$, and $Q^{(6)}$ into the resulting integral over the sphere $\S^{n-1}$. Refer to the proof of Proposition \ref{I_1 Q6}. \hfill $\diamond$
\end{rmk}

\subsection{Correction terms}
In this subsection, we explicitly solve the linearized equation \eqref{linearized eqn k,sQ6} to derive the correction terms.
We begin by writing $L_1[h]$ in terms of $\Dot{\ricci}$ and $\Dot{R}$, following the ideas in Section \ref{sec:curv}.
\begin{lemma}\label{lemma:Q6 L1 formula}
Let $g=\exp(h)$ and $\tr \,h=0$. Then
\begin{equation}\label{Q6 L1 formula}
\begin{aligned}
L_1[h]&=\frac{16}{n-2}\Dot{\ricci}_{ij}[h]\pa^2_{ij}\Delta -\frac{3n^2-12n+44}{4(n-2)(n-1)}\Dot{R}[h]\Delta^2\\
&\ +\frac{16}{n-2}\Dot{\ricci}_{ij,k}[h]\pa^3_{ijk} -\frac{3n^2-28n+60}{2(n-2)(n-1)} \Dot{R}_{,j}[h]\pa_j \Delta \\
&\ +\left[\frac{8}{n-4}\Delta\Dot{\ricci}_{ij}[h] -\frac{(n-8)(n-6)}{(n-4)(n-1)}\Dot{R}_{,ij}[h]\right]\pa^2_{ij} -\frac{3n^2-26n+72}{4(n-4)(n-1)}\Delta \Dot{R}[h]\Delta\\
&\ -\frac{n-8}{n-1}\Delta\Dot{R}_{,j}[h]\pa_j -\frac{n-6}{4(n-1)}\Delta^2 \Dot{R}[h].
\end{aligned}
\end{equation}
\end{lemma}
\begin{proof}
Applying \eqref{eq:L1 Q6}, \eqref{eq:diver T nabla u}, and Lemma \ref{Q6 curva}, we rewrite $L_1[h]$ as
\begin{equation}\label{Q6 L1 formula1}
\begin{aligned}
L_1[h] &= \frac{n-2}{2}\tr\Dot{A}\Delta^2-2(\Dot{T_2})_{ij}\pa^2_{ij}\Delta +\left[(n-2)(\tr\Dot{A})_{,j}-2\delta_j\Dot{T_2}\right]\pa_j \Delta -2(\Dot{T_2})_{ij,k}\pa^3_{ijk} \\
&\ +\frac{n-2}{2}\Delta\tr\Dot{A}\Delta -\left[(\Dot{T_4})_{ij}+\Delta (\Dot{T_2})_{ij}+2(\Dot{T_2})_{il,jl} \right]\pa^2_{ij}-\(\Delta \delta_j\Dot{T_2}+\delta_j \Dot{T_4}\)\pa_j \\
&\ -\frac{n-6}{2}\Delta^2\tr \Dot{A}.
\end{aligned}
\end{equation}
From Lemma \ref{T2T4}, \eqref{eq:T2T4 diver}, and $\tr\dot{A}=\frac{1}{2(n-1)}\dot{R}$, we obtain that
\begin{equation}\label{Q6 L1 formula2}
\begin{aligned}
&\ (\Dot{T_4})_{ij}+\Delta (\Dot{T_2})_{ij}+2(\Dot{T_2})_{il,jl}\\
&=-\frac{16}{(n-4)(n-2)}\Delta \Dot{\ricci}_{ij}+\frac{n^3-12n^2+44n-32}{2(n-4)(n-2)(n-1)}\Delta\Dot{R}\delta_{ij} +\frac{8}{(n-4)(n-1)}\pa^2_{ij}\Dot{R}\\
&\ -\frac{8}{n-2}\Delta\Dot{\ricci}_{ij} +\frac{n^2-4n+12}{2(n-2)(n-1)}\Delta\Dot{R}\delta_{ij} +\frac{n-10}{n-1}\pa^2_{ij} \Dot{R}.
\end{aligned}
\end{equation}
Putting \eqref{Q6 L1 formula2} into \eqref{Q6 L1 formula1}, and using Lemma \ref{T2T4} and \eqref{eq:T2T4 diver} once again, we derive \eqref{Q6 L1 formula}.
\end{proof}

Using \eqref{conformal normal}--\eqref{conformal normal coro}, we can further simplify $L_1[h]u$ for $u$ radial as follows. We omit the proof.
\begin{cor}\label{cor:Q6 spherical expansion}
If $u$ is radial, then
\begin{align*}
L_1[h]u
&\begin{medsize}
\displaystyle =-\frac{3n^2-12n+44}{4(n-2)(n-1)} \Dot{R}[h] u'''' - \left[\frac{3(n-2)}{2}\Dot{R}[h]+\frac{3n^2-28n+60}{2(n-2)(n-1)}\Dot{R}_{,i}[h]x_i\right] \frac{u'''}{r}
\end{medsize} \\
&\begin{medsize}
\displaystyle \ - \left[\alpha_n^1\Dot{R}[h]+\alpha_n^2\Dot{R}_{,i}[h]x_i +\frac{(n-8)(n-6)}{(n-4)(n-1)}\Dot{R}_{,ij}[h]x_ix_j +\frac{3n^2-26n+72}{4(n-4)(n-1)}r^2\Delta\Dot{R}[h]\right] \frac{u''}{r^2}
\end{medsize} \\
&\begin{medsize}
\displaystyle \ + \left[\alpha_n^1\Dot{R}[h]+\alpha_n^2\Dot{R}_{,i}[h]x_i +\frac{(n-8)(n-6)}{(n-4)(n-1)}\Dot{R}_{,ij}[h]x_ix_j -\beta_nr^2\Delta\Dot{R}[h]-\frac{n-8}{n-1}r^2\Delta\Dot{R}_{,i}[h]x_i\right] \frac{u'}{r^3}
\end{medsize} \\
&\begin{medsize}
\displaystyle \ -\frac{n-6}{4(n-1)}\Delta^2\Dot{R}[h]u.
\end{medsize}
\end{align*}
Here, $\alpha_n^1:=\frac{3 n^4- 33 n^3 + 100 n^2- 68 n + 144}{4(n-4)(n-2)}$, $\alpha_n^2:=\frac{3 n^3 - 40 n^2+ 140 n -48}{2(n-4)(n-2)}$, and $\beta_n:=\frac{3 n^3 - 25 n^2+ 10 n +152}{4(n-4)(n-1)}$.
\end{cor}

Similarly to Section \ref{sec:corr}, we consider the spherical harmonic decomposition of a homogeneous polynomial $H^{(k+2)}_{ij,ij} \in \mcp_k$ for $k=2,\ldots,K$, where $K=n-8$: There exist homogeneous harmonic polynomials $p^{(k-2s)} \in \mch_{k-2s}$ such that
\[H^{(k+2)}_{ij,ij}(x) = \sum_{s=0}^{\lfloor\frac{k-2}{2}\rfloor}r^{2s}p^{(k-2s)}(x),\]
where $r = |x|$. By Lemma \ref{orthogonal} and Lemma \ref{IBP}, $p^{(0)}$ and $p^{(1)}$ vanish.

By using Corollary \ref{cor:Q6 spherical expansion}, we decompose $L_1[H^{(k+2)}]u$ into
\[L_1[H^{(k+2)}]u=\sum_{s=0}^{[\frac{k-2}{2}]} \mfl_1[r^{2s}p^{(k-2s)}]u,\]
where
\begin{align*}
&\begin{medsize}
\displaystyle \ \mfl_1[r^{2s}p^{(k-2s)}]u
\end{medsize} \\
&\begin{medsize}
\displaystyle :=\left[-\frac{3n^2-12n+44}{4(n-2)(n-1)} r^{2s} u''''-\left\{\frac{3(n-2)}{2}+k\frac{3n^2-28n+60}{2(n-2)(n-1)}\right\}r^{2s-1}u'''\right.
\end{medsize} \\
&\begin{medsize}
\displaystyle -\left\{\alpha_n^1+k\alpha_n^2+k(k-1)\frac{(n-8)(n-6)}{(n-4)(n-1)} +2s(2k-2s+n-2)\frac{3n^2-26n+72}{4(n-4)(n-1)}\right\}r^{2s-2}u''
\end{medsize} \\
&\begin{medsize}
\displaystyle +\left\{\alpha_n^1+k\alpha_n^2+k(k-1)\frac{(n-8)(n-6)}{(n-4)(n-1)} -2s(2k-2s+n-2)\beta_n-2s(2k-2s+n-2)(k-2)\frac{n-8}{n-1}\right\}r^{2s-3}u'
\end{medsize} \\
&\begin{medsize}
\displaystyle \left.-2s(2k-2s+n-2)(2s-2)(2k-2s+n-4)\frac{n-6}{4(n-1)}r^{2s-4}u\right]p^{(k-2s)}.
\end{medsize}
\end{align*}
In particular,
\begin{align*}
&\ \mfl_1[r^{2s}p^{(k-2s)}](1+r^2)^{-\frac{n-6}{2}}\\
&=-\frac{n-6}{4(n-1)}\left[\msfc_1r^{2s-4}(1+r^2)^{-\frac{n-6}{2}}+\msfc_2r^{2s-2}(1+r^2)^{-\frac{n-4}{2}}\right.\\
&\hspace{100pt} \left. +\msfc_3r^{2s}(1+r^2)^{-\frac{n-2}{2}}+\msfc_4r^{2s+2}(1+r^2)^{-\frac{n}{2}}+\msfc_5r^{2s+4}(1+r^2)^{-\frac{n+2}{2}}\right] p^{(k-2s)};\\
&\ \mfl_1[r^{2s}p^{(k-2s)}](1+r^2)^{-\frac{n-4}{2}}\\
&=-\frac{1}{4(n-1)} \left[(n-6)\msfc_1r^{2s-4}(1+r^2)^{-\frac{n-4}{2}} +(n-4)\msfc_2r^{2s-2}(1+r^2)^{-\frac{n-2}{2}}\right.\\
&\hspace{40pt} \left. +(n-2)\msfc_3r^{2s}(1+r^2)^{-\frac{n}{2}} + n\msfc_4r^{2s+2}(1+r^2)^{-\frac{n+2}{2}} + (n+2)\msfc_5r^{2s+4}(1+r^2)^{-\frac{n+4}{2}}\right]p^{(k-2s)},
\end{align*}
where
\begin{align*}
\msfc_1(n,k,s)&:=2s(2k-2s+n-2)(2s-2)(2k-2s+n-4);\\
\msfc_2(n,k,s)&:=-2s(2k-2s+n-2)(4(n-8)k+3n^2-18n+8);\\
\msfc_3(n,k,s)&:=4(n-8)(n-6)k^2+2(n-6)(3n^2-12n-40)k \\
&\ +2s(2k-2s+n-2)(3n^2-26n+72) +3 n^4 - 12 n^3 - 44 n^2+ 176 n + 192;\\
\msfc_4(n,k)&:=-2(n-4)[(3n^2-28n+60)k+3(n^3-2n^2-4n+40)];\\
\msfc_5(n)&:=(n-4)n(3n^2-12n+44).
\end{align*}

\begin{defn}[Right-hand side of the linearized equation]\label{linearized eqn RHS Q6}
\

\noindent
\noindent We define two $(s+5)\times 1$ column vectors $\Vec{b}=(b_i)$ and $\Vec{b}'=(b'_i)$, where the elements $b_i=b_i(n,k,s) \in \R$ and $b'_i=b'_i(n,k,s) \in \R$ are defined by the relation
\begin{align*}
\mfl_1[r^{2s}p^{(k-2s)}](1+r^2)^{-\frac{n-6}{2}} &:=- \frac{n-6}{4(n-1)}\sum_{i=1}^{s+5}b_i(1+r^2)^{-\frac{n+8-2i}{2}}p^{(k-2s)};\\
\mfl_1[r^{2s}p^{(k-2s)}](1+r^2)^{-\frac{n-4}{2}} &:=- \frac{1}{4(n-1)}\sum_{i=1}^{s+5}b_i'(1+r^2)^{-\frac{n+10-2i}{2}}p^{(k-2s)}. \tag*{$\diamond$}
\end{align*}
\end{defn}
The vector $\Vec{b}$ is given by $b_1=0$, $b_2=0$,
\begin{align*}
b_3&=(-1)^{s+2}\msfc_5;\\
b_4&=(-1)^{s+1}\left[\binom{s+2}{1}\msfc_5+\msfc_4\right];\\
b_5&=(-1)^s\left[\binom{s+2}{2}\msfc_5+\binom{s+1}{1}\msfc_4+\msfc_3\right],
\end{align*}
and when $s\geq 1$,
$$
b_6=(-1)^{s-1}\left[\binom{s+2}{3}\msfc_5+\binom{s+1}{2}\msfc_4+\binom{s}{1}\msfc_3+\msfc_2\right],
$$
and when $s\geq 2$, for $i=7,\ldots, s+5$,
\[b_{i}=(-1)^{s+5-i} \left[\binom{s+2}{i-3}\msfc_5+\binom{s+1}{i-4}\msfc_4+\binom{s}{i-5}\msfc_3 +\binom{s-1}{i-6}\msfc_2+\binom{s-2}{i-7}\msfc_1\right].\]

The vector $\Vec{b}'$ is given by $b_1'=0$, $b'_2=0$,
\begin{align*}
b_3'&=(-1)^{s+2}(n+2)\msfc_5;\\
b_4'&=(-1)^{s+1}\left[\binom{s+2}{1}(n+2)\msfc_5+n\msfc_4\right];\\
b_5'&=(-1)^s\left[\binom{s+2}{2}(n+2)\msfc_5+\binom{s+1}{1}n\msfc_4+(n-2)\msfc_3\right],
\end{align*}
and when $s\geq 1$,
$$
b_6'=(-1)^{s-1}\left[\binom{s+2}{3}(n+2)\msfc_5+\binom{s+1}{2}n\msfc_4 +\binom{s}{1}(n-2)\msfc_3+(n-4)\msfc_2\right],
$$
and when $s\geq 2$, for $i=7,\ldots, s+5$,
\begin{align*}
b_{i}'&=(-1)^{s+5-i} \left[\binom{s+2}{i-3}(n+2)\msfc_5+\binom{s+1}{i-4}n\msfc_4 +\binom{s}{i-5}(n-2)\msfc_3 \right. \\
&\hspace{125pt} \left. +\binom{s-1}{i-6}(n-4)\msfc_2+\binom{s-2}{i-7}(n-6)\msfc_1\right].
\end{align*}

\medskip
Let $w=w_{1,0}=(1+|\cdot|^2)^{-(n-6)/2}$ be the normalized bubble (see \eqref{eq:bubble} with $\msfk=3$). We are now ready to find an explicit solution to the linearized equation
\begin{equation}\label{linearized eqn k,sQ6}
\Delta^3\Psi^{k,s}+\tmfc_6(n)w^{\frac{12}{n-6}}\Psi^{k,s} = -\mfl_1[r^{2s}p^{(k-2s)}]w,
\end{equation}
which is the key proposition of this subsection. Recall that $K=n-8$.
\begin{prop}\label{prop:PsiQ6}
Given $n\geq 12$, $k=2,\ldots,K-2$, and $s=0,\ldots, \lfloor\frac{k-2}{2}\rfloor$, we assume $\Delta p^{(k-2s)}=0$.
Then equation \eqref{linearized eqn k,sQ6} has a solution of the form
$$
\Psi^{k,s}=\frac{n-6}{4(n-1)}\sum_{j=1}^{s+3}\Gamma_j(1+r^2)^{-\frac{n-2-2j}{2}}p^{(k-2s)},
$$
where the coefficients $\Gamma_j=\Gamma_j(n,k,s)$ satisfy the same relations \eqref{eq:Gammanew} as in the $Q^{(4)}$ case.
\end{prop}
\begin{proof}
For $j=1,\ldots,s+3$, let $a=\frac{n-2-2j}{2}$ and $b=k-2s$. By Lemma \ref{lemma:newbasis}, we find
\begin{align*}
&\begin{medsize}
\displaystyle \ (\Delta^3+\tmfc_6(n)(1+r^2)^{-6})((1+r^2)^{-\frac{n-2-2j}{2}}p^{(k-2s)})
\end{medsize} \\
&\begin{medsize}
\displaystyle =(\tmfc_6(n)-(n-2j-2)(n-2j)(n-2j+2)(n-2j+4)(n-2j+6)(n-2j+8))(1+r^2)^{-\frac{n+10-2j}{2}}p^{(k-2s)}
\end{medsize} \\
&\begin{medsize}
\displaystyle \ -6(n-2j-2)(n-2j)(n-2j+2)(n-2j+4)(n-2j+6)(k-2s+j-2)(1+r^2)^{-\frac{n+8-2j}{2}}p^{(k-2s)}
\end{medsize} \\
&\begin{medsize}
\displaystyle \ -12(n-2j-2)(n-2j)(n-2j+2)(n-2j+4)(k-2s+j-1)(k-2s+j-2)(1+r^2)^{-\frac{n+6-2j}{2}}p^{(k-2s)}
\end{medsize} \\
&\begin{medsize}
\displaystyle \ -8(n-2j-2)(n-2j)(n-2j+2)(k-2s+j)(k-2s+j-1)(k-2s+j-2)(1+r^2)^{-\frac{n+4-2j}{2}}p^{(k-2s)}.
\end{medsize}
\end{align*}
Then we define an $(s+5)\times (s+3)$ matrix $A=(a_{i,j})$:
\[(\Delta^3+\tmfc_6(n)(1+r^2)^{-6})((1+r^2)^{-\frac{n-2-2j}{2}}p^{(k-2s)}) = \sum_{i=1}^{s+5}a_{i,j}(1+r^2)^{-\frac{n+8-2i}{2}}p^{(k-2s)}.\]
We have
\begin{align*}
a_{1,1}&=-6(n-4)(n-2)n(n+2)(n+4)(k-2s-1);\\
a_{2,1}&=-12(n-4)(n-2)n(n+2)(k-2s)(k-2s-1);\\
a_{3,1}&=-8(n-4)(n-2)n(k-2s+1)(k-2s)(k-2s-1),
\end{align*}
and for $j=2,\ldots,s+3$,
\begin{align*}
a_{j-1,j}&=\tmfc_6(n)-(n-2j-2)(n-2j)(n-2j+2)(n-2j+4)(n-2j+6)(n-2j+8);\\
a_{j,j}&=-6(n-2j-2)(n-2j)(n-2j+2)(n-2j+4)(n-2j+6)(k-2s+j-2);\\
a_{j+1,j}&=-12(n-2j-2)(n-2j)(n-2j+2)(n-2j+4)(k-2s+j-1)(k-2s+j-2);\\
a_{j+2,j}&=-8(n-2j-2)(n-2j)(n-2j+2)(k-2s+j)(k-2s+j-1)(k-2s+j-2).
\end{align*}
To prove the proposition, we need to solve the following overdetermined linear system:
\begin{equation}\label{eq:AGbQ6}
A\Gamma=\Vec{b},
\end{equation}
where $\Gamma=(\Gamma_j)$ is an $(s+3)\times 1$ column vector. Using a computer software such as Mathematica, we see that the following cancellation appears:
\begin{align*}
(k-2s-1)\Gamma_1&=2\Gamma_2;\\
(k-2s)\Gamma_2&=\frac{4(n-1)}{n-2}\Gamma_3,
\end{align*}
which leads the existence of a solution to \eqref{eq:AGbQ6}.

Following the proof of Proposition \ref{prop:Psi}, we can plug the above two equations back to \eqref{eq:AGbQ6} to establish the desired relations \eqref{eq:Gammanew}.
\end{proof}
\begin{rmk}\label{rmk:lin sol rel Q6}
Let us recall the linearized equations of the Yamabe problem, the $Q^{(4)}$-curvature and $Q^{(6)}$-curvature problem: For $\msfk=1,2,3$,
\[
\Delta^{\msfk}\Psi^{2\msfk,k,s} + (-1)^{\msfk-1}\tmfc_{2\msfk}(n)w^{\frac{4\msfk}{n-2\msfk}}\Psi^{2\msfk,k,s} = -\mfl_1^{(2\msfk)}[r^{2s}p^{(k-2s)}]w.
\]
Then we can write our solutions in the unified forms:
\begin{equation}\label{relation of Psi Q6 1}
\Psi^{2\msfk,k,s} = \frac{n-2\msfk}{4(n-1)} \sum_{j=1}^{s+3}\Gamma^{(2\msfk)}_j (1+r^2)^{-\frac{n+4-2\msfk-2j}{2}}p^{(k-2s)}.
\end{equation}
Moreover, the coefficients satisfy that
\begin{equation}\label{relation of Psi Q6 2}
\Gamma^{(2)}_j=\Gamma^{(4)}_j=\Gamma^{(6)}_j \quad \text{for } j=1,\dots, s+3.
\end{equation}
By \eqref{relation of Psi Q6 1}--\eqref{relation of Psi Q6 2}, we have the identities
\begin{equation}\label{relation of Psi Q6}
\frac{1}{n-2}\Psi^{2,k,s}=\frac{1}{n-4}(1+r^2)^{-1}\Psi^{4,k,s}=\frac{1}{n-6}(1+r^2)^{-2}\Psi^{6,k,s},
\end{equation}
generalizing \eqref{relation of Psi}.

On the other hand, suppose that $S=\msd V$ for some vector field $V$ on $\R^n$ (i.e., $S$ is in the image of the conformal Killing operator defined in \eqref{eq:cko}) and $x_jS_{ij}=0$.
We observe from Appendix \ref{sec:geoexp} that the function
\begin{equation}\label{relation of Psi Q6 3}
\Psi^{2\msfk} := -V_ix_i(n-2\msfk)(1+r^2)^{-\frac{n-2\msfk+2}{2}} + \frac{n-2\msfk}{2n}\delta V (1+r^2)^{-\frac{n-2\msfk}{2}}
\end{equation}
solves the linear equation
\[\Delta^{2\msfk} \Psi + (-1)^{\msfk-1}\tmfc_{2\msfk}(n) w^{\frac{4\msfk}{n-2\msfk}} \Psi=-L_1[S]w.\]
By \eqref{relation of Psi Q6 3}, we have
\[\frac{1}{n-2}\Psi^{2}=\frac{1}{n-4}(1+r^2)^{-1}\Psi^{4}=\frac{1}{n-6}(1+r^2)^{-2}\Psi^{6}\]
which is the same relation as \eqref{relation of Psi Q6}. \hfill $\diamond$
\end{rmk}

Define
\begin{equation}\label{eq:PsiH Q6}
\Psi[H^{(k+2)}] := \sum_{s=0}^{\lfloor\frac{k-2}{2}\rfloor}\Psi^{k,s}.
\end{equation}
From Proposition \ref{prop:PsiQ6}, we immediately deduce
\[\Delta^3\Psi[H^{(k+2)}] + \tmfc_6(n)w^{\frac{12}{n-6}}\Psi[H^{(k+2)}] = - L_1[H^{(k+2)}]w \quad \text{in } \R^n,\]
and the following estimate.
\begin{cor}
For $n\ge 12$ and $k=2,\ldots,K-2$, there exists a constant $C=C(n,k)>0$ such that
$$
\left|\pa_{\beta}\Psi[H^{(k+2)}]\right|(x) \le C|H^{(k+2)}|(1+|x|)^{k+8-n-|\beta|}
$$
for a multi-index $\beta$ with $|\beta|=0,\ldots,6$. Here, $|H^{(k+2)}|$ is the quantity defined in \eqref{eq:Hnorm}.
\end{cor}

\subsection{Refined blowup analysis}\label{subsec:blowup Q6}
Throughout Subsections \ref{subsec:blowup Q6} and \ref{subsec:Weyl Q6}, we assume that $\sigma_a \to \bsi \in M$ is an isolated simple blowup point for a sequence $\{u_a\}_{a \in \N}$ of solutions to \eqref{eq:maina} with $\msfk=3$.
We work in $g_a$-normal coordinates $x$ centered at $\sigma_a$, chosen such that $\det g_a = 1$.

\medskip
We set $u_a(x) = u_a(\exp^{g_a}_{\sigma_a}(x))$, $\ep_a = u_a(0)^{-\frac{2}{n-6}} > 0$, $\tig_a(y) = g_a(\ep_a y)$, $P_a^{(6)} = P_{\tig_a}^{(6)}$, $\mce_a = \mce_{\tig_a}$, and $U_a(y) = \ep_a^{\frac{n-6}{2}} u_a(\ep_a y)$.

We note that $K=n-8 \ge d=\lfloor\frac{n-6}{2}\rfloor$ for all $n \ge 10$. Let $g_a = \exp(h_a)$ and $\tig_a=\exp(\tih_a)$ so that $\tih_a(y)=h_a(\ep_a y)$. We define the $2$-tensor $H_a$ by \eqref{eq:H}, where $h$ is replaced by $h_a$, and write
\[\wth_a^{(k)}(y) := H_a^{(k)}(\ep_a y) = \ep_a^k H_a^{(k)}(y), \quad k=2,\ldots, K,\]
and $\wth_a := \sum_{k=2}^{K} \wth_a^{(k)}$. For $n = 10$ or $11$, we set $\wtPsi_a := 0$. For $n \ge 12$, we set
\[\wtPsi_a := \sum_{k=4}^{K} \Psi[\wth_a^{(k)}],\]
where $\Psi[\wth_a^{(k)}]$ is the function in \eqref{eq:PsiH Q6} with $H^{(k+2)}$ replaced by $\wth_a^{(k)}$. Then, $\wtPsi_a$ is an explicit rational function on $\R^n$ solving
\[\Delta^3\wtPsi_a + \tmfc_6(n)w^{\frac{12}{n-6}}\wtPsi_a = - \sum_{k=4}^{K} L_1[\wth_a^{(k)}]w \quad \text{in } \R^n.\]

As in Proposition \ref{prop:refest} and Corollary \ref{cor:scale back}, one can prove the following results by modifying the argument in \cite[Section 5]{LX}. We omit the proofs.
\begin{prop}\label{prop:refest Q6}
Let $n\ge 10$ and $\sigma_a \to \bsi \in M$ be an isolated simple blowup point for a sequence $\{u_a\}$ of solutions to \eqref{eq:maina} with $\msfk=3$. Then there exist $\msfd_0 \in (0,1)$ and $C>0$ such that
$$
\left|\pa_{\beta}(U_a-w-\wtPsi_a)\right|(y) \le C\sum_{k=2}^{d-1}\ep_a^{2k}|H_a^{(k)}|^2(1+|y|)^{2k+6-n-|\beta|} +C\ep_a^{n-7}(1+|y|)^{-1-|\beta|}$$
for all $a \in \N$, multi-indices $\beta$ with $|\beta| = 0,1,\ldots,6$, and $|y|\le \msfd_0\ep_a^{-1}$.
\end{prop}

\begin{cor}\label{cor:scale back Q6}
Scaling back to the $x$-coordinates, we set $w_a(x):=\ep_a^{-\frac{n-6}{2}}w(\ep_a^{-1}x)$ and $\Psi_a(x)=\ep_a^{-\frac{n-6}{2}}\wtPsi(\ep_a^{-1}x)$ so that
$$
\Delta^3\Psi_a + \tmfc_6(n)w_a^{\frac{12}{n-6}}\Psi_a = - \sum_{k=4}^{K} L_1[H_a^{(k)}]w_a \quad \text{in } \R^n.
$$
Under the conditions of Proposition \ref{prop:refest Q6}, there exist $\msfd_0 \in (0,1)$ and $C>0$ such that
$$
\left|\pa_{\beta}(u_a-w_a-\Psi_a)\right|(x) \le C\ep_a^{\frac{n-6}{2}+|\beta|} \left[\sum_{k=2}^{d-1}|H_a^{(k)}|^2(\ep_a+|x|)^{2k+6-n-|\beta|} +(\ep_a+|x|)^{-1-|\beta|}\right]
$$
for all $a \in \N$, multi-indices $\beta$ with $|\beta| = 0,1,\ldots,6$, and $|x|\le \msfd_0$.
\end{cor}

\subsection{Weyl vanishing theorem}\label{subsec:Weyl Q6}
Let $\msfd_0 \in (0,1)$ be the small number from Proposition \ref{prop:refest Q6}. For $\ep > 0$ small, we define
\begin{align*}
\mci_{1,\ep,a} &:= \int_{B^n(0,\msfd_0\ep^{-1})} Z\mce_a(w) dy; \\
\mci_{2,\ep,a} &:= \int_{B^n(0,\msfd_0\ep^{-1})} \(y \cdot \nabla \wtPsi_a + \frac{n-6}{2}\wtPsi_a\) \mce_a(w) dy; \\
\mci_{3,\ep,a} &:= \int_{B^n(0,\msfd_0\ep^{-1})} Z\mce_a(\wtPsi_a) dy.
\end{align*}
With the local Pohozaev-Pucci-Serrin identity \eqref{eq:Poho Q6}--\eqref{eq:Poho2 Q6} in hand, one can argue as in Subsection \ref{subsec:Weyl21} to achieve
\begin{equation}\label{eq:Pohoineq2 Q6}
\begin{medsize}
\displaystyle O(\ep_a^{n-6}) \ge \mci_{1,\ep_a,a} + \mci_{2,\ep_a,a} + \mci_{3,\ep_a,a} + O\(\sum_{k=2}^{\lfloor\frac{n-6}{3}\rfloor} \ep_a^{3k}|\log\ep_a| |H_a^{(k)}|^3\) + O\(\ep_a^2|\log\ep_a| \sum_{k=2}^{d-1} \ep_a^{2k}|H_a^{(k)}|^2\).
\end{medsize}
\end{equation}

\begin{defn}[Pohozaev quadratic form]
\

\noindent
Given any symmetric $2$-tensors $h,\, h'$, we define $L_2[h,h']$ by using \eqref{eq:L2 Q6}. For $\ep > 0$ small, let
\begin{equation}\label{eq:I_1 Q6}
I_{1,\ep}[h,h'] := -\int_{B^n(0,\msfd_0\ep^{-1})} ZL_2[h,h']w.
\end{equation}
Whenever $\Psi[h]$ and $\Psi[h']$ are well-defined via \eqref{eq:PsiH Q6}, we also set
\begin{equation}\label{eq:I_2 Q6}
I_{2,\ep}[h,h'] := -\frac{1}{2}\int_{B^n(0,\msfd_0\ep^{-1})} \left[\(y_i\pa_i+\frac{n-6}{2}\)\Psi[h] L_1[h']w + \(y_i\pa_i+\frac{n-6}{2}\)\Psi[h'] L_1[h]w\right];
\end{equation}
\begin{equation}\label{eq:I_3 Q6}
I_{3,\ep}[h,h'] := -\frac{1}{2}\int_{B^n(0,\msfd_0\ep^{-1})} \(\Psi[h] L_1[h']Z+\Psi[h'] L_1[h]Z\),
\end{equation}
and
\[I_{\ep}[h,h'] := I_{1,\ep}[h,h'] + I_{2,\ep}[h,h'] + I_{3,\ep}[h,h']. \tag*{$\diamond$}\]
\end{defn}

Let $\theta_k=1$ if $k=\frac{n-6}{2}$ and $\theta_k=0$ otherwise. The following lemma corresponds to Lemma \ref{lemma:diffmciI}. We omit its proof.
\begin{lemma}\label{lemma:diffmciI Q6}
Given any small $\eta \in (0,1)$, there exists $C = C(n,M,g) > 0$ such that
\[\left|\mci_{i,\ep_a,a} - \sum_{k,m=2}^d I_{i,\ep_a}\big[\wth^{(k)}_a,\wth^{(m)}_a\big]\right| \le C\eta \sum_{k=2}^d \ep_a^{2k}|\log\ep_a|^{\theta_k}|H_a^{(k)}|^2 + C(\msfd_0\eta^{-1}+\msfd_0^{12-n})\ep_a^{n-6}\]
for all $a \in \N$ and $i = 1,2,3$.
\end{lemma}

As in Proposition \ref{Positive definiteness}, we establish a key estimate, whose proof is deferred to Subsections \ref{sec:proof of posi Q6} and \ref{sec:three cases Q6}.
\begin{prop}\label{Positive definiteness Q6}
For any $10 \le n \le 26$, there exists a constant $C=C(n)>0$ such that
$$
I_{\ep_a}\big[\wth_a,\wth_a\big] \ge C^{-1} \sum_{k=2}^d \ep_a^{2k}|\log\ep_a|^{\theta_k}|H_a^{(k)}|^2 + O(\ep_a^{n-6}) \quad \text{for all } a \in \N.
$$
\end{prop}

We are ready to establish the Weyl vanishing theorem for the sixth-order setting.
\begin{thm}\label{thm:Weyl2 Q6}
Let $10 \le n \le 26$ and $\sigma_a \to \bsi \in M$ be an isolated simple blowup point for a sequence $\{u_a\}$ of solutions to \eqref{eq:maina} with $\msfk=3$. Then for all $a \in \N$ and $k = 0,\ldots,\lfloor \frac{n-10}{2} \rfloor$,
$$
|\nabla^k_{g_a}W_{g_a}|^2(\sigma_a) \le C\ep_a^{n-10-2k}|\log\ep_a|^{-\theta_{k+2}}.
$$
Consequently,
\[\nabla_g^k W_g(\bsi) = 0.\]
\end{thm}
\begin{proof}
Following the proof of Theorem \ref{thm:Weyl2}, we incorporate Lemma \ref{lemma:diffmciI Q6} for $\eta \in (0,1)$ small and Proposition \ref{Positive definiteness Q6} with \eqref{eq:Pohoineq2 Q6} to prove Theorem \ref{thm:Weyl2 Q6}.
\end{proof}

The corollary below, corresponding to Corollary \ref{cor:scale back improved}, follows directly from Theorem \ref{thm:Weyl2 Q6} and Corollary \ref{cor:scale back Q6}.
\begin{cor}
Under the conditions of Theorem \ref{thm:Weyl2 Q6}, there exist $\msfd_0 \in (0,1)$ and $C>0$ such that
$$
\left|\pa_{\beta}(u_a-w_a-\Psi_a)\right|(x) \le C\ep_a^{\frac{n-6}{2}+|\beta|}(\ep_a+|x|)^{-1-|\beta|}
$$
for all $a \in \N$, multi-indices $\beta$ with $|\beta| = 0,1,\ldots,6$, and $|x|\le \msfd_0$.
\end{cor}

\subsection{Proof of Proposition \ref{Positive definiteness Q6}}\label{sec:proof of posi Q6}
To prove Proposition \ref{Positive definiteness Q6}, we need to obtain analogous results to Proposition \ref{I_1} and Corollary \ref{I_1 coro}.

For brevity, we will often drop the subscript $a$. In addition to \eqref{eq:inner} and \eqref{eq:inner2}, we also define
\[\bla\pa^{\ell}\delta \bar{H},\pa^{\ell}\delta \bar{H}'\bra = \sum_{|\beta|=\ell} \int_{\S^{n-1}} \pa_{\beta}\delta_j\bar{H} \pa_{\beta}\delta_j\bar{H}' \quad \text{and} \quad
\bla\pa^{\ell}\delta^2\bar{H},\pa^{\ell}\delta^2\bar{H}'\bra = \sum_{|\beta|=\ell} \int_{\S^{n-1}} \pa_{\beta}\delta^2\bar{H} \pa_{\beta}\delta^2\bar{H}'\]
for $\ell \in \N$ and two matrices $\bar{H}$ and $\bar{H}'$ of polynomials on $\R^n$.

\begin{prop}\label{I_1 Q6}
Given $H^{(k)}\in \mcv_k$ and $H^{(m)}\in \mcv_m$ for $k,m=2,\ldots, d$, we define the quadratic form $J_1[H^{(k)},H^{(m)}]$ through the following three steps.

First, we fix six constants depending only on $n$ and $k+m$:
\begin{align*}
c_1(n,k+m) &:= 4(n-6)^2(n-4)^2 \left[-\frac{n^2-4}{4(n-4)}\mci_{n}^{n+k+m-1}-\frac{n}{n-4}\mci_{n}^{n+k+m+1}\right.\\
&\left.\hspace{105pt}+\frac{n^2-20}{4(n-4)}\mci_{n}^{n+k+m+3}+\mci_{n}^{n+k+m+5}\right]^{1-\theta_{\frac{k+m}{2}}};\\
c_2(n,k+m) &:= -8(n-6)^3 \left[-\frac{n^2(n-2)}{16(n-6)}\mci_{n-1}^{n+k+m-3}+\frac{n^2(n-10)}{16(n-6)}\mci_{n-1}^{n+k+m-1}\right.\\
&\left.\hspace{75pt}+\frac{n^2-6n-4}{2(n-6)}\mci_{n-1}^{n+k+m+1}+\mci_{n-1}^{n+k+m+3}\right]^{1-\theta_{\frac{k+m}{2}}};\\
c_3(n,k+m) &:= -4(n-6)^2(n-5) \left[-\frac{n^2-2n-4}{4(n-5)}\mci_{n-2}^{n+k+m-3}\right.\\
&\left.\hspace{110pt}+\frac{n(n-6)}{4(n-5)}\mci_{n-2}^{n+k+m-1}+\mci_{n-2}^{n+k+m+1}\right]^{1-\theta_{\frac{k+m}{2}}};\\
c_4(n,k+m) &:= -\frac{(n-6)^2(n-4)}{2} \(-\mci_{n-3}^{n+k+m-3}+\mci_{n-3}^{n+k+m-1}\)^{1-\theta_{\frac{k+m}{2}}};\\
c_5(n,k+m) &:= \frac{(n-6)^2}{2} \(-\mci_{n-4}^{n+k+m-5}+\mci_{n-4}^{n+k+m-3}\)^{1-\theta_{\frac{k+m}{2}}};\\
c_6(n,k+m) &:= -\frac{n-6}{2} \(-\mci_{n-5}^{n+k+m-7}+\mci_{n-5}^{n+k+m-5}\)^{1-\theta_{\frac{k+m}{2}}}.
\end{align*}
Second, we introduce ten constants depending only on $n$, $k$, and $m$:
\begin{align*}
\bc_1(n,k,m) &:= c_6(n,k+m)(k+m-4)(n+k+m-6)(k+m-2)(n+k+m-4)\frac{n-6}{4(n-1)}\\
&\ +c_3(n,k+m)\frac{n^2-2n-8+(n-10)(k+m-2)}{2(n-1)} - c_2(n,k+m)\frac{n-2}{4(n-1)}\\
&\ +c_1(n,k+m)\frac{n^2-4n+12}{2(n-2)(n-1)} + c_5(n,k+m)(k+m-2)(n+k+m-4)^2\frac{(n-6)}{2(n-1)}\\
&\ +c_4(n,k+m)(k+m-2)(n+k+m-4)\frac{(n-6)(n-4)(n-2)+16}{2(n-4)(n-2)(n-1)}\\
&\ +c_4(n,k+m)\frac{8(n-2)(k+m)^2-8(n-6)(k+m)+16(n-4)}{(n-4)(n-2)(n-1)};\\
\bc_2(n,k,m)&:=c_6(n,k+m)(k+m-4)(n+k+m-6)\frac{2(n-6)}{(n-2)^2}\\
&\ +c_5(n,k+m)\frac{4(n-6)(k+m)+4(n^2-8n+20)}{(n-2)^2} + c_4(n,k+m)\frac{4(n-4)^2+16}{(n-4)(n-2)^2};\\
\bc_3(n,k,m) &:= c_1(n,k+m)\frac{2km}{n-2} + c_4(n,k+m)(k+m)(n+k+m-2)\frac{4km}{(n-4)(n-2)};\\
\bc_4(n,k,m) &:= -c_3(n,k+m)\frac{2}{n-2}+c_4(n,k+m)\frac{8}{(n-4)(n-2)};\\
\bc_5(n,k,m) &:= -c_5(n,k+m)\frac{4}{(n-4)(n-2)};\\
\bc_6(n,k,m) &:= c_6(n,k+m)\frac{4(n-6)}{(n-4)(n-2)^2};\\
\bc_7(n,k,m) &:= c_4(n,k+m)\frac{8}{(n-4)(n-2)};\\
\bc_8(n,k,m) &:= -c_5(n,k+m)\frac{8(n+k+m-4)km}{(n-2)^2} - c_4(n,k+m)\left[\frac{8km}{(n-4)(n-2)}+\frac{8km}{(n-2)^2}\right];\\
\bc_9(n,k,m) &:= \begin{medsize}
\displaystyle -c_6(n,k+m)\frac{n-6}{4} \left[\frac{2(3n-4)}{(n-4)(n-2)(n-1)^2}+\frac{3n-2}{8(n-1)^2}\right](n+k+m-6)(k+m-4)
\end{medsize} \\
&\begin{medsize}
\displaystyle \ -c_6(n,k+m) \frac{n-6}{2}\frac{(n-6)(n+k+m-6)(n(k+m)-4)}{(n-4)(n-2)(n-1)^2}
\end{medsize} \\
&\begin{medsize}
\displaystyle \ -c_5(n,k+m) \left[\frac{3n^3-12n^2-36n+64}{16(n-1)^2}-\frac{2(k+m)}{(n-4)(n-1)}\right]
\end{medsize} \\
&\begin{medsize}
\displaystyle \ -c_5(n,k+m) \left[\frac{(3n^2-28n+28)(k+m-4)}{16(n-1)^2}+\frac{(n-4)(n(k+m)-4)}{(n-2)(n-1)^2} +\frac{8(n+k+m-4)}{(n-2)^2(n-1)^2}\right]
\end{medsize} \\
&\begin{medsize}
\displaystyle \ -c_5(n,k+m) \frac{(3n-4)((n-6)(k+m)+(n^2-8n+20))}{(n-2)^2(n-1)^2} - c_4(n,k+m)\frac{(3n-4)(n^2-8n+20)}{(n-4)(n-2)^2(n-1)^2}
\end{medsize} \\
&\begin{medsize}
\displaystyle \ -c_4(n,k+m) \left[\frac{3n^2-12n+28}{16(n-1)^2}+\frac{8}{(n-4)(n-2)(n-1)^2} +\frac{8}{(n-2)^2(n-1)^2}\right];
\end{medsize} \\
\bc_{10}(n,k,m) &:= c_6(n,k+m)\frac{n-6}{2} \left[\frac{2(3n-4)}{(n-4)(n-2)(n-1)^2}+\frac{3n-2}{8(n-1)^2}\right]\\
&\ -c_6(n,k+m)\frac{n-6}{2} \left[\frac{2(3n-4)}{(n-2)^2(n-1)^2}+\frac{(n-6)(n+4)}{4(n-4)(n-1)^2}\right].
\end{align*}
Finally, we define
\begin{equation}\label{eq:J1HkHm}
\begin{aligned}
J_1[H^{(k)},H^{(m)}]
&:= \bc_1(n,k,m) \int_{\S^{n-1}}\Ddot{R}^{(k,m)} + \bc_2(n,k,m) \bla\Dot{\ricci}^{(k)},\Dot{\ricci}^{(m)}\bra \\
&\ + \bc_3(n,k,m) \bla H^{(k)},H^{(m)}\bra \\
&\ + \bc_4(n,k,m) \left[k\bla H^{(k)},\Dot{\ricci}^{(m)}\bra + m\bla\Dot{\ricci}^{(k)},H^{(m)}\bra\right] \\
&\ + \bc_5(n,k,m) \left[k\bla H^{(k)},\Delta\Dot{\ricci}^{(m)}\bra + m\bla\Delta\Dot{\ricci}^{(k)},H^{(m)}\bra\right] \\
&\ + \bc_6(n,k,m) \left[\bla\Dot{\ricci}^{(k)},\Delta\Dot{\ricci}^{(m)}\bra + \bla\Delta\Dot{\ricci}^{(k)},\Dot{\ricci}^{(m)}\bra\right] \\
&\ + \bc_7(n,k,m) \left[k^2\bla H^{(k)},\Dot{\ricci}^{(m)}\bra + m^2\bla\Dot{\ricci}^{(k)},H^{(m)}\bra\right] \\
&\ + \bc_8(n,k,m) \bla\delta H^{(k)},\delta H^{(m)}\bra + \bc_9(n,k,m) \bla\delta^2H^{(k)},\delta^2H^{(m)}\bra \\
&\ + \bc_{10}(n,k,m) \bla\pa\delta^2H^{(k)},\pa\delta^2H^{(m)}\bra.
\end{aligned}
\end{equation}
Then, for $k, m = 2, \ldots, d$, it holds that
\begin{equation}\label{eq:I_1J_1 Q6}
I_{1,\ep}\big[\wth^{(k)},\wth^{(m)}\big] = \ep^{k+m}|\log\ep|^{\theta_{\frac{k+m}{2}}} J_1[H^{(k)},H^{(m)}] + O(\ep^{n-6}).
\end{equation}
\end{prop}
\begin{proof}
To evaluate $I_{1,\ep}[\wth^{(k)},\wth^{(m)}]$ given in \eqref{eq:I_1 Q6}, we will examine each term in the expansion of $L_2[h,h]$ from \eqref{eq:L2 Q6}. Since $w$ is radial and $x_ih_{ij}=0$, it holds that
\begin{align*}
\Delta(\tr\Dot{A}^{(k)}(h_{ij}^{(m)}\pa_iw)_{,j})&=0;\\
\Dot{\diver_g}\left\{T_2(\nabla_g ,\cdot)\right\}^{(k)} \circ \Dot{\diver_g}\left\{g(\nabla_g w,\cdot)\right\}^{(m)}&=0.
\end{align*}
Here, $\Dot{A}^{(k)}:=\Dot{A}[H^{(k)}]$, and analogous definitions apply to $\Dot{\diver_g}\{T_2(\nabla_g ,\cdot)\}^{(k)}$ and to the other superscripts $(k)$ appearing below (cf. Subsection \ref{subsec:proppd}). Integrating by parts, and using the radial symmetry of $Z$ and $x_ih_{ij}=0$, we find
\begin{align*}
&\ \int_{B^n(0,\msfd_0\ep^{-1})} Z\Dot{\diver_g}\left\{g(\nabla_g,\cdot)\right\}^{(k)} \circ \Dot{\diver_g}\left\{T_2(\nabla_g w,\cdot)\right\}^{(m)} \\
&=\int_{B^n(0,\msfd_0\ep^{-1})} \Dot{\diver_g}\left\{g(\nabla_g Z,\cdot)\right\}^{(k)} \Dot{\diver_g}\left\{T_2(\nabla_g w,\cdot)\right\}^{(m)} + O(\ep^{n-6}) = O(\ep^{n-6});\\
&\ \int_{B^n(0,\msfd_0\ep^{-1})} Z(H^{(k)}_{ij}(\tr\Dot{A}^{(m)}\Delta w)_{,j})_{,i}\\
&=- \int_{B^n(0,\msfd_0\ep^{-1})} Z_{,i}H^{(k)}_{ij}(\tr\Dot{A}^{(m)}\Delta w)_{,j} + O(\ep^{n-6}) = O(\ep^{n-6}).
\end{align*}
Furthermore, we employ \eqref{eq:DdotR} and integration by parts again to obtain
\begin{align*}
\int_{B^n(0,\msfd_0\ep^{-1})}Z\Delta[(\tr\Ddot{A}^{(k,m)}-\Dot{A}^{(k)}\cdot h^{(m)})\Delta w] &=\frac{1}{2(n-1)} \int_{B^n(0,\msfd_0\ep^{-1})}\Ddot{R}^{(k,m)}\Delta Z\Delta w+O(\ep^{n-6});\\
\int_{B^n(0,\msfd_0\ep^{-1})} Z\Delta\circ \Ddot{\diver_g}\left\{T_2(\nabla_g w,\cdot)\right\}^{(k,m)} &=\int_{B^n(0,\msfd_0\ep^{-1})}\Delta Z \Ddot{\diver_g}\left\{T_2(\nabla_g w,\cdot)\right\}^{(k,m)}+O(\ep^{n-6}).
\end{align*}
So, we have
\begin{equation}\label{eq:Q6 polar proof 1}
\begin{aligned}
&\ \ep^{-(k+m)}I_{1,\ep}\big[\wth^{(k)},\wth^{(m)}\big] \\
&=-\frac{n-2}{4(n-1)}\int_{B^n(0,\msfd_0\ep^{-1})}\Ddot{R}^{(k,m)}\Delta Z\Delta w+\frac{n-6}{2}\int_{B^n(0,\msfd_0\ep^{-1})} \big(\Ddot{Q^{(6)}}\big)^{(k,m)} Z w\\
&\ +\int_{B^n(0,\msfd_0\ep^{-1})}\Delta Z \Ddot{\diver_g}\left\{T_2(\nabla_g w,\cdot)\right\}^{(k,m)} + \int_{B^n(0,\msfd_0\ep^{-1})} Z \Ddot{\diver_g}\left\{T_2(\nabla_g \Delta w ,\cdot)\right\}^{(k,m)} \\
&\ +\int_{B^n(0,\msfd_0\ep^{-1})}Z \Ddot{\diver_g}\left\{T_4(\nabla_g w ,\cdot)\right\}^{(k,m)}+O(\ep^{n-6}).
\end{aligned}
\end{equation}

Next, by \eqref{eq:diver T nabla u}, we see that
\begin{equation}\label{eq:Q6 polar proof 2}
\Ddot{\diver_g}\left\{T(\nabla_g u,\cdot)\right\} = \Ddot{T}_{ij}y_iy_j\(u'\frac{1}{r}\)'\frac{1}{r} + \left[\tr \,\Ddot{T} + (\Ddot{\diver_gT})_iy_i + \frac{1}{2}\Dot{T} \cdot (y_i\pa_i-2)h\right] u'\frac{1}{r}
\end{equation}
for a symmetric $2$-tensor $T$ and a radial function $u(y)=u(r)$. Also, from Lemmas \ref{Q6 curva} and \ref{T2T4} together with Corollary \ref{cor:t2t4}, it follows that
\begin{equation}\label{eq:Q6 polar proof 3}
\begin{aligned}
&\deg\Ddot{R}^{(k,m)}=k+m-2,\quad \deg \big(\Ddot{Q^{(6)}}\big)^{(k,m)}=k+m-6,\\
&\deg(\Ddot{T_2})_{ij}y_iy_j=k+m,\quad \hspace{5pt} \deg(\Ddot{T_4})_{ij}y_iy_j=k+m-2,\\
&\deg\left[\tr \,\Ddot{T_2}+(\Ddot{\diver_gT_2})_iy_i + \frac{1}{2}\Dot{T_2} \cdot (y_i\pa_i-2)h\right]=k+m-2,\\
&\deg\left[\tr \,\Ddot{T_4}+(\Ddot{\diver_gT_4})_iy_i + \frac{1}{2}\Dot{T_4} \cdot (y_i\pa_i-2)h\right]=k+m-4.
\end{aligned}
\end{equation}

Inserting \eqref{eq:Q6 polar proof 2} and \eqref{eq:Q6 polar proof 3} into \eqref{eq:Q6 polar proof 1}, we observe that
\begin{align}
&\ \ep^{-(k+m)} I_{1,\ep}\big[\wth^{(k)},\wth^{(m)}\big] \nonumber \\
&= -\frac{n-2}{4(n-1)} \int_{\S^{n-1}}\Ddot{R}^{(k,m)} \int_0^{\msfd_0\ep^{-1}}r^{n+k+m-3}\Delta Z\Delta w \nonumber \\
&\ + \int_{\S^{n-1}}(\Ddot{T_2})^{(k,m)}_{ij}y_iy_j \int_0^{\msfd_0\ep^{-1}}r^{n+k+m-2} \left[\Delta Z\(w'\frac{1}{r}\)'+ Z\((\Delta w)'\frac{1}{r}\)'\right] \nonumber \\
&\ + \int_{\S^{n-1}} \left[\tr \,\Ddot{T_2}^{(k,m)}+(\Ddot{\diver_gT_2})^{(k,m)}_iy_i + \frac{1}{4}\left\{\Dot{T_2}^{(k)} \cdot (y_i\pa_i-2)H^{(m)} + \Dot{T_2}^{(m)} \cdot (y_i\pa_i-2)H^{(k)}\right\}\right] \nonumber \\
&\hspace{250pt} \times \int_0^{\msfd_0\ep^{-1}}r^{n+k+m-4}\left[\Delta Z w'+Z (\Delta w)'\right] \nonumber \\
&\ + \int_{\S^{n-1}}(\Ddot{T_4})^{(k,m)}_{ij}y_iy_j \int_0^{\msfd_0\ep^{-1}}r^{n+k+m-4}Z\(w'\frac{1}{r}\)' \label{eq:Q6 polar proof 4} \\
&\ + \int_{\S^{n-1}} \left[\tr \,\Ddot{T_4}^{(k,m)}+(\Ddot{\diver_gT_4})^{(k,m)}_iy_i +
\frac{1}{4}\left\{\Dot{T_4}^{(k)} \cdot (y_i\pa_i-2)H^{(m)} + \Dot{T_4}^{(m)} \cdot (y_i\pa_i-2)H^{(k)}\right\}\right] \nonumber \\
&\hspace{315pt} \times \int_0^{\msfd_0\ep^{-1}}r^{n+k+m-6}Zw' \nonumber \\
&\ +\frac{n-6}{2} \int_{\S^{n-1}}\big(\Ddot{Q^{(6)}}\big)^{(k,m)} \int_0^{\msfd_0\ep^{-1}}r^{n+k+m-7}Zw + O(\ep^{n-6}). \nonumber
\end{align}
Moreover, the radial integrals appearing in \eqref{eq:Q6 polar proof 4} can be computed as follows:
\begin{equation}\label{eq:c1c6}
\begin{aligned}
&\int_0^{\msfd_0\ep^{-1}}r^{n+k+m-2} \left[\Delta Z\(w'\frac{1}{r}\)'+ Z\((\Delta w)'\frac{1}{r}\)'\right] = c_1(n,k+m)|\log \ep|^{\theta_{\frac{k+m}{2}}}+O(\ep^{n-6-k-m});\\
&\int_0^{\msfd_0\ep^{-1}}r^{n+k+m-3}\Delta Z\Delta w = c_2(n,k+m)|\log \ep|^{\theta_{\frac{k+m}{2}}}+O(\ep^{n-6-k-m});\\
&\int_0^{\msfd_0\ep^{-1}}r^{n+k+m-4}\left[\Delta Zw'+Z(\Delta w)'\right] = c_3(n,k+m)|\log \ep|^{\theta_{\frac{k+m}{2}}}+O(\ep^{n-6-k-m});\\
&\int_0^{\msfd_0\ep^{-1}}r^{n+k+m-4}Z\(w'\frac{1}{r}\)' = c_4(n,k+m)|\log \ep|^{\theta_{\frac{k+m}{2}}}+O(\ep^{n-6-k-m});\\
&\int_0^{\msfd_0\ep^{-1}}r^{n+k+m-6}Zw' = c_5(n,k+m)|\log \ep|^{\theta_{\frac{k+m}{2}}}+O(\ep^{n-6-k-m});\\
&\int_0^{\msfd_0\ep^{-1}}r^{n+k+m-7}Zw = c_6(n,k+m)|\log \ep|^{\theta_{\frac{k+m}{2}}}+O(\ep^{n-6-k-m}).
\end{aligned}
\end{equation}
In the end, the five sphere integrals in \eqref{eq:Q6 polar proof 4} associated with $T_2$, $T_4$, and $Q^{(6)}$ will be expressed in terms of sphere integrals involving $\Ddot{R}$, $\Dot{R}$, $\Dot{\ricci}$, and $H$ in Lemmas \ref{lemma:Q6 sphere}, \ref{lemma:T2 sphere}, and \ref{lemma:T4 sphere}.

Combining the above information, we can deduce \eqref{eq:I_1J_1 Q6}.
\end{proof}

\begin{cor}\label{I_1 coro Q6}
For $k, m = 2, \ldots, d$, it holds that
\begin{equation}\label{eq:I_1 coro Q6}
\begin{aligned}
J_1[H^{(k)},H^{(m)}] &= \bc'_1(n,k,m)\bla H^{(k)},H^{(m)}\bra+\bc'_2(n,k,m)\bla \mcl_kH^{(k)},H^{(m)}\bra\\
&\ +\bc'_3(n,k,m)\bla \mcl_kH^{(k)},\mcl_mH^{(m)}\bra \\
&\ +\bc'_4(n,k,m)\left[\bla \mcl_k^2H^{(k)},\mcl_m H^{(m)}\bra+\bla \mcl_k H^{(k)},\mcl_m^2H^{(m)}\bra\right]\\
&\ +\bc'_5(n,k,m)\bla \delta H^{(k)},\delta H^{(m)}\bra+\bc'_6(n,k,m)\bla \pa \delta H^{(k)},\pa \delta H^{(m)}\bra\\
&\ +\bc'_7(n,k,m)\bla \delta^2 H^{(k)},\delta^2 H^{(m)}\bra+\bc'_{8}(n,k,m)\bla \pa \delta^2 H^{(k)},\pa \delta^2H^{(m)}\bra,
\end{aligned}
\end{equation}
where
\begin{align*}
\bc'_1(n,k,m) &:= \bc_3(n,k,m)-\frac{1}{8}(k+m)(n+k+m-2)\bc_1(n,k,m)\\
&\ -\frac{1}{2}(m-k)(n+k+m-2) \bigg[(n+k+m-6)km\bc_5(n,k,m)-\frac{1}{4}\bc_1(n,k,m)\\
&\hspace{140pt} + k\left\{\bc_4(n,k,m)-2(n-4)\bc_5(n,k,m)\right\} \\
&\hspace{140pt} + k^2\left\{\bc_7(n,k,m)-(n+k+m-2)\bc_5(n,k,m)\right\}\bigg];\\
\bc'_2(n,k,m) &:= 2\left[km(n+k+m-6)\bc_5(n,k,m)-\frac{1}{4}\bc_1(n,k,m)\right]\\
&\ +(k+m)\left[\bc_4(n,k,m)-2(n-4)\bc_5(n,k,m)\right]\\
&\ +(k^2+m^2)\left[\bc_7(n,k,m)-(n+k+m-2)\bc_5(n,k,m)\right];\\
\bc'_3(n,k,m) &:= \bc_2(n,k,m)-4(n+k+m-4)\bc_6(n,k,m)-2(k+m)\bc_5(n,k,m);\\
\bc'_4(n,k,m) &:= -2\bc_6(n,k,m);\\
\bc'_5(n,k,m) &:= \bc_8(n,k,m) + \frac{km}{2}\bc_2(n,k,m) + km(2n+k+m-8)\bc_5(n,k,m)\\
&\ + \left[\frac{km}{2}(n+k+m-4)(k+m-2)-km(n+k-2)(n+m-2)\right.\\
&\hspace{20pt} +km(2n+k+m-4)\bigg]\bc_6(n,k,m);\\
\bc'_6(n,k,m) &:= -km\bc_6(n,k,m);\\
\bc'_7(n,k,m) &:= \bc_9(n,k,m) + \frac{1}{n-1}\bc_2(n,k,m) + \frac{(n-3)(k+m)}{n-1}\bc_5(n,k,m)\\
&\ + \left[\frac{(n-3)^2}{2(n-1)^2}(k+m-4)(n+k+m-6) + (k+m)+\frac{2(k+m-2)}{n-1} \right.\\
&\quad - \frac{(n-3)(n-2)}{(n-1)^2}(k^2+m^2-5(k+m)+12) -\frac{(n-3)n}{n-1}(k+m-2) \\
&\quad + \frac{k+m-4}{n-1}(n+k+m-6) - \frac{n-3}{2(n-1)}(k^2+m^2+(k+m)(n-4))\bigg]\bc_6(n,k,m);\\
\bc'_{8}(n,k,m) &:= \bc_{10}(n,k,m)-\frac{n^2-4n+7}{(n-1)^2}\bc_6(n,k,m).
\end{align*}
\end{cor}
\begin{proof}
It is a simple consequence of Proposition \ref{I_1 Q6} and Lemmas \ref{ricci ricci lemma}, \ref{lemma:comm}, and \ref{Delta ricci lemma}.
\end{proof}

\begin{rmk}\label{rmk:composed decomp}
Three remarks concerning Proposition \ref{I_1 Q6} and Corollary \ref{I_1 coro Q6} are in order.

\noindent 1. The computations carried out here are substantially more involved than those in the $Q^{(4)}$ case.

In Proposition \ref{I_1} (corresponding to Proposition \ref{I_1 Q6}), only 4 terms appears:
\[\int_{\S^{n-1}}\Ddot{R}^{(k,m)}, \quad \int_{\S^{n-1}}\Dot{\ricci}^{(k)} \cdot \Dot{\ricci}^{(m)}, \quad \int_{\S^{n-1}}H^{(k)} \cdot H^{(m)}, \quad \text{and} \quad \int_{\S^{n-1}}\delta^2H^{(k)} \cdot \delta^2H^{(m)}.\]
In Corollary \ref{I_1 coro} (corresponding to Corollary \ref{I_1 coro Q6}), only 5 terms appears:
\[\bla H^{(k)},H^{(m)}\bra, \quad \bla\mcl_kH^{(k)},H^{(m)}\bra, \quad \bla\mcl_k^2H^{(k)},H^{(m)}\bra, \quad \bla\delta H^{(k)},\delta H^{(m)}\bra, \quad \text{and} \quad \bla\delta^2H^{(k)},\delta^2H^{(m)}\bra.\]
For the comparison with the Yamabe case, refer to the paragraphs after Corollary \ref{I_1 coro}.

\noindent 2. In fact, the decomposition of $J_1[H^{(k)},H^{(m)}]$ in \eqref{eq:I_1 coro Q6} has the following equivalent form:
\begin{align*}
J_1[H^{(k)},H^{(m)}] &= \bc''_1(n,k,m)\bla H^{(k)},H^{(m)}\bra+\bc''_2(n,k,m)\bla \mcl_kH^{(k)},H^{(m)}\bra\\
&\ +\bc''_3(n,k,m)\bla \mcl_k^2H^{(k)},H^{(m)}\bra+\bc''_4(n,k,m)\bla \mcl_k^3H^{(k)}, H^{(m)}\bra\\
&\ +\bc''_5(n,k,m)\bla \delta H^{(k)},\delta H^{(m)}\bra+\bc''_6(n,k,m)\bla \pa \delta H^{(k)},\pa \delta H^{(m)}\bra\\
&\ +\bc''_7(n,k,m)\bla \delta^2 H^{(k)},\delta^2 H^{(m)}\bra+\bc''_{8}(n,k,m)\bla \pa \delta^2 H^{(k)},\pa \delta^2H^{(m)}\bra
\end{align*}
for some $\bc''_1,\ldots,\bc''_8 \in \R$. However, we will keep using \eqref{eq:I_1 coro Q6} for computational simplicity.

\noindent 3. Given $4 \le \msfk < \frac{n}{2}$, let $L^{2\msfk}_2$ be the second order expansion of $P_g^{(2\msfk)}-\Delta_g^{\msfk}$ in $h$.
We \textbf{conjecture} that the second-order expansions of $\int_{\R^n}ZL^{2\msfk}_2[H^{(k)},H^{(m)}]w$ for all $Q^{(2\msfk)}$-curvatures share the following pattern:
\[\begin{medsize}
\displaystyle \sum_{\ell=0}^{\msfk} c_{\ell}(n,k,m)\bla\mcl_k^{\ell}H^{(k)},H^{(m)}\bra + \sum_{\ell=0}^{\msfk-1} c'_{\ell}(n,k,m)\bla\pa^{\ell}\delta H^{(k)},\pa^{\ell}\delta H^{(m)}\bra
+ \sum_{\ell=0}^{\msfk-2} c''_{\ell}(n,k,m)\bla\pa^{\ell}\delta^2H^{(k)},\pa^{\ell}\delta^2H^{(m)}\bra.
\end{medsize} \tag*{$\diamond$}\]
\end{rmk}

\begin{proof}[Proof of Proposition \ref{Positive definiteness Q6}]
Proceeding as in the proof of Proposition \ref{Positive definiteness}, we derive \eqref{eq:Idecom} for the $Q^{(6)}$-case.
The positive definiteness of each term on the right-hand side of \eqref{eq:Idecom} then follows from Corollary \ref{I_1 coro Q6} together with Lemmas \ref{lemma of case 16}--\ref{lemma of case 36} of Subsection \ref{sec:three cases Q6}.
This completes the proof of Proposition \ref{Positive definiteness Q6}.
\end{proof}

\subsection{Local sign restriction}
In this subsection, we list the results corresponding to those in Section \ref{sec:lsr}, omitting the proof.
\begin{prop}\label{prop:lsr Q6}
Let $10\le n \le 26$ and $\sigma_a \to \bsi \in M$ be an isolated simple blowup point for a sequence $\{u_a\}$ of solutions to \eqref{eq:maina} with $\msfk=3$, where $\det g_a = 1$ near $\sigma_a \in M$.
Assume that $u_a(\sigma_a)u_a \to \mfg$ in $C^5_{\mathrm{loc}}(B^{g_{\infty}}_{\msfd_0}(\bsi) \setminus \{\bsi\})$, where $\msfd_0 \in (0,1)$ is the number in Proposition \ref{prop:refest Q6} and $\mfg$ is the function in Proposition \ref{prop:isosim decay}. Then
$$
\liminf_{r\to 0}\mbp^{(6)}(r,\mfg)\ge 0.
$$
\end{prop}

\begin{prop}\label{global prop Q6}
For any given $\ep, R>0$, there exists a constant $C=C(M,g,\ep,R)>0$ such that if a solution $u$ of equation \eqref{eq:main6} satisfies $\max_{\sigma\in M} u(\sigma)\ge C$, then there exist local maximum points $\sigma^1,\ldots,\sigma^N \in M$, where $N \in \N$ depends on $u$, such that
\begin{itemize}
\item[(i)] $\{B_{r_i}(\sigma^i)\}_{1\le i\le N}$ are disjoint, where $r_i=Ru(\sigma^i)^{-\frac{2}{n-6}}$;
\item[(ii)] For each $i=1,\ldots,N$, it holds that $\norm{C^6(B_R)}{u(\sigma^i)^{-1} u(u(\sigma^i)^{-\frac{2}{n-6}}\cdot)-w}<\ep$;
\item[(iii)] $u(\sigma) \le Cd_g(\sigma,\{\sigma^1,\ldots,\sigma^N\})^{-\frac{n-6}{2}}$ for any $\sigma\in M \setminus \{\sigma^1,\ldots,\sigma^N\}$.
\end{itemize}
\end{prop}

\begin{cor}
Let $10 \le n \le 26$ and $\sigma_a \to \bsi \in M$ be an isolated blowup point for a sequence $\{u_a\}$ of solutions to \eqref{eq:maina} with $\msfk=3$, where $\det g_a = 1$ near $\sigma_a \in M$. Then $\sigma_a \to \bsi \in M$ is isolated simple.
\end{cor}

\begin{cor}
Let $10 \le n\le 26$. Let $\ep,\, R,\, C=C(M,g,\ep,R),\, u$ and $\{\sigma^1,\ldots,\sigma^N\}$ be as in Proposition \ref{global prop Q6}.
Suppose that $\ep > 0$ is sufficiently small and $R>0$ is sufficiently large. Then, there exists $\oc=\oc(M,g,\ep, R) > 0$ such that if $\max_{\sigma\in M} u(\sigma)\ge C$, then $d_g(\sigma^i,\sigma^j)\ge \oc$ for any $i\neq j\in \{1,\ldots, N\}$.
\end{cor}

\subsection{Proof of Theorem \ref{thm:main6}}
Because $\mathrm{Ker}P_g^{(6)} = \{0\}$, the Green's function $G_g^{(6)}$ of the sixth-order GJMS operator $-P_g^{(6)}$ has the expansion
\[G_g(x,0)=\vsi_n^{-1}|x|^{6-n}\left[1+\sum_{k=4}^{n-6}\psi^{(k)}(x)\right]+A+B\log|x|+O(|x|),\]
in normal coordinates $x$ centered at a point $\bsi \in M$, provided $\det g = 1$ near $\bsi$. Here, $\vsi_n=8(n-6)(n-4)(n-2)|\S^{n-1}|$, $A, B \in \R$, $\psi^{(k)}\in \mcp_k$, and $\int_{\S^{n-1}}\psi^{(n-6
)}=0$. If $n$ is odd, we also have $B=0$.
For its proof, refer to \cite[Proposition 2.1]{CH}.

Using Theorem \ref{thm:Weyl2 Q6}, we can obtain the following lemma.
\begin{lemma}
Let $10\leq n\leq 26$ and $\sigma_a \to \bsi \in M$ be an isolated simple blowup point for a sequence $\{u_a\}$ of solutions to \eqref{eq:maina} with $\msfk=3$, where $\det g_a = 1$ near $\sigma_a \in M$.
Assume that $u_a(\sigma_a)u_a \to \vsi_nG_{g_{\infty}}^{(6)}(\cdot,\bsi)+\mfh$ in $C^5_{\mathrm{loc}}(B^{g_{\infty}}_{\msfd_0}(\bsi) \setminus \{\bsi\})$,
where $\msfd_0 \in (0,1)$ is the number in Proposition \ref{prop:refest Q6}, $G_{g_{\infty}}^{(6)}$ is the Green's function of the GJMS operator $-P_{g_{\infty}}^{(6)}$,
and $\mfh \in C^7(B^{g_{\infty}}_{\msfd_0}(\bsi))$ is a nonnegative function from Proposition \ref{prop:isosim decay}. Then it holds that
\begin{equation}\label{expansion of Green's function Q6}
\begin{aligned}
G_{g_{\infty}}^{(6)}(x,0)&=\vsi_n^{-1}|x|^{6-n}\left[1+\sum_{k=d+1}^{n-6}\psi^{(k)}(x)\right]+A+O^{(6)}(|x|),\\
\text{where} &\quad \int_{\S^{n-1}}\psi^{(k)}=0 \quad \text{and} \quad \int_{\S^{n-1}}x_i\psi^{(k)}=0.
\end{aligned}
\end{equation}
\end{lemma}

However, a direct computation using equations \eqref{eq:Poho Q6} and \eqref{expansion of Green's function Q6} tells us that
\[\lim_{r\to 0}\mbp^{(6)}(r,G_{g_{\infty}}^{(6)})=-\frac{n-6}{2}A<0,\]
which contradicts with Proposition \ref{prop:lsr Q6}. This finishes the proof of Theorem \ref{thm:main6}.

\subsection{Evaluation of the Pohozaev quadratic form}\label{sec:three cases Q6}
By slightly modifying the argument in Section \ref{sec:tech}, we examine the positivity of the Pohozaev quadratic form in three mutually exclusive cases, which is crucial in the proof of Proposition \ref{Positive definiteness Q6}.
We set $\theta_k=1$ if $k=\frac{n-6}{2}$ and $\theta_k=0$ otherwise.

\begin{lemma}\label{lemma of case 16}
Assume that $10\le n\le 26$ and $s=2,\ldots, d$. Let $M_q^{(s)}$ be the matrix defined in \eqref{eq:MqsEDs} and $\ep > 0$ small. Then there exists a constant $C=C(n,s) > 0$ such that
\begin{multline*}
\sum_{q,q'=0}^{\lfloor \frac{d-s}{2} \rfloor} |\log\ep|^{\theta_{q+q'+s}} J_1\left[\ep^{2q+s}|x|^{2q}M_q^{(s)},\ep^{2q'+s}|x|^{2q'}M_{q'}^{(s)}\right] \\
\ge C^{-1} \sum_{q=0}^{\lfloor \frac{d-s}{2} \rfloor} \ep^{2(2q+s)}|\log\ep|^{\theta_{2q+s}} \left\||x|^{2q}M_q^{(s)}\right\|^2.
\end{multline*}
\end{lemma}
\begin{proof}
Fixing any $s = 2,\ldots, d$, let $k=2q+s$, $m=2q'+s$, $\lambda_q= -q(n+2q+2s-2)$, and $\lambda_{q'}= -q'(n+2q'+2s-2)$. By applying Corollary \ref{I_1 coro Q6} and \eqref{eq:hatDq1}--\eqref{eq:hatDq2}, we compute
\begin{align}
&\ J_1\left[|x|^{2q}M_q^{(s)},|x|^{2q'}M_{q'}^{(s)}\right] \nonumber \\
&= \left[\bc'_1(n,k,m)+\bc'_2(n,k,m)\lambda_q+\bc'_3(n,k,m)\lambda_q\lambda_{q'} +\bc'_4(n,k,m)\lambda_q\lambda_{q'}(\lambda_q+\lambda_{q'})\right] \bla M_q^{(s)},M_{q'}^{(s)} \bra \nonumber \\
&:= (m^{D,s}_{qq'})' \bla M_q^{(s)},M_{q'}^{(s)} \bra. \label{eq:mDs0 Q6}
\end{align}
Then we set
\begin{equation}\label{eq:mDs Q6}
m^{D,s}_{qq'}:=N_0^{\theta_{\frac{k+m}{2}}}(m^{D,s}_{qq'})',
\end{equation}
where $N_0 \in \N$ is taken to be large enough; for example, $N_0 = 10^{10}$ suffices.

Using Mathematica, we observe that matrices $(m^{D,s}_{qq'})$ are positive-definite for all $s=2,\ldots, d$ when $n\le 26$. In addition, $(m^{D,2}_{qq'})$ has a negative eigenvalue when $n\ge 27$.

Following the rest of the proof of Lemma \ref{lemma of case 1}, we complete the proof.
\end{proof}

\begin{lemma}
Assume that $12 \le n \le 29$ and $s=1,\ldots, d-2$. Let $V_q^{(s+1)}$ be the vector field defined in \eqref{eq:MqsEDs} and $\ep > 0$ small. Then there exists a constant $C=C(n,s) > 0$ such that
\begin{multline*}
\sum_{q,q'=1}^{\lfloor \frac{d-s}{2} \rfloor} |\log\ep|^{\theta_{q+q'+s}} J_1\left[\proj\left[|x|^{2q}\msd V_q^{(s+1)}\right], \proj\left[|x|^{2q'}\msd V_{q'}^{(s+1)}\right]\right] \\
\ge C^{-1} \sum_{q=1}^{\lfloor \frac{d-s}{2} \rfloor} \ep^{2(2q+s)}|\log\ep|^{\theta_{2q+s}} \left\|\proj\left[|x|^{2q}\msd V_q^{(s+1)}\right]\right\|^2.
\end{multline*}
\end{lemma}
\begin{proof}
Fixing any $s = 1,\ldots, d-2$, let $k=2q+s$, $m=2q'+s$, $\lambda_q= -\frac{1}{2}(n+s+2q-2)(s+2q)$, and $\lambda_{q'}= -\frac{1}{2}(n+s+2q'-2)(s+2q')$.
By applying Corollary \ref{I_1 coro Q6}, \eqref{eq:hatWq1}--\eqref{eq:hatWq2}, and Lemma \ref{lemma:addition info of eigenvector}, we compute
\begin{align*}
&\ J_1\left[\proj\left[|x|^{2q}\msd V_q^{(s+1)}\right],\proj\left[|x|^{2q'}\msd V_{q'}^{(s+1)}\right]\right]\\
&=\left\{\left[\bc'_1(n,k,m)+\bc'_2(n,k,m)\lambda_q+\bc'_3(n,k,m)\lambda_q\lambda_{q'} +\bc'_4(n,k,m)\lambda_q\lambda_{q'}(\lambda_q+\lambda_{q'})\right]2s(n+s)\right.\\
&\hspace{15pt} +\bc'_5(n,k,m)s^2(n+s)^2+\bc'_6(n,k,m)s^2(n+s)^2\\
&\left.\hspace{70pt} \times \left[4(q-1)(q'-1)+2(s+1)(q+q'-2)+(s+1)(n+2s)\right]\right\}\bla V_q^{(s+1)},V_{q'}^{(s+1)} \bra\\
&:= (m^{W,s}_{qq'})' \bla V_q^{(s+1)},V_{q'}^{(s+1)} \bra.
\end{align*}
Then we set $m^{W,s}_{qq'}$ by exploiting \eqref{eq:mDs Q6} in which all the superscripts $D$ are replaced with $W$.

Using Mathematica, we observe that matrices $(m^{W,s}_{qq'})$ are positive-definite for all $s=1,\ldots, d-2$ when $n\le 29$. In addition, $(m^{W,1}_{qq'})$ has a negative eigenvalue when $n\ge 30$.

Following the rest of the proof of Lemma \ref{lemma of case 1}, we complete the proof.
\end{proof}

\begin{lemma}\label{lemma of case 36}
Assume that $14 \le n \le 33$ and $s=2,\ldots, d-2$. Let $\check{H}_q^{(2q+s)}$ be the matrix defined in \eqref{eq:PqsEHs} and $\ep > 0$ small. Then there exists a constant $C=C(n,s) > 0$ such that
\[\sum_{q,q'=1}^{\lfloor \frac{d-s}{2} \rfloor} I_{\ep}\left[\check{H}_q^{(2q+s)},\check{H}_{q'}^{(2q'+s)}\right] \ge C^{-1} \sum_{q=1}^{\lfloor \frac{d-s}{2} \rfloor} \ep^{2(2q+s)}|\log\ep|^{\theta_{2q+s}} \big\|\hat{H}_q^{(2q+s)}\big\|^2.\]
\end{lemma}
\begin{proof}
Fixing any $s=2,\ldots,d-2$, let $k=2q+s$ and $m=2q'+s$.
Also, let $\hat{H}_q^{(2q+s)} = \proj[|x|^{2q+2}\nabla^2P_q^{(s)}]$ be an eigenvector of the operator $\mcl_k$ with eigenvalue $A_{2q+s,q}=(s-1)[2-\frac{n-2}{n-1}(n+s-1)]-(q+1)(n+2q+2s-4)$.
We write $\lambda_q= A_{2q+s,q}$, $\lambda_{q'}= A_{2q'+s,q'}$, and $\ka_s=\frac{n-2}{n-1}s(s-1)(n+s-1)(n+s-2)$.

By applying Corollary \ref{I_1 coro Q6} and \eqref{eq:hatHq1}--\eqref{eq:hatHq2} and Lemma \ref{lemma:addition info of eigenvector}, we find
\begin{align}
&\ J_1\left[\hat{H}_q^{(k)},\hat{H}_{q'}^{(m)}\right] \nonumber \\
&=\left\{\left[\bc'_1(n,k,m)+\bc'_2(n,k,m)\lambda_q + \bc'_3(n,k,m)\lambda_q\lambda_{q'}+\bc'_4(n,k,m)\lambda_q\lambda_{q'}(\lambda_q+\lambda_{q'})\right]\ka_s\right. \nonumber \\
&\hspace{15pt} +\bc'_5(n,k,m)\frac{\ka_s^2}{s(n+s-2)} +\bc'_6(n,k,m)\frac{\ka_s^2}{s^2(n+s-2)^2}\left[s(s-1)(n+2s-2)(n+2s-4)\right. \nonumber \\
&\hspace{25pt} +s(n+2s-2)(4qq'-2(q+q')-s^2+2s(q+q'+1)) \label{eq:J_1 H Q6} \\
&\hspace{25pt} \left.+s^2(2s(1-(q+q'))-4qq'+2(q+q')+n-4)\right]+\bc'_7(n,k,m)\ka_s^2 \nonumber \\
&\hspace{15pt} +\left.\bc'_8(n,k,m)\ka_s^2[4(q-1)(q'-1)+2s(q+q'-2)+s(n+2s-2)]\right\}\bla P_q^{(s+1)},P_{q'}^{(s+1)} \bra \nonumber \\
&:= (m^{H,s,1}_{qq'})' \bla P_q^{(s+1)},P_{q'}^{(s+1)} \bra. \nonumber
\end{align}
Then we set $m^{H,s,1}_{qq'}$ as in \eqref{eq:mDs Q6}.

For the $I_{2,\ep}[\check{H}^{(k)}_{q},\check{H}^{(m)}_{q'}]$ term, we remind from \eqref{eq:ddhatH} and $s=(k-2)-2(q-1)$ that
\begin{equation}\label{eq:L_1Hw Q6}
\begin{aligned}
L_1[\hat{H}^{(k)}_q]w &= \mfl_1[\delta^2\hat{H}^{(k)}_q]w = \ka_s \mfl_1[|x|^{2(q-1)}P_q^{(s)}]w \\
&= -\frac{(n-6)\ka_s}{4(n-1)}\sum_{i=1}^{q+4}b_i(n,k-2,q-1)(1+r^2)^{-\frac{n+8-2i}{2}}P_q^{(s)},\\
\Psi[\hat{H}^{(k)}_q] &= \frac{(n-6)\ka_s}{4(n-1)} \sum_{j=1}^{q+2}\Gamma_j(n,k-2,q-1)(1+r^2)^{-\frac{n-2-2j}{2}}P_q^{(s)},
\end{aligned}
\end{equation}
where $b_i(n,k-2,q-1)$ and $\Gamma_j(n,k-2,q-1)$ are the numbers appearing in Definition \ref{linearized eqn RHS Q6} and Proposition \ref{prop:PsiQ6}, respectively. We also observe
\begin{multline*}
\(x_i\pa_i+\frac{n-6}{2}\)\((1+r^2)^{-\frac{n-2-2j}{2}}P_q^{(s)}\) \\
=(n-2-2j)(1+r^2)^{-\frac{n-2j}{2}}P_q^{(s)} - \frac{1}{2}(n+2-4j-2s)(1+r^2)^{-\frac{n-2-2j}{2}}P_q^{(s)}.
\end{multline*}
Making use of polar coordinates, we determine
\begin{align*}
&\ -\int_{B^n(0,\msfd_0\ep^{-1})} \(y_i\pa_i+\frac{n-6}{2}\)\Psi[\check{H}^{(k)}_{q}] L_1[\check{H}^{(m)}_{q'}]w dy \\
&= -\ep^{k+m}\left[\frac{(n-6)\ka_s}{4(n-1)}\right]^2 \sum_{i=1}^{q'+4}\sum_{j=1}^{q+2} b_i\Gamma_j \left[-(n-2-2j) \int_0^{\msfd_0\ep^{-1}} \frac{r^{n-1+2s}}{(1+r^2)^{n+4-i-j}} dr \right. \\
&\hspace{140pt} \left. + \frac{1}{2} (n+2-4j-2s)\int_0^{\msfd_0\ep^{-1}} \frac{r^{n-1+2s}}{(1+r^2)^{n+3-i-j}} dr\right] \bla P^{(s)}_q,P^{(s)}_{q'} \bra \\
&= \ep^{k+m}|\log\ep|^{\theta_{\frac{k+m}{2}}} (m^{H,s,2}_{qq'})' \bla P^{(s)}_q,P^{(s)}_{q'} \bra + O(\ep^{n-6}),
\end{align*}
where $b_i = b_i(n,m-2,q'-1)$, $\Gamma_j = \Gamma_j(n,k-2,q-1)$, and
\begin{multline*}
(m^{H,s,2}_{qq'})' \\
:= \begin{cases}
\begin{medsize}
\displaystyle \frac{1}{2}\left[\frac{(n-6)\ka_s}{4(n-1)}\right]^2 \sum_{i=1}^{q'+4}\sum_{j=1}^{q+2}b_i\Gamma_j \left[\frac{(n-2s-2i-2j+6)(n-2-2j)}{n+3-i-j}-(n+2-4j-2s)\right] \mci_{n+3-i-j}^{n-1+2s}
\end{medsize} \\
\hfill \begin{medsize}
\displaystyle \text{if } k+m=2(q+q'+s) < n-6,
\end{medsize} \\
\begin{medsize}
\displaystyle 0
\end{medsize}
\hfill \begin{medsize}
\displaystyle \text{if } k+m=2(q+q'+s) = n-6.
\end{medsize}
\end{cases}
\end{multline*}
We define
\begin{equation}\label{eq:mHs23 Q6}
m^{H,s,2}_{qq'} := \frac{1}{2} N_0^{\theta_\frac{k+m}{2}}\left[(m^{H,s,2}_{qq'})' + (m^{H,s,2}_{q'q})'\right]
\end{equation}
where $N_0 \in \N$ is taken to be large enough; for example, $N_0 = 10^{10}$ suffices. Clearly, $m^{H,s,2}_{qq'}=0$ when $k+m = n-6$. Then, by \eqref{eq:I_2 Q6},
\begin{equation}\label{eq:I_2H Q6}
I_{2,\ep}\left[\check{H}^{(k)}_{q},\check{H}^{(m)}_{q'}\right] = \ep^{k+m} \left[\frac{|\log\ep|}{N_0}\right]^{\theta_{\frac{k+m}{2}}} m^{H,s,2}_{qq'} \bla P^{(s)}_q,P^{(s)}_{q'} \bra + O(\ep^{n-6}).
\end{equation}

For the $I_{3,\ep}[\check{H}^{(k)}_{q},\check{H}^{(m)}_{q'}]$ term, we need to evaluate
\begin{equation}\label{eq:L_1HZ0 Q6}
L_1[\hat{H}_q^{(k)}]Z=-\frac{n-6}{2}L_1[\hat{H}_q^{(k)}] (1+r^2)^{-\frac{n-6}{2}}+(n-6)L_1[\hat{H}_q^{(k)}](1+r^2)^{-\frac{n-4}{2}}.
\end{equation}
We have \eqref{eq:L_1Hw Q6} and
\begin{equation}\label{eq:L_1HZ01 Q6}
L_1[\hat{H}^{(k)}_q](1+r^2)^{-\frac{n-4}{2}} = -\frac{\ka_s}{4(n-1)}\sum_{i=1}^{q+4} b'_i(n,k-2,q-1)(1+r^2)^{-\frac{n+10-2i}{2}}P_q^{(s)},
\end{equation}
where $b'_i(n,k-2,q-1)$ is the number appearing in Definition \ref{linearized eqn RHS Q6}. It follows that
\begin{align*}
&\ -\int_{B^n(0,\msfd_0\ep^{-1})} \Psi[\check{H}^{(k)}_{q}]L_1[\check{H}^{(m)}_{q'}]Z dy \\
&= -\ep^{k+m} \left[\frac{(n-6)\ka_s}{4(n-1)}\right]^2 \sum_{i=1}^{q'+4}\sum_{j=1}^{q+2} \Gamma_j \left[-b'_i \int_0^{\msfd_0\ep^{-1}} \frac{r^{n-1+2s}}{(1+r^2)^{n+4-i-j}} dr \right. \\
&\hspace{160pt} \left. + \frac{n-6}{2}b_i \int_0^{\msfd_0\ep^{-1}} \frac{r^{n-1+2s}}{(1+r^2)^{n+3-i-j}} dr\right] \bla P^{(s)}_q,P^{(s)}_{q'} \bra\\
&= \ep^{k+m}|\log\ep|^{\theta_{\frac{k+m}{2}}} (m^{H,s,3}_{qq'})' \bla P^{(s)}_q,P^{(s)}_{q'} \bra + O(\ep^{n-6}),
\end{align*}
where $b_i = b_i(n,m-2,q'-1)$, $b'_i = b'_i(n,m-2,q'-1)$, $\Gamma_j = \Gamma_j(n,k-2,q-1)$, and
\[(m^{H,s,3}_{qq'})' := \begin{cases}
\begin{medsize}
\displaystyle \frac{1}{2} \left[\frac{(n-6)\ka_s}{4(n-1)}\right]^2 \sum_{i=1}^{q'+4}\sum_{j=1}^{q+2}\Gamma_j \left[\frac{n-2s-2i-2j+6}{n+3-i-j}b'_i-(n-6)b_i\right] \mci_{n+3-i-j}^{n-1+2s}
\end{medsize} \\
\hfill \begin{medsize}
\displaystyle \text{if } k+m=2(q+q'+s) < n-6,
\end{medsize} \\
\begin{medsize}
\displaystyle -\frac{n-6}{2} \left[\frac{(n-6)\ka_s}{4(n-1)}\right]^2 \Gamma_{q+2}b_{q'+4}
\end{medsize}
\hfill \begin{medsize}
\displaystyle \text{if } k+m=2(q+q'+s) = n-6.
\end{medsize}
\end{cases}\]

If we define $m^{H,s,3}_{qq'}$ as in \eqref{eq:mHs23 Q6}, then by \eqref{eq:I_3 Q6},
\begin{equation}\label{eq:I_3H Q6}
I_{3,\ep}\left[\check{H}^{(k)}_{q},\check{H}^{(m)}_{q'}\right] = \ep^{k+m} \left[\frac{|\log\ep|}{N_0}\right]^{\theta_{\frac{k+m}{2}}} m^{H,s,3}_{qq'} \bla P^{(s)}_q,P^{(s)}_{q'} \bra + O(\ep^{n-6}).
\end{equation}

Adding up \eqref{eq:I_1J_1 Q6}, \eqref{eq:J_1 H Q6}, \eqref{eq:I_2H Q6}, and \eqref{eq:I_3H Q6}, we arrive at
$$
(I_{1,\ep}+I_{2,\ep}+I_{3,\ep})\left[\check{H}_q^{(k)},\check{H}_{q'}^{(m)}\right] = \ep^{k+m} \left[\frac{|\log\ep|}{N_0}\right]^{\theta_{\frac{k+m}{2}}} m^{H,s}_{qq'} \bla P^{(s)}_q,P^{(s)}_{q'} \bra + O(\ep^{n-6}),
$$
where $m^{H,s}_{qq'} := m^{H,s,1}_{qq'} + m^{H,s,2}_{qq'} + m^{H,s,3}_{qq'}$. Using Mathematica, we observe that matrices $(m^{H,s}_{qq'})$ are positive-definite for all $s=2,\ldots, d-2$ when $n\le 33$. In addition, $(m^{H,2}_{qq'})$ has a negative eigenvalue when $n\ge 34$.

Following the rest of the proof of Lemma \ref{lemma of case 16}, we complete the proof.
\end{proof}

\section{Non-compactness of the Constant $Q^{(6)}$-curvature Problem}\label{sec:noncomp6}
Building on the computations from the previous section, we now construct an example of $L^{\infty}$-unbounded solutions to \eqref{eq:main6} on $M = \S^n$ for all $n \ge 27$, thereby proving Theorem \ref{thm:main6n}.
This serves as a sixth-order analogue of the results of Brendle \cite{Br2} and Brendle and Marques \cite{BM} for the Yamabe problem, and that of Wei and Zhao \cite{WZ} for the $Q^{(4)}$-curvature problem.

\subsection{The background metric and the energy functional}
Let $W_{iklj}$ be a fixed multilinear form satisfying all algebraic properties of the Weyl tensor. Define
\begin{align*}
|W|_{\sharp} &:= \Bigg[\sum_{i,j,k,l=1}^n (W_{iklj} + W_{ilkj})^2\Bigg]^\frac{1}{2};\\
(W\times W)_{\msfpq} &:= \sum_{i,j,k=1}^n (W_{ik\msfp j} + W_{i\msfp kj})(W_{ik\msfq j} + W_{i\msfq kj})
\end{align*}
for $\msfp, \msfq = 1,\ldots,n$. Moreover, we assume that the matrix $((W\times W)_{\msfpq})_{\msfp,\msfq=1}^n$ is non-singular.
Since $((W\times W)_{\msfpq})_{\msfp,\msfq=1}^n$ can be viewed as a Gram matrix of a collection of vectors in the space of trilinear forms endowed with the natural inner product, it is positive definite. Consequently, $|W|_{\sharp}^2=\tr (W\times W)>0$.
Let us recall from \eqref{eq:mcv decom2} that $\mcd_k = \{D\in \mcv_k\mid \delta D=0\}$. By the algebraic properties of $W$, we know that
$$
(W_{iklj}x_kx_l)_{i,j=1}^n \in\{M \in \mcd_{2}\mid \Delta M_{ij}=0 \text{ for each } i,j=1,\ldots,n\}.
$$

We fix $s=2$ and $q=0,\ldots,4$. Denoting $k=2q+s$, we set matrix functions on $\R^n$:
\begin{equation}\label{eq:Hkij}
H^{(k)}_{ij}(x) := \msfa_q|x|^{2q}W_{iabj}x_ax_b\in \mcd_k,
\end{equation}
where the coefficients $\msfa_q \in \R$ will be determined later. Then we define
\begin{equation}\label{eq:Hij}
H_{ij}(x) := \sum_{q=0}^{4} H^{(2q+s)}_{ij}(x) = \sum_{q=0}^{4}\msfa_q|x|^{2q}W_{iabj}x_ax_b,
\end{equation}
(cf. \eqref{eq:H} and Lemma \ref{D hat}).

The tensor $H^{(k)}_{ij}$ in \eqref{eq:Hkij} exhibits commutative properties that will be useful in our subsequent computations: Let $q'=0,\ldots,4$ and $m=2q'+s$. Let $l,\, \bar{l}\geq 0$ be integers and $\msfp,\msfq = 1,\ldots,n$. Then,
\begin{align}
\Delta^{l} H^{(k)}_{\msfp i}\Delta^{\bar{l}} H^{(m)}_{\msfq i}&=\Delta^{l} H^{(k)}_{\msfq i}\Delta^{\bar{l}} H^{(m)}_{\msfp i}; \label{eq:commutative laplace} \\
\pa_{\msfp} H^{(k)}\cdot \pa_{\msfq} H^{(m)}&=\pa_{\msfq} H^{(k)}\cdot \pa_{\msfp} H^{(m)}. \label{eq:commutative pa_pq}
\end{align}

\begin{rmk}\label{rmk:log}
In Proposition \ref{I_1 Q6}, we required that $k=2,\ldots,\lfloor\frac{n-6}{2}\rfloor$. Here, we choose $k \leq 10$ to ensure that $2k<n-6$ holds when $n \geq 27$, which guarantees that logarithmic-order terms do not appear. \hfill $\diamond$
\end{rmk}

We consider a smooth Riemannian metric $g_0=\exp(h)$ on $\R^n$, where $h$ is a trace-free symmetric $2$-tensor on $\R^n$ such that $h(x)=0$ for $|x| \ge 1$ and
\begin{equation}\label{eq:hij}
h_{ij}(x)=\mu \ep^{10}H_{ij}(\ep^{-1}x) \quad \text{for } |x|\leq \rho.
\end{equation}
Here, $\ep$ and $\rho$ are small numbers satisfying $0 < \ep \ll \rho \leq \frac{1}{4}$ and $\mu:=\ep^{\frac{1}{4}}$. In addition, we impose the condition
\begin{equation}\label{eq:alpha0}
\sum_{m=0}^7 |\nabla^mh(x)| \le \alpha_0 \quad \text{for } x \in \R^n,
\end{equation}
where $\alpha_0$ is a small number satisfying $0 < \ep \ll \alpha_0 \leq 1$. Note that $h_{ii}(x) = 0$, $h_{ij}(x)x_j = 0$ and $\delta_i h(x)=0$ for $|x|\leq \rho$, and $dv_{g_0} = dx$.

We recall the bubble $w_{\dx}$ and the normalized bubble $w=w_{1,0}$ from \eqref{eq:bubble} (with $\msfk=3$).
Let $\Xi: \S^n \setminus \{(0,\ldots,0,1)\} \to \R^n$ be the stereographic projection, whose inverse is given as
\begin{equation}\label{eq:Xi}
\Xi^{-1}(x) = \(\frac{2x}{|x|^2+1},\frac{|x|^2-1}{|x|^2+1}\) \quad \text{for } x \in \R^n.
\end{equation}
Then, $g_{\S^n} = 4\, \Xi^*(w^{\frac{4}{n-6}}g_{\R^n})$. Besides, the pullback metric $g := 4\, \Xi^*(w^{\frac{4}{n-6}}g_0)$ in $\S^n \setminus \{(0,\ldots,0,1)\}$ smoothly extends to a metric on $\S^n$ (still denoted by $g$),
because $g_0$ equals the standard metric $g_{\R^n}$ in $\R^n$ outside of the unit ball $B_1$.

\medskip
Our next concern is the positivity of the operators $-P_{g_0}^{(6)}$ and $-P_g^{(6)}$, along with the positivity of their associated Green's functions $G_{g_0}^{(6)}$ and $G_g^{(6)}$, respectively.
\begin{lemma}\label{lemma:Green Q6}
There exists a constant $\alpha_0 = \alpha_0(n) > 0$ for which the following properties hold:
\begin{itemize}
\item[(i)] There exists a constant $C = C(n,\alpha_0) > 0$ such that
\begin{equation}\label{eq:coer Q6}
\int_{\R^n} u(-P_{g_0}^{(6)}u) dx \ge C^{-1} \sum_{m=0}^3 \|\nabla^mu\|_{L^{\frac{2n}{n-2(3-m)}}(\R^n)}^2 \quad \text{for all } u \in \dot{H}^3(\R^n)
\end{equation}
and
\begin{equation}\label{eq:coer2 Q6}
\int_{\S^n} u(-P_g^{(6)}u) dv_g \ge C^{-1}\|u\|_{H^3(\S^n)}^2 \quad \text{for all } u \in H^3(\S^n).
\end{equation}
In particular, $\textup{Ker}P_{g_0}^{(6)} = \textup{Ker}P_g^{(6)} = \{0\}$, and the Green's functions $G_{g_0}^{(6)}$ and $G_g^{(6)}$ exist.
\item[(ii)] It holds that $G_g^{(6)}(X,Y) > 0$ for all $X, Y \in \S^n$ such that $X \ne Y$.
    In particular, $G_{g_0}^{(6)}(x,y) > 0$ for all $x, y \in \R^n$ such that $x \ne y$.
\item[(iii)] If $u \in C^6(\S^n)$ satisfies $-P_g^{(6)}u \ge 0$ on $\S^n$, then either $u > 0$ or $u \equiv 0$ on $\S^n$.
\end{itemize}
\end{lemma}
\begin{proof}
(i) By \eqref{eq:P6}, \eqref{eq:alpha0}, $\text{supp}(h) \subset B_1$, H\"older's inequality and the Sobolev inequality, it holds that
\[\int_{\R^n} u(-P_{g_0}^{(6)}u) dx = \int_{\R^n} |\nabla\Delta u|^2 dx + O(\alpha_0) \sum_{m=0}^3 \int_{B_1} |\nabla^m u|^2 dx \ge C^{-1} \sum_{m=0}^3 \|\nabla^mu\|_{L^{\frac{2n}{n-2(3-m)}}(\R^n)}^2\]
for all $u \in \dot{H}^3(\R^n)$, which is \eqref{eq:coer Q6}.

On the other hand, according to Branson \cite{Bra},
\begin{equation}\label{eq:GJMS6 sphere}
-P_{g_{\S^n}}^{(6)}u = \prod_{m=1}^3 \left[-\Delta_{\S^n} + \frac{(n-2m)(n+2m-2)}{4}\right]u.
\end{equation}
It follows that
\[\int_{\S^n} u(-P_g^{(6)}u) dv_g = \int_{\S^n} u(-P_{g_{\S^n}}^{(6)}u) dv_{g_{\S^n}} + O(\alpha_0)\|u\|_{H^3(\S^n)}^2 \ge C^{-1}\|u\|_{H^3(\S^n)}^2\]
for all $u \in H^3(\S^n)$, which is \eqref{eq:coer2 Q6}.

The existence of $G_g^{(6)}$ follows from \eqref{eq:coer2 Q6}, the Lax-Milgram theorem, and elliptic regularity; refer to the proof of \cite[Proposition 2.4]{GM} for the fourth-order setting.
Moreover, owing to the conformal covariance property \eqref{eq:ccp}, we have
\begin{equation}\label{eq:Gg0}
G_{g_0}^{(6)}(x,y) = 2^{n-6}w(x)w(y)G_g^{(6)}(\Xi^{-1}(x),\Xi^{-1}(y)) \quad \text{for } x, y \in \R^n,\ x \ne y.
\end{equation}

\medskip \noindent (ii) Because each second-order factor on the right-hand side of \eqref{eq:GJMS6 sphere} has a positive Green's function, the Green's function $G_{g_{\S^n}}^{(6)}$ of $-P_{g_{\S^n}}^{(6)}$ is positive as well.
Also, since $P_g^{(6)}$ is a small lower-order perturbation of $P_{g_{\S^n}}^{(6)}$ in $L^{\infty}(\S^n)$, a suitable simplification of the argument in \cite[Section 5]{GS} (see also \cite[Section 5.1]{GGS}), which relies on the Neumann series, shows that $G_g^{(6)}$ is likewise positive.
The positivity of $G_{g_0}^{(6)}$ is an immediate consequence of \eqref{eq:Gg0}.

\medskip \noindent (iii) This follows directly from (ii).
\end{proof}

Throughout this section, we always assume that $n \ge 27$ and the parameters $(\dx)$ belongs to an admissible set
\[\mca := \(1-\vep_0, 1+\vep_0\) \times B^n(0,\vep_0),\]
where $\vep_0 \in (0,1)$ is a sufficiently small number.

Inspired by \cite{Br, Sc4}, we give the following definition.
\begin{defn}[Modified bubbles]
Let
\begin{equation}\label{eq:vdx}
v_{\dx}(y) = \zeta_{\rho^{-1}\ep}(|y-\xi|)w_{\dx}(y) + [1-\zeta_{\rho^{-1}\ep}(|y-\xi|)]\delta^{\frac{n-6}{2}}\ga_n^{-1}G_{\tig_0}^{(6)}(y,\xi) \quad \text{for } y \in \R^n,
\end{equation}
where $\tig_0 := g_0(\ep\cdot)$, $\ga_n=2^{-7}\pi^{-n/2}\Gamma(\frac{n-6}{2})$, $G_{\tig_0}^{(6)}$ is the Green's function of $-P_{\tig_0}^{(6)}$,
$\zeta \in C_c^{\infty}([0,\infty))$ is a non-negative cutoff function satisfying $\zeta(r)=1$ for $r \in [0,1]$ and $\zeta(r)=0$ for $r \in [2,\infty)$, and $\zeta_{\rho^{-1}\ep}:=\zeta(\rho^{-1}\ep\cdot)$.
Given $\ep > 0$ small, we also set
\[v_{\edx}(x) = \ep^{-\frac{n-6}{2}}v_{\dx}(\ep^{-1}x) \quad \text{for } x \in \R^n. \tag*{$\diamond$}\]
\end{defn}
\noindent To build an example of $L^{\infty}$-unbounded solutions to \eqref{eq:main6} on $M = \S^n$, we will search for a solution to \eqref{eq:main6} in $(M,g) = (\R^n,g_0)$ of the form $v_{\edx} + \Psi_{\edx}$, where $\Psi_{\edx}$ is a correction term to be determined later.
The current definition of $v_{\edx}$ plays a crucial role in reducing the error in the neck region (see the proof of Lemma \ref{lemma:mcrest}) and thereby ensuring the positivity of our solution in $\R^n$.

For future use, we provide basic properties of $G_{\tig_0}^{(6)}$ here.
\begin{lemma}
For $\ep > 0$ small and $(\dx) \in \mca$, it holds that
\begin{equation}\label{eq:Gtig0}
G_{\tig_0}^{(6)}(y,\xi) = \ep^{n-6}G_{g_0}^{(6)}(\ep y,\ep\xi) \quad \text{for } y \in \R^n \setminus \{\xi\}.
\end{equation}
Moreover, there is a constant $C = C(n,g_0,\vep_0) > 1$ such that
\begin{equation}\label{eq:Gtig01}
G_{\tig_0}^{(6)}(y,\xi) \ge \frac{C^{-1}}{|y-\xi|^{n-6}} \quad \text{and} \quad \left|\nabla^m_y G_{\tig_0}^{(6)}(y,\xi)\right| \le \frac{C}{|y-\xi|^{n-6+m}}
\end{equation}
for $|y-\xi| \ge \frac{\rho}{\ep}$ and $m = 0,1,\ldots,6$, and
\begin{equation}\label{eq:Gtig02}
\left|\nabla^m_y\left[w_{\dx}(y)-\delta^{\frac{n-6}{2}}\ga_n^{-1}G_{\tig_0}^{(6)}(y,\xi)\right]\right| \le \frac{C\alpha_0\ep^2}{|y-\xi|^{n-8+m}}
\end{equation}
for $\frac{\rho}{2\ep} \le |y-\xi| \le \frac{2\rho}{\ep}$ and $m = 0,1,\ldots,6$. Here, $\alpha_0 > 0$ is the number appearing in \eqref{eq:alpha0}.
\end{lemma}
\begin{proof}
The identity \eqref{eq:Gtig0} is straightforward to verify.

Assume that $|x-\ep\xi| \ge \rho$. We have that $G_{g_0}^{(6)}(x,\ep\xi) = 2^{n-6}w(\ep\xi)w(x) G_g^{(6)}(\Xi^{-1}(x),\Xi^{-1}(\ep\xi))$, as follows from \eqref{eq:Gg0}.
Furthermore, a direct computation using \eqref{eq:bubble} and \eqref{eq:Xi} shows that $w(x)G_g^{(6)}(\Xi^{-1}(x),\Xi^{-1}(\ep\xi)) \ge C|x|^{6-n}$ and $|\nabla^m_x [w(x)G_g^{(6)}(\Xi^{-1}(x),\Xi^{-1}(\ep\xi))]| \le C|x|^{6-n-m}$.
Thus, \eqref{eq:Gtig01} follows from \eqref{eq:Gtig0}.

Next, suppose that $\frac{\rho}{2} \le |x-\ep\xi| \le 2\rho$. Because of \eqref{eq:alpha0}, $h_{ij}(0)=h_{ij,k}(0)=0$, and $|\ep\xi| \le \ep \ll \frac{\rho}{4} \le |x|$, the parametrix expansion of $G_{g_0}^{(6)}$ yields
\begin{equation}\label{eq:Gg01}
G_{g_0}^{(6)}(x,\ep\xi) = \frac{\ga_n}{|x-\ep\xi|^{n-6}} + O^{(6)}\(\frac{\alpha_0}{|x|^{n-8}}\).
\end{equation}
Inequality \eqref{eq:Gtig02} is a consequence of \eqref{eq:bubble}, \eqref{eq:Gg01}, and \eqref{eq:Gtig0}.
\end{proof}

In the following lemma, we expand the energy, defined in \eqref{eq:mcf}, of modified bubbles $\mcf_{g_0}(v_{\edx})$.
\begin{lemma}\label{lemma:mcfexp}
Let $L_1[h]$ and $L_2[h,h]$ be the differential operators from Lemma \ref{mce6}. For $(\dx) \in \mca$, we define
\begin{align*}
\Dot{\mcf}[h](\dx) &:=\frac{1}{2} \int_{\R^n} \(-w_{\dx} L_1[h]w_{\dx}-2h_{ij}\pa_i\Delta^2w_{\dx} \pa_j w_{\dx}-h_{ij}\pa_j \Delta w_{\dx}\pa_i\Delta w_{\dx}\);\\
\Ddot{\mcf}[h](\dx) &:=\frac{1}{2} \int_{\R^n} \(-w_{\dx} L_2[h,h]w_{\dx}+h_{ik}h_{kj}\pa_i\Delta^2w_{\dx} \pa_j w_{\dx}+\frac{1}{2}h_{ik}h_{kj}\pa_j \Delta w_{\dx}\pa_i\Delta w_{\dx}\) \\
&\ +\frac{1}{2}\int_{\R^n}\left[2h_{ij}\pa_iw_{\dx} \pa^2_{jk} (h_{kl}\pa_l\Delta w_{\dx})+h_{ij}\pa_iw_{\dx} \pa^2_{jk} \Delta(h_{kl}\pa_l w_{\dx})\right].
\end{align*}
If we also set $\Dot{\msfF} := \Dot{\mcf}[H]$ and $\Ddot{\msfF} := \Ddot{\mcf}[H]$ by replacing $h$ with $H$, then it holds that
\begin{equation}\label{eq:mcfexp}
\mcf_{g_0}(v_{\edx}) = \mcf_0(w) + \mu\ep^{10} \Dot{\msfF}(\dx) + \mu^2\ep^{20} \Ddot{\msfF}(\dx) + O\bigg(\mu^3\ep^{20(1+\vep_1)} + \bigg(\frac{\ep}{\rho}\bigg)^{n-6}\bigg)
\end{equation}
uniformly for $(\dx) \in \mca$, where $\mcf_0(w) := \frac{3}{n}\int_{\R^n} |\nabla\Delta w|^2 dx > 0$ and $\vep_1 \in (0,\frac{1}{4})$ is a small number.
\end{lemma}
\begin{proof}
Plugging $v_{\edx}$ into the energy functional \eqref{eq:mcf} and using \eqref{eq:Gtig01}, we find
\begin{align*}
\mcf_{g_0}(v_{\edx}) &= \left[\frac{1}{2} \int_{\R^n} |\nabla\Delta w|^2 - \frac{n-6}{2n}\mfc_6(n) \int_{\R^n} w^{\frac{2n}{n-6}}\right]
+ \frac{1}{2} \int_{\R^n} w_{\dx} (-P_{\tig_0}^{(6)}+\Delta^3)w_{\dx} \\
&\ + \frac{1}{2} \int_{B^n(\xi,\rho\ep^{-1})^c} v_{\dx} (-P_{\tig_0}^{(6)}v_{\dx}) - \frac{n-6}{2n}\mfc_6(n) \int_{B^n(\xi,\rho\ep^{-1})^c} v_{\dx}^{\frac{2n}{n-6}} \\
&= \mcf_0(w) + \Dot{\mcf}[h](\dx) + \Ddot{\mcf}[h](\dx) + O\bigg(\mu^3\ep^{20(1+\vep_1)} + \bigg(\frac{\ep}{\rho}\bigg)^{n-6}\bigg).
\end{align*}
Moreover,
\begin{align*}
\Dot{\mcf}[h](\dx) &= \frac{1}{2} \(\int_{B^n(0,\rho)} + \int_{B^n(0,\rho)^c}\) \(-w_{\dx} L_1[h]w_{\dx}-2h_{ij}\pa_i\Delta^2w_{\dx} \pa_j w_{\dx}-h_{ij}\pa_j \Delta w_{\dx}\pa_i\Delta w_{\dx}\) \\
&= \begin{medsize}
\displaystyle \frac{1}{2} \mu\ep^{10} \int_{B^n(0,\rho\ep^{-1})} \(-w_{\dx} L_1[H]w_{\dx}-2H_{ij}\pa_i\Delta^2w_{\dx} \pa_j w_{\dx}-H_{ij}\pa_j \Delta w_{\dx}\pa_i\Delta w_{\dx}\) + O\bigg(\bigg(\frac{\ep}{\rho}\bigg)^{n-6}\bigg)
\end{medsize} \\
&= \mu\ep^{10} \Dot{\msfF}(\dx) + O\bigg(\bigg(\frac{\ep}{\rho}\bigg)^{n-6}\bigg),
\end{align*}
and similarly,
\[\Ddot{\mcf}[h](\dx) = \mu^2\ep^{20} \Ddot{\msfF}(\dx) + O\bigg(\bigg(\frac{\ep}{\rho}\bigg)^{n-6}\bigg).\]
As a result, \eqref{eq:mcfexp} follows.
\end{proof}

Making use of the condition $\delta H=0$, we see that the first-order term $\Dot{\msfF}(\dx)$ vanishes.
\begin{lemma}\label{lemma:DotE}
For each $(\dx) \in \mca$, it holds that
\begin{equation}\label{eq:DotE}
\Dot{\msfF}(\dx)=0.
\end{equation}
\end{lemma}
\begin{proof}
Since $\Dot{R}[H]=0$, we see from \eqref{Q6 L1 formula} that
\begin{equation}\label{eq:DotE3}
L_1[H]=\frac{16}{n-2}\Dot{\ricci}_{ij}[H]\pa^2_{ij}\Delta+\frac{16}{n-2}\Dot{\ricci}_{ij,k}[H]\pa^3_{ijk} +\frac{8}{n-4}\Delta\Dot{\ricci}_{ij}[H]\pa^2_{ij}.
\end{equation}
Therefore,
\begin{align*}
&\ \Dot{\msfF}(\dx) \\
&\begin{medsize}
\displaystyle =-\frac{1}{2} \int_{\R^n} \(\frac{16}{n-2}\Dot{\ricci}_{ij}[H]w_{\dx}\pa^2_{ij}\Delta w_{\dx} + \frac{16}{n-2}\Dot{\ricci}_{ij,k}[H]w_{\dx}\pa^3_{ijk}w_{\dx} + \frac{8}{n-4}\Delta\Dot{\ricci}_{ij}[H]w_{\dx}\pa^2_{ij}w_{\dx}\)
\end{medsize} \\
&\ -\frac{1}{2} \int_{\R^n} \(2H_{ij}\pa_i\Delta^2w_{\dx} \pa_j w_{\dx}+H_{ij}\pa_j \Delta w_{\dx}\pa_i\Delta w_{\dx}\).
\end{align*}

Let us deal with $\int_{\R^n}\Dot{\ricci}_{ij,k}[H] w_{\dx}\pa^3_{ijk} w_{\dx}$ first. Integrating by parts, we obtain
\begin{multline}\label{eq:DotE1}
\int_{\R^n}\Dot{\ricci}_{ij}[H] \Delta\pa_i w_{\dx}\pa_j w_{\dx} \\
=\int_{\R^n}\pa_i w_{\dx}\(\Delta \Dot{\ricci}_{ij}[H] \pa_j w_{\dx}+2\Dot{\ricci}_{ij,k}[H] \pa^2_{jk} w_{\dx}+\Dot{\ricci}_{ij}[H]\Delta \pa_j w_{\dx}\).
\end{multline}
Since $\Dot{\ricci}_{ij}[H]$ is symmetric, the third term on the right-hand side of \eqref{eq:DotE1} cancels out with the term on the left-hand side. Hence,
\begin{equation}\label{eq:DotE2}
\int_{\R^n}\Delta \Dot{\ricci}_{ij}[H] \pa_i w_{\dx}\pa_j w_{\dx}=-2\int_{\R^n}\Dot{\ricci}_{ij,k}[H] \pa_i w_{\dx}\pa^2_{jk} w_{\dx}.
\end{equation}
Using $\delta_j\Dot{\ricci}[H]=0$ and integrating by parts on both sides of \eqref{eq:DotE2}, we observe that
$$
\int_{\R^n}\Delta \Dot{\ricci}_{ij}[H] w_{\dx}\pa^2_{ij} w_{\dx}=-2\int_{\R^n}\Dot{\ricci}_{ij,k}[H] w_{\dx}\pa^3_{ijk} w_{\dx}.
$$
It follows that
\begin{align*}
\Dot{\msfF}(\dx) &=-\frac{1}{2} \int_{\R^n} \left[\frac{16}{n-2}\Dot{\ricci}_{ij}[H]w_{\dx}\pa^2_{ij}\Delta w_{\dx} +\frac{16}{(n-4)(n-2)}\Delta\Dot{\ricci}_{ij}[H]w_{\dx}\pa^2_{ij}w_{\dx}\right] \\
&\ -\frac{1}{2} \int_{\R^n} \(2H_{ij}\pa_i\Delta^2w_{\dx} \pa_j w_{\dx}+H_{ij}\pa_j \Delta w_{\dx}\pa_i\Delta w_{\dx}\).
\end{align*}

We next prove that $\int_{\R^n}\Delta\Dot{\ricci}_{ij}[H]w_{\dx}\pa^2_{ij}w_{\dx}$ vanishes, by adapting the ideas from \cite[Proposition 13]{Br2} and \cite[Lemma 4.14]{WZ}.
On one hand, integrating by parts and using $\delta_j\Dot{\ricci}[H]=0$, we obtain
$$
\int_{\R^n}\Delta\Dot{\ricci}_{ij}[H]w_{\dx}\pa^2_{ij}w_{\dx} = -\int_{\R^n}\Delta\Dot{\ricci}_{ij}[H]\pa_iw_{\dx}\pa_j w_{\dx}.
$$
On the other hand, directly computing $\pa^2_{ij}w_{\dx}$ and $\pa_{i}w_{\dx}$ and using $\tr \Dot{\ricci}=0$, we obtain
$$
\int_{\R^n}\Delta\Dot{\ricci}_{ij}[H]w_{\dx}\pa^2_{ij}w_{\dx} = \frac{n-4}{n-6}\int_{\R^n}\Delta\Dot{\ricci}_{ij}[H]\pa_iw_{\dx}\pa_j w_{\dx}.
$$
Combining the two equations above, we arrive at $\int_{\R^n}\Delta\Dot{\ricci}_{ij}[H]w_{\dx}\pa^2_{ij}w_{\dx}=0$.

To show that $\int_{\R^n} \Dot{\ricci}_{ij}[H]w_{\dx}\pa^2_{ij}\Delta w_{\dx}$ vanishes, as \cite[Lemma 4.11]{WZ}, we combine the following three identities:
\begin{align*}
\int_{\R^n} \Dot{\ricci}_{ij}[H]w_{\dx}\pa^2_{ij}\Delta w_{\dx}&=-\int_{\R^n} \Dot{\ricci}_{ij}[H]\pa_iw_{\dx}\pa_{j}\Delta w_{\dx}=\int_{\R^n} \Dot{\ricci}_{ij}\pa^2_{ij}w_{\dx}\Delta w_{\dx};\\
\begin{medsize}
\displaystyle \int_{\R^n} \Dot{\ricci}_{ij}[H]w_{\dx}\pa^2_{ij}\Delta w_{\dx}
\end{medsize}
&\begin{medsize}
\displaystyle =\frac{2(n-2)}{n-6}\int_{\R^n} \Dot{\ricci}_{ij}[H]\pa_iw_{\dx}\pa_j\Delta w_{\dx}-\frac{n-2}{n-6}\int_{\R^n} \Dot{\ricci}_{ij}[H]\pa^2_{ij}w_{\dx}\Delta w_{\dx}.
\end{medsize}
\end{align*}

Finally, to show that both $\int_{\R^n}H_{ij}\pa_i\Delta^2w_{\dx} \pa_j w_{\dx}$ and $\int_{\R^n}H_{ij}\pa_j \Delta w_{\dx}\pa_i\Delta w_{\dx}$ vanish, we combine the following five identities:
\begin{align*}
\int_{\R^n}H_{ij}\pa_j \Delta w_{\dx}\pa_i\Delta w_{\dx}&=-\int_{\R^n}H_{ij}\Delta w_{\dx}\pa^2_{ij}\Delta w_{\dx};\\
\int_{\R^n}H_{ij}\pa^2_{ij}\Delta^2w_{\dx} w_{\dx}&=-\int_{\R^n}H_{ij}\pa_i\Delta^2w_{\dx} \pa_j w_{\dx}=\int_{\R^n}H_{ij}\Delta^2w_{\dx} \pa^2_{ij} w_{\dx};\\
\begin{medsize}
\displaystyle \int_{\R^n}H_{ij}\pa_j \Delta w_{\dx}\pa_i\Delta w_{\dx}
\end{medsize}
&\begin{medsize}
\displaystyle =c\int_{\R^n}H_{ij}\pa^2_{ij}\Delta^2w_{\dx} w_{\dx}+b\int_{\R^n}H_{ij}\pa_i\Delta^2w_{\dx} \pa_j w_{\dx}+a\int_{\R^n}H_{ij}\Delta^2w_{\dx} \pa^2_{ij}w_{\dx};
\end{medsize} \\
\begin{medsize}
\displaystyle \int_{\R^n}H_{ij}\Delta w_{\dx}\pa^2_{ij}\Delta w_{\dx}
\end{medsize}
&\begin{medsize}
\displaystyle =c'\int_{\R^n}H_{ij}\pa^2_{ij}\Delta^2w_{\dx} w_{\dx}+b'\int_{\R^n}H_{ij}\pa_i\Delta^2w_{\dx} \pa_j w_{\dx}+a'\int_{\R^n}H_{ij}\Delta^2w_{\dx} \pa^2_{ij}w_{\dx};
\end{medsize}
\end{align*}
where
\begin{align*}
c &= -\frac{1}{8}(n-4)(n+2)\frac{n-6}{n}, & b &= \frac{1}{4}(n-4)(n+2), & a &= -\frac{1}{8}(n^2-20),\\
c'&= -\frac{1}{8}(n-6)n, & b' &= \frac{1}{4}n^2, & a' &= -\frac{1}{8}(n+4)(n-2).
\end{align*}

Consequently, \eqref{eq:DotE} holds.
\end{proof}

\subsection{Linear theory}
The main objective of this subsection is to establish the invertibility of the operator $P_{g_0}^{(6)}+\tmfc_6(n) v_{\edx}^{\frac{12}{n-6}}$ in suitable weighted $L^{\infty}(\R^n)$ spaces.

\begin{defn}[Weighted $L^{\infty}(\R^n)$-norms]\label{def of norms}
\

Given a small number $\ep > 0$ and a parameter $(\dx) \in \mca$, we set two weighted $L^{\infty}(\R^n)$-norms:
\begin{align*}
\|\wtPsi\|_* &:= \sup_{y \in \R^n} \left[\chi_{\{|y-\xi| < \frac{\rho}{2\ep}\}} \(\frac{\mu\ep^{10}}{1+|y-\xi|^{n-16}} + \frac{\alpha_0\ep^{n-6}}{\rho^{n-6}}\)^{-1}
+ \chi_{\{|y-\xi| \ge \frac{\rho}{2\ep}\}} \frac{|y-\xi|^{n-6}}{\alpha_0}\right] |\wtPsi(y)|; \\
\|f\|_{**} &\begin{medsize}
\displaystyle := \sup_{y \in \R^n} \left[\chi_{\{|y-\xi| < \frac{\rho}{2\ep}\}} \frac{1+|y-\xi|^{n-10}}{\mu\ep^{10}}
+ \chi_{\{|y-\xi| \ge \frac{\rho}{2\ep},\, |y| < \frac{1}{\ep}\}}
\frac{|y-\xi|^{n-2}}{\alpha_0\ep^2}
+ \chi_{\{|y| \ge \frac{1}{\ep}\}} \frac{|y-\xi|^{n+\ka}}{\alpha_0}\right] |f(y)|,
\end{medsize}
\end{align*}
where $\mu=\ep^{\frac{1}{4}}$, $\ka \in (0,1)$ is a small constant, and $\rho \in (\ep,1]$ and $\alpha_0 > 0$ are the numbers appearing in \eqref{eq:hij}--\eqref{eq:alpha0}.
Besides, $\chi_S$ stands for the characteristic function of a set $S \subset \R^n$. \hfill $\diamond$
\end{defn}
\noindent The choice of the above norms was influenced by \cite{WZ, KMW2}.

\medskip
We consider the linear problem
\begin{equation}\label{eq:lin}
\begin{cases}
\displaystyle P_{\tig_0}^{(6)}\wtPsi + \tmfc_6(n) v_{\dx}^{\frac{12}{n-6}}\wtPsi = f + \sum_{i=0}^n t_iv_{\dx}^{\frac{12}{n-6}}Z_{\dx}^i &\text{in } \R^n,\\
\displaystyle \int_{\R^n} v_{\dx}^{\frac{12}{n-6}}Z_{\dx}^i \wtPsi = 0 &\text{for } i = 0, 1,\ldots, n,
\end{cases}
\end{equation}
where $\tig_0 = g_0(\ep\cdot)$, $\wtPsi \in \dot{H}^3(\R^n) \cap C^0(\R^n)$ and $f \in C^0(\R^n)$ satisfy $\|\wtPsi\|_* + \|f\|_{**} < \infty$, $(t_0, \ldots, t_n) \in \R^{n+1}$, and
$Z_{\dx}^i$ for $i=0,1,\ldots,n$ are the functions defined in \eqref{eq:bubbleZ}.

\begin{prop}\label{prop:lin}
Assume that $\ep > 0$ and $\rho, \alpha_0 \gg \ep > 0$ are small enough. Given any $(\dx) \in \mca$ and $f \in C^0(\R^n)$ such that $\|f\|_{**} < \infty$,
equation \eqref{eq:lin} admits a unique solution $\wtPsi \in \dot{H}^3(\R^n) \cap C^0(\R^n)$ and $(t_0, \ldots, t_n) \in \R^{n+1}$. Moreover, there exists a constant $C = C(n,\vep_0) > 0$ such that
\begin{equation}\label{eq:apriori}
\|\wtPsi\|_* \le C \|f\|_{**}.
\end{equation}
\end{prop}
\begin{proof}
The existence result follows from a standard application of the Fredholm alternative, once a priori estimate \eqref{eq:apriori} has been established.
The remainder of the proof is therefore devoted to deriving \eqref{eq:apriori}, which we only sketch since it is a slight modification of \cite[Proposition 5.1]{WZ}.
Our argument is simpler, however, as it uses the Green’s function of $P_{g_0}^{(6)}$, which was built in Lemma \ref{lemma:Green Q6}, instead of that of $-\Delta_{g_0}^3$.

\medskip
Suppose that \eqref{eq:apriori} is false. Then, there exist sequences $\{\ep_a\}_{a \in \N} \subset (0,1)$, $\{(\delta_a,\xi_a)\}_{a \in \N} \subset \mca$, $\{(t_i)_a\} \subset \R$ for $i=0,\ldots,n$,
$\{\wtPsi_a\}_{a \in \N} \subset \dot{H}^3(\R^n) \cap C^0(\R^n)$, and $\{f_a\}_{a \in \N} \subset C^0(\R^n)$ satisfying \eqref{eq:lin} such that $\|\wtPsi_a\|_* = 1$ for all $a \in \N$ and $\ep_a + \|f_a\|_{**} \to 0$ as $a \to \infty$.
For the sake of brevity, we will omit the subscripts $a \in \N$ for the rest of the proof.

First, we claim that
\begin{equation}\label{eq:apriori1}
\sum_{i=0}^n |t_i| = o(\mu\ep^{10}) \quad \text{as } \ep \to 0.
\end{equation}
Testing \eqref{eq:lin} with $Z_{\dx}^j$ for a fixed $j=0,\ldots,n$, using \eqref{eq:bubbleZeq}, \eqref{eq:vdx} and \eqref{eq:Gtig01}, and integrating by parts, we obtain
\[\int_{\R^n} (P_{\tig_0}^{(6)}-\Delta^3)Z_{\dx}^j \wtPsi + \tmfc_6(n) \int_{B^n(\xi,\rho\ep^{-1})^c} \(v_{\dx}^{\frac{12}{n-6}}-w_{\dx}^{\frac{12}{n-6}}\)\wtPsi = \int_{\R^n} fZ_{\dx}^j + \sum_{i=0}^n t_i \int_{\R^n} v_{\dx}^{\frac{12}{n-6}}Z_{\dx}^iZ_{\dx}^j.\]
Applying Lemma \ref{mce6} (see also \eqref{eq:DotE3}), we estimate each integral to find
\[\left[o(\mu\ep^{10}) + O\(\frac{\alpha_0\ep^{n-6}}{\rho^{n-6}}\)\right] \|\wtPsi\|_* = O\(\mu\ep^{10}+\frac{\alpha_0\ep^{n-6}}{\rho^{n-8}}\)\|f\|_{**} + t_jC_j + \sum_{i=0}^n |t_i| O\(\frac{\alpha_0\ep^n}{\rho^n}\)\]
for some numbers $C_0 > 0$ and $C_1 = \cdots = C_n > 0$ depending only on $n$, from which \eqref{eq:apriori1} follows.

Second, we claim that for $\ep > 0$ small enough, the maximum point $y_0$ of the function
\[\left[\chi_{\{|y-\xi| < \frac{\rho}{2\ep}\}} \(\frac{\mu\ep^{10}}{1+|y-\xi|^{n-16}} + \frac{\alpha_0\ep^{n-6}}{\rho^{n-6}}\)^{-1}
+ \chi_{\{|y-\xi| \ge \frac{\rho}{2\ep}\}} \frac{|y-\xi|^{n-6}}{\alpha_0}\right] |\wtPsi(y)| \quad \text{for } y \in \R^n\]
must belong to a large ball $B(\xi,r_0)$ for some $r_0 > 1$. For this purpose, we bound $\wtPsi$ by using the Green's representation formula: By Lemma \ref{lemma:Green Q6}, \eqref{eq:Green} and \eqref{eq:apriori1}, it holds that
\begin{align*}
|\wtPsi(y)| &\le C \left[\int_{\R^n} \frac{1}{|y-z|^{n-6}} \frac{|\wtPsi(z)|}{1+|z-\xi|^{12}} dz + \int_{\R^n} \frac{|f(z)|}{|y-z|^{n-6}} dz\right] \\
&\ + o(\mu\ep^{10}) \int_{\R^n} \frac{1}{|y-z|^{n-6}} \frac{1}{1+|z-\xi|^{n+6}} dz
\end{align*}
for $y \in \R^n$. Additionally,
\begin{align*}
&\ \int_{\R^n} \frac{1}{|y-z|^{n-6}} \frac{|\wtPsi(z)|}{1+|z-\xi|^{12}} dz \\
&\le C\|\wtPsi\|_* \left[\chi_{\{|y-\xi| < \frac{\rho}{\ep}\}} \(\frac{\mu\ep^{10}}{1+|y-\xi|^{n-10}} + \frac{\alpha_0\ep^{n-6}}{\rho^{n-6}} \frac{1}{1+|y-\xi|^6} + \frac{\alpha_0\ep^n}{\rho^n}\) \right. \\
&\hspace{200pt} \left. + \chi_{\{|y-\xi| \ge \frac{\rho}{\ep}\}} \frac{\alpha_0\ep^6}{\rho^6} \frac{1}{|y-\xi|^{n-6}}\right],
\end{align*}
\[\int_{\R^n} \frac{|f(z)|}{|y-z|^{n-6}} dz \le C\|f\|_{**} \left[\chi_{\{|y-\xi| < \frac{\rho}{\ep}\}} \(\frac{\mu\ep^{10}}{1+|y-\xi|^{n-16}} + \frac{\alpha_0\ep^{n-6}}{\rho^{n-8}}\) + \chi_{\{|y| \ge \frac{\rho}{\ep}\}} \frac{\alpha_0}{|y-\xi|^{n-6}}\right],\]
and
\[\int_{\R^n} \frac{1}{|y-z|^{n-6}} \frac{1}{1+|z-\xi|^{n+6}} dz \le \frac{C}{1+|y-\xi|^{n-6}}.\]
Consequently, if $|y_0-\xi| \ge \frac{\rho}{2\ep}$, then
\begin{multline*}
\frac{\alpha_0}{|y_0-\xi|^{n-6}} = |\wtPsi(y_0)| = o(1) \left[\chi_{\{\frac{\rho}{2\ep} \le |y-\xi| < \frac{\rho}{\ep}\}}(y_0) \(\frac{\mu\ep^{n-6}}{\rho^{n-16}} + \frac{\alpha_0\ep^{n-6}}{\rho^{n-8}}\) \right. \\
\left. + \chi_{\{|y-\xi| \ge \frac{\rho}{\ep}\}}(y_0) \frac{\alpha_0}{|y_0-\xi|^{n-6}}\right] + \frac{o(\mu\ep^{10})}{1+|y_0-\xi|^{n-6}},
\end{multline*}
so it cannot happen for $\ep > 0$ small. If $|y_0-\xi| < \frac{\rho}{2\ep}$, then
\begin{multline*}
\frac{\mu\ep^{10}}{1+|y_0-\xi|^{n-16}} + \frac{\alpha_0\ep^{n-6}}{\rho^{n-6}} = |\wtPsi(y_0)| \\
\le C\(\frac{\mu\ep^{10}}{1+|y_0-\xi|^{n-10}} + \frac{\alpha_0\ep^{n-6}}{\rho^{n-6}} \frac{1}{1+|y_0-\xi|^6} + \frac{\alpha_0\ep^n}{\rho^n}\) + \frac{o(\mu\ep^{10})}{1+|y_0-\xi|^{n-6}},
\end{multline*}
which tells us that $y_0 \in B(\xi,r_0)$ for some large $r_0 > 1$.

Now, a standard argument based on the non-degeneracy result for the linear equation \eqref{eq:bubbleZeq} (stated in Subsection \ref{subsec:bubble}) shows that $\|\wtPsi\|_* \to 0$ as $\ep \to 0$, leading to a contradiction.
\end{proof}

\subsection{Nonlinear problem}
We keep assuming that $\ep > 0$ and $\rho, \alpha_0 \gg \ep > 0$ are small enough. For $(\dx) \in \mca$, let $\mcr_{\dx}$ be the error term given by
\[\mcr_{\dx} := -P_{\tig_0}^{(6)}v_{\dx}-v_{\dx}^{\frac{n+6}{n-6}} \quad \text{in } \R^n,\]
where $\tih := h(\ep\cdot)$.

\begin{lemma}\label{lemma:mcrest}
It holds that
\begin{equation}\label{eq:mcrest}
\left\|\mcr_{\dx}\right\|_{**} \le C
\end{equation}
uniformly for $\ep > 0$ small and $(\dx) \in \mca$.
\end{lemma}
\begin{proof}
We first assume that $|y-\xi| < \frac{\rho}{2\ep}$. Then, $|\ep y| \le \rho$ and $v_{\dx}(y) = w_{\dx}(y)$.
Owing to \eqref{eq:DotE3} and $\delta H=0$ (that yields $\Dot{\ricci}_{ij}[H] = -\frac{1}{2}\Delta H_{ij}$), we have
\begin{align}
\mcr_{\dx}(y) &= -L_1[\tih]w_{\dx}(y) -(\Delta_{\tig_0}^3-\Delta^3)w_{\dx}(y) + O\Bigg(\sum_{\beta=0}^4 \sum_{\substack{0 \le \alpha_1,\alpha_2 \le 6-\beta \\ \alpha_1+\alpha_2=6-\beta}} \(|\pa^{\alpha_1}\tih||\pa^{\alpha_2}\tih| \pa^{\beta}w_{\dx}\)(y)\Bigg) \nonumber \\
&= \mcr_{\dx}^1(y) + O(\mu^2\ep^{20}(1+|y-\xi|)^{20-n}), \label{eq:mcrexp}
\end{align}
where
\begin{equation}\label{eq:mcr1}
\begin{aligned}
\mcr_{\dx}^1(y) &:= \mu\ep^{10} \left[\frac{8}{n-2}\Delta H_{ij}\pa^2_{ij}\Delta w_{\dx} + \frac{8}{n-2}\Delta H_{ij,k}\pa^3_{ijk}w_{\dx} + \frac{4}{n-4}\Delta^2H_{ij}\pa^2_{ij}w_{\dx}\right](y) \\
&\ + \mu\ep^{10} \left[\Delta^2(H_{ij}\pa^2_{ij} w_{\dx}) + \Delta(H_{ij}\pa^2_{ij}\Delta w_{\dx}) + H_{ij}\pa^2_{ij}\Delta^2w_{\dx}\right](y).
\end{aligned}
\end{equation}
It follows that
\[|\mcr_{\dx}(y)| \le \frac{C\mu\ep^{10}}{1+|y-\xi|^{n-10}}.\]

We next assume that $y$ belongs to the neck region $\{y: |y-\xi| \ge \frac{\rho}{2\ep},\, |\ep y| \le 1\}$.
Setting $v_{\dx}^1(y) := w_{\dx}(y) - \delta^{\frac{n-6}{2}}\ga_n^{-1}G_{\tig_0}^{(6)}(y,\xi)$, we observe that $v_{\dx}(y) = \zeta_{\rho^{-1}\ep}(|y-\xi|)v_{\dx}^1(y) + \delta^{\frac{n-6}{2}}\ga_n^{-1}G_{\tig_0}^{(6)}(y,\xi)$.
Thus, \eqref{eq:Gtig02} implies
\begin{align*}
|\mcr_{\dx}(y)| &\le \left|P_{\tig_0}^{(6)}\left[\zeta_{\rho^{-1}\ep}(|y-\xi|)v_{\dx}^1(y)\right]\right| + \frac{C}{|y-\xi|^{n+6}} \\
&\le C\sum_{\beta=0}^6 \left|(\pa^{6-\beta}\tig_0)(y)\right| \left|\pa^{\beta}\left[\zeta_{\rho^{-1}\ep}(|y-\xi|)v_{\dx}^1(y)\right]\right| + \frac{C}{|y-\xi|^{n+6}}
\le \frac{C\alpha_0\ep^2}{|y-\xi|^{n-2}}.
\end{align*}

Finally, if $|\ep y| \ge 1$, then $v_{\dx}(y) = \delta^{\frac{n-6}{2}}\ga_n^{-1}G_{\tig_0}^{(6)}(y,\xi)$, and so
\[\mcr_{\dx}(y) = -\(\delta^{\frac{n-6}{2}}\ga_n^{-1}G_{\tig_0}^{(6)}(y,\xi)\)^{\frac{n+6}{n-6}} = O\(\frac{1}{|y-\xi|^{n+6}}\) = O\(\frac{\alpha_0}{|y-\xi|^{n+\ka}}\)\]
for $\ka \in (0,1)$.

Combining the above computations leads to \eqref{eq:mcrest}.
\end{proof}

We next study the unique solvability of an intermediate nonlinear problem
\begin{equation}\label{eq:nonlin}
\begin{cases}
\displaystyle P_{\tig_0}^{(6)}\wtPsi + \tmfc_6(n) v_{\dx}^{\frac{12}{n-6}}\wtPsi = \mcr_{\dx} + N_{\dx}(\wtPsi) + \sum_{i=0}^n t_iv_{\dx}^{\frac{12}{n-6}}Z_{\dx}^i &\text{in } \R^n,\\
\displaystyle \int_{\R^n} v_{\dx}^{\frac{12}{n-6}}Z_{\dx}^i \wtPsi = 0 &\text{for } i = 0, 1,\ldots, n,
\end{cases}
\end{equation}
where
\[N_{\dx}(\wtPsi) := -\mfc_6(n) \left[|v_{\dx}+\wtPsi|^{\frac{12}{n-6}}(v_{\dx}+\wtPsi) - v_{\dx}^{\frac{n+6}{n-6}}-\(\frac{n+6}{n-6}\) v_{\dx}^{\frac{12}{n-6}}\wtPsi\right].\]
\begin{prop}
For $\ep > 0$ small and $(\dx) \in \mca$, equation \eqref{eq:nonlin} admits a unique solution $\wtPsi_{\dx} \in \dot{H}^3(\R^n) \cap C^0(\R^n)$ and $((t_0)_{\dx}, \ldots, (t_n)_{\dx}) \in \R^{n+1}$ such that
\begin{equation}\label{eq:wtPsi}
\|\wtPsi_{\dx}\|_* \le C
\end{equation}
for some $C = C(n,\vep_0) > 0$. Furthermore, the map $(\dx) \in \mca \mapsto \wtPsi_{\dx}$ is of class $C^1$ and
\[\|\nabla_{(\dx)}\wtPsi_{\dx}\|_* \le C\]
for some $C = C(n,\vep_0) > 0$.
\end{prop}
\begin{proof}
One can follow the proofs of \cite[Propositions 6.1 and 6.2]{WZ}. Particularly, one can derive
\begin{equation}\label{eq:Nest}
\|N_{\dx}(\wtPsi)\|_{**} \le C\left\|(1+|\cdot-\xi|)^{n-18}\wtPsi^2\right\|_{**} \le C\ep^{6-\ka}\|\wtPsi\|_*^2 \le C\ep^{6-\ka},
\end{equation}
where $\ka \in (0,1)$ is the small number from the $**$-norm in Definition \ref{def of norms}. We omit the details.
\end{proof}

We now turn to refining the estimate of $\wtPsi_{\dx}$. Let $\mcr_{\dx}^1$ be the function in \eqref{eq:mcr1}.
By the proof of Lemma \ref{lemma:mcrest}, we have that $\|\mcr_{\dx}^1\|_{**} \le C$.
Thus, a slight modification of the proof of Proposition \ref{prop:lin} yields the unique existence of the solution $\wtPsi_{\dx}^1 \in \dot{H}^3(\R^n) \cap C^0(\R^n)$ and $(t_0^1, \ldots, t_n^1) \in \R^{n+1}$ to the linear problem
\begin{equation}\label{eq:wtPsi1}
\begin{cases}
\displaystyle \Delta^3\wtPsi_{\dx}^1 + \tmfc_6(n) v_{\dx}^{\frac{12}{n-6}}\wtPsi_{\dx}^1 = \zeta_{\rho^{-1}\ep}(|\cdot-\xi|)\mcr_{\dx}^1 + \sum_{i=0}^n t_i^1v_{\dx}^{\frac{12}{n-6}}Z_{\dx}^i &\text{in } \R^n,\\
\displaystyle \int_{\R^n} v_{\dx}^{\frac{12}{n-6}}Z_{\dx}^i \wtPsi_{\dx}^1 = 0 &\text{for } i = 0, 1,\ldots, n,
\end{cases}
\end{equation}
such that $\|\wtPsi_{\dx}^1\|_* \le C$. Here, $\zeta_{\rho^{-1}\ep}$ denotes the cutoff function from \eqref{eq:vdx}.
\begin{defn}[Modified weighted $L^{\infty}(\R^n)$-norms]
\

Given $\ep > 0$ small and $(\dx) \in \mca$, we set two additional weighted $L^{\infty}(\R^n)$-norms:
\begin{align*}
\|\wtPsi\|_*' &:= \sup_{y \in \R^n} \left[\chi_{\{|y-\xi| < \frac{\rho}{2\ep}\}} \(\frac{\mu^2\ep^{20}}{1+|y-\xi|^{n-26}} + \frac{\alpha_0\ep^{n-6}}{\rho^{n-6}}\)^{-1}
+ \chi_{\{|y-\xi| \ge \frac{\rho}{2\ep}\}} \frac{|y-\xi|^{n-6}}{\alpha_0}\right] |\wtPsi(y)|; \\
\|f\|_{**}' &\begin{medsize}
\displaystyle := \sup_{y \in \R^n} \left[\chi_{\{|y-\xi| < \frac{\rho}{2\ep}\}} \frac{1+|y-\xi|^{n-20}}{\mu^2\ep^{20}}
+ \chi_{\{|y-\xi| \ge \frac{\rho}{2\ep},\, |y| < \frac{1}{\ep}\}}
\frac{|y-\xi|^{n-2}}{\alpha_0\ep^2}
+ \chi_{\{|y| \ge \frac{1}{\ep}\}} \frac{|y-\xi|^{n+\ka}}{\alpha_0} \right] |f(y)|.
\end{medsize}
\end{align*}
Here, the parameters $\mu$, $\ka$, $\rho$, and $\alpha_0$ are the ones in Definition \ref{def of norms}. \hfill $\diamond$
\end{defn}
\begin{lemma}
It holds that
\begin{equation}\label{eq:wtPsiest ref}
\left\|\wtPsi_{\dx}-\wtPsi_{\dx}^1\right\|_*' \le C
\end{equation}
uniformly for $\ep > 0$ small and $(\dx) \in \mca$.
\end{lemma}
\begin{proof}
Standard interior elliptic estimates for $\wtPsi_{\dx}^1$ yield
\begin{equation}\label{eq:wtPsidec2}
\begin{medsize}
\displaystyle \left|\pa_{\beta}\wtPsi_{\dx}^1\right|(y) \le C\left[\chi_{\{|y-\xi| < \frac{\rho}{2\ep}\}} \frac{1}{1+|y-\xi|^{|\beta|}}\(\frac{\mu\ep^{10}}{1+|y-\xi|^{n-16}} + \frac{\alpha_0\ep^{n-6}}{\rho^{n-6}}\)
+ \chi_{\{|y-\xi| \ge \frac{\rho}{2\ep}\}} \frac{\alpha_0}{|y-\xi|^{n-6+|\beta|}}\right]
\end{medsize}
\end{equation}
for a multi-index $\beta$ with $|\beta|=0,\ldots,6$. By using \eqref{eq:mcrexp} and \eqref{eq:wtPsidec2}, and arguing as in the proof of Proposition \ref{prop:lin}, we observe
\begin{align*}
\left\|\wtPsi_{\dx}-\wtPsi_{\dx}^1\right\|_*' &\le C\(\left\|\mcr_{\dx}-\zeta_{\rho^{-1}\ep}(|\cdot-\xi|)\mcr_{\dx}^1\right\|_{**}' + \left\|(P_{\tig_0}^{(6)}-\Delta^3)\wtPsi_{\dx}^1\right\|_{**}' + \|N_{\dx}(\wtPsi)\|_{**}'\) \\
&\le C,
\end{align*}
which is \eqref{eq:wtPsiest ref}.
\end{proof}

\subsection{Variational reduction}
We reduce the problem to a finite-dimensional variational problem over the admissible set $\mca$ of parameters.
Let $\Psi_{\edx}(x) := \ep^{-\frac{n-6}{2}}\wtPsi_{\dx}(\ep^{-1}x)$ for $x \in \R^n$.
\begin{lemma}\label{lem:u epsilon}
Assume that $\ep > 0$ and $\rho, \alpha_0 \gg \ep > 0$ are small enough. Then the following assertions hold:
\begin{itemize}
\item[(i)] The map $(\dx) \in \mca \mapsto \mcf_{g_0}(v_{\edx} + \Psi_{\edx}) \in \R$ is continuously differentiable, where $\mcf_{g_0}$ is the energy functional defined in \eqref{eq:mcf}.
\item[(ii)] If $(\delta(\ep),\xi(\ep))$ is a critical point of the above map, then $u_{\ep}:=v_{\ep\delta(\ep),\ep\xi(\ep)} + \Psi_{\ep\delta(\ep),\ep\xi(\ep)}$ is a critical point of $\mcf_{g_0}$. Moreover, the function $u_{\ep}$ is a solution to \eqref{eq:main6} in $(M,g) = (\R^n,g_0)$ such that
$$
\sup_{|x|<\ep}u_{\ep}(x)\geq C^{-1}\ep^{-\frac{n-6}{2}},
$$
where $C=C(n,\ep_0)>0$.
\end{itemize}
\end{lemma}
\begin{proof}
The proof is standard. Note that $v_{\ep\delta(\ep),\ep\xi(\ep)} + \Psi_{\ep\delta(\ep),\ep\xi(\ep)} > 0$ in $\R^n$ provided $\alpha_0 > 0$ small, since \eqref{eq:vdx}, \eqref{eq:Gtig01}, and \eqref{eq:wtPsi} imply that there exists
\[v_{\ep\delta(\ep),\ep\xi(\ep)}(x) \ge \frac{C^{-1}\ep^{\frac{n-6}{2}}}{(\ep^2+|x-\ep\xi|^2)^{\frac{n-6}{2}}} \quad \text{and} \quad |\Psi_{\ep\delta(\ep),\ep\xi(\ep)}(x)| \le \frac{C\alpha_0\ep^{\frac{n-6}{2}}}{(\ep^2+|x-\ep\xi|^2)^{\frac{n-6}{2}}} \quad \text{in } \R^n. \qedhere\]
\end{proof}

We set $\omcr_{\dx}^1 := (\mu\ep^{10})^{-1}\mcr_{\dx}^1$ and $\opsi_{\dx}^1 := (\mu\ep^{10})^{-1}\wtPsi_{\dx}^1$.
\begin{prop}
For each fixed small number $\ep > 0$, it holds that
\begin{equation}\label{eq:mcfexp2}
\begin{aligned}
\mcf_{g_0}(v_{\edx} + \Psi_{\edx}) &= \mcf_0(w) + \mu^2\ep^{20} \left[\Ddot{\msfF}(\dx) + \frac{1}{2} \int_{\R^n} \omcr_{\dx}^1\opsi_{\dx}^1 dx\right] \\
&\ + O\bigg(\mu^3\ep^{20(1+\vep_1)} + \bigg(\frac{\ep}{\rho}\bigg)^{n-6}\bigg)
\end{aligned}
\end{equation}
uniformly for $(\dx) \in \mca$. Here, the number $\mcf_0(w) > 0$ and the function $\Ddot{\msfF}$ are given in Lemma \ref{lemma:mcfexp}, and $\vep_1 \in (0,\frac{1}{4})$ is a small number.
\end{prop}
\begin{proof}
Testing the first equation of \eqref{eq:nonlin} with $\wtPsi_{\dx}$, we obtain
\begin{equation}\label{eq:nonlin2}
\int_{\R^n} \wtPsi_{\dx} P_{\tig_0}^{(6)}\wtPsi_{\dx} dx + \tmfc_6(n) \int_{\R^n} v_{\dx}^{\frac{12}{n-6}}\wtPsi_{\dx}^2 dx = \int_{\R^n} [\mcr_{\dx} + N_{\dx}(\wtPsi_{\dx})]\wtPsi_{\dx} dx.
\end{equation}
In view of \eqref{eq:nonlin2}, \eqref{eq:wtPsiest ref}, \eqref{eq:Nest}, and Taylor's theorem, we observe that
\begin{equation}\label{eq:mcfexp21}
\begin{aligned}
&\ \mcf_{g_0}(v_{\edx} + \Psi_{\edx}) - \mcf_{g_0}(v_{\edx}) \\
&= \frac{1}{2} \int_{\R^n} \zeta_{\rho^{-1}\ep}(|x-\xi|)\mcr_{\dx}^1\wtPsi_{\dx}^1 dx + \frac{1}{2} \int_{\R^n} \left[\mcr_{\dx}-\zeta_{\rho^{-1}\ep}(|x-\xi|)\mcr_{\dx}^1\right] \wtPsi_{\dx}^1 dx \\
&\ + \frac{1}{2} \int_{\R^n} \mcr_{\dx}\(\wtPsi_{\dx}-\wtPsi_{\dx}^1\) dx - \frac{1}{2} \int_{\R^n} N_{\dx}(\wtPsi_{\dx})\wtPsi_{\dx} dx \\
&\ - \frac{n-6}{2n} \mfc_6(n) \int_{\R^n} \left[|v_{\dx}+\wtPsi_{\dx}|^{\frac{2n}{n-6}} - v_{\dx}^{\frac{2n}{n-6}} - \frac{2n}{n-6} v_{\dx}^{\frac{n+6}{n-6}}\wtPsi_{\dx} - \frac{n(n+6)}{(n-6)^2} v_{\dx}^{\frac{12}{n-6}}\wtPsi_{\dx}^2\right] dx \\
&= \frac{1}{2} \mu^2\ep^{20} \int_{\R^n} \omcr_{\dx}^1\opsi_{\dx}^1 dx + O\bigg(\mu^2\ep^{26-\ka} + \bigg(\frac{\ep}{\rho}\bigg)^{n-6}\bigg).
\end{aligned}
\end{equation}
Consequently, \eqref{eq:mcfexp2} follows from Lemmas \ref{lemma:mcfexp} and \ref{lemma:DotE}, and \eqref{eq:mcfexp21}.
\end{proof}
In the subsequent subsections, we will adjust the value of $\msfa_0,\ldots,\msfa_4 \in \R$ in \eqref{eq:Hkij} so that $(\dx) = (1,0)$ becomes a strict local minimum of the function
$$
(\dx) \in \mca \mapsto \Ddot{\msfF}(\dx) + \frac{1}{2} \int_{\R^n} \omcr_{\dx}^1\opsi_{\dx}^1 dx.
$$
Thanks to the following elementary lemmas, it will be sufficient to evaluate $\pa_{\delta}\Ddot{\msfF}(\delta,0)$, $\pa^2_{\delta}\Ddot{\msfF}(\delta,0)$ and $\pa^2_{\xi_i\xi_j}\Ddot{\msfF}(\delta,0)$.
\begin{lemma}\label{delta xi direction}
It holds that $\Ddot{\msfF}(\dx)=\Ddot{\msfF}(\delta,-\xi)$ for all $(\dx) \in (0,\infty) \times \R^n$. Consequently,
\[\pa_{\xi_i}\Ddot{\msfF}(\delta,0)=0 \quad \text{and} \quad \pa^2_{\delta \xi_i}\Ddot{\msfF}(\delta,0)=0\]
for any $\delta \in (0,\infty)$.
\end{lemma}
\begin{proof}
It immediately follows from the identity $H_{ij}(x)=H_{ij}(-x)$ for any $x\in \R^n$.
\end{proof}

\begin{lemma}\label{lem:R Psi direction}
It holds that
\begin{equation}\label{eq:R Psi direction}
\left.\pa_{(\dx)}\int_{\R^n} \omcr_{\dx}^1\opsi_{\dx}^1 dx\right|_{(\dx)=(\delta,0)}=0 \quad \text{and} \quad \left.\pa^2_{(\dx)}\int_{\R^n} \omcr_{\dx}^1\opsi_{\dx}^1 dx\right|_{(\dx)=(\delta,0)}=0
\end{equation}
for any $\delta \in (0,\infty)$.
\end{lemma}
\begin{proof}
With $\tr H=0$, $x_jH_{ij}=0$, and $\delta_iH=0$, we see that each term in the definition of $\mcr_{\dx}^1$ from \eqref{eq:mcr1} is at least quadratic in $\xi$, implying $\mcr_{\delta,0}^1=\pa_{\xi}\mcr_{\delta,0}^1=0$ in $\R^n$.
Hence, applying a variant of Proposition \ref{prop:lin} to \eqref{eq:wtPsi1} and the equation for $\pa_{\xi}\wtPsi_{\delta,0}^1$, we get $\opsi_{\delta,0}^1=\pa_{\xi}\opsi_{\delta,0}^1=0$ in $\R^n$.
Due to the arbitrariness of $\delta$, we also have that $\pa_{\delta}\omcr_{\delta,0}^1=\pa_{\delta}\opsi_{\delta,0}^1=0$ in $\R^n$.

The identities in \eqref{eq:R Psi direction} follow from the preceding discussion and the application of the chain rule.
\end{proof}
\noindent In Subsection \ref{subsec:noncomp6d1}, we will calculate the derivatives of $\Ddot{\msfF}(\delta,0)$ along the $\delta$-direction.
In Subsection \ref{subsec:noncomp6d2}, we shall compute the Hessian matrix of $\Ddot{\msfF}(\delta,0)$ along the $\xi$-direction.

\subsection{Derivatives in the scaling direction}\label{subsec:noncomp6d1}
Since $w_{\delta,0}$ is radial, we have
$$
\Ddot{\msfF}(\delta,0) =-\frac{1}{2}\sum_{\substack{k,m=2\\ k,m: \textup{even}}}^{10}\int_{\R^n}w_{\delta,0} L_2[H^{(k)},H^{(m)}]w_{\delta,0}.
$$

\begin{prop}\label{prop:the k+m relation}
Let $k,m = 2,\ldots,10$. For $H^{(k)}\in \mcv_k$ and $H^{(m)}\in\mcv_m$, it holds that
\begin{equation}\label{eq: k+m relation}
\int_{\R^n}Z_{\delta,0}L_2[H^{(k)},H^{(m)}]w_{\delta,0}=-\frac{k+m}{2}\delta^{-1}\int_{\R^n}w_{\delta,0} L_2[H^{(k)},H^{(m)}]w_{\delta,0}.
\end{equation}
\end{prop}
\begin{proof}
On one hand, due to the homogeneity
$$
\int_{\R^n}w_{\delta,0}L_2[H^{(k)},H^{(m)}]w_{\delta,0}=\delta^{k+m}\int_{\R^n}w L_2[H^{(k)},H^{(m)}]w,
$$
we obtain that
\begin{align*}
\pa_{\delta}\int_{\R^n}w_{\delta,0}L_2[H^{(k)},H^{(m)}]w_{\delta,0}&=(k+m)\delta^{k+m-1}\int_{\R^n}w L_2[H^{(k)},H^{(m)}]w\\
&=(k+m)\delta^{-1}\int_{\R^n}w_{\delta,0} L_2[H^{(k)},H^{(m)}]w_{\delta,0}.
\end{align*}
On the other hand, since $\pa_{\delta}w_{\delta,0}=-Z_{\delta,0}$, we have
$$
\pa_{\delta}\int_{\R^n}w_{\delta,0}L_2[H^{(k)},H^{(m)}]w_{\delta,0}=-2\int_{\R^n}Z_{\delta,0}L_2[H^{(k)},H^{(m)}]w_{\delta,0}.
$$
Combining the above two equations, we deduce \eqref{eq: k+m relation}.
\end{proof}

\begin{rmk}\label{rmk:k+m relation}
The quantity
$
\int_{\R^n}w_{\delta,0} L_2[H^{(k)},H^{(m)}]w_{\delta,0}
$
can be evaluated in the same way as in Proposition \ref{I_1 Q6}. First, we determine the numbers $c_{w,i}(n,k+m)$ for $i = 1,\ldots,6$ and $k,m = 2,\ldots,10$ through the following relations:
\begin{align*}
c_{w,1}(n,k+m) &:= \int_0^{\infty}r^{n+k+m-2} \left[\Delta w\(w'\frac{1}{r}\)'+ w\((\Delta w)'\frac{1}{r}\)'\right]; \\
c_{w,2}(n,k+m) &:= \int_0^{\infty}r^{n+k+m-3} |\Delta w|^2; \\
c_{w,3}(n,k+m) &:= \int_0^{\infty}r^{n+k+m-4} \left[\Delta ww'+w(\Delta w)'\right]; \\
c_{w,4}(n,k+m) &:= \int_0^{\infty}r^{n+k+m-4}w\(w'\frac{1}{r}\)'; \\
c_{w,5}(n,k+m) &:= \int_0^{\infty}r^{n+k+m-6}ww'; \\
c_{w,6}(n,k+m) &:= \int_0^{\infty}r^{n+k+m-7}w^2;
\end{align*}
cf. \eqref{eq:c1c6} (We note that no logarithmic term arises here, as ensured by Remark \ref{rmk:log}).
Second, we obtain the coefficients $\bc_{w,i}(n,k,m)$ by the same computations as in Proposition \ref{I_1 Q6}, replacing only $c_i(n,k+m)$ with $c_{w,i}(n,k+m)$.
Finally, we deduce $\bc'_{w,i}(n,k,m)$ by the same computations as in Corollary \ref{I_1 coro Q6}, changing $\bc_i(n,k,m)$ to $\bc_{w,i}(n,k,m)$. It follows that
\begin{align*}
&\ -\delta^{-(k+m)}\int_{\R^n}w_{\delta,0}L_2[H^{(k)},H^{(m)}]w_{\delta,0}\\
&= \bc'_{w,1}(n,k,m)\bla H^{(k)},H^{(m)}\bra+\bc'_{w,2}(n,k,m)\bla \mcl_kH^{(k)},H^{(m)}\bra\\
&\ +\bc'_{w,3}(n,k,m)\bla \mcl_kH^{(k)},\mcl_mH^{(m)}\bra \\
&\ +\bc'_{w,4}(n,k,m)\left[\bla \mcl_k^2H^{(k)},\mcl_m H^{(m)}\bra+\bla \mcl_k H^{(k)},\mcl_m^2H^{(m)}\bra\right]\\
&\ +\bc'_{w,5}(n,k,m)\bla \delta H^{(k)},\delta H^{(m)}\bra+\bc'_{w,6}(n,k,m)\bla \pa \delta H^{(k)},\pa \delta H^{(m)}\bra\\
&\ +\bc'_{w,7}(n,k,m)\bla \delta^2 H^{(k)},\delta^2 H^{(m)}\bra+\bc'_{w,8}(n,k,m)\bla \pa \delta^2 H^{(k)},\pa \delta^2H^{(m)}\bra.
\end{align*}

After the lengthy computation, one can verify that
$$
\bc'_{i}(n,k,m)=-\frac{k+m}{2}\bc'_{w,i}(n,k,m) \quad \text{for all } i=1,\ldots, 8,
$$
where $\bc'_{i}(n,k,m)$ are from Corollary \ref{I_1 coro Q6}. This is an alternative way to prove Proposition \ref{prop:the k+m relation}.

It is noteworthy that not all $c_i(n,k,m)$ and $\bc_i(n,k,m)$ satisfy this ``variational-type'' relation; for example, $c_{1}(n,k,m)\neq -\frac{k+m}{2}c_{w,1}(n,k,m)$.
This reflects the distinctive role of $\bc_{i}'(n,k,m)$ and indicates a hierarchical structure in the decomposition given in Corollary \ref{I_1 coro Q6}. \hfill $\diamond$
\end{rmk}

\begin{cor}\label{delta direction}
When $n\geq 27$, there exist a number $\msfa_0 \in \R$, determined in \eqref{eq:a0}, and a vector $\Vec{\msfa} = [\msfa_0,\msfa_1,\ldots,\msfa_4]:=[\msfa_0,-3634, 803, -62, 1]$ such that
$$
\Ddot{\msfF}(1,0)<0, \quad \pa_{\delta}\Ddot{\msfF}(1,0)=0,\quad \text{and} \quad \pa^2_{\delta}\Ddot{\msfF}(1,0)>0.
$$
\end{cor}
\begin{proof}
Because of \eqref{eq: k+m relation} and the proof of Lemma \ref{lemma of case 16}, we can write
\begin{align*}
\ -\frac{1}{2}\int_{\R^n}w_{\delta,0}L_2[H^{(k)},H^{(m)}]w_{\delta,0}
&=\frac{\delta}{k+m}\int_{\R^n}Z_{\delta,0}L_2[H^{(k)},H^{(m)}]w_{\delta,0}\\
&=-\frac{1}{k+m}m^{D,2}_{qq'}\delta^{k+m}\msfa_q\msfa_{q'}\int_{\S^{n-1}} W_{iabj}W_{icdj}y_ay_by_cy_d,
\end{align*}
where $m^{D,2}_{qq'} \in \R$ is the quantity defined by \eqref{eq:mDs0 Q6}--\eqref{eq:mDs Q6}. By Lemma \ref{IBP}, we also have that
$$
\int_{\S^{n-1}} W_{iabj}W_{icdj}y_ay_by_cy_d=\frac{|\S^{n-1}|}{2n(n+2)}|W|_{\sharp}^2.
$$

Recalling the notation $k=2q+2$ and $m=2q'+2$, we define a quantity
$$
m^{w,D,2}_{qq'}:=-\frac{1}{k+m}m^{D,2}_{qq'}
$$
and a polynomial in $\delta$,
$$
P(\delta):=\sum_{q,q'=0}^{4}m^{w,D,2}_{qq'} \msfa_q\msfa_{q'} \delta^{k+m}.
$$
Then we have that
$$
\Ddot{\msfF}(\delta,0) =-\frac{1}{2}\sum_{\substack{k,m=2\\ k,m: \textup{even}}}^{10}\int_{\R^n}w_{\delta,0} L_2[H^{(k)},H^{(m)}]w_{\delta,0}=P(\delta)\frac{|\S^{n-1}|}{2n(n+2)}|W|_{\sharp}^2.
$$
A direct calculation yields
\begin{align*}
P'(1) &= \sum_{q,q'=0}^{4}(k+m)m^{w,D,2}_{qq'}\msfa_q\msfa_{q'} = -\sum_{q,q'=0}^{4}m^{D,2}_{qq'}\msfa_q\msfa_{q'};\\
P''(1) &= \sum_{q,q'=0}^{4}(k+m)(k+m-1)m^{w,D,2}_{qq'}\msfa_q\msfa_{q'} = -\sum_{q,q'=0}^{4}(k+m-1)m^{D,2}_{qq'}\msfa_q\msfa_{q'}.
\end{align*}

Next, we determine $\msfa_0$ such that $P'(1)=0$. Substituting the values of $\msfa_1,\ldots,\msfa_4$ into $P'(1)$ reduces it to a quadratic polynomial in $\msfa_0$.
Using the Mathematica implementation of Lemma \ref{lemma of case 16}, we verify that the discriminant of $P'(1)$ is positive for $n \geq 27$. Accordingly, we select $\msfa_0$ to be the larger root of the equation $P'(1)=0$:
\begin{equation}\label{eq:a0}
\msfa_0 := \frac{\msfA_1+\sqrt{\msfA_2}}{\msfA_3},
\end{equation}
where
\begin{align*}
\msfA_1 &:= (n - 26) (n - 24) (n - 22) (n - 20) (n + 4) \\
&\ \times \(11991 n^9-852294 n^8+24029888 n^7-334408272 n^6+2238186992 n^5-4359884256 n^4\right.\\
&\quad \left.-20759728000 n^3+78857215488 n^2+8339503104 n-124262055936\);\\
\msfA_2 &:= (n - 26) (n - 24) (n - 22) (n - 20) (n + 4) \\
&\ \times (10989225 n^{23}-2987756460 n^{22}+378031055952 n^{21}-29507030164560 n^{20}\\
&\quad +1587504902043088 n^{19}-62286471762681984 n^{18}+1838499343113943552 n^{17}\\
&\quad -41499228022465995264 n^{16}+720749814386841727744 n^{15}-9608374930260373355520 n^{14}\\
&\quad +97259447027246828171264 n^{13}-732525312300365433016320 n^{12}\\
&\quad +3980505368562305038315520 n^{11}-15065834669595927057334272 n^{10}\\
&\quad +40459950924707392220168192 n^9-103397876171264024797249536 n^8\\
&\quad +339819087637723673505300480 n^7 -740098364899745956009869312 n^6\\
&\quad -394627146196739293779591168 n^5 +4735360453239104195348398080 n^4\\
&\quad -3600382007256808416243351552 n^3-13598394732655532252847931392 n^2\\
&\quad +27513309069085642550126051328 n-15449764981452555902999592960);\\
\msfA_3 &:= (n-26) (n - 24) (n - 22) (n - 20) (n - 18) (n - 16) (n - 14) (n - 12) (n - 4) (n - 2) \\
&\ \times (3 n^4-24 n^3-4 n^2+208 n+384).
\end{align*}

Similarly, $P(1)$ and $P''(1)$ are also quadratic functions in $\msfa_0$. Picking the above $\msfa_0$, we can use Mathematica to prove that
\[P(1)<0 \quad \text{and} \quad P''(1)>0\]
when $n\geq 27$. The proof is concluded.
\end{proof}

\begin{rmk}
In the Yamabe case and $Q^{(4)}$ case, the non-compactness phenomenon happens from $n\geq 25$. In the $Q^{(6)}$ case, the required lowest dimension increases by two.

However, if we take a vector $\Vec{\msfa}=[\msfa_0,1]$, Mathematica shows that the discriminant of the quadratic polynomial $-\frac{1}{2}\sum_{q,q'=0}^{1}m^{D,2}_{qq'}\msfa_q\msfa_{q'}$ in $\msfa_0$ is positive whenever $n \geq 52$. This dimension threshold coincides with those in the Yamabe case and the $Q^{(4)}$ case. \hfill $\diamond$
\end{rmk}

\subsection{Hessian matrix in the translation direction}\label{subsec:noncomp6d2}
Let $\msfp, \msfq = 1,\ldots,n$ and $L_2[H,H]$ be the operator defined by \eqref{eq:L2 Q6}. Since
$$
\pa_{\xi_i}w_{\dx}(y)=-\pa_iw_{\dx}(y) \quad \text{for } y \in \R^n
$$
and $L_2[H,H]$ is self-adjoint, we have that
\begin{align*}
\pa^2_{\xi_{\msfp}\xi_{\msfq}}\Ddot{\msfF}(\delta,0) &=-\int_{\R^n} \pa^2_{\msfpq}w_{\delta,0}L_2[H,H]w_{\delta,0} - \int_{\R^n}\pa_{\msfp}w_{\delta,0}L_2[H,H]\pa_{\msfq} w_{\delta,0}\\
&\ +\int_{\R^n}H_{ik}H_{kj}\pa^2_{i\msfp}\Delta^2w_{\delta,0} \pa^2_{j\msfq} w_{\delta,0} + \frac{1}{2}\int_{\R^n}H_{ik}H_{kj}\pa^2_{i\msfp} \Delta w_{\delta,0}\pa^2_{j\msfq}\Delta w_{\delta,0}\\
&\ +2\int_{\R^n}H_{ij}\pa^2_{i\msfp}w_{\delta,0} \pa^2_{jk} (H_{kl}\pa^2_{l\msfq}\Delta w_{\delta,0}) + \int_{\R^n}H_{ij}\pa^2_{i\msfp}w_{\delta,0} \pa^2_{jk} \Delta(H_{kl}\pa^2_{l\msfq} w_{\delta,0}).
\end{align*}
Because
\begin{equation}\label{divergence free cancellation}
\pa_k(H_{kl}\pa^2_{l\msfq}u)=\pa_k\(H_{k\msfq}u'\frac{1}{r}\)=\delta_{\msfq}H u'\frac{1}{r}+H_{k\msfq}\pa_k\(u'\frac{1}{r}\)=0
\end{equation}
for a radial function $u$, we can obtain
\begin{align}
\pa^2_{\xi_{\msfp}\xi_{\msfq}}\Ddot{\msfF}(\delta,0) &=-\sum_{\substack{k,m=2\\ k,m: \textup{even}}}^{10}\int_{\R^n} \pa^2_{\msfpq}w_{\delta,0}L_2[H^{(k)},H^{(m)}]w_{\delta,0} + \sum_{\substack{k,m=2\\ k,m: \textup{even}}}^{10}\int_{\R^n}\pa_{\msfp}w_{\delta,0}L_2[H^{(k)},H^{(m)}]\pa_{\msfq} w_{\delta,0} \nonumber \\
&\ +\frac{1}{2}\sum_{\substack{k,m=2\\ k,m: \textup{even}}}^{10}\int_{\R^n} H^{(k)}_{\msfp i}H^{(m)}_{\msfq i} \left[2(\Delta^2w_{\delta,0})' w_{\delta,0}'+\{(\Delta w_{\delta,0})'\}^2\right]\frac{1}{r^2}. \label{eq:FHess}
\end{align}
The last term on the right-hand side of \eqref{eq:FHess} can be expressed directly in polar coordinates.
Thus, in Parts 1 and 2 below, we focus on evaluating the first and second terms, respectively, and in Part 3, we assemble all terms to compute $\pa^2_{\xi_{\msfp}\xi_{\msfq}}\Ddot{\msfF}(\delta,0)$.

\subsubsection{Part 1: Evaluation of $\int_{\R^n} \pa^2_{\msfpq}w_{\delta,0}L_2[H^{(k)},H^{(m)}]w_{\delta,0}$}
\begin{lemma}\label{lemma:rn11}
We define
\begin{align*}
\RN{1}_{\msfpq}^{(k,m)} &:= \begin{medsize}
\displaystyle \frac{n-2}{4(n-1)}\delta_{\msfpq} \int_{\R^n}\Ddot{R}^{(k,m)}(\Delta w_{\delta,0})'\Delta w_{\delta,0}\frac{1}{r}
- \frac{n-6}{2}\delta_{\msfpq} \int_{\R^n} \big(\Ddot{Q^{(6)}}\big)^{(k,m)} w_{\delta,0}'w_{\delta,0}\frac{1}{r}
\end{medsize} \\
&\begin{medsize}
\displaystyle \ -\delta_{\msfpq} \int_{\R^n}(\Ddot{T_2})^{(k,m)}_{ij} y_iy_j\left[(\Delta w_{\delta,0})'\(w_{\delta,0}'\frac{1}{r}\)'+w_{\delta,0}'\left\{(\Delta w_{\delta,0})'\frac{1}{r}\right\}'\right]\frac{1}{r^2}
\end{medsize} \\
&\begin{medsize}
\displaystyle \ -\delta_{\msfpq} \int_{\R^n} \left[\tr \,\Ddot{T_2}^{(k,m)}+(\Ddot{\diver_{g_0}}T_2)^{(k,m)}_iy_i + \frac{1}{4}\left\{\Dot{T_2}^{(k)} \cdot (y_i\pa_i-2)H^{(m)} + \Dot{T_2}^{(m)} \cdot (y_i\pa_i-2)H^{(k)}\right\}\right]
\end{medsize} \\
&\begin{medsize}
\displaystyle \hspace{240pt} \times \left[(\Delta w_{\delta,0})'w_{\delta,0}'+ w_{\delta,0}'(\Delta w_{\delta,0})'\right]\frac{1}{r^2}
\end{medsize} \\
&\begin{medsize}
\displaystyle \ -\delta_{\msfpq} \int_{\R^n}(\Ddot{T_4})^{(k,m)}_{ij}y_iy_j w_{\delta,0}'\(w_{\delta,0}'\frac{1}{r}\)'\frac{1}{r^2}
\end{medsize} \\
&\begin{medsize}
\displaystyle \ -\delta_{\msfpq} \int_{\R^n} \left[\tr \,\Ddot{T_4}^{(k,m)}+(\Ddot{\diver_{g_0}T_4})^{(k,m)}_iy_i + \frac{1}{4}\left\{\Dot{T_4}^{(k)} \cdot (y_i\pa_i-2)H^{(m)} + \Dot{T_4}^{(m)} \cdot (y_i\pa_i-2)H^{(k)}\right\}\right] (w_{\delta,0}')^2\frac{1}{r^2},
\end{medsize}
\end{align*}
and
\begin{align*}
\RN{2}_{\msfpq}^{(k,m)} &:=\begin{medsize}
\displaystyle \frac{n-2}{4(n-1)} \int_{\R^n} \Ddot{R}^{(k,m)}y_{\msfp}y_{\msfq} \left[(\Delta w_{\delta,0})''-(\Delta w_{\delta,0})'\frac{1}{r}\right] \Delta w_{\delta,0}\frac{1}{r^2}
\end{medsize} \\
&\begin{medsize}
\displaystyle \ - \frac{n-6}{2}\int_{\R^n} \big(\Ddot{Q^{(6)}}\big)^{(k,m)} y_{\msfp}y_{\msfq} \(w_{\delta,0}''- w_{\delta,0}'\frac{1}{r}\)w_{\delta,0}\frac{1}{r^2}
\end{medsize} \\
&\begin{medsize}
\displaystyle \ -\int_{\R^n}(\Ddot{T_2})^{(k,m)}_{ij}y_iy_jy_{\msfp}y_{\msfq} \left[\left\{(\Delta w_{\delta,0})''- (\Delta w_{\delta,0})'\frac{1}{r}\right\}\(w_{\delta,0}'\frac{1}{r}\)'+\(w_{\delta,0}''- w_{\delta,0}'\frac{1}{r}\)\left\{(\Delta w_{\delta,0})'\frac{1}{r}\right\}'\right]\frac{1}{r^3}
\end{medsize} \\
&\begin{medsize}
\displaystyle \ -\int_{\R^n} \left[\tr \,\Ddot{T_2}^{(k,m)}+(\Ddot{\diver_{g_0}T_2})^{(k,m)}_iy_i + \frac{1}{4}\left\{\Dot{T_2}^{(k)} \cdot (y_i\pa_i-2)H^{(m)} + \Dot{T_2}^{(m)} \cdot (y_i\pa_i-2)H^{(k)}\right\}\right] y_{\msfp}y_{\msfq}
\end{medsize} \\
&\begin{medsize}
\displaystyle \hspace{135pt} \times \left[\left\{(\Delta w_{\delta,0})''- (\Delta w_{\delta,0})'\frac{1}{r}\right\}w_{\delta,0}'+\(w_{\delta,0}''- w_{\delta,0}'\frac{1}{r}\)(\Delta w_{\delta,0})'\right]\frac{1}{r^3}
\end{medsize} \\
&\begin{medsize}
\displaystyle \ -\int_{\R^n}(\Ddot{T_4})^{(k,m)}_{ij}y_iy_jy_{\msfp}y_{\msfq} \(w_{\delta,0}''-w_{\delta,0}'\frac{1}{r}\) \(w_{\delta,0}'\frac{1}{r}\)'\frac{1}{r^3}
\end{medsize} \\
&\begin{medsize}
\displaystyle \ -\int_{\R^n} \left[\tr \,\Ddot{T_4}^{(k,m)}+(\Ddot{\diver_{g_0}T_4})^{(k,m)}_iy_i + \frac{1}{4}\left\{\Dot{T_4}^{(k)} \cdot (y_i\pa_i-2)H^{(m)} + \Dot{T_4}^{(m)} \cdot (y_i\pa_i-2)H^{(k)}\right\}\right]y_{\msfp}y_{\msfq}
\end{medsize} \\
&\begin{medsize}
\displaystyle \hspace{290pt} \times \(w_{\delta,0}''- w_{\delta,0}'\frac{1}{r}\)w_{\delta,0}'\frac{1}{r^3},
\end{medsize}
\end{align*}
where $r=|y|$. For $k,m=2,\ldots,10$ even, it holds that
\begin{equation}\label{eq:part1}
\int_{\R^n} \pa^2_{\msfpq}w_{\delta,0}L_2[H^{(k)},H^{(m)}]w_{\delta,0}=\RN{1}_{\msfpq}^{(k,m)}+\RN{2}_{\msfpq}^{(k,m)}.
\end{equation}
\end{lemma}
\begin{proof}
To derive \eqref{eq:part1}, we will use \eqref{eq:L2 Q6} and the radial symmetry of $w_{\delta,0}$.

\medskip
Since $\tr H=\delta H=0$ and $y_iH_{ij}=0$, it holds that $\tr \Dot{A}[H]=0$ and $y_i(\Dot{T_2}[H])_{ij}=0$. Hence, if $u$ is radial, then
\begin{align*}
\Delta(\tr\Dot{A}[H](h_{ij}\pa_i u)_{,j}) &= (h_{ij}(\tr\Dot{A}[H]\Delta u)_{,j})_{,i}=0,\\
\Dot{\diver_{g_0}}\left\{g(\nabla u,\cdot)\right\}[H]\circ \Dot{\diver_{g_0}}\left\{T_2(\nabla u,\cdot)\right\}[H] &= \Dot{\diver_{g_0}}\left\{T_2(\nabla ,\cdot)\right\}[H]\circ \Dot{\diver_{g_0}}\left\{g(\nabla u,\cdot)\right\}[H]=0.
\end{align*}
By \eqref{eq:L2 Q6}, we have
\begin{align*}
L_2[H,H]u &= \frac{n-2}{4(n-1)}\Delta(\Ddot{R}[H,H]\Delta u)-\frac{n-6}{2}\Ddot{Q^{(6)}}[H,H]u \\
&\ -\Delta\circ \Ddot{\diver_{g_0}}\left\{T_2(\nabla u,\cdot)\right\}[H,H] -\Ddot{\diver_{g_0}}\left\{T_2(\nabla\Delta u,\cdot)\right\}[H,H] \\
&\ -\Ddot{\diver_{g_0}}\left\{T_4(\nabla u,\cdot)\right\}[H,H].
\end{align*}
Integrating by parts, we obtain
\begin{align*}
&\ \int_{\R^n} \pa^2_{\msfpq}w_{\delta,0}L_2[H,H]w_{\delta,0}\\
&= \frac{n-2}{4(n-1)}\int_{\R^n}\Ddot{R}[H,H]\pa^2_{\msfpq}\Delta w_{\delta,0}\Delta w_{\delta,0} -\frac{n-6}{2}\int_{\R^n}\Ddot{Q^{(6)}}[H,H] \pa^2_{\msfpq}w_{\delta,0}w_{\delta,0} \\
&\ -\int_{\R^n}\pa^2_{\msfpq}\Delta w_{\delta,0}\Ddot{\diver_{g_0}}\left\{T_2(\nabla w_{\delta,0},\cdot)\right\}[H,H]
-\int_{\R^n} \pa^2_{\msfpq}w_{\delta,0}\Ddot{\diver_{g_0}}\left\{T_2(\nabla \Delta w_{\delta,0} ,\cdot)\right\}[H,H] \\
&\ -\int_{\R^n}\pa^2_{\msfpq}w_{\delta,0}\Ddot{\diver_{g_0}}\left\{T_4(\nabla w_{\delta,0},\cdot)\right\}[H,H].
\end{align*}
Since $\pa^2_{\msfpq} u=y_{\msfp}y_{\msfq}\(u''- u'\frac{1}{r}\)\frac{1}{r^2}+\delta_{\msfpq}u'\frac{1}{r}$ for a radial function $u$, we can separate the terms involving $\delta_{\msfpq}$ and those involving $y_{\msfp}y_{\msfq}$ and then apply the polarization identity \eqref{eq:polar} to deduce \eqref{eq:part1}.
\end{proof}

We represent the term $\RN{1}_{\msfpq}^{(k,m)}$ by using sphere integrals.
\begin{lemma}\label{lemma:rn1}
First, we fix six constants depending only on $n$ and $k+m$:
\begin{align*}
c_{1}^{\RN{1}}(n,k+m) &:= 2 (n-6)^2(n-4) \left[(n^2-8)\mci_{n-1}^{n+k+m-1}+4(n-3)\mci_{n-1}^{n+k+m+1}\right];\\
c_{2}^{\RN{1}}(n,k+m) &:= - (n-6)^2(n-4) \left[n(n+2)\mci_{n-1}^{n+k+m-3}+8(n+1) \mci_{n-1}^{n+k+m-1}+16\mci_{n-1}^{n+k+m+1}\right];\\
c_{3}^{\RN{1}}(n,k+m) &:= -2 (n-6)^2(n-4) \left[(n+2)\mci_{n-2}^{n+k+m-3}+4\mci_{n-2}^{n+k+m-1}\right];\\
c_{4}^{\RN{1}}(n,k+m) &:= -(n-6)^2(n-4) \mci_{n-3}^{n+k+m-3};\\
c_{5}^{\RN{1}}(n,k+m) &:= (n-6)^2 \mci_{n-4}^{n+k+m-5};\\
c_{6}^{\RN{1}}(n,k+m) &:= -(n-6) \mci_{n-5}^{n+k+m-7}.
\end{align*}
Second, for $i=1,\ldots,7$, we define the coefficients $-\bc_{i}^{\RN{1}}(n,k,m)$ analogously to $\bc_i(n,k,m)$ in the statement of Proposition \ref{I_1 Q6}, replacing $c_i(n,k+m)$ with $c_{i}^{\RN{1}}(n,k+m)$. Then we have
\begin{align}
\delta^{-(k+m-2)} \RN{1}_{\msfpq}^{(k,m)} &= \left[\bc^{\RN{1}}_1(n,k,m)\int_{\S^{n-1}}\Ddot{R}^{(k,m)} + \bc^{\RN{1}}_2(n,k,m) \bla\Dot{\ricci}^{(k)},\Dot{\ricci}^{(m)}\bra \right. \nonumber \\
&\hspace{15pt} + \bc^{\RN{1}}_3(n,k,m) \bla H^{(k)},H^{(m)}\bra + \bc^{\RN{1}}_4(n,k,m) \left\{k\bla H^{(k)},\Dot{\ricci}^{(m)}\bra + m\bla\Dot{\ricci}^{(k)},H^{(m)}\bra\right\} \nonumber \\
&\hspace{15pt} + \bc^{\RN{1}}_5(n,k,m) \left\{k\bla H^{(k)},\Delta\Dot{\ricci}^{(m)}\bra + m\bla\Delta\Dot{\ricci}^{(k)},H^{(m)}\bra\right\} \label{eq:RN1_pq} \\
&\hspace{15pt} + \bc^{\RN{1}}_6(n,k,m)\left\{\bla\Dot{\ricci}^{(k)},\Delta\Dot{\ricci}^{(m)}\bra + \bla\Delta\Dot{\ricci}^{(k)},\Dot{\ricci}^{(m)}\bra\right\} \nonumber \\
&\hspace{15pt} \left. + \bc^{\RN{1}}_7(n,k,m) \left\{k^2\bla H^{(k)},\Dot{\ricci}^{(m)}\bra + m^2\bla\Dot{\ricci}^{(k)},H^{(m)}\bra\right\}\right]\delta_{\msfpq}. \nonumber
\end{align}
\end{lemma}
\begin{proof}
Expressing $\RN{1}_{\msfpq}^{(k,m)}$ in polar coordinates produces sphere integrals that agree with those in Proposition \ref{I_1 Q6}, up to an overall minus sign. The radial integrals in \eqref{eq:c1c6} are to be substituted with
\begin{align*}
c_{1}^{\RN{1}}(n,k+m)\delta^{k+m-2} &= \int_{0}^{\infty}r^{n+k+m-3}\left[(\Delta w_{\delta,0})'\(w_{\delta,0}'\frac{1}{r}\)'+w_{\delta,0}'\((\Delta w_{\delta,0})'\frac{1}{r}\)'\right];\\
c_{2}^{\RN{1}}(n,k+m)\delta^{k+m-2} &= \int_0^{\infty}r^{n+k+m-4}(\Delta w_{\delta,0})'\Delta w_{\delta,0};\\
c_{3}^{\RN{1}}(n,k+m)\delta^{k+m-2} &= 2\int_0^{\infty}r^{n+k+m-5}(\Delta w_{\delta,0})'w_{\delta,0}';\\
c_{4}^{\RN{1}}(n,k+m)\delta^{k+m-2} &= \int_0^{\infty} r^{n+k+m-5}w_{\delta,0}'\(w_{\delta,0}'\frac{1}{r}\)';\\
c_{5}^{\RN{1}}(n,k+m)\delta^{k+m-2} &= \int_0^{\infty}r^{n+k+m-7}(w_{\delta,0}')^2;\\
c_{6}^{\RN{1}}(n,k+m)\delta^{k+m-2} &= \int_0^{\infty}r^{n+k+m-8} w_{\delta,0}'w_{\delta,0}.
\end{align*}
Consequently, replacing each $\bc_i(n,k+m)$ with $\bc_{i}^{\RN{1}}(n,k+m)$ in the right-hand side of \eqref{eq:J1HkHm} (and noting that $\delta H=0$) leads directly to \eqref{eq:RN1_pq}.
\end{proof}

We represent the term $\RN{2}_{\msfpq}^{(k,m)}$ by using sphere integrals.
\begin{lemma}\label{lemma:rn2}
First, we fix six constants depending only on $n$ and $k+m$:
\begin{align*}
c_{1}^{\RN{2}}(n,k+m) &:= \begin{medsize}
\displaystyle -2 (n-6)^2(n-4)^2(n-2) \left[(n+4)\mci_{n}^{n+k+m+1}+4\mci_{n}^{n+k+m+3}\right];
\end{medsize}\\
c_{2}^{\RN{2}}(n,k+m) &:= \begin{medsize}
\displaystyle (n-6)^2(n-4)(n-2) \left[n(n+4)\mci_{n}^{n+k+m-1}+8(n+2)\mci_{n}^{n+k+m+1}+16\mci_{n}^{n+k+m+3}\right];
\end{medsize}\\
c_{3}^{\RN{2}}(n,k+m) &:= \begin{medsize}
\displaystyle 2 (n-6)^2(n-4) \left[(n^2-8)\mci_{n-1}^{n+k+m-1}+4(n-3)\mci_{n-1}^{n+k+m+1}\right];
\end{medsize}\\
c_{4}^{\RN{2}}(n,k+m) &:= \begin{medsize}
\displaystyle (n-6)^2(n-4)^2\mci_{n-2}^{n+k+m-1};
\end{medsize}\\
c_{5}^{\RN{2}}(n,k+m) &:= \begin{medsize}
\displaystyle -(n-6)^2(n-4) \mci_{n-3}^{n+k+m-3};
\end{medsize}\\
c_{6}^{\RN{2}}(n,k+m) &:= \begin{medsize}
\displaystyle (n-6)(n-4)\mci_{n-4}^{n+k+m-5}.
\end{medsize}
\end{align*}
Second, we define
\begin{align*}
\bc^{\RN{2}}_1(n,k,m) &:= -c^{\RN{2}}_6(n,k+m)(k+m-6)(n+k+m-4)(k+m-4)(n+k+m-2)\frac{n-6}{4(n-1)}\\
&\ -c^{\RN{2}}_3(n,k+m)\frac{n^2-2n-8+(n-10)(k+m-2)}{2(n-1)}+c^{\RN{2}}_2(n,k+m)\frac{n-2}{4(n-1)}\\
&\ -c^{\RN{2}}_1(n,k+m)\frac{n^2-4n+12}{2(n-2)(n-1)}\\
&\ -c^{\RN{2}}_5(n,k+m)(k+m-4)(n+k+m-2)(n+k+m-4)\frac{(n-6)}{2(n-1)}\\
&\ -c^{\RN{2}}_4(n,k+m)(k+m-4)(n+k+m-2)\frac{(n-6)(n-4)(n-2)+16}{2(n-4)(n-2)(n-1)}\\
&\ -c^{\RN{2}}_4(n,k+m)\frac{8(n-2)(k+m)^2-8(n-6)(k+m)+16(n-4)}{(n-4)(n-2)(n-1)};\\
\bc^{\RN{2}}_2(n,k,m) &:= -c^{\RN{2}}_6(n,k+m)(k+m-6)(n+k+m-4)\frac{2(n-6)}{(n-2)^2}\\
&\ -c^{\RN{2}}_5(n,k+m)\frac{4(n-6)(k+m)+4(n^2-8n+20)}{(n-2)^2} -c^{\RN{2}}_4(n,k+m)\frac{4(n^2-8n+20)}{(n-4)(n-2)^2};\\
\bc^{\RN{2}}_3(n,k,m) &:= -c^{\RN{2}}_1(n,k+m)\frac{2km}{n-2} -c^{\RN{2}}_4(n,k+m)(k+m-2)(n+k+m)\frac{4km}{(n-4)(n-2)};\\
\bc^{\RN{2}}_4(n,k,m) &:= c^{\RN{2}}_3(n,k+m)\frac{2}{n-2}-c^{\RN{2}}_4(n,k+m)\frac{8}{(n-4)(n-2)};\\
\bc^{\RN{2}}_5(n,k,m) &:= c^{\RN{2}}_5(n,k+m)\frac{4}{(n-4)(n-2)};\\
\bc^{\RN{2}}_6(n,k,m) &:= -c^{\RN{2}}_6(n,k+m)\frac{4(n-6)}{(n-4)(n-2)^2};\\
\bc^{\RN{2}}_7(n,k,m) &:= -c^{\RN{2}}_4(n,k+m)\frac{8}{(n-4)(n-2)};\\
\bc^{\RN{2}}_8(n,k,m) &:= -c^{\RN{2}}_6(n,k+m)(k+m-4)(n+k+m-4)\frac{n-6}{n-1}\\
&\ -c^{\RN{2}}_4(n,k+m)\frac{(n-6)(n-4)(n-2)+16}{(n-4)(n-2)(n-1)} -c^{\RN{2}}_5(n,k+m)(n+k+m-4)\frac{n-6}{n-1};\\
\bc^{\RN{2}}_9(n,k,m) &:= -c^{\RN{2}}_6(n,k+m)\frac{4(n-6)}{(n-2)^2};\\
\bc^{\RN{2}}_{10}(n,k,m) &:= -c^{\RN{2}}_4(n,k+m)\frac{8km}{(n-4)(n-2)}.
\end{align*}
Then we have
\begin{align}
&\begin{medsize}
\displaystyle \ \delta^{-(k+m-2)}\RN{2}_{\msfpq}^{(k,m)}
\end{medsize} \nonumber \\
&\begin{medsize}
\displaystyle = \bc^{\RN{2}}_1(n,k,m) \int_{\S^{n-1}}\Ddot{R}^{(k,m)}y_{\msfp}y_{\msfq} + \bc^{\RN{2}}_2(n,k,m)\int_{\S^{n-1}}\Dot{\ricci}^{(k)}\cdot\Dot{\ricci}^{(m)}y_{\msfp}y_{\msfq} + \bc^{\RN{2}}_3(n,k,m)\int_{\S^{n-1}}H^{(k)}\cdot H^{(m)} y_{\msfp}y_{\msfq}
\end{medsize} \nonumber \\
&\begin{medsize}
\displaystyle \ + \bc^{\RN{2}}_4(n,k,m)\int_{\S^{n-1}} \big(kH^{(k)} \cdot \Dot{\ricci}^{(m)}+m\Dot{\ricci}^{(k)} \cdot H^{(m)}\big)y_{\msfp}y_{\msfq}
\end{medsize} \nonumber \\
&\begin{medsize}
\displaystyle \ + \bc^{\RN{2}}_5(n,k,m)\int_{\S^{n-1}} \big(kH^{(k)} \cdot \Delta\Dot{\ricci}^{(m)}+m\Delta\Dot{\ricci}^{(k)} \cdot H^{(m)}\big)y_{\msfp}y_{\msfq}
\end{medsize} \label{eq:RN2_pq} \\
&\begin{medsize}
\displaystyle \ + \bc^{\RN{2}}_6(n,k,m)\int_{\S^{n-1}} \big(\Dot{\ricci}^{(k)} \cdot \Delta\Dot{\ricci}^{(m)}+\Delta\Dot{\ricci}^{(k)} \cdot \Dot{\ricci}^{(m)}\big)y_{\msfp}y_{\msfq}
\end{medsize} \nonumber \\
&\begin{medsize}
\displaystyle \ + \bc^{\RN{2}}_7(n,k,m)\int_{\S^{n-1}} \big(k^2H^{(k)} \cdot \Dot{\ricci}^{(m)}+m^2\Dot{\ricci}^{(k)} \cdot H^{(m)}\big)y_{\msfp}y_{\msfq}
\end{medsize} \nonumber \\
&\begin{medsize}
\displaystyle \ + \left[\bc^{\RN{2}}_8(n,k,m) \int_{\S^{n-1}}\Ddot{R}^{(k,m)} + \bc^{\RN{2}}_9(n,k,m) \bla\Dot{\ricci}^{(k)},\Dot{\ricci}^{(m)}\bra + \bc^{\RN{2}}_{10}(n,k,m) \bla\Dot{\ricci}^{(k)},\Dot{\ricci}^{(m)}\bra\right]\delta_{\msfpq}.
\end{medsize} \nonumber
\end{align}
\end{lemma}
\begin{proof}
We express $\RN{2}_{\msfpq}^{(k,m)}$ in polar coordinates:
\begin{align}
&\begin{medsize}
\displaystyle \ \delta^{-(k+m-2)} \RN{2}_{\msfpq}^{(k,m)}
\end{medsize} \nonumber \\
&\begin{medsize}
\displaystyle = \frac{n-2}{4(n-1)}\int_{\S^{n-1}}\Ddot{R}^{(k,m)}y_{\msfp}y_{\msfq} \cdot c_{2}^{\RN{2}}(n,k+m)
-\frac{n-6}{2}\int_{\S^{n-1}}\big(\Ddot{Q^{(6)}}\big)^{(k,m)} y_{\msfp}y_{\msfq} \cdot c_{6}^{\RN{2}}(n,k+m)
\end{medsize} \nonumber \\
&\begin{medsize}
\displaystyle \ -\int_{\S^{n-1}}(\Ddot{T_2})^{(k,m)}_{ij}y_iy_jy_{\msfp}y_{\msfq} \cdot c_{1}^{\RN{2}}(n,k+m)
\end{medsize} \label{eq:RN2_pq1} \\
&\begin{medsize}
\displaystyle \ -\int_{\S^{n-1}} \left[\tr \,\Ddot{T_2}^{(k,m)}+(\Ddot{\diver_{g_0}T_2})^{(k,m)}_iy_i +\frac{1}{4}\left\{\Dot{T_2}^{(k)} \cdot (y_i\pa_i-2)H^{(m)} + \Dot{T_2}^{(m)} \cdot (y_i\pa_i-2)H^{(k)}\right\}\right]y_{\msfp}y_{\msfq} \cdot c_{3}^{\RN{2}}(n,k+m)
\end{medsize} \nonumber \\
&\begin{medsize}
\displaystyle \ -\int_{\S^{n-1}}(\Ddot{T_4})^{(k,m)}_{ij}y_iy_jy_{\msfp}y_{\msfq}\cdot c_{4}^{\RN{2}}(n,k+m)
\end{medsize} \nonumber \\
&\begin{medsize}
\displaystyle \ -\int_{\S^{n-1}} \left[\tr \,\Ddot{T_4}^{(k,m)}+(\Ddot{\diver_{g_0}T_4})^{(k,m)}_iy_i +\frac{1}{4}\left\{\Dot{T_4}^{(k)} \cdot (y_i\pa_i-2)H^{(m)} + \Dot{T_4}^{(m)} \cdot (y_i\pa_i-2)H^{(k)}\right\}\right]y_{\msfp}y_{\msfq} \cdot c_{5}^{\RN{2}}(n,k+m),
\end{medsize} \nonumber
\end{align}
where
\begin{align*}
c_{1}^{\RN{2}}(n,k+m) &= \int_0^{\infty}r^{n+k+m-2}\left[\left\{(\Delta w)''- (\Delta w)'\frac{1}{r}\right\} \(w'\frac{1}{r}\)'+\(w''- w'\frac{1}{r}\)\left\{(\Delta w)'\frac{1}{r}\right\}'\right];\\
c_{2}^{\RN{2}}(n,k+m) &= \int_0^{\infty}r^{n+k+m-3}\left\{(\Delta w)''- (\Delta w)'\frac{1}{r}\right\}\Delta w;\\
c_{3}^{\RN{2}}(n,k+m) &= \int_0^{\infty}r^{n+k+m-4}\left[\left\{(\Delta w)''- (\Delta w)'\frac{1}{r}\right\}w'+\(w''- w'\frac{1}{r}\)(\Delta w)'\right];\\
c_{4}^{\RN{2}}(n,k+m) &= \int_0^{\infty}r^{n+k+m-4}\(w''- w'\frac{1}{r}\)\(w'\frac{1}{r}\)';\\
c_{5}^{\RN{2}}(n,k+m) &= \int_0^{\infty} r^{n+k+m-6} \(w''- w'\frac{1}{r}\)w';\\
c_{6}^{\RN{2}}(n,k+m) &= \int_{0}^{\infty}r^{n+k+m-7}\(w''- w'\frac{1}{r}\)w.
\end{align*}
In the end, the five sphere integrals in \eqref{eq:RN2_pq1} associated with $T_2$, $T_4$, and $Q^{(6)}$ will be expressed in terms of sphere integrals involving $\Ddot{R}$, $\Dot{R}$, $\Dot{\ricci}$, and $H$ in Lemma \ref{lemma:sphere with xpxq}.

Combining the above information, we can deduce \eqref{eq:RN2_pq}.
\end{proof}

\subsubsection{Part 2: Evaluation of $\int_{\R^n}\pa_{\msfp}w_{\delta,0}L_2[H^{(k)},H^{(m)}]\pa_{\msfq} w_{\delta,0}$}
\begin{lemma}
We define
\begin{align*}
\RN{3}_{\msfpq}^{(k,m)} &:=\frac{n-2}{4(n-1)}\int_{\R^n}\Ddot{R}^{(k,m)}y_{\msfp}y_{\msfq}\{(\Delta w_{\delta,0})'\}^2\frac{1}{r^2} -\frac{n-6}{2}\int_{\R^n}\big(\Ddot{Q}^{(6)}\big)^{(k,m)}y_{\msfp}y_{\msfq}(w_{\delta,0}')^2\frac{1}{r^2}\\
&\ +2\int_{\R^n}(\Ddot{T_2})^{(k,m)}_{ij}y_iy_jy_{\msfp}y_{\msfq}\left[(\Delta w_{\delta,0})''-(\Delta w_{\delta,0})'\frac{1}{r}\right]\(w_{\delta,0}''-w_{\delta,0}'\frac{1}{r}\)\frac{1}{r^4}\\
&\ +2\int_{\R^n}(\Ddot{T_2})^{(k,m)}_{\msfpq}(\Delta w_{\delta,0})'w_{\delta,0}'\frac{1}{r^2}\\
&\ +2\int_{\R^n}(\Ddot{T_2})^{(k,m)}_{i\msfp}y_iy_{\msfq}\left[\left\{(\Delta w_{\delta,0})''-(\Delta w_{\delta,0})'\frac{1}{r}\right\}w_{\delta,0}'+\(w_{\delta,0}''-w_{\delta,0}'\frac{1}{r}\)(\Delta w_{\delta,0})'\right]\frac{1}{r^3}\\
&\ +\int_{\R^n}(\Ddot{T_4})^{(k,m)}_{ij}y_iy_jy_{\msfp}y_{\msfq} \(w_{\delta,0}''-w_{\delta,0}'\frac{1}{r}\)^2\frac{1}{r^4} + 2\int_{\R^n}(\Ddot{T_4})^{(k,m)}_{i\msfp}y_iy_{\msfq} \(w_{\delta,0}''-w_{\delta,0}'\frac{1}{r}\)w_{\delta,0}'\frac{1}{r^3} \\
&\ +\int_{\R^n}(\Ddot{T_4})^{(k,m)}_{\msfpq}(w_{\delta,0}')^2\frac{1}{r^2},
\end{align*}
where $r=|y|$. For $k,m=2,\ldots,10$ even, it holds that
\begin{equation}\label{eq:part2}
\int_{\R^n}\pa_{\msfp}w_{\delta,0}L_2[H^{(k)},H^{(m)}]\pa_{\msfq} w_{\delta,0}=\RN{3}_{\msfpq}^{(k,m)}.
\end{equation}
\end{lemma}
\begin{proof}
Because $\pa_{\msfq} w_{\delta,0}$ is not radial, we cannot use the argument in the proof of Lemma \ref{lemma:rn11}.
To overcome this difficulty, we make use of the variational structure of the constant $Q^{(6)}$-curvature problem, with the associated energy functional defined in \eqref{eq:mcf}--\eqref{eq:mcf2}.

Doing the second-order expansion on both sides of \eqref{eq:mcf2}, we discover
\begin{align*}
&\begin{medsize}
\displaystyle \ -\int_{\R^n}\pa_{\msfp}w_{\delta,0}L_2[H,H]\pa_{\msfq} w_{\delta,0}
\end{medsize} \\
&\begin{medsize}
\displaystyle = \int_{\R^n}\left[-2\Ddot{T_2}[H,H](\nabla\Delta \pa_{\msfp}w_{\delta,0},\nabla \pa_{\msfq}w_{\delta,0})-\Ddot{T_4}[H,H](\nabla \pa_{\msfp}w_{\delta,0},\nabla \pa_{\msfq}w_{\delta,0})\right]
\end{medsize} \\
&\begin{medsize}
\displaystyle \ + \int_{\R^n}\left[-\frac{n-2}{4(n-1)}\Ddot{R}[H,H]\Delta\pa_{\msfp} w_{\delta,0}\Delta\pa_{\msfq} w_{\delta,0}+\frac{n-6}{2}\Ddot{Q}^{(6)}[H,H]\pa_{\msfp}w_{\delta,0}\pa_{\msfq}w_{\delta,0}\right]
\end{medsize} \\
&\begin{medsize}
\displaystyle \ + \int_{\R^n}\left[2(\Dot{T_2})_{ij}[H](H_{il}\pa^2_{\msfp l}\Delta w_{\delta,0}\pa^2_{\msfq j}w_{\delta,0}+H_{jl}\pa^2_{\msfp i}\Delta w_{\delta,0}\pa^2_{\msfq l}w_{\delta,0})
+(\Dot{T_4})_{ij}[H](H_{il}\pa^2_{\msfp l} w_{\delta,0}\pa^2_{\msfq j}w_{\delta,0}+H_{jl}\pa^2_{\msfp i} w_{\delta,0}\pa^2_{\msfq l}w_{\delta,0})\right]
\end{medsize} \\
&\begin{medsize}
\displaystyle \ + \int_{\R^n}\left[2\Dot{T_2}[H](\nabla \pa_j(H_{ij}\pa^2_{i\msfp}w_{\delta,0}),\nabla \pa_{\msfq}w_{\delta,0})
+ \frac{n-2}{4(n-1)}\Dot{R}[H]\left\{\pa_j(H_{ij}\pa^2_{i\msfp}w_{\delta,0})\Delta\pa_{\msfq} w_{\delta,0} + \pa_j(H_{ij}\pa^2_{i\msfq}w_{\delta,0})\Delta\pa_{\msfp} w_{\delta,0}\right\}\right].
\end{medsize}
\end{align*}
The third term vanishes because $(\Dot{T_2})_{ij}[H]y_i=(\Dot{T_4})_{ij}[H]y_i=0$. The fourth term vanishes because of \eqref{divergence free cancellation}. In summary, we have that
\begin{align*}
&\ \int_{\R^n}\pa_{\msfp}w_{\delta,0}L_2[H,H]\pa_{\msfq} w_{\delta,0} \\
&= \int_{\R^n} \left[2(\Ddot{T_2})_{ij}[H,H]\pa^2_{\msfp i}\Delta w_{\delta,0} \pa^2_{\msfq j}w_{\delta,0} +(\Ddot{T_4})_{ij}[H,H]\pa^2_{\msfp i}w_{\delta,0}\pa^2_{\msfq j}w_{\delta,0}\right] \\
&\ +\int_{\R^n} \left[\frac{n-2}{4(n-1)}\Ddot{R}[H,H]y_{\msfp}y_{\msfq}\{(\Delta w_{\delta,0})'\}^2\frac{1}{r^2}-\frac{n-6}{2}\Ddot{Q}^{(6)}[H,H]y_{\msfp}y_{\msfq}(w_{\delta,0}')^2\frac{1}{r^2}\right].
\end{align*}
From this, we can derive \eqref{eq:part2}.
\end{proof}

We represent the term $\RN{3}_{\msfpq}^{(k,m)}$ by using sphere integrals.
\begin{lemma}\label{lemma:rn3}
First, we fix eight constants depending only on $n$ and $k+m$:
\begin{align*}
c_{1}^{\RN{3}}(n,k+m) &:= -(n-6)^2(n-4)^2(n-2) \left[(n+4)\mci_{n}^{n+k+m+1}+4\mci_{n}^{n+k+m+3}\right] = \frac{1}{2}c_{1}^{\RN{2}}(n,k+m);\\
c_{2}^{\RN{3}}(n,k+m) &:= (n-6)^2(n-4)^2 \left[(n+2)^2\mci_{n}^{n+k+m-1}+8(n+2)\mci_{n}^{n+k+m+1}+16\mci_{n}^{n+k+m+3}\right];\\
c_{3}^{\RN{3}}(n,k+m) &:= 2(n-6)^2(n-4) \left[(n^2-8)\mci_{n-1}^{n+k+m-1}+4(n-3)\mci_{n-1}^{n+k+m+1}\right]=c_3^{\RN{2}}(n,k+m);\\
c_{4}^{\RN{3}}(n,k+m) &:= (n-6)^2(n-4)^2 \mci_{n-2}^{n+k+m-1}=c_{4}^{\RN{2}}(n,k+m);\\
c_{5}^{\RN{3}}(n,k+m) &:= -(n-6)^2(n-4) \mci_{n-3}^{n+k+m-3}=c_5^{\RN{2}}(n,k+m);\\
c_{6}^{\RN{3}}(n,k+m) &:= (n-6)^2\mci_{n-4}^{n+k+m-5};\\
c_{7}^{\RN{3}}(n,k+m) &:= -2(n-6)^2(n-4) \left[(n+2)\mci_{n-2}^{n+k+m-3}+4\mci_{n-2}^{n+k+m-1}\right];\\
c_{8}^{\RN{3}}(n,k+m) &:= (n-6)^2\mci_{n-4}^{n+k+m-5}.
\end{align*}
Second, we define
\begin{align*}
\bc^{\RN{3}}_1(n,k,m) &:= -c^{\RN{3}}_6(n,k+m)(k+m-6)(n+k+m-4)(k+m-4)(n+k+m-2)\frac{n-6}{4(n-1)}\\
&\ +c^{\RN{3}}_3(n,k+m)\frac{n^2-4n+12}{(n-2)(n-1)}+c^{\RN{3}}_2(n,k+m)\frac{n-2}{4(n-1)}\\
&\ +c^{\RN{3}}_1(n,k+m)\frac{n^2-4n+12}{(n-2)(n-1)} +c^{\RN{3}}_5(n,k+m)(k+m-4)(n+k+m-2)\frac{n-6}{n-1}\\
&\ +c^{\RN{3}}_5(n,k+m)(k+m-2)(n+k+m-2)\frac{16}{(n-4)(n-2)(n-1)}\\
&\ +c^{\RN{3}}_5(n,k+m)(k+m-1)(n+k+m-2)\frac{8}{(n-4)(n-1)}\\
&\ +c^{\RN{3}}_4(n,k+m)(k+m-4)(n+k+m-2)\frac{(n-6)(n-4)(n-2)+16}{2(n-4)(n-2)(n-1)}\\
&\ +c^{\RN{3}}_4(n,k+m)\frac{8(n-2)(k+m)^2-8(n-6)(k+m)+16(n-4)}{(n-4)(n-2)(n-1)}\\
&\ +c^{\RN{3}}_8(n,k+m)(n+k+m-4)(n+k+m-2)\frac{8}{(n-4)(n-1)};\\
\bc^{\RN{3}}_2(n,k,m) &:= -c^{\RN{3}}_6(n,k+m)(k+m-6)(n+k+m-4)\frac{2(n-6)}{(n-2)^2}\\
&\ +c^{\RN{3}}_5(n,k+m)\frac{8(n^2-8n+20)}{(n-4)(n-2)^2} +c^{\RN{3}}_4(n,k+m)\frac{4(n^2-8n+20)}{(n-4)(n-2)^2};\\
\bc^{\RN{3}}_3(n,k,m) &:= c^{\RN{3}}_1(n,k+m)\frac{4km}{n-2} +c^{\RN{3}}_4(n,k+m)(k+m-2)(n+k+m)\frac{4km}{(n-4)(n-2)};\\
\bc^{\RN{3}}_4(n,k,m) &:= c^{\RN{3}}_4(n,k+m)\frac{8}{(n-4)(n-2)};\\
\bc^{\RN{3}}_6(n,k,m) &:= -c^{\RN{3}}_6(n,k+m)\frac{4(n-6)}{(n-4)(n-2)^2};\\
\bc^{\RN{3}}_7(n,k,m) &:= c^{\RN{3}}_4(n,k+m)\frac{8}{(n-4)(n-2)};\\
\bc^{\RN{3}}_8(n,k,m) &:= -c^{\RN{3}}_6(n,k+m)(k+m-4)(n+k+m-4)\frac{n-6}{n-1}\\
&\ +c^{\RN{3}}_4(n,k+m)\frac{(n-6)(n-4)(n-2)+16}{(n-4)(n-2)(n-1)}\\
&\ +c^{\RN{3}}_5(n,k+m)\frac{2(n-6)}{n-1} -c^{\RN{3}}_5(n,k+m)(k+m-1)\frac{8}{(n-4)(n-1)}\\
&\ +c^{\RN{3}}_7(n,k+m)\frac{n^2-4n+12}{2(n-2)(n-1)} -c^{\RN{3}}_8(n,k+m)(n+k+m-4)\frac{8}{(n-4)(n-1)}\\
&\ +c^{\RN{3}}_8(n,k+m)(k+m-2)(n+k+m-4)\frac{(n-6)(n-4)(n-2)+16}{2(n-4)(n-2)(n-1)};\\
\bc^{\RN{3}}_9(n,k,m) &:= -c^{\RN{3}}_6(n,k+m)\frac{4(n-6)}{(n-2)^2} +c^{\RN{3}}_8(n,k+m) \left[\frac{4(n-4)}{(n-2)^2}+\frac{16}{(n-4)(n-2)^2}\right];\\
\bc^{\RN{3}}_{10}(n,k,m) &:= c^{\RN{3}}_4(n,k+m)\frac{8km}{(n-4)(n-2)};\\
\bc^{\RN{3}}_{11}(n,k,m) &:= c^{\RN{3}}_3(n,k+m)\frac{2}{n-2} +c^{\RN{3}}_5(n,k+m)(k+m-2)(n+k+m-2)\frac{4}{(n-4)(n-2)};\\
\bc^{\RN{3}}_{12}(n,k,m) &:= c^{\RN{3}}_5(n,k+m)\frac{16}{(n-4)(n-2)};\\
\bc^{\RN{3}}_{13}(n,k,m) &:= -c^{\RN{3}}_7(n,k+m)\frac{4}{n-2} -c^{\RN{3}}_8(n,k+m)(k+m-2)(n+k+m-4)\frac{8}{(n-4)(n-2)};\\
\bc^{\RN{3}}_{14}(n,k,m) &:= c^{\RN{3}}_8(n,k+m)(n+k+m-4)\frac{8}{(n-4)(n-2)};\\
\bc^{\RN{3}}_{15}(n,k,m) &:= c^{\RN{3}}_8(n,k+m) \left[-\frac{48}{(n-2)^2}+\frac{16}{(n-2)^2}+\frac{32}{(n-4)(n-2)}\right];\\
\bc^{\RN{3}}_{16}(n,k,m) &:= c^{\RN{3}}_7(n,k+m)\frac{2}{n-2} +c^{\RN{3}}_8(n,k+m)(k+m-2)(n+k+m-4)\frac{4}{(n-4)(n-2)};\\
\bc^{\RN{3}}_{17}(n,k,m) &:= c^{\RN{3}}_8(n,k+m)\frac{8}{(n-4)(n-2)}.
\end{align*}
Then we have
\begin{align}
&\begin{medsize}
\displaystyle \ \delta^{-(k+m-2)}\RN{3}_{\msfpq}^{(k,m)}
\end{medsize} \nonumber \\
&\begin{medsize}
\displaystyle = \bc^{\RN{3}}_1(n,k,m)\int_{\S^{n-1}}\Ddot{R}^{(k,m)}y_{\msfp}y_{\msfq} +\bc^{\RN{3}}_2(n,k,m)\int_{\S^{n-1}}\Dot{\ricci}^{(k)}\cdot\Dot{\ricci}^{(m)} y_{\msfp}y_{\msfq} +\bc^{\RN{3}}_3(n,k,m)\int_{\S^{n-1}}H^{(k)}\cdot H^{(m)} y_{\msfp}y_{\msfq}
\end{medsize} \nonumber \\
&\begin{medsize}
\displaystyle \ +\bc^{\RN{3}}_4(n,k,m)\int_{\S^{n-1}} \big(kH^{(k)} \cdot \Dot{\ricci}^{(m)} +m\Dot{\ricci}^{(k)} \cdot H^{(m)}\big) y_{\msfp}y_{\msfq}
\end{medsize} \nonumber \\
&\begin{medsize}
\displaystyle \ +\bc^{\RN{3}}_6(n,k,m)\int_{\S^{n-1}} \big(\Dot{\ricci}^{(k)} \cdot \Delta\Dot{\ricci}^{(m)}+\Delta\Dot{\ricci}^{(k)} \cdot \Dot{\ricci}^{(m)}\big)y_{\msfp}y_{\msfq}
\end{medsize} \nonumber \\
&\begin{medsize}
\displaystyle \ +\bc^{\RN{3}}_7(n,k,m)\int_{\S^{n-1}} \big(k^2H^{(k)} \cdot \Dot{\ricci}^{(m)}+m^2\Dot{\ricci}^{(k)} \cdot H^{(m)}\big) y_{\msfp}y_{\msfq}
\end{medsize} \nonumber \\
&\begin{medsize}
\displaystyle \ +\left[\bc^{\RN{3}}_8(n,k,m)\int_{\S^{n-1}}\Ddot{R}^{(k,m)} +\bc^{\RN{3}}_9(n,k,m) \bla\Dot{\ricci}^{(k)},\Dot{\ricci}^{(m)}\bra +\bc^{\RN{3}}_{10}(n,k,m) \bla H^{(k)},H^{(m)} \bra\right]\delta_{\msfpq}
\end{medsize} \label{eq:RN3_pq}\\
&\begin{medsize}
\displaystyle \ +\bc^{\RN{3}}_{11}(n,k,m)\int_{\S^{n-1}} \big(k H^{(k)}\cdot y_{\msfp}\pa_{\msfq} H^{(m)}+m\, y_{\msfp}\pa_{\msfq} H^{(k)} \cdot H^{(m)}\big)
\end{medsize} \nonumber \\
&\begin{medsize}
\displaystyle \ +\bc^{\RN{3}}_{12}(n,k,m)\int_{\S^{n-1}} \big(k\, y_{\msfp}\pa_{\msfq} H^{(k)} \cdot \Dot{\ricci}^{(m)}+m\Dot{\ricci}^{(k)} \cdot y_{\msfp}\pa_{\msfq} H^{(m)}\big)
\end{medsize} \nonumber \\
&\begin{medsize}
\displaystyle \ +\bc^{\RN{3}}_{13}(n,k,m)\int_{\S^{n-1}} \big(H_{\msfp l}^{(k)}\Dot{\ricci}_{\msfq l}^{(m)}+\Dot{\ricci}_{\msfq l}^{(k)}H_{\msfp l}^{(m)}\big) +\bc^{\RN{3}}_{14}(n,k,m)\int_{\S^{n-1}}\big(kH_{\msfp l}^{(k)}\Dot{\ricci}_{\msfq l}^{(m)}+m\Dot{\ricci}_{\msfq l}^{(k)}H_{\msfp l}^{(m)}\big)
\end{medsize} \nonumber \\
&\begin{medsize}
\displaystyle \ +\bc^{\RN{3}}_{15}(n,k,m)\int_{\S^{n-1}}\Dot{\ricci}^{(k)}_{\msfp l}\Dot{\ricci}^{(m)}_{\msfq l}
+\bc^{\RN{3}}_{16}(n,k,m)\int_{\S^{n-1}}\pa_{\msfp} H^{(k)}\cdot \pa_{\msfq} H^{(m)}
\end{medsize} \nonumber \\
&\begin{medsize}
\displaystyle \ +\bc^{\RN{3}}_{17}(n,k,m)\int_{\S^{n-1}}\big(\Dot{\ricci}^{(k)}\cdot \pa^2_{\msfpq} H^{(m)}+\pa^2_{\msfpq} H^{(k)}\cdot \Dot{\ricci}^{(m)}\big).
\end{medsize} \nonumber
\end{align}
\end{lemma}
\begin{proof}
We express $\RN{3}_{\msfpq}^{(k,m)}$ in polar coordinates:
\begin{align}
&\ \delta^{-(k+m-2)} \RN{3}_{\msfpq}^{(k,m)} \nonumber \\
&=\frac{n-2}{4(n-1)}\int_{\S^{n-1}}\Ddot{R}^{(k,m)}y_{\msfp}y_{\msfq} \cdot c_{2}^{\RN{3}}(n,k+m) -\frac{n-6}{2}\int_{\S^{n-1}}\big(\Ddot{Q}^{(6)}\big)^{(k,m)}y_{\msfp}y_{\msfq} \cdot c_{6}^{\RN{3}}(n,k+m) \nonumber \\
&\ +2\int_{\S^{n-1}}(\Ddot{T_2})^{(k,m)}_{ij}y_iy_jy_{\msfp}y_{\msfq} \cdot c_{1}^{\RN{3}}(n,k+m) +2\int_{\S^{n-1}}(\Ddot{T_2})^{(k,m)}_{i\msfp}y_iy_{\msfq} \cdot c_{3}^{\RN{3}}(n,k+m) \label{eq:RN3_pq1} \\
&\ +\int_{\S^{n-1}}(\Ddot{T_4})^{(k,m)}_{ij}y_iy_jy_{\msfp}y_{\msfq} \cdot c_{4}^{\RN{3}}(n,k+m) +2\int_{\S^{n-1}}(\Ddot{T_4})^{(k,m)}_{i\msfp}y_iy_{\msfq} \cdot c_{5}^{\RN{3}}(n,k+m) \nonumber \\
&\ +\int_{\S^{n-1}}(\Ddot{T_2})^{(k,m)}_{\msfpq} \cdot c_{7}^{\RN{3}}(n,k+m) + \int_{\S^{n-1}}(\Ddot{T_4})^{(k,m)}_{\msfpq} \cdot c_{8}^{\RN{3}}(n,k+m), \nonumber
\end{align}
where
\begin{align*}
c_{1}^{\RN{3}}(n,k+m)\delta^{k+m-2} &= \int_0^{\infty}r^{n+k+m-3} \left[(\Delta w_{\delta,0})''-(\Delta w_{\delta,0})'\frac{1}{r}\right] \(w_{\delta,0}''-w_{\delta,0}'\frac{1}{r}\);\\
c_{2}^{\RN{3}}(n,k+m)\delta^{k+m-2} &= \int_0^{\infty}r^{n+k+m-3}\{(\Delta w_{\delta,0})'\}^2;\\
c_{3}^{\RN{3}}(n,k+m)\delta^{k+m-2} &=
\begin{medsize}
\displaystyle \int_0^{\infty}r^{n+k+m-4}\left[\left\{(\Delta w_{\delta,0})''- (\Delta w_{\delta,0})'\frac{1}{r}\right\}w_{\delta,0}'+\(w_{\delta,0}''- w_{\delta,0}'\frac{1}{r}\)(\Delta w_{\delta,0})'\right];
\end{medsize}\\
c_{4}^{\RN{3}}(n,k+m)\delta^{k+m-2} &= \int_0^{\infty}r^{n+k+m-5}\(w_{\delta,0}''-w_{\delta,0}'\frac{1}{r}\)^2;\\
c_{5}^{\RN{3}}(n,k+m)\delta^{k+m-2} &= \int_0^{\infty} r^{n+k+m-6} \(w_{\delta,0}''- w_{\delta,0}'\frac{1}{r}\)w_{\delta,0}';\\
c_{6}^{\RN{3}}(n,k+m)\delta^{k+m-2} &= \int_0^{\infty}r^{n+k+m-7}(w_{\delta,0}')^2;\\
c_{7}^{\RN{3}}(n,k+m)\delta^{k+m-2} &= 2\int_0^{\infty}r^{n+k+m-5}(\Delta w_{\delta,0})'w_{\delta,0}';\\
c_{8}^{\RN{3}}(n,k+m)\delta^{k+m-2} &= \int_0^{\infty}r^{n+k+m-7}(w_{\delta,0}')^2.
\end{align*}
Among the eight sphere integrals in \eqref{eq:RN3_pq1}, four already appeared in \eqref{eq:RN2_pq1} and were evaluated in Lemma \ref{lemma:sphere with xpxq}.
In Lemma \ref{lemma:sphere int new four}, we compute the remaining four sphere integrals:
$$
\int_{\S^{n-1}}(\Ddot{T_2})^{(k,m)}_{i\msfp}y_iy_{\msfq}, \quad \int_{\S^{n-1}}(\Ddot{T_4})^{(k,m)}_{i\msfp}y_iy_{\msfq}, \quad \int_{\S^{n-1}}(\Ddot{T_2})^{(k,m)}_{\msfpq}, \quad \text{and} \quad \int_{\S^{n-1}}(\Ddot{T_4})^{(k,m)}_{\msfpq}.
$$

Combining the above information, we can deduce \eqref{eq:RN3_pq}.
\end{proof}

\subsubsection{Part 3: Evaluation of $\pa^2_{\xi_{\msfp}\xi_{\msfq}}\Ddot{\msfF}(\delta,0)$}
In Lemma \ref{lemma:FHess2}, we show that the quantity $\pa^2_{\xi_{\msfp}\xi_{\msfq}}\Ddot{\msfF}(\delta,0)$ can be written in terms of seven sphere integrals (up to a change in the order of the superscripts ${(k)}$ and $(m)$):
\begin{equation}\label{eq:seven}
\begin{aligned}
&\int_{\S^{n-1}}\Ddot{R}^{(k,m)}, \quad \int_{\S^{n-1}}\Ddot{R}^{(k,m)}y_{\msfp}y_{\msfq}, \quad \int_{\S^{n-1}}H^{(k)}\cdot H^{(m)}, \quad \int_{\S^{n-1}}H^{(k)}\cdot H^{(m)} y_{\msfp}y_{\msfq}, \\
&\int_{\S^{n-1}}H^{(k)}_{\msfp l}H^{(m)}_{\msfq l}, \quad \int_{\S^{n-1}}H^{(k)}\cdot y_{\msfp}\pa_{\msfq} H^{(m)}, \quad \text{and} \quad \int_{\S^{n-1}}H^{(k)}\cdot \pa^2_{\msfpq}H^{(m)}.
\end{aligned}
\end{equation}
Apparently, $\int_{\S^{n-1}}\Ddot{R}^{(k,m)}y_{\msfp}y_{\msfq}$, $\int_{\S^{n-1}}H^{(k)}\cdot H^{(m)} y_{\msfp}y_{\msfq}$, and $\int_{\S^{n-1}}H^{(k)}\cdot \pa^2_{\msfpq}H^{(m)}$ are symmetric in $\msfp$ and $\msfq$.
By \eqref{eq:commutative laplace}--\eqref{eq:commutative pa_pq} and the divergence theorem, $\int_{\S^{n-1}}H^{(k)}_{\msfp l}H^{(m)}_{\msfq l}$ and $\int_{\S^{n-1}}H^{(k)}\cdot y_{\msfp}\pa_{\msfq} H^{(m)}$ are also symmetric in $\msfp$ and $\msfq$.

\begin{lemma}\label{lemma:FHess2}
We define
\begin{equation}\label{eq:part3}
\RN{4}_{\msfpq}^{(k,m)}:=\frac{\delta^{k+m-2}}{2} \int_{\S^{n-1}}H^{(k)}_{\msfp l}H^{(m)}_{\msfq l} \int_{0}^{\infty}r^{n+k+m-3}\left[2(\Delta^2w)' w'+(\Delta w)'^2\right].
\end{equation}
Then, it holds that
\begin{equation}\label{eq:FHess2}
\pa^2_{\xi_{\msfp}\xi_{\msfq}}\Ddot{\msfF}(\delta,0) =\sum_{\substack{k,m=2\\ k,m: \textup{even}}}^{10} \left[-\RN{1}_{\msfpq}^{(k,m)}-\RN{2}_{\msfpq}^{(k,m)}-\RN{3}_{\msfpq}^{(k,m)}+\RN{4}_{\msfpq}^{(k,m)}\right],
\end{equation}
where $\lambda_q= -q(n+2q+2)$, $\lambda_{q'}= -q'(n+2q'+2)$,
\begin{align}
\delta^{-(k+m-2)}\RN{1}_{\msfpq}^{(k,m)} &= \left[\bc^{\RN{1}}_2(n,k,m)\lambda_q\lambda_{q'} +\bc^{\RN{1}}_3(n,k,m)+\bc^{\RN{1}}_4(n,k,m)(k\lambda_{q'}+m\lambda_q)\right. \nonumber \\
&\hspace{15pt} -2\bc^{\RN{1}}_5(n,k,m)(k\lambda_{q'-1}\lambda_{q'}+m\lambda_{q-1}\lambda_{q}) -2\bc^{\RN{1}}_6(n,k,m)(\lambda_{q'-1}+\lambda_{q-1})\lambda_{q'}\lambda_q \nonumber \\
&\hspace{15pt} \left.+\bc^{\RN{1}}_7(n,k,m)(k^2\lambda_{q'}+m^2\lambda_q)\right] \int_{\S^{n-1}}H^{(k)}\cdot H^{(m)}\delta_{\msfpq} \label{eq:rn1} \\
&\ +\bc^{\RN{1}}_1(n,k,m)\int_{\S^{n-1}}\Ddot{R}^{(k,m)}\delta_{\msfpq}, \nonumber
\end{align}
\begin{align}
\delta^{-(k+m-2)}\RN{2}_{\msfpq}^{(k,m)}
&= \left[\bc^{\RN{2}}_2(n,k,m)\lambda_q\lambda_{q'} +\bc^{\RN{2}}_3(n,k,m)+\bc^{\RN{2}}_4(n,k,m)(k\lambda_{q'}+m\lambda_q)\right. \nonumber \\
&\hspace{15pt} -2\bc^{\RN{2}}_5(n,k,m)(k\lambda_{q'-1}\lambda_{q'}+m\lambda_{q-1}\lambda_{q}) -2\bc^{\RN{2}}_6(n,k,m)(\lambda_{q'-1}+\lambda_{q-1})\lambda_{q'}\lambda_q \nonumber \\
&\hspace{15pt} \left.+\bc^{\RN{2}}_7(n,k,m)(k^2\lambda_{q'}+m^2\lambda_q)\right]\int_{\S^{n-1}}H^{(k)}\cdot H^{(m)}y_{\msfp}y_{\msfq} \label{eq:rn2} \\
&\ +\bc^{\RN{2}}_1(n,k,m)\int_{\S^{n-1}}\Ddot{R}^{(k,m)}y_{\msfp}y_{\msfq} +\bc^{\RN{2}}_8(n,k,m)\int_{\S^{n-1}}\Ddot{R}^{(k,m)}\delta_{\msfpq} \nonumber \\
&\ +\left[\bc^{\RN{2}}_9(n,k,m)\lambda_q\lambda_{q'}+\bc^{\RN{2}}_{10}(n,k,m)\right]\int_{\S^{n-1}}H^{(k)}\cdot H^{(m)}\delta_{\msfpq}, \nonumber
\end{align}
\begin{align}
&\ \delta^{-(k+m-2)}\RN{3}_{\msfpq}^{(k,m)} \nonumber \\
&=\left[\bc^{\RN{3}}_2(n,k,m)\lambda_q\lambda_{q'}+\bc^{\RN{3}}_3(n,k,m)+\bc^{\RN{3}}_4(n,k,m)(k\lambda_{q'}+m\lambda_q)\right. \nonumber \\
&\hspace{15pt} -\left.2\bc^{\RN{3}}_6(n,k,m)(\lambda_{q'-1}+\lambda_{q-1})\lambda_{q'}\lambda_q+\bc^{\RN{3}}_7(n,k,m)(k^2\lambda_{q'}+m^2\lambda_q)\right] \int_{\S^{n-1}}H^{(k)}\cdot H^{(m)} y_{\msfp}y_{\msfq} \nonumber \\
&\ +\bc^{\RN{3}}_1(n,k,m)\int_{\S^{n-1}}\Ddot{R}^{(k,m)}y_{\msfp}y_{\msfq} +\bc^{\RN{3}}_8(n,k,m)\int_{\S^{n-1}}\Ddot{R}^{(k,m)}\delta_{\msfpq} \nonumber \\
&\ +\left[\bc^{\RN{3}}_9(n,k,m)\lambda_q\lambda_{q'}+\bc^{\RN{3}}_{10}(n,k,m)\right]\int_{\S^{n-1}}H^{(k)}\cdot H^{(m)}\delta_{\msfpq} \label{eq:rn3} \\
&\ +\left[\bc^{\RN{3}}_{11}(n,k,m)m+\bc^{\RN{3}}_{12}(n,k,m)k \lambda_{q'} +\frac{1}{2}\bc^{\RN{3}}_{16}(n,k,m)(n+k+m-2)\right]\int_{\S^{n-1}}H^{(m)} \cdot y_{\msfp}\pa_{\msfq} H^{(k)} \nonumber \\
&\ +\left[\bc^{\RN{3}}_{11}(n,k,m)k+\bc^{\RN{3}}_{12}(n,k,m)m \lambda_{q} +\frac{1}{2}\bc^{\RN{3}}_{16}(n,k,m)(n+k+m-2)\right]\int_{\S^{n-1}} H^{(k)} \cdot y_{\msfp}\pa_{\msfq} H^{(m)} \nonumber \\
&\ +\left[\bc^{\RN{3}}_{13}(n,k,m)(\lambda_{q}+\lambda_{q'})+\bc^{\RN{3}}_{14}(n,k,m)(m\lambda_q+k\lambda_{q'}) +\bc^{\RN{3}}_{15}(n,k,m)\lambda_q\lambda_{q'}\right]\int_{\S^{n-1}}H_{\msfp l}^{(k)}H_{\msfq l}^{(m)} \nonumber \\
&\ +\left[\bc^{\RN{3}}_{17}(n,k,m)\lambda_{q'}-\frac{1}{2}\bc^{\RN{3}}_{16}(n,k,m)\right]\int_{\S^{n-1}}H^{(m)}\cdot \pa^2_{\msfpq} H^{(k)} \nonumber \\
&\ +\left[\bc^{\RN{3}}_{17}(n,k,m)\lambda_{q}-\frac{1}{2}\bc^{\RN{3}}_{16}(n,k,m)\right]\int_{\S^{n-1}}H^{(k)}\cdot \pa^2_{\msfpq} H^{(m)}, \nonumber
\end{align}
and
\begin{equation}\label{eq:rn4}
\begin{aligned}
&\ \delta^{-(k+m-2)}\RN{4}_{\msfpq}^{(k,m)}\\
&=\frac{1}{2}(n-6)^2(n-4)\left[(3n^3+8n^2-20n-48)\mci_{n}^{n+k+m-1}+16 (n^2-8)\mci_{n}^{n+k+m+1}
\right.\\
&\left.\hspace{95pt}+32(n-3)\mci_{n}^{n+k+m+3}\right]\int_{\S^{n-1}}H_{\msfp l}^{(k)}H_{\msfq l}^{(m)}.
\end{aligned}
\end{equation}
\end{lemma}
\begin{proof}
Combining \eqref{eq:FHess}, \eqref{eq:part1}, \eqref{eq:part2}, and \eqref{eq:part3}, we deduce \eqref{eq:FHess2}.

By combining Lemmas \ref{lemma:rn1}--\ref{lemma:rn2} with the identities
\begin{equation}\label{eq:three equations}
\Dot{\ricci}[H]_{ij} = -\frac{1}{2}\Delta H_{ij}, \quad \mcl_k H^{(k)}=|y|^2\Dot{\ricci}[H^{(k)}]=\lambda_q H^{(k)}, \quad |y|^4\Delta \Dot{\ricci}[H^{(k)}]=-2\lambda_{q-1}\lambda_q H^{(k)}
\end{equation}
for $y \in \S^{n-1}$, we can obtain \eqref{eq:rn1} and \eqref{eq:rn2}.

By applying Lemma \ref{IBP}, we derive
\begin{align*}
\int_{\S^{n-1}}\pa_{\msfp} H^{(k)} \cdot \pa_{\msfq} H^{(m)} &= \frac{1}{2}(n+k+m-2)\int_{\S^{n-1}} \left[H^{(k)}\cdot y_{\msfp}\pa_{\msfq} H^{(m)} + H^{(m)}\cdot y_{\msfp}\pa_{\msfq} H^{(k)}\right]\\
&\ -\frac{1}{2}\int_{\S^{n-1}} \left[H^{(k)}\cdot \pa^2_{\msfpq} H^{(m)} + H^{(m)}\cdot \pa^2_{\msfpq} H^{(k)}\right].
\end{align*}
Combining this with Lemma \ref{lemma:rn3} and \eqref{eq:three equations}, we can obtain \eqref{eq:rn3}.

In the end, \eqref{eq:rn4} follows directly from evaluating $\int_{0}^{\infty}r^{n+k+m-3}[2(\Delta^2w )' w'+(\Delta w)'^2]$.
\end{proof}

It remains to evaluate the seven sphere integrals in \eqref{eq:seven}. From \cite[Lemmas 9.5--9.6]{WZ}, we borrow the following lemma.
\begin{lemma}
Let $\bar{H}^{(2)}_{ij} = \bar{H}_{ij} :=W_{iabj}y_ay_b$. It holds that
\begin{align*}
\int_{\S^{n-1}}\bar{H}\cdot \bar{H} &= \frac{|\S^{n-1}|}{2n(n+2)}|W|_{\sharp}^2;\\
\int_{\S^{n-1}}(\bar{H}_{ij,l})^2 &= \frac{|\S^{n-1}|}{n}|W|^2_{\sharp};\\
\int_{\S^{n-1}}\bar{H}\cdot \bar{H} y_{\msfp}y_{\msfq} &= \frac{2|\S^{n-1}|}{n(n+2)(n+4)}(W\times W)_{\msfpq}+\frac{|\S^{n-1}|}{2n(n+2)(n+4)}|W|^2_{\sharp}\delta_{\msfpq};\\
\int_{\S^{n-1}}(\bar{H}_{ij,l})^2y_{\msfp}y_{\msfq} &= \frac{2|\S^{n-1}|}{n(n+2)}(W\times W)_{\msfpq}+\frac{|\S^{n-1}|}{n(n+2)}|W|^2_{\sharp}\delta_{\msfpq};\\
\int_{\S^{n-1}}\bar{H}_{\msfp i}\bar{H}_{\msfq i} &= \frac{|\S^{n-1}|}{2n(n+2)}(W\times W)_{\msfpq};\\
\int_{\S^{n-1}}\bar{H}\cdot y_{\msfp}\pa_{\msfq}\bar{H} &= \frac{|\S^{n-1}|}{n(n+2)}(W\times W)_{\msfpq};\\
\int_{\S^{n-1}}\bar{H}\cdot \pa^2_{\msfpq}\bar{H} &= 0.
\end{align*}
\end{lemma}
\begin{cor}\label{cor:seven sph integral}
Recall that $k=2q+2$ and $m=2q'+2$ where $q,q'\in \{0,\ldots,4\}$. It holds that
\begin{align*}
\int_{\S^{n-1}}H^{(k)}\cdot H^{(m)} &= \msfa_q\msfa_{q'}|\S^{n-1}||W|_{\sharp}^2\frac{1}{2n(n+2)},\\
\int_{\S^{n-1}}\Ddot{R}^{(k,m)} &= -\msfa_q\msfa_{q'}|\S^{n-1}||W|_{\sharp}^2 \left[\frac{(q+q'+qq')}{2n(n+2)}+\frac{1}{4n}\right],\\
\int_{\S^{n-1}}H^{(k)}\cdot H^{(m)} y_{\msfp}y_{\msfq} &= \msfa_q\msfa_{q'}|\S^{n-1}|(W\times W)_{\msfpq}\frac{2}{n(n+2)(n+4)}\\
&\ +\msfa_q\msfa_{q'}|\S^{n-1}||W|^2_{\sharp}\delta_{\msfpq}\frac{1}{2n(n+2)(n+4)},\\
\int_{\S^{n-1}}\Ddot{R}^{(k,m)}y_{\msfp}y_{\msfq} &= -\msfa_q\msfa_{q'}|\S^{n-1}|(W\times W)_{\msfpq} \left[\frac{2(q+q'+qq')}{n(n+2)(n+4)}+\frac{1}{2n(n+2)}\right]\\
&\ -\msfa_q\msfa_{q'}|\S^{n-1}||W|^2_{\sharp}\delta_{\msfpq} \left[\frac{(q+q'+qq')}{2n(n+2)(n+4)}+\frac{1}{4n(n+2)}\right],\\
\int_{\S^{n-1}}H_{\msfp i}^{(k)}H_{\msfq i}^{(m)} &= \msfa_q\msfa_{q'}|\S^{n-1}|(W\times W)_{\msfpq}\frac{1}{2n(n+2)},\\
\int_{\S^{n-1}}H^{(k)}\cdot y_{\msfp}\pa_{\msfq} H^{(m)} &= \msfa_q\msfa_{q'}|\S^{n-1}|(W\times W)_{\msfpq} \left[\frac{1}{n(n+2)}+\frac{4q'}{n(n+2)(n+4)}\right]\\
&\ +\msfa_q\msfa_{q'}|\S^{n-1}||W|^2_{\sharp}\delta_{\msfpq}\frac{q'}{n(n+2)(n+4)},\\
\int_{\S^{n-1}} H^{(k)}\cdot \pa^2_{\msfpq} H^{(m)} &= \msfa_q\msfa_{q'}|\S^{n-1}|(W\times W)_{\msfpq} \left[\frac{8q'(q'-1)}{n(n+2)(n+4)}+\frac{4q'}{n(n+2)}\right]\\
&\ +\msfa_q\msfa_{q'}|\S^{n-1}||W|^2_{\sharp}\delta_{\msfpq} \left[\frac{2q'(q'-1)}{n(n+2)(n+4)}+\frac{q'}{n(n+2)}\right].
\end{align*}
\end{cor}
\begin{proof}
We remind from \eqref{eq:Hkij} that $H^{(k)}_{ij}=\msfa_q|x|^{2q}\bar{H}^{(2)}_{ij}$ and $H^{(m)}_{ij}=\msfa_{q'}|x|^{2q'}\bar{H}^{(2)}_{ij}$. By using $\Ddot{R}^{(k,m)} = -\frac{1}{4}H_{ij,l}^{(k)}H_{ij,l}^{(m)}$, we have
\begin{align*}
\int_{\S^{n-1}}H^{(k)}\cdot H^{(m)} &= \msfa_q\msfa_{q'}\int_{\S^{n-1}}\bar{H}\cdot \bar{H};\\
\int_{\S^{n-1}}\Ddot{R}^{(k,m)} &= -\msfa_q\msfa_{q'}(q+q'+qq')\int_{\S^{n-1}}\bar{H}\cdot \bar{H} -\frac{1}{4}\msfa_q\msfa_{q'}\int_{\S^{n-1}}(\bar{H}_{ij,l})^2;\\
\int_{\S^{n-1}}H^{(k)}\cdot H^{(m)} y_{\msfp}y_{\msfq} &= \msfa_q\msfa_{q'}\int_{\S^{n-1}}\bar{H}\cdot \bar{H} y_{\msfp}y_{\msfq};\\
\int_{\S^{n-1}}\Ddot{R}^{(k,m)}y_{\msfp}y_{\msfq} & =-\msfa_q\msfa_{q'}(q+q'+qq')\int_{\S^{n-1}}\bar{H}\cdot \bar{H} y_{\msfp}y_{\msfq}-\frac{1}{4}\msfa_q\msfa_{q'}\int_{\S^{n-1}}(\bar{H}_{ij,l})^2y_{\msfp}y_{\msfq};\\
\int_{\S^{n-1}}H_{\msfp i}^{(k)}H_{\msfq i}^{(m)} &= \msfa_q\msfa_{q'}\int_{\S^{n-1}}\bar{H}_{\msfp i}\bar{H}_{\msfq i};\\
\int_{\S^{n-1}}H^{(k)} \cdot y_{\msfp}\pa_{\msfq} H^{(m)} &= \msfa_q\msfa_{q'}\int_{\S^{n-1}}\bar{H}\cdot y_{\msfp}\pa_{\msfq}\bar{H}+2\msfa_q\msfa_{q'}q'\int_{\S^{n-1}}\bar{H}\cdot \bar{H}y_{\msfp}y_{\msfq};\\
\int_{\S^{n-1}} H^{(k)} \cdot \pa^2_{\msfpq} H^{(m)} &= 4\msfa_q\msfa_{q'}q'(q'-1)\int_{\S^{n-1}}\bar{H}\cdot \bar{H}y_{\msfp}y_{\msfq}+2\msfa_q\msfa_{q'}q'\int_{\S^{n-1}}\bar{H}\cdot \bar{H}\delta_{\msfpq}\\
&\ +4\msfa_q\msfa_{q'}q'\int_{\S^{n-1}}\bar{H}\cdot y_{\msfp}\pa_{\msfq}\bar{H}.
\end{align*}
Applying the previous lemma, we can prove the assertion.
\end{proof}

Putting \eqref{eq:rn1}--\eqref{eq:rn4} together with Corollary \ref{cor:seven sph integral}, we have the following result.
\begin{lemma}\label{lemma:rn123}
It holds that
\begin{align*}
\pa^2_{\xi_{\msfp}\xi_{\msfq}}\Ddot{\msfF}(\delta,0) &= |\S^{n-1}|(W\times W)_{\msfpq}\sum_{q,q'=0}^{4}\delta^{k+m-2} \(-\msfm^{\RN{2},1}_{q,q'}-\msfm^{\RN{3},1}_{q,q'}+\msfm^{\RN{4}}_{q,q'}\) \msfa_q\msfa_{q'}\\
&\ +|\S^{n-1}||W|_{\sharp}^2\delta_{\msfpq}\sum_{q,q'=0}^{4}\delta^{k+m-2} \(-\msfm^{\RN{1}}_{q,q'}-\msfm^{\RN{2},2}_{q,q'}-\msfm^{\RN{3},2}_{q,q'}\) \msfa_q\msfa_{q'},
\end{align*}
where the six $5\times 5$ coefficient matrices $\msfm^{\RN{1}}_{q,q'},\, \msfm^{\RN{2},1}_{q,q'},\, \msfm^{\RN{2},2}_{q,q'},\, \msfm^{\RN{3},1}_{q,q'},\, \msfm^{\RN{3},2}_{q,q'}$, and $\msfm^{\RN{4}}_{q,q'}$ are given as follows:
\begin{align*}
\msfm^{\RN{1}}_{q,q'} &:= \frac{1}{2n(n+2)}\left[\bc^{\RN{1}}_2(n,k,m)\lambda_q\lambda_{q'} +\bc^{\RN{1}}_3(n,k,m)+\bc^{\RN{1}}_4(n,k,m)(k\lambda_{q'}+m\lambda_q)\right.\\
&\hspace{15pt} -2\bc^{\RN{1}}_5(n,k,m)(k\lambda_{q'-1}\lambda_{q'}+m\lambda_{q-1}\lambda_{q}) -2\bc^{\RN{1}}_6(n,k,m)(\lambda_{q'-1}+\lambda_{q-1})\lambda_{q'}\lambda_q\\
&\hspace{15pt} +\left.\bc^{\RN{1}}_7(n,k,m)(k^2\lambda_{q'}+m^2\lambda_q)\right] -\left[\frac{q+q'+qq'}{2n(n+2)}+\frac{1}{4n}\right]\bc^{\RN{1}}_1(n,k,m);\\
\msfm^{\RN{2},1}_{q,q'} &:= \frac{2}{n(n+2)(n+4)}\left[\bc^{\RN{2}}_2(n,k,m)\lambda_q\lambda_{q'} +\bc^{\RN{2}}_3(n,k,m)+\bc^{\RN{2}}_4(n,k,m)(k\lambda_{q'}+m\lambda_q)\right.\\
&\hspace{15pt} -2\bc^{\RN{2}}_5(n,k,m)(k\lambda_{q'-1}\lambda_{q'}+m\lambda_{q-1}\lambda_{q}) -2\bc^{\RN{2}}_6(n,k,m)(\lambda_{q'-1}+\lambda_{q-1})\lambda_{q'}\lambda_q\\
&\hspace{15pt} +\left.\bc^{\RN{2}}_7(n,k,m)(k^2\lambda_{q'}+m^2\lambda_q)\right] -\left[\frac{2(q+q'+qq')}{n(n+2)(n+4)}+\frac{1}{2n(n+2)}\right]\bc^{\RN{2}}_1(n,k,m);\\
\msfm^{\RN{2},2}_{q,q'} &:= \frac{1}{2n(n+2)(n+4)}\left[\bc^{\RN{2}}_2(n,k,m)\lambda_q\lambda_{q'} +\bc^{\RN{2}}_3(n,k,m)+\bc^{\RN{2}}_4(n,k,m)(k\lambda_{q'}+m\lambda_q)\right.\\
&\hspace{15pt} -2\bc^{\RN{2}}_5(n,k,m)(k\lambda_{q'-1}\lambda_{q'}+m\lambda_{q-1}\lambda_{q}) -2\bc^{\RN{2}}_6(n,k,m)(\lambda_{q'-1}+\lambda_{q-1})\lambda_{q'}\lambda_q\\
&\hspace{15pt} \left.+\bc^{\RN{2}}_7(n,k,m)(k^2\lambda_{q'}+m^2\lambda_q)\right] -\left[\frac{q+q'+qq'}{2n(n+2)(n+4)}+\frac{1}{4n(n+2)}\right]\bc^{\RN{2}}_1(n,k,m)\\
&\ -\left[\frac{q+q'+qq'}{2n(n+2)}+\frac{1}{4n}\right]\bc^{\RN{2}}_8(n,k,m) +\frac{1}{2n(n+2)}\left[\bc^{\RN{2}}_9(n,k,m)\lambda_q\lambda_{q'}+\bc^{\RN{2}}_{10}(n,k,m)\right];\\
\msfm^{\RN{3},1}_{q,q'} &:= \frac{2}{n(n+2)(n+4)}\left[\bc^{\RN{3}}_2(n,k,m)\lambda_q\lambda_{q'} +\bc^{\RN{3}}_3(n,k,m)+\bc^{\RN{3}}_4(n,k,m)(k\lambda_{q'}+m\lambda_q)\right.\\
&\hspace{15pt} -\left.2\bc^{\RN{3}}_6(n,k,m)(\lambda_{q'-1}+\lambda_{q-1})\lambda_{q'}\lambda_q +\bc^{\RN{3}}_7(n,k,m)(k^2\lambda_{q'}+m^2\lambda_q)\right]\\
&\ -\left[\frac{2(q+q'+qq')}{n(n+2)(n+4)}+\frac{1}{2n(n+2)}\right]\bc^{\RN{3}}_1(n,k,m)\\
&\ +\frac{1}{2n(n+2)}\left[\bc^{\RN{3}}_{13}(n,k,m)(\lambda_{q}+\lambda_{q'})+\bc^{\RN{3}}_{14}(n,k,m)(m \lambda_q+k\lambda_{q'})+\bc^{\RN{3}}_{15}(n,k,m)\lambda_q\lambda_{q'}\right]\\
&\ +\left[\bc^{\RN{3}}_{11}(n,k,m)m+\bc^{\RN{3}}_{12}(n,k,m)k \lambda_{q'} +\frac{1}{2}\bc^{\RN{3}}_{16}(n,k,m)(n+k+m-2)\right]\frac{n+4+4q}{n(n+2)(n+4)}\\
&\ +\left[\bc^{\RN{3}}_{11}(n,k,m)k+\bc^{\RN{3}}_{12}(n,k,m)m \lambda_{q} +\frac{1}{2}\bc^{\RN{3}}_{16}(n,k,m)(n+k+m-2)\right]\frac{n+4+4q'}{n(n+2)(n+4)}\\
&\ +\left[\bc^{\RN{3}}_{17}(n,k,m)\lambda_{q'}-\frac{1}{2}\bc^{\RN{3}}_{16}(n,k,m)\right] \cdot \left[\frac{8q(q-1)}{n(n+2)(n+4)}+\frac{4q}{n(n+2)}\right]\\
&\ +\left[\bc^{\RN{3}}_{17}(n,k,m)\lambda_{q}-\frac{1}{2}\bc^{\RN{3}}_{16}(n,k,m)\right] \cdot \left[\frac{8q'(q'-1)}{n(n+2)(n+4)}+\frac{4q'}{n(n+2)}\right];\\
\msfm^{\RN{3},2}_{q,q'} &:= \frac{1}{2n(n+2)(n+4)}\left[\bc^{\RN{3}}_2(n,k,m)\lambda_q\lambda_{q'} +\bc^{\RN{3}}_3(n,k,m)+\bc^{\RN{3}}_4(n,k,m)(k\lambda_{q'}+m\lambda_q)\right.\\
&\hspace{15pt} -\left.2\bc^{\RN{3}}_6(n,k,m)(\lambda_{q'-1}+\lambda_{q-1})\lambda_{q'}\lambda_q +\bc^{\RN{3}}_7(n,k,m)(k^2\lambda_{q'}+m^2\lambda_q)\right]\\
&\ -\left[\frac{q+q'+qq'}{2n(n+2)(n+4)}+\frac{1}{4n(n+2)}\right]\bc^{\RN{3}}_1(n,k,m)\\
&\ -\left[\frac{q+q'+qq'}{2n(n+2)}+\frac{1}{4n}\right]\bc^{\RN{3}}_8(n,k,m) +\frac{1}{2n(n+2)}\left[\bc^{\RN{3}}_9(n,k,m)\lambda_q\lambda_{q'}+\bc^{\RN{3}}_{10}(n,k,m)\right]\\
&\ +\left[\bc^{\RN{3}}_{11}(n,k,m)m+\bc^{\RN{3}}_{12}(n,k,m)k \lambda_{q'} +\frac{1}{2}\bc^{\RN{3}}_{16}(n,k,m)(n+k+m-2)\right]\frac{q}{n(n+2)(n+4)}\\
&\ +\left[\bc^{\RN{3}}_{11}(n,k,m)k+\bc^{\RN{3}}_{12}(n,k,m)m \lambda_{q} +\frac{1}{2}\bc^{\RN{3}}_{16}(n,k,m)(n+k+m-2)\right]\frac{q'}{n(n+2)(n+4)}\\
&\ +\left[\bc^{\RN{3}}_{17}(n,k,m)\lambda_{q'}-\frac{1}{2}\bc^{\RN{3}}_{16}(n,k,m)\right] \cdot \left[\frac{2q(q-1)}{n(n+2)(n+4)}+\frac{q}{n(n+2)}\right]\\
&\ +\left[\bc^{\RN{3}}_{17}(n,k,m)\lambda_{q}-\frac{1}{2}\bc^{\RN{3}}_{16}(n,k,m)\right] \cdot \left[\frac{2q'(q'-1)}{n(n+2)(n+4)}+\frac{q'}{n(n+2)}\right];\\
\msfm^{\RN{4}}_{q,q'} &:= \frac{(n-6)^2(n-4)}{4n(n+2)}\left[(3n^3+8n^2-20n-48)\mci_{n}^{n+k+m-1}+16 (n^2-8)\mci_{n}^{n+k+m+1}\right.\\
&\hspace{265pt} \left.+32(n-3)\mci_{n}^{n+k+m+3}\right].
\end{align*}
\end{lemma}

\begin{prop}\label{xi direction}
Assume that $n\geq 27$. Let $\Vec{\msfa}=[\msfa_0,-3634, 803, -62, 1]$ be the vector from Corollary \ref{delta direction}. Then,
$$
\pa^2_{\xi_{\msfp}\xi_{\msfq}}\Ddot{\msfF}(1,0) \text{ is positive definite}.
$$
\end{prop}
\begin{proof}
We set $\msfm^{1}_{q,q'}:=-\msfm^{\RN{2},1}_{q,q'}-\msfm^{\RN{3},1}_{q,q'}+\msfm^{\RN{4}}_{q,q'}$ and $\msfm^{2}_{q,q'}:=-\msfm^{\RN{1}}_{q,q'}-\msfm^{\RN{2},2}_{q,q'}-\msfm^{\RN{3},2}_{q,q'}$. Employing the Mathematica implementation of Lemma \ref{lemma:rn123}, we obtain that
\begin{align*}
&\ \sum_{q,q'=0}^{4}(\msfm^{1}_{q,q'}+\msfm^{2}_{q,q'})\msfa_q\msfa_{q'}\\
&=\frac{(n-6)^2}{1024 (n-2) (n+2) (n+4)} \frac{\Gamma(\frac{n}{2}-12) \Gamma(\frac{n}{2}+3)}{\Gamma(n+1)}\\
&\ \times [96 (n-24) (n-22) (n-20) (n-18) (n-16) (n-14) (n-12) (n-10) (n-3) (n-2)^2 \msfa_0^2\\
&\hspace{15pt} + (n-24) (n-22) (n-20) (n-18) (n-2) \\
&\hspace{15pt} \times (16575 n^9-1195250 n^8+34609680 n^7-498405392 n^6 +3446826416 n^5\\
&\hspace{25pt} -6825460384 n^4 -32854222976 n^3+140880911360 n^2 -106071106560 n+18351046656) \msfa_0\\
&\hspace{15pt} -62587200 n^{14}+9336113958 n^{13}-606097935636 n^{12}\\
&\hspace{15pt} +22249394723672 n^{11}-498762775048144 n^{10}+6750998185878816 n^9\\
&\hspace{15pt} -47179375607237312 n^8+11091250449666688 n^7+2371218672260016384 n^6\\
&\hspace{15pt} -14692080003168780288 n^5+16691759536983257088 n^4+91210104770428968960 n^3\\
&\hspace{15pt} -251248956860387328000 n^2+196206255022978105344 n-29818376183737221120],
\end{align*}
and
\begin{align*}
&\ \sum_{q,q'=0}^{4}\msfm^{2}_{q,q'}\msfa_q\msfa_{q'}\\
&=\frac{(n-6)^2}{1024 (n-2) (n+2) (n+4)} \frac{\Gamma(\frac{n}{2}-12) \Gamma(\frac{n}{2}+3)}{\Gamma(n+1)}\\
&\ \times [(n-24) (n-22) (n-20) (n-18) (n-2) \\
&\hspace{15pt} \times (3315 n^9-210762 n^8+5346128 n^7-67844560 n^6+431303152 n^5\\
&\hspace{25pt} -1014186016 n^4-2114580096 n^3+14491618304 n^2-18999469056 n+4521295872)\msfa_0\\
&\hspace{15pt} -9586980 n^{14}+1357155654 n^{13}-82964460564 n^{12}\\
&\hspace{15pt} +2837453335768 n^{11}-58221379164240 n^{10}+692018215328032 n^9\\
&\hspace{15pt} -3505745916586176 n^8-18173061039999360 n^7+375084465957971200 n^6\\
&\hspace{15pt} -1981203034836353024 n^5+2356795745930203136 n^4+11804945493335236608 n^3\\
&\hspace{15pt} -35984415983967043584 n^2+27690912838249611264 n-1017973919522488320].
\end{align*}
Recalling the definition of $\msfa_0$ from \eqref{eq:a0}, we can verify that
$$
(\msfm^{1}_{q,q'}+\msfm^{2}_{q,q'})\msfa_q\msfa_{q'}>0 \text{ and }\msfm^{2}_{q,q'}\msfa_q\msfa_{q'}>0
$$
provided $n\geq 27$. By Lemma \ref{lemma:A+trA}, $\pa^2_{\xi_{\msfp}\xi_{\msfq}}\Ddot{\msfF}(1,0)$ is positive definite.
\end{proof}

\subsection{Proof of Theorem \ref{thm:main6n}}
Combining Lemmas \ref{delta xi direction}--\ref{lem:R Psi direction}, Corollary \ref{delta direction}, and Proposition \ref{xi direction}, we immediately establish the following corollary.
\begin{cor}\label{cor:local minimum}
Under the same condition, $(1,0)\in \mca $ is a strict local minimum of the function
$$
(\dx) \mapsto \Ddot{\msfF}(\dx) + \frac{1}{2} \int_{\R^n} \omcr_{\dx}^1\opsi_{\dx}^1 dx.
$$
\end{cor}

We are now ready to complete the proof of Theorem \ref{thm:main6n}.
\begin{proof}[Proof of Theorem \ref{thm:main6n}]
Fix a non-negative cutoff function $\zeta \in C_c^{\infty}([0,\infty))$ satisfying $\zeta(r)=1$ for $r \in [0,1]$ and $\zeta(r)=0$ for $r \in [2,\infty)$.
Given any $N_0 \in \N$, we define a trace-free symmetric $2$-tensor on $\R^n$ by
$$
h_{ij}(x)=\sum_{N=N_0}^{\infty} \zeta(4N^2|x-x_N|)2^{-(10+\frac{1}{4})N}H_{ij}(2^N(x-x_N)),
$$
where $x_N:=(N^{-1},0,\ldots,0)\in \R^n$, $H_{ij}$ is the $2$-tensor defined in \eqref{eq:Hij}, and $\msfa_0,\ldots,\msfa_4$ in \eqref{eq:Hij} are the coefficients given in Corollary \ref{delta direction}.
It is straightforward to verify that $h$ is smooth in $\R^n$.

Then, we choose $N_0 \in \N$ sufficiently large so that, for each $N\geq N_0$, a translation $h_{N}(x):=h(x+x_N)$ of $h$ satisfies \eqref{eq:hij}--\eqref{eq:alpha0} with $\ep=2^{-N}$, $\rho=\frac{1}{4}N^{-2}$, and $\alpha_0 = \alpha_0(N_0)$ small.

We define $g_N:=\exp(h_N)$. By Corollary \ref{cor:local minimum} and estimate \eqref{eq:mcfexp2}, there exists a critical point $(\delta(\ep),\xi(\ep))\in \mca$ of the map $(\dx) \mapsto \mcf_{g_N}(v_{\edx} + \Psi_{\edx})$,
where $\mcf$ and $v_{\edx}$ are given in \eqref{eq:mcf} and \eqref{eq:vdx}, respectively, and $\Psi_{\edx}$ is obtained by solving \eqref{eq:nonlin}.
Hence, according to Lemma \ref{lem:u epsilon}, $u_{N}:=v_{\ep\delta(\ep),\ep\xi(\ep)} + \Psi_{\ep\delta(\ep),\ep\xi(\ep)}$ is a critical point of $\mcf_{g_N}$.

Finally, by translation invariance, for any $N\geq N_0$, the functions $u_{N}(\cdot-x_N)$ are solutions to \eqref{eq:main6} with $(M,g)=(\R^n,g_0:=\exp(h))$, and their $L^{\infty}$-norms diverge as $N \to \infty$.
\end{proof}

\appendix
\section{Useful Tools}\label{sec:useful}
In this appendix, we collect several useful results needed for the proof of Theorems \ref{thm:main4}, \ref{thm:main6}, \ref{thm:main6n}, and \ref{thm:Weyl}.

\medskip
The proofs of Lemma \ref{Euler's homo thm}--Lemma \ref{lemma:newbasis} follow directly from elementary computations, so we omit them.
\begin{lemma}[Euler's homogeneous function theorem]\label{Euler's homo thm}
If $p^{(k)}\in \mcp_{k}$, then
$$
x_i\pa_i p^{(k)}=kp^{(k)}.
$$
\end{lemma}

\begin{lemma}\label{harmonic poly}
If $|x|^{2s} p^{(k-2s)}\in \mcp_{k}$ for some $p^{(k-2s)}\in \mch_{k-2s}$, then
$$
\Delta \left[|x|^{2s} p^{(k-2s)}\right]=2s(2k-2s+n-2)|x|^{2s-2}p^{(k-2s)}.
$$
\end{lemma}

\begin{lemma}\label{IBP}
Assume that $n \ge 2$. For $i=1,\ldots,n$ and $F^{(k)}\in \mathcal{P}_k$, it holds that
$$
\int_{\S^{n-1}}x_iF^{(k)}=\int_{B_1}\pa_iF^{(k)} = \frac{1}{n+k-1}\int_{\S^{n-1}}\pa_iF^{(k)}.
$$
Therefore, for $p^{(k)}\in \mathcal{P}_{k}$ with $k \in \N$,
$$
\int_{\S^{n-1}}p^{(k)}=\frac{1}{k(n+k-2)}\int_{\S^{n-1}}\Delta p^{(k)}.
$$
\end{lemma}

\begin{lemma}\label{lemma:mcinl}
Let $\mci_n^l$ be the integral defined in \eqref{eq:mci}. It holds that
\begin{align*}
\mci_n^l&=\frac{2n-l-3}{2n-2}\mci_{n-1}^l;\\
\mci_n^l&=\frac{2n-l-3}{l+1}\mci_{n}^{l+2}.
\end{align*}
\end{lemma}

\begin{cor}\label{cor:mcinl}
It holds that
\begin{align*}
\mci_n^l&=\frac{l-1}{2n-l-1}\mci_{n}^{l-2}=\frac{l-1}{2n-2}\mci_{n-1}^{l-2},\\
\int_0^{\infty}\frac{(1-r^2)r^l}{(1+r^2)^{n}}dr&=\frac{2n-2l-4}{2n-l-3}\mci_n^l.
\end{align*}
\end{cor}

\begin{lemma}\label{lemma:newbasis}
Let $a> 0$, $b\ge 2$, and $p^{(b)} \in \mch_b$. Then
\begin{align*}
\Delta((1+r^2)^{-a}p^{(b)})
&=-2a(2a+2)(1+r^2)^{-a-2}p^{(b)} -2a(n-2a+2b-2)(1+r^2)^{-a-1}p^{(b)};\\
\Delta^2((1+r^2)^{-a}p^{(b)})&=2a(2a+2)(2a+4)(2a+6)(1+r^2)^{-a-4}p^{(b)}\\
&\ +2a(2a+2)(2a+4)2(n-2a+2b-4)(1+r^2)^{-a-3}p^{(b)}\\
&\ +2a(2a+2)(n-2a+2b-2)(n-2a+2b-4)(1+r^2)^{-a-2}p^{(b)}.
\end{align*}
Furthermore,
\begin{align*}
&\ \Delta^3((1+r^2)^{-a}p^{(b)})\\
&=-2a(2a+2)(2a+4)(2a+6)(2a+8)(2a+10)(1+r^2)^{-a-6}p^{(b)}\\
&\ -2a(2a+2)(2a+4)(2a+6)(2a+8)3(n-2a+2b-6)(1+r^2)^{-a-5}p^{(b)}\\
&\ -2a(2a+2)(2a+4)(2a+6)(n-2a+2b-4)3(n-2a+2b-6)(1+r^2)^{-a-4}p^{(b)}\\
&\ -2a(2a+2)(2a+4)(n-2a+2b-2)(n-2a+2b-4)(n-2a+2b-6)(1+r^2)^{-a-3}p^{(b)}.
\end{align*}
\end{lemma}

\begin{lemma}\label{lemma:tr}
For $l \in \N$, let $A$ be an $l \times l$ positive-definite symmetric matrix and $B$ an $l \times l$ positive semi-definite symmetric matrix. Then
\[\tr(AB) \ge \lambda_1(A)\, \tr(B),\]
where $\lambda_1(A) > 0$ is the least eigenvalue of $A$.
\end{lemma}
\begin{proof}
We express $A$ as $A = Q\Lambda Q^T$, where $\Lambda := \text{diag}(\lambda_1,\ldots,\lambda_l)$ is a diagonal matrix and $Q$ is an orthogonal matrix. We may assume that $0 < \lambda_1 \le \cdots \le \lambda_l$.
Set $B' = Q^TBQ$. Then, $B'$ is positive semi-definite, so
\[\tr(AB) = \tr(Q\Lambda Q^TB) = \tr(\Lambda Q^TBQ) = \tr(\Lambda B')
= \sum_{i=1}^l \lambda_iB_{ii}' \ge \lambda_1\tr(B') = \lambda_1\tr(B). \qedhere\]
\end{proof}

\begin{lemma}\label{lemma:A+trA}
For $l \in \N$, let $A$ be an $l \times l$ positive-definite symmetric matrix. Let $a,b\in \R$ with $a+b>0$ and $b>0$. Then
\[\text{the matrix } aA+b\,(\tr A)I \text{ is also positive-definite,}\]
where $I$ is the $l \times l$ identity matrix.
\end{lemma}
\begin{proof}
We express $A$ as $A = Q\Lambda Q^T$, where $\Lambda := \text{diag}(\lambda_1,\ldots,\lambda_l)$ is a diagonal matrix and $Q$ is an orthogonal matrix. Then $\tr A= \tr\, \Lambda$. Combining with $a>-b$ and $b>0$, we obtain that
\[Q^T[aA+b\,(\tr A)I]Q=a\Lambda+b\,(\tr\,\Lambda)I>[(\tr\,\Lambda)I-\Lambda]b>0. \qedhere\]
\end{proof}

\section{Curvature-driven Quantities in Conformal Normal Coordinates}\label{sec:cnc}
For the result in this section, we recall the notations given in Sections \ref{sec:curv}, Subsection \ref{subsec:proppd}, and Subsection \ref{subsec:curv6}.

\medskip
The following lemma follows from a straightforward calculation, we omit the proof.
\begin{lemma}\label{lemma:CNC appendix}
In conformal normal coordinates, there holds that
\begin{align*}
h_{il,j}x_l &= -h_{ij};\\
2\Dot{\Gamma}_{ij}^lx_l &= -(x_l\pa_l+2)h_{ij};\\
2\Dot{\Gamma}_{il}^jx_l &= (x_l\pa_l)h_{ij};\\
2\Dot{\ricci}_{ij}x_i &= (x_i\pa_i+1)\delta_jh;\\
4\Ddot{\ricci}_{ij}x_ix_j &= -x_l\pa_l h\cdot x_p\pa_p h.
\end{align*}
Consequently,
\begin{align*}
2\Gamma_{ij,l}^p\Dot{A}_{lp}x_ix_j &= -2\Dot{A}\cdot (x_i\pa_i)h;\\
(h_{lp}\Dot{A}_{ij,l})_{,p}x_ix_j &= (h_{lp}\Dot{A}_{ij,l}x_ix_j-2h_{lp}\Dot{A}_{lj}x_j)_{,p}+2\delta_lh \Dot{A}_{lj}x_j+2\Dot{A}\cdot h;\\
2\Dot{\Gamma}_{il,l}^p\Dot{A}_{pj}x_ix_j &= (x_i\pa_i-1)\delta_ph\Dot{A}_{pj}x_j;\\
2\Dot{\Gamma}_{il}^p\Dot{A}_{jp,l}x_ix_j &= [(x_i\pa_i)h_{lp}\Dot{A}_{jp}x_j]_{,l}-(x_i\pa_i+1)\delta_ph \Dot{A}_{jp}x_j-\Dot{A}\cdot (x_i\pa_i)h.
\end{align*}
\end{lemma}

\begin{lemma}
It holds that
\begin{align}
\int_{\S^{n-1}}\Ddot{R}^{(k,m)} &= \int_{\S^{n-1}} \left[\frac{1}{2}\delta_i H^{(k)}\delta_i H^{(m)} - \frac{1}{4}H_{ij,l}^{(k)}H_{ij,l}^{(m)}\right]; \label{eq:ricci h1} \\
\bla\Dot{\ricci}^{(k)},H^{(m)}\bra &= -\int_{\S^{n-1}} \left[2\Ddot{R}^{(k,m)} + \frac{k(n+k+m-2)}{2}H^{(k)}_{ij}H^{(m)}_{ij}\right]. \label{eq:ricci h2}
\end{align}
\end{lemma}
\begin{proof}
Identity \eqref{eq:ricci h1} is a direct corollary of \eqref{conformal normal}, the definition of $\Ddot{R}$ in \eqref{eq:R2}, and Lemma \ref{IBP}.

We next take \eqref{eq:ricci h2} into account. Due to the definition of $\Dot{\ricci}$ in \eqref{eq:Ric1}, we have
\begin{align*}
\Dot{\ricci}^{(k)}_{ij}H_{ij}^{(m)} &= H^{(k)}_{il,jl}H^{(m)}_{ij} - \frac{1}{2}H^{(k)}_{ij,ll}H^{(m)}_{ij}\\
&= \pa_j(H^{(k)}_{il,l}H^{(m)}_{ij}) - \frac{1}{2}\pa_l(H^{(k)}_{ij,l}H^{(m)}_{ij}) - H^{(k)}_{il,l}H^{(m)}_{ij,j} + \frac{1}{2}H^{(k)}_{ij,l}H^{(m)}_{ij,l}.
\end{align*}
Integrate it over $\S^{n-1}$, we discover
$$
\int_{\S^{n-1}}\Dot{\ricci}^{(k)}_{ij}H^{(m)}_{ij} = -\int_{\S^{n-1}} \left[2\Ddot{R}^{(k,m)} +\frac{1}{2}\pa_l(H^{(k)}_{ij,l}H^{(m)}_{ij})\right].
$$
Then, by invoking Lemma \ref{IBP} and Lemma \ref{Euler's homo thm}, we deduce \eqref{eq:ricci h2}.
\end{proof}

\begin{lemma}\label{ricci ricci lemma}
Let $\mcl_k$ be the linear operator introduced in \eqref{eq:mclk}. It holds that
\begin{align*}
\int_{\S^{n-1}}\Ddot{R}^{(k,m)} &= -\frac{1}{4}\left[\bla \mcl_kH^{(k)},H^{(m)}\bra+\bla H^{(k)}, \mcl_mH^{(m)}\bra\right] \\
&\ -\frac{1}{8}(k+m)(n+k+m-2)\bla H^{(k)},H^{(m)}\bra;\\
\bla\Dot{\ricci}^{(k)},\Dot{\ricci}^{(m)}\bra &= \bla \mcl_kH^{(k)},\mcl_mH^{(m)}\bra+\frac{km}{2}\bla\delta H^{(k)},\delta H^{(m)}\bra+\frac{1}{n-1}\bla\delta^2H^{(k)},\delta^2H^{(m)}\bra.
\end{align*}
\end{lemma}
\begin{proof}
Let us derive the first identity. According to the previous lemma, we have
$$
\int_{\S^{n-1}}\Ddot{R}^{(k,m)} = \frac{1}{4}\int_{\S^{n-1}} \left[\pa_{l}(H_{ij,j}^{(k)} H_{il}^{(m)})+\pa_{i}(H_{lj,j}^{(k)} H_{il}^{(m)}) -\pa_{j}(H_{il,j}^{(k)}H_{il}^{(m)})-2\Dot{\ricci}^{(k)}_{il}H_{il}^{(m)}\right].
$$
By employing the definition of $\mcl_k$ in \eqref{eq:mclk}, (\ref{conformal normal}), Lemma \ref{IBP}, and Lemma \ref{Euler's homo thm}, we derive
$$
\int_{\S^{n-1}}\Ddot{R}^{(k,m)} = -\frac{1}{4} \int_{\S^{n-1}} \left[2(\mcl_kH^{(k)})_{ij}H^{(m)}_{ij} + k(n+k+m-2)H^{(k)}_{ij}H^{(m)}_{ij}\right].
$$
Then, using the symmetry of $\Ddot{R}^{(k,m)}$ with respect to $k$ and $m$, we can obtain the first identity.

We next turn to the second identity. We use the definition of $\mcl_k$ and $\mcl_m$ to obtain
\begin{align*}
&\ \int_{\S^{n-1}} (\mcl_kH^{(k)})_{ij}(\mcl_mH^{(m)})_{ij} \\
&= \int_{\S^{n-1}} \left[\Dot{\ricci}^{(k)}_{ij}\Dot{\ricci}^{(m)}_{ij} + \frac{km}{2}\delta_i H^{(k)}\delta_i H^{(m)} + \frac{1}{n-1}\delta^2H^{(k)}\delta^2H^{(m)}\right] \\
&\ -\int_{\S^{n-1}}\Dot{\ricci}^{(k)}_{ij} \left[m x_j\delta_iH^{(m)} - \frac{1}{n-1}\delta^2H^{(m)}(x_ix_j-|x|^2\delta_{ij})\right]\\
&\ - \int_{\S^{n-1}}\Dot{\ricci}^{(m)}_{ij} \left[k x_j\delta_iH^{(k)} - \frac{1}{n-1}\delta^2H^{(k)}(x_ix_j-|x|^2\delta_{ij})\right]\\
&\ + \int_{\S^{n-1}} \left[\frac{k}{n-1}x_j\delta_iH^{(k)}\delta^2H^{(m)} + \frac{m}{n-1}x_j\delta_iH^{(m)}\delta^2H^{(k)}\right] (x_ix_j-|x|^2\delta_{ij}).
\end{align*}
Here, the fourth term vanishes because of (\ref{conformal normal}) and (\ref{conformal normal coro}). For the second and third terms, we need the following claim:
\begin{claim}\label{2ricci1}
$$
2\Dot{\ricci}^{(k)}_{ij}x_j=k\delta_iH^{(k)}.
$$
\end{claim}
\noindent \medskip \textsc{Proof of Claim \ref{2ricci1}.} By the definition of $\Dot{\ricci}$ in \eqref{eq:Ric1}, we know that
$$
2\Dot{\ricci}^{(k)}_{ij}x_j = x_j\pa_jH_{il,l}^{(k)} + \left[\pa^2_{li}(x_jH_{jl}^{(k)})-H_{ll,i}^{(k)}-\delta_iH^{(k)}\right] - \left[\Delta(x_jH^{(k)}_{ij})-2\delta_i H^{(k)}\right].
$$
Then, using (\ref{conformal normal}) and Lemma \ref{Euler's homo thm}, we finish the proof of the claim.

Combining Claim \ref{2ricci1} with (\ref{conformal normal}), (\ref{conformal normal coro}) and $\Dot{\ricci}^{(k)}_{ii}=\delta^2H^{(k)}$, we can compute the second term as
\begin{align*}
&\ \int_{\S^{n-1}}\Dot{\ricci}^{(k)}_{ij} \left[m x_j\delta_iH^{(m)} - \frac{1}{n-1}\delta^2H^{(m)}(x_ix_j-|x|^2\delta_{ij})\right] \\
&= \int_{\S^{n-1}} \left[\frac{km}{2}\delta_i H^{(k)}\delta_i H^{(m)} + \frac{1}{n-1}\delta^2H^{(k)}\delta^2H^{(m)}\right].
\end{align*}
By symmetry, the third term can be handled similarly to the second term.

Finally, by putting all four terms together, we complete the proof of the second identity.
\end{proof}

We now establish the commutator relations.
\begin{lemma}\label{lemma:comm}
Let $A^{(k)}$ and $B^{(m)}$ be matrices whose entries are homogeneous symmetric polynomials of degrees $k$ and $m$, respectively. Then,
\begin{equation}\label{eq:Laplace IBP}
\bla \Delta A^{(k)},B^{(m)}\bra-\bla A^{(k)},\Delta B^{(m)}\bra = (k-m)(n+k+m-2)\bla A^{(k)},B^{(m)}\bra;
\end{equation}
\begin{equation}\label{eq:mcl_k IBP}
\bla \mcl_kH^{(k)},H^{(m)}\bra-\bla H^{(k)},\mcl_mH^{(m)}\bra = -\frac{1}{2}(k-m)(n+k+m-2)\bla H^{(k)},H^{(m)}\bra;
\end{equation}
\begin{equation}\label{eq:diver of mcl_k}
\bla \delta(\mcl_kH^{(k)}),\delta H^{(m)}\bra = -\frac{n-3}{2(n-1)}\bla\delta^2H^{(k)},\delta^2 H^{(m)}\bra - \frac{1}{2}k(n+k-2) \bla\delta H^{(k)},\delta H^{(m)}\bra;
\end{equation}
\begin{align}
\begin{medsize}
\displaystyle \bla \delta^2(\mcl_kH^{(k)}),\delta^2H^{(m)}\bra
\end{medsize} &
\begin{medsize}
\displaystyle = \left[\frac{(n-3)(k-2)}{2(n-1)}(n+k+m-6)-\frac{n-2}{n-1}(k-2)(k-3)-n(k-1)\right] \bla \delta^2H^{(k)},\delta^2H^{(m)}\bra
\end{medsize} \nonumber \\
&\begin{medsize}
\displaystyle \ -\frac{n-3}{2(n-1)}\bla\pa\delta^2H^{(k)},\pa\delta^2H^{(m)}\bra.
\end{medsize} \label{eq:double diver of mcl_k}
\end{align}
\end{lemma}
\begin{proof}
The proofs of \eqref{eq:Laplace IBP} and \eqref{eq:mcl_k IBP} are straightforward. We omit them.

To justify \eqref{eq:diver of mcl_k} and \eqref{eq:double diver of mcl_k}, we need to compute $\delta(\mcl_k H^{(k)})$ and $\delta^2(\mcl_k H^{(k)})$. By the definition \eqref{eq:mclk} of $\mcl_k$, we have that
$$
\delta_j(\mcl_k H^{(k)}) = -\frac{1}{2}k(n+k-2)\delta_jH^{(k)} + \frac{n-3}{2(n-1)}|x|^2\delta^2H^{(k)}_{,j} + \(\frac{k-2}{n-1}-\frac{k-2}{2}\) x_j\delta^2H^{(k)},
$$
and
$$
\delta^2(\mcl_k H^{(k)}) = \frac{n-3}{2(n-1)}|x|^2\Delta\delta^2H^{(k)} - \left[\frac{n-2}{n-1}(k-2)(k-3)+n(k-1)\right]\delta^2H^{(k)}.
$$
Then, using $x_j\delta_j H^{(k)}=0$ and Lemma \ref{IBP}, we obtain \eqref{eq:diver of mcl_k} and \eqref{eq:double diver of mcl_k}.
\end{proof}

\begin{lemma}\label{Delta ricci lemma}
It holds that
\begin{align}
\bla\Dot{\ricci}^{(k)},H^{(m)}\bra &= \bla \mcl_kH^{(k)},H^{(m)}\bra; \label{eq:ric H}\\
\bla\Delta\Dot{\ricci}^{(k)},H^{(m)}\bra &=-2\bla \mcl_kH^{(k)},\mcl_mH^{(m)}\bra+(k-m-2)(n+k+m-4)\bla\mcl_kH^{(k)},H^{(m)}\bra \nonumber \\
&\ +\frac{n-3}{n-1}\bla\delta^2H^{(k)},\delta^2 H^{(m)}\bra+k(n+k-4)\bla\delta H^{(k)},\delta H^{(m)}\bra, \label{eq:lap ric H}
\end{align}
and
\begin{equation}\label{eq:lap ric ric}
\begin{aligned}
&\begin{medsize}
\displaystyle \ \bla\Delta\Dot{\ricci}^{(k)},\Dot{\ricci}^{(m)}\bra
\end{medsize} \\
&\begin{medsize}
\displaystyle =-2\bla \mcl_k^2H^{(k)},\mcl_m H^{(m)}\bra-2(n+2k-4)\bla\mcl_kH^{(k)},\mcl_mH^{(m)}\bra
\end{medsize} \\
&\begin{medsize}
\displaystyle \ +\frac{km}{2}[(n+k+m-4)(k-1)-(n+k-4)(n+m-2)]\bla\delta H^{(k)},\delta H^{(m)}\bra
\end{medsize} \\
&\begin{medsize}
\displaystyle \ +\left[m+\frac{2(k-1)}{n-1}+\frac{k-2}{n-1}(n+k+m-6) -\frac{n-3}{2(n-1)}k(n+k-4)\right] \bla\delta^2H^{(k)},\delta^2 H^{(m)}\bra
\end{medsize} \\
&\begin{medsize}
\displaystyle \ +\frac{n-3}{n-1}\left[\frac{(n-3)(m-2)}{2(n-1)}(n+k+m-6)-\frac{n-2}{n-1}(m-2)(m-3)-n(m-1)\right] \bla\delta^2H^{(k)},\delta^2H^{(m)}\bra
\end{medsize} \\
&\begin{medsize}
\displaystyle \ -\frac{km}{2}\bla\pa\delta H^{(k)},\pa \delta H^{(m)}\bra - \frac{n^2-4n+7}{2(n-1)^2} \bla\pa\delta^2H^{(k)},\pa\delta^2H^{(m)}\bra.
\end{medsize}
\end{aligned}
\end{equation}
\end{lemma}
\begin{proof}
The derivation of \eqref{eq:ric H} is straightforward. We omit it.

Let us consider \eqref{eq:lap ric H}. Applying \eqref{eq:mclk} and \eqref{eq:Laplace IBP} yields
\begin{equation}\label{eq:lap ric H1}
\begin{aligned}
\bla \Delta\Dot{\ricci}^{(k)}, H^{(m)}\bra &= \bla \mcl_k H^{(k)},\Delta H^{(m)}\bra + \frac{k}{2}\bla x_j\delta_iH^{(k)}+x_i\delta_j H^{(k)},\Delta H^{(m)}\bra \\
&\ +(k-m-2)(n+k+m-4)\bla\mcl_kH^{(k)},H^{(m)}\bra.
\end{aligned}
\end{equation}
For the first term in the right-hand side of \eqref{eq:lap ric H1}, we use Lemma \ref{IBP}, \eqref{eq:mclk}, and \eqref{eq:Ric1} to get
\begin{equation}\label{eq:lap ric H2}
\begin{aligned}
\bla\mcl_kH^{(k)},\Delta H^{(m)}\bra &= -2\bla\mcl_kH^{(k)},\mcl_mH^{(m)}\bra + \bla\mcl_kH^{(k)},\pa_j\delta_iH^{(m)}+\pa_i\delta_jH^{(m)}\bra\\
&=-2\bla\mcl_kH^{(k)},\mcl_mH^{(m)}\bra - 2\bla\delta(\mcl_k H^{(k)}),\delta H^{(m)}\bra.
\end{aligned}
\end{equation}
For the second term, we use $x_jH_{ij}^{(m)}=0$ to see
\begin{equation}\label{eq:lap ric H3}
\bla x_j\delta_i H^{(k)},\Delta H^{(m)}\bra=-2\bla\delta H^{(k)},\delta H^{(m)}\bra.
\end{equation}
By putting \eqref{eq:lap ric H3} and the combination of \eqref{eq:lap ric H2} and \eqref{eq:diver of mcl_k} into \eqref{eq:lap ric H1}, we deduce \eqref{eq:lap ric H}.

It remains to establish \eqref{eq:lap ric ric}. Employing \eqref{eq:mclk} and \eqref{eq:Laplace IBP}, we have that
\begin{align*}
\bla\Delta\Dot{\ricci}^{(k)},\Dot{\ricci}^{(m)}\bra &= \bla \Delta\Dot{\ricci}^{(k)},\mcl_mH^{(m)}\bra + \frac{m}{2} \bla\Delta\Dot{\ricci}^{(k)}, x_j\delta_i H^{(m)}+x_i\delta_jH^{(m)}\bra \\
&\ -\frac{1}{n-1} \bla\Delta\Dot{\ricci}^{(k)}, (x_ix_j-|x|^2\delta_{ij})\delta^2H^{(m)}\bra.
\end{align*}
Each of the three terms on the right-hand side is evaluated below: First, we infer from \eqref{eq:lap ric H} that
\begin{align*}
\bla \Delta\Dot{\ricci}^{(k)}, \mcl_m H^{(m)}\bra &= -2\bla\mcl_kH^{(k)},\mcl_m^2H^{(m)}\bra + (k-m-2)(n+k+m-4)\bla\mcl_k H^{(k)},\mcl_mH^{(m)}\bra\\
&\ + \frac{n-3}{n-1} \bla\delta^2H^{(k)},\delta^2\mcl_mH^{(m)}\bra + k(n+k-4)\bla\delta H^{(k)},\delta\mcl_mH^{(m)}\bra.
\end{align*}
Combining this with \eqref{eq:mcl_k IBP}--\eqref{eq:double diver of mcl_k}, we obtain that
\begin{align*}
&\ \begin{medsize}
\displaystyle \bla \Delta\Dot{\ricci}^{(k)}, \mcl_m H^{(m)}\bra
\end{medsize} \\
&\begin{medsize}
\displaystyle= -2\bla\mcl_k^2H^{(k)},\mcl_m H^{(m)}\bra - 2(n+2k-4)\bla\mcl_kH^{(k)},\mcl_mH^{(m)}\bra - \frac{km}{2}(n+k-4)(n+m-2)\bla\delta H^{(k)},\delta H^{(m)}\bra
\end{medsize} \\
&\begin{medsize}
\displaystyle \ +\frac{n-3}{n-1}\left[\frac{(n-3)(m-2)}{2(n-1)}(n+k+m-6)-\frac{n-2}{n-1}(m-2)(m-3)-n(m-1)\right] \bla \delta^2H^{(k)},\delta^2H^{(m)}\bra
\end{medsize} \\
&\begin{medsize}
\displaystyle \ -\frac{(n-3)^2}{2(n-1)^2}\bla\pa\delta^2H^{(k)},\pa\delta^2H^{(m)}\bra - \frac{n-3}{2(n-1)}k(n+k-4)\bla \delta^2H^{(k)},\delta^2H^{(m)}\bra.
\end{medsize}
\end{align*}
Second, Lemmas \ref{lemma:CNC appendix} and \ref{IBP} lead to
\begin{align*}
&\ \frac{m}{2} \bla\Delta\Dot{\ricci}^{(k)}, x_j\delta_i H^{(m)}+x_i\delta_jH^{(m)}\bra\\
&= \frac{km}{2} \bla\Delta(\delta H^{(k)}),\delta H^{(m)}\bra + m\bla\delta^2H^{(k)},\delta^2H^{(m)}\bra\\
&= \frac{km}{2}(n+k+m-4)(k-1) \bla\delta H^{(k)},\delta H^{(m)}\bra - \frac{km}{2}\bla\pa\delta H^{(k)},\pa\delta H^{(m)}\bra + m\bla\delta^2H^{(k)},\delta^2H^{(m)}\bra.
\end{align*}
Finally, by Lemmas \ref{lemma:CNC appendix} and \ref{IBP} again, we arrive at
\begin{align*}
&\ -\frac{1}{n-1} \bla\Delta\Dot{\ricci}^{(k)},(x_ix_j-|x|^2\delta_{ij})\delta^2 H^{(m)}\bra\\
&= \frac{1}{n-1} \bla(x_i\pa_i+k)\delta^2H^{(k)},\delta^2H^{(m)}\bra + \frac{1}{n-1} \bla\Delta(\delta^2H^{(k)}),\delta^2H^{(m)}\bra\\
&= \left[\frac{2(k-1)}{n-1}+\frac{k-2}{n-1}(n+k+m-6)\right] \bla\delta^2H^{(k)},\delta^2H^{(m)}\bra - \frac{1}{n-1}\bla\pa\delta^2H^{(k)},\pa\delta^2H^{(m)}\bra.
\end{align*}
The preceding computations together yield \eqref{eq:lap ric ric}.
\end{proof}

\section{Sphere Integrals Related to the $Q^{(6)}$-curvatures}\label{sec:sphere Q6}
We collect integrals of geometric quantities induced by the $Q^{(6)}$-curvature over the unit sphere, which play a role in Sections \ref{sec:comp6}--\ref{sec:noncomp6}.
In Lemmas \ref{lemma:Q6 sphere} and \ref{lemma:T4 sphere}, by \textbf{pure-divergence terms}, we mean terms that vanish whenever $\delta H^{(k)}=\delta H^{(m)}=0$.

\begin{lemma}\label{lemma:Q6 sphere}
Let $d=\lfloor\frac{n-6}{2}\rfloor$ and $k,m \in \{2,\ldots,d\}$. It holds that
\begin{align*}
\int_{\S^{n-1}}\big(\Ddot{Q^{(6)}}\big)^{(k,m)} &= \frac{1}{2(n-1)} \int_{\S^{n-1}}\Delta^2\Ddot{R}^{(k,m)} + \frac{4}{(n-2)^2} \int_{\S^{n-1}}\Delta \big(\Dot{\ricci}^{(k)}\cdot\Dot{\ricci}^{(m)}\big) \\
&\ + \frac{8}{(n-4)(n-2)^2}\left[\bla\Dot{\ricci}^{(k)},\Delta\Dot{\ricci}^{(m)}\bra + \bla\Delta\Dot{\ricci}^{(k)},\Dot{\ricci}^{(m)}\bra\right] \\
&\ +\textup{\textbf{(pure-divergence terms)}},
\end{align*}
where
\begin{align*}
&\begin{medsize}
\displaystyle \ \textup{\textbf{(pure-divergence terms)}}
\end{medsize} \\
&\begin{medsize}
\displaystyle := -\left[\left\{\frac{3n-4}{(n-4)(n-2)(n-1)^2}+\frac{3n-2}{16(n-1)^2}\right\} (k+m-4) + \frac{(n-6)(n(k+m)-4)}{(n-4)(n-2)(n-1)^2}\right](n+k+m-6)
\end{medsize} \\
&\begin{medsize}
\displaystyle \hspace{340pt} \times \int_{\S^{n-1}}\Dot{R}^{(k)}\Dot{R}^{(m)}
\end{medsize} \\
&\begin{medsize}
\displaystyle \ -\left[\frac{2(3n-4)}{(n-2)^2(n-1)^2}+\frac{(n-6)(n+4)}{4(n-4)(n-1)^2} -\frac{2(3n-4)}{(n-4)(n-2)(n-1)^2}-\frac{3n-2}{8(n-1)^2}\right] \int_{\S^{n-1}}\Dot{R}^{(k)}_{,i}\Dot{R}^{(m)}_{,i}.
\end{medsize}
\end{align*}
\end{lemma}
\begin{proof}
By Lemma \ref{Q6 curva},
\[\Ddot{Q^{(6)}}=\frac{16}{n-4}\Dot{A}\cdot \Delta \Dot{A} +\frac{8(n-6)}{n-4}\Dot{A}_{ij}\tr\Dot{A}_{,ij}+\RN{1}+\RN{2},\]
where
\begin{align*}
\RN{1} &:= \Delta^2(\tr\Ddot{A}-\Dot{A}\cdot h)+4\Delta(\Dot{A}\cdot\Dot{A}) -\frac{3n-2}{2}\tr\Dot{A}\Delta\tr\Dot{A}-(n-6)\tr\Dot{A}_{,i}\tr\Dot{A}_{,i};\\
\RN{2} &:= -[h_{ij}\Delta\tr\Dot{A}_{,i}+\Delta (h_{ij}\tr\Dot{A}_{,i})]_{,j}.
\end{align*}

By Lemma \ref{Schouten tensor}, it holds that $\tr\Dot{A}=\frac{1}{2(n-1)}\Dot{R}$, $\tr\Ddot{A}-\Dot{A}\cdot h=\frac{1}{2(n-1)}\Ddot{R}$, and
\[\Dot{A}\cdot \Dot{A} =\frac{1}{(n-2)^2}\Dot{\ricci}\cdot\Dot{\ricci} -\frac{3n-4}{4(n-2)^2(n-1)^2}\Dot{R}^2.\]
Using these identities, we can deal with the $\RN{1}$ term.

Integrating the $\RN{2}$ term over $\S^{n-1}$ and applying Lemma \ref{IBP} yields $\int_{\S^{n-1}} \RN{2} = 0$.

By employing Lemma \ref{Schouten tensor} again, we obtain
\begin{equation}\label{eq:Q6 sphere1}
\Dot{A}\cdot \Delta\Dot{A}=\frac{1}{(n-2)^2}\Dot{\ricci}\cdot\Delta\Dot{\ricci} -\frac{3n-4}{4(n-2)^2(n-1)^2}\Dot{R}\Delta\Dot{R}.
\end{equation}
Moreover, we infer from Lemmas \ref{IBP} and \ref{Euler's homo thm} that
\begin{equation}\label{eq:Q6 sphere2}
\begin{aligned}
\int_{\S^{n-1}}\Dot{R}^{(k)}\Delta\Dot{R}^{(m)} &= \int_{\S^{n-1}}\big(\Dot{R}^{(k)}_{,i}\Dot{R}^{(m)}\big)_{,i} -\int_{\S^{n-1}}\Dot{R}^{(k)}_{,i}\Dot{R}^{(m)}_{,i}\\
&= (n+k+m-6)(k-2)\int_{\S^{n-1}}\Dot{R}^{(k)}\Dot{R}^{(m)} -\int_{\S^{n-1}}\Dot{R}^{(k)}_{,i}\Dot{R}^{(m)}_{,i}.
\end{aligned}
\end{equation}
Combining \eqref{eq:Q6 sphere1} and \eqref{eq:Q6 sphere2}, we can treat the $\Dot{A}\cdot \Delta\Dot{A}$ term.

It suffices to handle the $\Dot{A}_{ij}\tr\Dot{A}_{,ij}$ term. Since $\Dot{A}_{ij}\tr\Dot{A}_{,ij}=(\Dot{A}_{ij}\tr\Dot{A}_{,i})_{,j}-\tr\Dot{A}_{,i}\tr\Dot{A}_{,i}$, we keep the $-\tr\Dot{A}_{,i}\tr\Dot{A}_{,i}$ part as it is and use the following claim to evaluate the $(\Dot{A}_{ij}\tr\Dot{A}_{,i})_{,j}$ part.
\begin{claim}\label{claim:Q6 sphere}
We have
$$
\int_{\S^{n-1}}\big(\Dot{A}^{(k)}_{ij}\tr\Dot{A}^{(m)}_{,i}\big)_{,j} =
-\frac{(n+k+m-6)((n-1)k+m-2)}{4(n-2)(n-1)^2} \int_{\S^{n-1}}\Dot{R}^{(k)}\Dot{R}^{(m)}.
$$
\end{claim}
Collecting the above computations completes the proof of Lemma \ref{lemma:Q6 sphere}.

\medskip \noindent \textsc{Proof of Claim \ref{claim:Q6 sphere}.} First, we use Lemma \ref{IBP} to obtain
$$
\int_{\S^{n-1}}\big(\Dot{A}^{(k)}_{ij}\tr\Dot{A}^{(m)}_{,i}\big)_{,j} =(n+k+m-6)\int_{\S^{n-1}}y_j\Dot{A}^{(k)}_{ij}\tr\Dot{A}^{(m)}_{,i}.
$$
Second, we use Lemma \ref{lemma:CNC appendix} to obtain
$$
y_j\Dot{A}_{ij}=\frac{(y_j\pa_j+1)}{2(n-2)}\delta_ih -\frac{1}{2(n-2)(n-1)}\Dot{R}y_i.
$$
In the final step, we use Lemmas \ref{lemma:CNC appendix} and \ref{IBP} and \eqref{conformal normal} to obtain
\begin{align*}
\int_{\S^{n-1}} (y_j\pa_j+1)\delta_ih\tr\Dot{A}_{,i}
&= \int_{\S^{n-1}} \left[\big\{(y_j\pa_j+1)\delta_ih\tr\Dot{A}\big\}_{,i} -\frac{(y_i\pa_i+2)}{2(n-1)}\delta^2h\Dot{R}\right] \\
&= -\int_{\S^{n-1}}\frac{(y_i\pa_i+2)}{2(n-1)}\delta^2h\Dot{R}.
\end{align*}
Combining the above three equations, we deduce
\begin{align*}
\int_{\S^{n-1}}\big(\Dot{A}^{(k)}_{ij}\tr\Dot{A}^{(m)}_{,i}\big)_{,j}
&= -\frac{n+k+m-6}{4(n-2)(n-1)}\int_{\S^{n-1}} \left[(y_i\pa_i+2)\delta^2H^{(k)}\Dot{R}^{(m)} +\frac{1}{n-1}y_i\Dot{R}^{(k)}\Dot{R}^{(m)}_{,i}\right] \\
&= -\frac{(n+k+m-6)((n-1)k+m-2)}{4(n-2)(n-1)^2} \int_{\S^{n-1}}\Dot{R}^{(k)}\Dot{R}^{(m)},
\end{align*}
proving the claim.
\end{proof}

\begin{lemma}\label{lemma:T2 sphere}
It holds that
$$
\int_{\S^{n-1}}(\Ddot{T_2})^{(k,m)}_{ij}y_iy_j
=\frac{n^2-4n+12}{2(n-2)(n-1)}\int_{\S^{n-1}}\Ddot{R}^{(k,m)} +\frac{2km}{n-2}\bla H^{(k)},H^{(m)}\bra
$$
and
\begin{align*}
&\begin{medsize}
\displaystyle \ \int_{\S^{n-1}} \left[\tr \,\Ddot{T_2}^{(k,m)}+(\Ddot{\diver_{g}T_2})^{(k,m)}_iy_i + \frac{1}{4}\left\{\Dot{T_2}^{(k)} \cdot (y_i\pa_i-2)H^{(m)} + \Dot{T_2}^{(m)} \cdot (y_i\pa_i-2)H^{(k)}\right\}\right]
\end{medsize} \\
&\begin{medsize}
\displaystyle = \frac{n^2-2n-8+(n-10)(k+m-2)}{2(n-1)} \int_{\S^{n-1}}\Ddot{R}^{(k,m)} -\frac{2}{n-2}\left[k\bla H^{(k)},\Dot{\ricci}^{(m)}\bra + m\bla\Dot{\ricci}^{(k)},H^{(m)}\bra\right].
\end{medsize}
\end{align*}
\end{lemma}
\begin{proof}
Using Lemma \ref{T2T4} and Corollary \ref{cor:t2t4}, we have that
\begin{align*}
(\Ddot{T_2})_{ij}y_iy_j &= -8\Ddot{A}_{ij}y_iy_j+(n-2)(\tr\Ddot{A}-\Dot{A}\cdot h)|y|^2;\\
\tr \,\Ddot{T_2} &= (n^2-2n-8)(\tr\Ddot{A}-\Dot{A}\cdot h)-8\Dot{A}\cdot h;\\
(\Ddot{\diver_{g} T_2})_iy_i &= (n-10)(y_i\pa_i)(\tr\Ddot{A}-\Dot{A}\cdot h);\\
\Dot{T_2}\cdot h &= -\frac{8}{n-2}\Dot{\ricci}\cdot h.
\end{align*}
By Lemma \ref{lemma:CNC appendix}, we also have that
$$
\Ddot{A}_{ij}y_iy_j=-\frac{1}{4(n-2)}y_l\pa_l h\cdot y_s\pa_s h-\frac{1}{2(n-2)(n-1)}\Ddot{R}|y|^2.
$$
Using the identities above, we can prove the lemma.
\end{proof}

\begin{lemma}\label{lemma:T4 sphere}
It holds that
\begin{equation}\label{eq:T4xx}
\begin{aligned}
&\begin{medsize}
\displaystyle \ \int_{\S^{n-1}}(\Ddot{T_4})^{(k,m)}_{ij}y_iy_j
\end{medsize} \\
&\begin{medsize}
\displaystyle = \left[\frac{(n-6)(n-4)(n-2)+16}{2(n-4)(n-2)(n-1)} (k+m-2)(n+k+m-4) \right.
\end{medsize} \\
&\begin{medsize}
\displaystyle \hspace{130pt} \left. +\frac{8(n-2)(k+m)^2-8(n-6)(k+m)+16(n-4)}{(n-4)(n-2)(n-1)}\right] \times \int_{\S^{n-1}}\Ddot{R}^{(k,m)}
\end{medsize} \\
&\begin{medsize}
\displaystyle \ +\frac{4km}{(n-4)(n-2)}(k+m)(n+k+m-2)\bla H^{(k)},H^{(m)}\bra +\frac{4(n^2-8n+20)}{(n-4)(n-2)^2}\bla\Dot{\ricci}^{(k)},\Dot{\ricci}^{(m)}\bra \end{medsize} \\
&\begin{medsize}
\displaystyle \ +\frac{8}{(n-4)(n-2)} \left[(k^2+k)\bla H^{(k)},\Dot{\ricci}^{(m)}\bra + (m^2+m)\bla\Dot{\ricci}^{(k)},H^{(m)}\bra \right] +\textup{\textbf{(pure-divergence terms)}},
\end{medsize}
\end{aligned}
\end{equation}
where
\begin{align*}
&\textup{\textbf{(pure-divergence terms)}}\\
&:=-\left[\frac{(3n-4)(n^2-8n+20)}{(n-4)(n-2)^2(n-1)^2} +\frac{3n^2-12n+28}{16(n-1)^2}+\frac{16(n-3)}{(n-4)(n-2)^2(n-1)^2}\right] \int_{\S^{n-1}}\Dot{R}^{(k)}\Dot{R}^{(m)}\\
&\ -\left[\frac{8km}{(n-4)(n-2)}+\frac{8km}{(n-2)^2}\right] \bla\delta H^{(k)},\delta H^{(m)}\bra.
\end{align*}
Also, there holds that
\begin{align}
&\ \int_{\S^{n-1}} \left[\tr \,\Ddot{T_4}^{(k,m)}+(\Ddot{\diver_{g}T_4})^{(k,m)}_iy_i + \frac{1}{4}\left\{\Dot{T_4}^{(k)} \cdot (y_i\pa_i-2)H^{(m)} + \Dot{T_4}^{(m)} \cdot (y_i\pa_i-2)H^{(k)}\right\}\right] \nonumber \\
&=-\frac{4}{(n-4)(n-2)} \left[k\bla H^{(k)},\Delta\Dot{\ricci}^{(m)}\bra + m\bla\Delta\Dot{\ricci}^{(k)},H^{(m)}\bra\right] \label{eq:T4diverx} \\
&\ +(k+m-2)(n+k+m-4)^2\frac{(n-6)}{2(n-1)}\int_{\S^{n-1}} \Ddot{R}^{(k,m)} \nonumber \\
&\ +\frac{4(n-6)(k+m)+4(n^2-8n+20)}{(n-2)^2} \bla\Dot{\ricci}^{(k)},\Dot{\ricci}^{(m)}\bra +\textup{\textbf{(pure-divergence terms)}}, \nonumber
\end{align}
where
\begin{align*}
&\begin{medsize}
\displaystyle \ \textup{\textbf{(pure-divergence terms)}}
\end{medsize} \\
&\begin{medsize}
\displaystyle := -\left[\frac{3n^3-12n^2-36n+64}{16(n-1)^2}-\frac{2(k+m)}{(n-4)(n-1)} +\frac{(3n^2-28n+28)(k+m-4)}{16(n-1)^2} \right.
\end{medsize} \\
&\begin{medsize}
\displaystyle \hspace{25pt} \left. +\frac{(n-4)(n(k+m)-4)}{(n-2)(n-1)^2} +\frac{8(n+k+m-4)}{(n-2)^2(n-1)^2}
+\frac{(3n-4)((n-6)(k+m)+(n^2-8n+20))}{(n-2)^2(n-1)^2}\right] \int_{\S^{n-1}}\Dot{R}^{(k)}\Dot{R}^{(m)}
\end{medsize} \\
&\begin{medsize}
\displaystyle \ +\frac{8(n+k+m-4)km}{(n-2)^2}\bla\delta H^{(k)},\delta H^{(m)}\bra.
\end{medsize}
\end{align*}
\end{lemma}
\begin{proof}
Let us prove equation \eqref{eq:T4xx} first. Using Lemmas \ref{T2T4} and \ref{lemma:CNC appendix}, and equation \eqref{eq:ddot bach}, we have that
\begin{align*}
(\Ddot{T_4})_{ij}y_iy_j &= -\frac{16}{n-4}\Ddot{B}_{ij}y_iy_j+\RN{1};\\
\Ddot{B}_{ij}y_iy_j &= \Delta\Ddot{A}_{ij}y_iy_j-2\Dot{A}\cdot h+\RN{2},
\end{align*}
where
\begin{align*}
\RN{1} &:= \left[(n-6)\Delta(\tr\Ddot{A}-\Dot{A}\cdot h)+4(n-4)\Dot{A}\cdot\Dot{A}\right]|y|^2 -\frac{1}{4}(3n^2-12n+28)(\tr\Dot{A})^2|y|^2\\
&\ -48(y_i\Dot{A}_{ij})^2-(n-6)(h_{ls}\tr\Dot{A}_{,s})_{,l}|y|^2;\\
\RN{2} &:= -(\tr\Ddot{A}-\Dot{A}\cdot h)_{,ij}y_iy_j -\Dot{A}\cdot\Dot{A}|y|^2-h_{ls,ij}\Dot{A}_{ls}y_iy_j-(n-4)(y_i\Dot{A}_{ij})^2 +(y_i\pa_i+1)\delta_sh \Dot{A}_{js}y_j\\
&\ +\frac{2}{n-2}(\tr\Dot{A})^2|y|^2 - \left[h_{ls}\Dot{A}_{ij,l}y_iy_j+2(y_i\pa_i-1)h_{ls}\Dot{A}_{lj}y_j\right]_{,s}.
\end{align*}
We can deal with terms $\RN{1}$ and $\RN{2}$ using the same techniques in proof of Lemmas \ref{lemma:Q6 sphere} and \ref{lemma:T2 sphere}.
Thus, it remains to evaluate $\Delta\Ddot{A}_{ij}y_iy_j-2\Dot{A}\cdot h$. We can write
$$
\Delta\Ddot{A}_{ij}y_iy_j=\Delta(\Ddot{A}_{ij}y_iy_j)-4\delta_i\Ddot{A}y_i-2\tr\Ddot{A}.
$$
Using the divergence formula \eqref{eq:diver schouten} of $A_{g}$, we have
$$
\delta_i\Ddot{A}y_i=(\tr\Ddot{A}-\Dot{A}\cdot h)_{,i}y_i-\Dot{A}\cdot h+\frac{1}{2}\Dot{A}\cdot (y_i\pa_i)h+(h_{ls}\Dot{A}_{il}y_i)_{,s}.
$$
In conclusion,
\begin{align*}
\Delta\Ddot{A}_{ij}y_iy_j-2\Dot{A}\cdot h &= \Delta(\Ddot{A}_{ij}y_iy_j)-4(\tr\Ddot{A}-\Dot{A}\cdot h)_{,i}y_i-2\Dot{A}\cdot (y_i\pa_i)h\\
&\ -2(\tr\Ddot{A}-\Dot{A}\cdot h)-4(h_{ls}\Dot{A}_{il}y_i)_{,s}.
\end{align*}
Using the above formulas, we can derive \eqref{eq:T4xx}.

Next, let us prove equation \eqref{eq:T4diverx}. Using Corollary \ref{cor:t2t4}, we have that
\begin{align*}
\tr \,\Ddot{T_4} &= -\frac{16}{n-4} h\cdot \Delta\Dot{A}+4(n^2-4n-12)\Dot{A}\cdot\Dot{A}+(n-6)n\Delta(\tr\Ddot{A}
-\Dot{A}\cdot h)-\frac{16}{n-4}\delta_ih\tr\Dot{A}_{,i}\\
&\ -\frac{1}{4}(3n^3-12n^2-36n+64)(\tr\Dot{A})^2 +\left[\frac{16}{n-4}-(n-6)n\right](h_{ls}\tr\Dot{A}_{,s})_{,l},\\
(\Ddot{\diver_{g} T_4})_iy_i &= 8(n-4)y_i\Dot{A}_{ij}\tr\Dot{A}_{,j}-32(\Dot{A}_{jl}y_i\Dot{A}_{ij})_{,l}+\RN{3},
\end{align*}
where
\begin{align*}
\RN{3} &:= (n-6)(y_i\pa_i)\left[\Delta (\tr\Ddot{A}-\Dot{A}\cdot h)\right]+4((n-6)y_i\pa_i+8)(\Dot{A}\cdot\Dot{A})\\
&\ -\frac{1}{4}(3n^2-28n+28)(y_i\pa_i)(\tr\Dot{A})^2-(n-6)(h_{lj}\tr\Dot{A}_{,l})_{,ij}y_i.
\end{align*}
We can deal with terms $\tr \,\Ddot{T_4}$ and $\RN{3}$ using the same techniques in proof of Lemmas \ref{lemma:Q6 sphere} and \ref{lemma:T2 sphere}.
For the term $y_i\Dot{A}_{ij}\tr\Dot{A}_{,j}$, we use the identity
\[\int_{\S^{n-1}}y_i\Dot{A}^{(k)}_{ij}\tr\Dot{A}^{(m)}_{,j} = -\frac{((n-1)k+m-2)}{4(n-2)(n-1)^2} \int_{\S^{n-1}}\Dot{R}^{(k)}\Dot{R}^{(m)}.\]
obtained in the proof of Claim \ref{claim:Q6 sphere}. For the term $(\Dot{A}_{jl}y_i\Dot{A}_{ij})_{,l}$, we use
\begin{multline*}
\frac{1}{n+k+m-4}\int_{\S^{n-1}}\big(\Dot{A}^{(k)}_{jl}y_i\Dot{A}^{(m)}_{ij}\big)_{,l} =\int_{\S^{n-1}}y_l\Dot{A}^{(k)}_{jl}y_i\Dot{A}^{(m)}_{ij} \\
=\frac{km}{4(n-2)^2}\bla\delta H^{(k)},\delta H^{(m)}\bra+\frac{1}{4(n-2)^2(n-1)^2}\int_{\S^{n-1}}\Dot{R}^{(k)}\Dot{R}^{(m)}.
\end{multline*}
Using the above formulas, we can derive \eqref{eq:T4diverx}.
\end{proof}

In the remainder of this section, we derive further estimates needed in Section \ref{sec:noncomp6}.
Under the divergence-free condition on $h$, the following corollary holds as a direct consequence of the results in Subsection \ref{subsec:curv6}.
\begin{cor}\label{cor:diver free sphere}
Assume that $\delta h=0$. Then there holds that
\begin{align*}
\Ddot{Q^{(6)}} &= \frac{1}{2(n-1)} \Delta^2 \Ddot{R} + \frac{4}{(n-2)^2} \Delta (\Dot{\ricci} \cdot \Dot{\ricci}) + \frac{16}{(n-4)(n-2)^2}\Dot{\ricci}\cdot\Delta\Dot{\ricci};\\
(\Ddot{T_2})_{ij}y_iy_j &= \frac{n^2-4n+12}{2(n-2)(n-1)} \Ddot{R}|y|^2 + \frac{2}{n-2}y_l\pa_lh \cdot y_s\pa_sh;\\
(\Ddot{T_4})_{ij}y_iy_j &= -\frac{16}{n-4}\Ddot{B}_{ij}y_iy_j + \left[\frac{n-6}{2(n-1)}
\Delta\Ddot{R} + \frac{4(n-4)}{(n-2)^2}\Dot{\ricci}\cdot\Dot{\ricci}\right] |y|^2;\\
\Ddot{B}_{ij}y_iy_j &= \Delta\left[-\frac{1}{4(n-2)}y_l\pa_lh \cdot y_s\pa_sh - \frac{1}{2(n-2)(n-1)}\Ddot{R}|y|^2\right] \\
&\ -\frac{1}{2(n-1)}\Ddot{R}_{,ij}y_iy_j - \frac{2}{n-1}\Ddot{R}_{,i}y_i - \frac{1}{n-1}\Ddot{R}\\
&\ -\frac{1}{n-2}\Dot{\ricci}\cdot y_iy_j \pa^2_{ij}h - \frac{2}{n-2}\Dot{\ricci}\cdot y_i\pa_ih - \frac{1}{(n-2)^2}\Dot{\ricci}\cdot\Dot{\ricci}|y|^2,
\end{align*}
and
\begin{align*}
&\ \tr \,\Ddot{T_2}+(\Ddot{\diver_{g}T_2})_iy_i+\frac{1}{2}\Dot{T_2} \cdot (y_i\pa_i-2)h \\
&= \frac{n^2-2n-8}{2(n-1)}\Ddot{R}+\frac{n-10}{2(n-1)}y_i\pa_i\Ddot{R}-\frac{4}{n-2}\Dot{\ricci}\cdot y_i\pa_ih;\\
&\ \tr \,\Ddot{T_4}+(\Ddot{\diver_{g}T_4})_iy_i+\frac{1}{2}\Dot{T_4} \cdot (y_i\pa_i-2)h \\
&= -\frac{8}{(n-4)(n-2)} \Delta\Dot{\ricci}\cdot y_i\pa_ih + \frac{4(n^2-4n-4)}{(n-2)^2} \Dot{\ricci}\cdot\Dot{\ricci}+ \frac{4(n-6)}{(n-2)^2}y_i\pa_i (\Dot{\ricci}\cdot\Dot{\ricci})\\
&\ +\frac{(n-6)n}{2(n-1)} \Delta\Ddot{R}+\frac{n-6}{2(n-1)} y_i\pa_i\Delta\Ddot{R}.
\end{align*}
\end{cor}

\begin{lemma}\label{lemma:sphere with xpxq}
Assume that $\delta H^{(k)}=\delta H^{(m)}=0$. Then there holds that
\begin{align*}
&\ \int_{\S^{n-1}}\big(\Ddot{Q^{(6)}}\big)^{(k,m)}y_{\msfp}y_{\msfq} \\
&= \frac{(k+m-6)(n+k+m-4)(k+m-4)(n+k+m-2)}{2(n-1)} \int_{\S^{n-1}}\Ddot{R}^{(k,m)}y_{\msfp}y_{\msfq}\\
&\ +\frac{2(k+m-4)(n+k+m-4)}{n-1} \int_{\S^{n-1}}\Ddot{R}^{(k,m)}\delta_{\msfpq}\\
&\ +\frac{4}{(n-2)^2}\left[(k+m-6)(n+k+m-4) \int_{\S^{n-1}}\Dot{\ricci}^{(k)}\cdot\Dot{\ricci}^{(m)}y_{\msfp}y_{\msfq} +2\bla\Dot{\ricci}^{(k)},\Dot{\ricci}^{(m)}\bra \delta_{\msfpq}\right] \\
&\ +\frac{8}{(n-4)(n-2)^2} \int_{\S^{n-1}}\big(\Dot{\ricci}^{(k)}\cdot\Delta\Dot{\ricci}^{(m)} + \Delta\Dot{\ricci}^{(k)}\cdot\Dot{\ricci}^{(m)}\big)y_{\msfp}y_{\msfq};
\end{align*}
\[\int_{\S^{n-1}}(\Ddot{T_2})^{(k,m)}_{ij}y_iy_jy_{\msfp}y_{\msfq}
=\frac{n^2-4n+12}{2(n-2)(n-1)}\int_{\S^{n-1}}\Ddot{R}^{(k,m)}y_{\msfp}y_{\msfq} +\frac{2km}{n-2}\int_{\S^{n-1}}H^{(k)}\cdot H^{(m)}y_{\msfp}y_{\msfq},\]
and
\begin{align*}
&\ \int_{\S^{n-1}} \left[\tr \,\Ddot{T_2}^{(k,m)}+(\Ddot{\diver_{g}T_2})^{(k,m)}_iy_i +\frac{1}{4}\left\{\Dot{T_2}^{(k)} \cdot (y_i\pa_i-2)H^{(m)} + \Dot{T_2}^{(m)} \cdot (y_i\pa_i-2)H^{(k)}\right\}\right]y_{\msfp}y_{\msfq} \\
&=\frac{n^2-2n-8+(n-10)(k+m-2)}{2(n-1)}\int_{\S^{n-1}}\Ddot{R}^{(k,m)}y_{\msfp}y_{\msfq} \\
&\ -\frac{2}{n-2} \int_{\S^{n-1}}\(kH^{(k)}\cdot\Dot{\ricci}^{(m)}+m\Dot{\ricci}^{(k)}\cdot H^{(m)}\)y_{\msfp}y_{\msfq}.
\end{align*}
In addition, we have
\begin{align*}
&\ \int_{\S^{n-1}}(\Ddot{T_4})^{(k,m)}_{ij}y_iy_jy_{\msfp}y_{\msfq}\\
&=\frac{(n-6)(n-4)(n-2)+16}{2(n-4)(n-2)(n-1)}
\int_{\S^{n-1}}\left[(k+m-4)(n+k+m-2)\Ddot{R}^{(k,m)}y_{\msfp}y_{\msfq} +2\Ddot{R}^{(k,m)}\delta_{\msfpq}\right]\\
&\ +\frac{8(n-2)(k+m)^2-8(n-6)(k+m)+16(n-4)}{(n-4)(n-2)(n-1)} \int_{\S^{n-1}}\Ddot{R}^{(k,m)}y_{\msfp}y_{\msfq}\\
&\ +\frac{4km}{(n-4)(n-2)}\left[(k+m-2)(n+k+m) \int_{\S^{n-1}}H^{(k)}\cdot H^{(m)}y_{\msfp}y_{\msfq} +2\bla H^{(k)},H^{(m)}\bra \delta_{\msfpq}\right] \\
&\ +\frac{8}{(n-4)(n-2)} \int_{\S^{n-1}}\left[(k^2+k)H^{(k)}\cdot\Dot{\ricci}^{(m)}+(m^2+m)\Dot{\ricci}^{(k)}\cdot H^{(m)}\right]y_{\msfp}y_{\msfq}\\
&\ +\frac{4(n^2-8n+20)}{(n-4)(n-2)^2} \int_{\S^{n-1}}\Dot{\ricci}^{(k)}\cdot\Dot{\ricci}^{(m)}y_{\msfp}y_{\msfq},
\end{align*}
and
\begin{align*}
&\ \int_{\S^{n-1}} \left[\tr \,\Ddot{T_4}^{(k,m)} + (\Ddot{\diver_{g}T_4})^{(k,m)}_iy_i +\frac{1}{4}\left\{\Dot{T_4}^{(k)} \cdot (y_i\pa_i-2)H^{(m)} + \Dot{T_4}^{(m)} \cdot (y_i\pa_i-2)H^{(k)}\right\}\right] y_{\msfp}y_{\msfq}\\
&= -\frac{4}{(n-4)(n-2)}\int_{\S^{n-1}} \big(kH^{(k)}\cdot\Delta\Dot{\ricci}^{(m)}+m\Delta\Dot{\ricci}^{(k)}\cdot H^{(m)}\big)y_{\msfp}y_{\msfq} \\
&\ +(n+k+m-4)\frac{(n-6)}{2(n-1)} \int_{\S^{n-1}}\left[(k+m-4)(n+k+m-2)\Ddot{R}^{(k,m)}y_{\msfp}y_{\msfq} +2\Ddot{R}^{(k,m)}\delta_{\msfpq}\right] \\
&\ +\frac{4(n-6)(k+m)+4(n^2-8n+20)}{(n-2)^2} \int_{\S^{n-1}}\Dot{\ricci}^{(k)}\cdot\Dot{\ricci}^{(m)}y_{\msfp}y_{\msfq}.
\end{align*}
\end{lemma}
\begin{proof}
Given $f\in \mcp_k$, Lemmas \ref{Euler's homo thm} and \ref{IBP} yield
\begin{equation}\label{eq:Lap xpxq}
\int_{\S^{n-1}}\Delta fy_{\msfp}y_{\msfq}=(k-2)(n+k)\int_{\S^{n-1}} fy_{\msfp}y_{\msfq}+2\delta_{\msfpq}\int_{\S^{n-1}}f.
\end{equation}
Straightforward computations using Corollary \ref{cor:diver free sphere}, Lemma \ref{Euler's homo thm}, and \eqref{eq:Lap xpxq} provide all the sphere integrals stated in the lemma.
\end{proof}

\begin{lemma}\label{lemma:sphere int new four}
Assume that $\delta H^{(k)}=\delta H^{(m)}=0$ and \eqref{eq:commutative laplace}--\eqref{eq:commutative pa_pq}. Then there holds that
\begin{equation}\label{eq:T2ipiq}
\begin{aligned}
\int_{\S^{n-1}}(\Ddot{T_2})^{(k,m)}_{i\msfp}y_iy_{\msfq} &= \frac{1}{n-2}\int_{\S^{n-1}}\left[kH^{(k)}\cdot y_{\msfp}\pa_{\msfq}H^{(m)}+m y_{\msfp}\pa_{\msfq}H^{(k)}\cdot H^{(m)}\right] \\
&\ +\frac{n^2-4n+12}{2(n-2)(n-1)} \int_{\S^{n-1}}\Ddot{R}^{(k,m)}y_{\msfp}y_{\msfq}
\end{aligned}
\end{equation}
and
\begin{equation}\label{eq:T2pq}
\begin{aligned}
\int_{\S^{n-1}}(\Ddot{T_2})^{(k,m)}_{\msfpq} &= \frac{2}{n-2}\int_{\S^{n-1}} \left[H_{\msfp l}^{(k)}\Delta H_{\msfq l}^{(m)}+H^{(m)}_{\msfp l}\Delta H_{\msfq l}^{(k)}+\pa_{\msfp} H^{(k)}\cdot \pa_{\msfq} H^{(m)}\right] \\
&\ +\frac{n^2-4n+12}{2(n-2)(n-1)} \int_{\S^{n-1}}\Ddot{R}^{(k,m)}\delta_{\msfpq}.
\end{aligned}
\end{equation}
In addition, we have
\begin{align}
&\ \int_{\S^{n-1}}(\Ddot{T_4})^{(k,m)}_{i\msfp}y_iy_{\msfq} \nonumber \\
&=\frac{n-6}{2(n-1)} \int_{\S^{n-1}}\left[(k+m-4)(n+k+m-2)\Ddot{R}^{(k,m)}y_{\msfp}y_{\msfq} +2\Ddot{R}^{(k,m)}\delta_{\msfpq}\right] \nonumber \\
&\ +\frac{8}{(n-4)(n-2)(n-1)}(k+m-2)(n+k+m-2) \int_{\S^{n-1}}\Ddot{R}^{(k,m)}y_{\msfp}y_{\msfq} \nonumber \\
&\ +\frac{8(k+m-1)}{2(n-4)(n-1)} \int_{\S^{n-1}}\left[(n+k+m-2)\Ddot{R}^{(k,m)}y_{\msfp}y_{\msfq} -\Ddot{R}^{(k,m)}\delta_{\msfpq}\right] \label{eq:T4ipiq} \\
&\ +\frac{2}{(n-4)(n-2)}(k+m-2)(n+k+m-2)\int_{\S^{n-1}} \left[kH^{(k)}\cdot y_{\msfp}\pa_{\msfq}H^{(m)}+m y_{\msfp}\pa_{\msfq}H^{(k)}\cdot H^{(m)}\right] \nonumber \\
&\ +\frac{4(n^2-8n+20)}{(n-4)(n-2)^2} \int_{\S^{n-1}}\Dot{\ricci}^{(k)}\cdot\Dot{\ricci}^{(m)}y_{\msfp}y_{\msfq} \nonumber \\
&\ +\frac{8}{(n-4)(n-2)} \int_{\S^{n-1}}\left[ky_{\msfp}\pa_{\msfq}H^{(k)} \cdot\Dot{\ricci}^{(m)} + m\Dot{\ricci}^{(k)}\cdot y_{\msfp}\pa_{\msfq}H^{(m)}\right], \nonumber
\end{align}
and
\begin{equation}\label{eq:T4pq}
\begin{aligned}
&\ \int_{\S^{n-1}}(\Ddot{T_4})^{(k,m)}_{\msfpq}\\
&= \left[\frac{n-6}{2(n-1)}+\frac{8}{(n-4)(n-2)(n-1)}\right](k+m-2)(n+k+m-4) \int_{\S^{n-1}}\Ddot{R}^{(k,m)}\delta_{\msfpq}\\
&\ +\frac{8}{(n-4)(n-1)}(n+k+m-4) \int_{\S^{n-1}}\left[(n+k+m-2)\Ddot{R}^{(k,m)}y_{\msfp}y_{\msfq} -\Ddot{R}^{(k,m)}\delta_{\msfpq}\right]\\
&\ +\frac{4}{(n-4)(n-2)}(k+m-2)(n+k+m-4) \\
&\hspace{135pt} \times \int_{\S^{n-1}}\left[H_{\msfp l}^{(k)}\Delta H_{\msfq l}^{(m)}+H^{(m)}_{\msfp l}\Delta H_{\msfq l}^{(k)}+\pa_{\msfp} H^{(k)}\cdot \pa_{\msfq} H^{(m)}\right]\\
&\ -\frac{48}{(n-2)^2}\int_{\S^{n-1}}\Dot{\ricci}^{(k)}_{\msfp l}\Dot{\ricci}^{(m)}_{\msfq l}
+\frac{4(n-4)}{(n-2)^2} \bla\Dot{\ricci}^{(k)},\Dot{\ricci}^{(m)}\bra \delta_{\msfpq}\\
&\ +\frac{16}{(n-4)(n-2)^2} \left[\bla\Dot{\ricci}^{(k)},\Dot{\ricci}^{(m)}\bra \delta_{\msfpq} +(n-4)\int_{\S^{n-1}}\Dot{\ricci}_{\msfp l}^{(k)}\Dot{\ricci}_{\msfq l}^{(m)}\right]\\
&\ +\frac{8}{(n-4)(n-2)} \int_{\S^{n-1}}\left[\pa^2_{\msfpq}H^{(k)}\cdot\Dot{\ricci}^{(m)} + \Dot{\ricci}^{(k)}\cdot\pa^2_{\msfpq}H^{(m)}-2\Delta H_{\msfp l}^{(k)}\Dot{\ricci}_{\msfq l}^{(m)}\right]\\
&\ +\frac{8}{(n-4)(n-2)}(n+k+m-4)\int_{\S^{n-1}}\left[kH_{\msfp l}^{(k)}\Dot{\ricci}_{\msfq l}^{(m)}+m\Dot{\ricci}_{\msfq l}^{(k)}H_{\msfp l}^{(m)}\right].
\end{aligned}
\end{equation}
\end{lemma}
\begin{proof}
We start by computing $\int_{\S^{n-1}}(\Ddot{T_2})_{i\msfp}y_iy_{\msfq}$ and $\int_{\S^{n-1}}(\Ddot{T_4})_{i\msfp}y_iy_{\msfq}$. Combining Lemmas \ref{T2T4} and \ref{Bach tensor} with $\delta h=0$, we know that
\begin{align*}
(\Ddot{T_2})_{i\msfp} y_i &= -\frac{8}{n-2}\Ddot{\ricci}_{i\msfp}y_i+\frac{n^2-4n+12}{2(n-2)(n-1)}\Ddot{R}y_{\msfp};\\
4\Ddot{\ricci}_{i\msfp}y_i &= -y_i\pa_i h\cdot \pa_{\msfp}h +(h_{l\msfp}y_i\pa_ih_{ls} -h_{ls}y_i\pa_ih_{l\msfp})_{,s};\\
(\Ddot{T_4})_{i\msfp} y_i &= -\frac{16}{n-4}\Ddot{B}_{i\msfp}y_i + \left[\frac{n-6}{2(n-1)}\Delta\Ddot{R} +\frac{4(n-4)}{(n-2)^2}\Dot{\ricci}\cdot\Dot{\ricci}\right]y_{\msfp};\\
\Ddot{B}_{i\msfp}y_i &= \Delta(\Ddot{A}_{i\msfp}y_i)-2\Ddot{A}_{i\msfp,i}-(\tr\Ddot{A}-\Dot{A}\cdot h)_{,i\msfp}y_i-\Dot{A}\cdot\Dot{A}y_{\msfp}\\
&\ -(h_{ls}\Dot{A}_{i\msfp,l})_{,s}y_i-h_{ls,i\msfp}y_i\Dot{A}_{ls}+2\Dot{\Gamma}_{i\msfp,l}^sy_i\Dot{A}_{ls} -\Dot{\Gamma}_{il,l}^sy_i\Dot{A}_{s\msfp}-2(\Dot{\Gamma}_{il}^sy_i\Dot{A}_{\msfp s,l}+\Dot{\Gamma}_{\msfp l}^sy_i\Dot{A}_{is,l})\\
&=\frac{1}{n-2}\Delta\left[\Ddot{\ricci}_{i\msfp}y_i-\frac{1}{2(n-1)}\Ddot{R}y_{\msfp}\right] -\frac{1}{2(n-1)}\Ddot{R}_{,\msfp}-\frac{1}{2(n-1)}(\Ddot{R}_{,i}y_i)_{,\msfp}\\
&\ -\frac{1}{(n-2)^2}\Dot{\ricci}\cdot\Dot{\ricci}y_{\msfp} -\frac{1}{n-2}\Dot{\ricci}_{ls}(y_i\pa_ih_{ls})_{,\msfp} +\frac{1}{n-2}(\Dot{\ricci}_{ls}y_i\pa_ih_{l\msfp} -\Dot{\ricci}_{l\msfp}y_i\pa_ih_{ls})_{,s}.
\end{align*}
Applying \eqref{eq:commutative pa_pq} and
\[\int_{\S^{n-1}}(h_{l\msfp}y_i\pa_ih_{ls}-h_{ls}y_i\pa_ih_{l\msfp})_{,s}y_{\msfq} =\int_{\S^{n-1}}(\Dot{\ricci}_{ls}y_i\pa_ih_{l\msfp}-\Dot{\ricci}_{l\msfp}y_i\pa_ih_{ls})_{,s}y_{\msfq}=0,\]
we can obtain \eqref{eq:T2ipiq} and
\begin{align*}
&\begin{medsize}
\displaystyle \ \int_{\S^{n-1}}(\Ddot{T_4})^{(k,m)}_{i\msfp}y_iy_{\msfq}
\end{medsize} \\
&\begin{medsize}
\displaystyle =\frac{n-6}{2(n-1)} \int_{\S^{n-1}}\left[(k+m-4)(n+k+m-2)\Ddot{R}^{(k,m)}y_{\msfp}y_{\msfq} +2\Ddot{R}^{(k,m)}\delta_{\msfpq}\right]
\end{medsize} \\
&\begin{medsize}
\displaystyle \ -\frac{16}{(n-4)(n-2)} \int_{\S^{n-1}}\left[\Delta\(\Ddot{\ricci}^{(k,m)}_{i\msfp}y_iy_{\msfq} -\frac{1}{2(n-1)}\Ddot{R^{(k,m)}}y_{\msfp}y_{\msfq}\) -2\(\Ddot{\ricci}^{(k,m)}_{i\msfp}y_i-\frac{1}{2(n-1)}\Ddot{R}^{(k,m)}y_{\msfp}\)_{,\msfq}\right]
\end{medsize} \\
&\begin{medsize}
\displaystyle \ +\frac{8(k+m-1)}{2(n-4)(n-1)} \int_{\S^{n-1}}\left[(n+k+m-2)\Ddot{R}^{(k,m)}y_{\msfp}y_{\msfq} -\Ddot{R}^{(k,m)}\delta_{\msfpq}\right]
\end{medsize} \\
&\begin{medsize}
\displaystyle \ +\frac{4(n^2-8n+20)}{(n-4)(n-2)^2} \int_{\S^{n-1}}\Dot{\ricci}^{(k)}\cdot\Dot{\ricci}^{(m)}y_{\msfp}y_{\msfq}
\end{medsize} \\
&\begin{medsize}
\displaystyle \ +\frac{8}{(n-4)(n-2)} \int_{\S^{n-1}}\left[ky_{\msfp}\pa_{\msfq}H^{(k)} \cdot\Dot{\ricci}^{(m)} + m\Dot{\ricci}^{(k)}\cdot y_{\msfp}\pa_{\msfq}H^{(m)}\right].
\end{medsize}
\end{align*}
Then, by employing Lemma \ref{IBP}, we have \eqref{eq:T4ipiq}.

Next, let us evaluate $\int_{\S^{n-1}}(\Ddot{T_2})_{\msfpq}$ and $\int_{\S^{n-1}}(\Ddot{T_4})_{\msfpq}$. Combining Lemmas \ref{T2T4}, \ref{Bach tensor}, and \ref{lemma:CNC appendix} with $\delta h=0$, we know that
\begin{align*}
(\Ddot{T_2})_{\msfpq} &= -\frac{8}{n-2}\Ddot{\ricci}_{\msfpq}+\frac{n^2-4n+12}{2(n-2)(n-1)}\Ddot{R}\delta_{\msfpq};\\
4\Ddot{\ricci}_{\msfpq} &= -h_{\msfp l}\Delta h_{\msfq l}-h_{\msfq l}\Delta h_{\msfp l}-h_{ls,\msfp}h_{ls,\msfq}+\RN{1};\\
(\Ddot{T_4})_{\msfpq} &= -\frac{16}{n-4}\Ddot{B}_{\msfpq}-\frac{48}{(n-2)^2}\Dot{\ricci}_{\msfp l}\Dot{\ricci}_{\msfq l}+\(\frac{n-6}{2(n-1)}\Delta \Ddot{R}+\frac{4(n-4)}{(n-2)^2}\Dot{\ricci}\cdot\Dot{\ricci}\)\delta_{\msfpq};\\
\Ddot{B}_{\msfpq} &= \Delta\Ddot{A}_{\msfpq}-(\tr\Ddot{A}-\Dot{A}\cdot h)_{,\msfpq}-\Dot{A}\cdot\Dot{A}\delta_{\msfpq}-(n-4)\Dot{A}_{\msfp l}\Dot{A}_{l\msfq}\\
&\ -h_{ls,\msfpq}\Dot{A}_{ls}+\frac{1}{2}(\Delta h_{l\msfp}\Dot{A}_{l\msfq}+\Delta h_{l\msfq}\Dot{A}_{l\msfp}) -2(\Dot{\Gamma}_{\msfp l}^s\Dot{A}_{\msfq s}+\Dot{\Gamma}_{\msfq l}^s\Dot{A}_{\msfp s})_{,l} -(h_{ls}\Dot{A}_{\msfpq,l})_{,s}+2(\Dot{\Gamma}_{\msfpq}^s\Dot{A}_{ls})_{,l}\\
&= \frac{1}{n-2}\Delta\left[\Ddot{\ricci}_{\msfpq}-\frac{1}{2(n-1)}\Ddot{R}\delta_{\msfpq}\right] -\frac{1}{2(n-1)}\Ddot{R}_{,\msfpq}-\frac{1}{(n-2)^2}\Dot{\ricci}\cdot\Dot{\ricci}\delta_{\msfpq}\\
&\ -\frac{n-4}{(n-2)^2}\Dot{\ricci}_{\msfp l}\Dot{\ricci}_{\msfq l} -\frac{1}{n-2}h_{ls,\msfpq}\Dot{\ricci}_{ls} +\frac{1}{2(n-2)}(\Delta h_{l\msfp}\Dot{\ricci}_{l\msfq}+\Delta h_{l\msfq}\Dot{\ricci}_{l\msfp})\\
&-\frac{2}{n-2}(\Dot{\Gamma}_{\msfp l}^s\Dot{\ricci}_{\msfq s}+\Dot{\Gamma}_{ql}^s\Dot{\ricci}_{\msfp s})_{,l}+\RN{2}.
\end{align*}
where
\begin{align*}
\RN{1} &:= (h_{ls,\msfp}h_{l\msfq}+h_{ls,\msfq}h_{l\msfp}-2h_{\msfp l}h_{\msfq s,l})_{,s} +[h_{ls}(2h_{\msfpq,l}-h_{l\msfp,\msfq}-h_{l\msfq,\msfp})]_{,s}, \\
\RN{2} &:= -\frac{1}{n-2}(h_{ls}\Dot{\ricci}_{\msfpq,l})_{,s} +\frac{2}{n-2}(\Dot{\Gamma}_{\msfpq}^s\Dot{\ricci}_{ls})_{,l}.
\end{align*}
We infer from Lemma \ref{IBP} that $\int_{\S^{n-1}}\RN{1}=\int_{\S^{n-1}}\RN{2}=0$.
Then, by applying Lemmas \ref{lemma:CNC appendix} and \ref{IBP}, and equations \eqref{eq:commutative laplace}--\eqref{eq:commutative pa_pq}, we can prove \eqref{eq:T2pq} and \eqref{eq:T4pq}.
\end{proof}

\section{Eigenspaces of the Operator $\mcl_k$}\label{sec:eigenspaces}
We start by discussing a projection onto the space $\mcv_k$ introduced at the beginning of Subsection \ref{subsec:proppd}.
\begin{lemma}[Lemma A.5 in \cite{KMS}]\label{projection lemma}
Let $\bar{H}$ be a symmetric matrix whose elements are all homogeneous polynomials on $\R^n$ of degree $k$. Suppose there are $p,t\in \mcp_{k-2}$, $q_j\in \mcp_{k-1}$ such that
\begin{align*}
\bar{H}_{ii}&=p|x|^2;\\
x_i\bar{H}_{ij}&=q_j|x|^2;\\
x_ix_j\bar{H}_{ij}&=t|x|^4.
\end{align*}
Then
$$
\bar{H}_{ij}+\mu_ix_j+\mu_jx_i+\nu\delta_{ij} \in \mcv_k,
$$
where $\mu_j := -q_j+\frac{p+(n-2)t}{2(n-1)}x_j$ and $\nu := \frac{t-p}{n-1}|x|^2$. We define
\begin{equation}\label{eq:projH}
(\proj\bar{H})_{ij} := \bar{H}_{ij}+\mu_ix_j+\mu_jx_i+\nu\delta_{ij}.
\end{equation}
\end{lemma}

\subsection{Three eigenspaces}
We study the eigenspaces $\mcv_k/\mcw_k$, $\mcw_k/\mcd_k$, and $\mcd_k$ of the operator $\mcl_k$ (see \eqref{eq:mclk} and \eqref{eq:mcv decom2}), which induces the orthogonal decomposition \eqref{eq:mcv decom} of the space $\mcv_k$.
\begin{lemma}[Lemma A.6 in \cite{KMS}]\label{H hat}
Let $\bar{H} \in \mcv_k$. Then there exist $\bar{W} \in \mcw_k$ and $\hat{H}_q\in \mcv_k/\mcw_k$ for $q=1,\ldots,\lfloor\frac{k-2}{2}\rfloor$ such that
$$
\bar{H} = \bar{W}+\sum_{q=1}^{\lfloor\frac{k-2}{2}\rfloor}\hat{H}_q,
$$
where $\langle \hat{H}_{q_1},\hat{H}_{q_2} \rangle=0$ for $q_1\neq q_2$. Furthermore, it holds that
\begin{align}
\mcl_k\hat{H}_q&=A_{k,q}\hat{H}_q; \label{eq:hatHq1} \\
\hat{H}_q&=\proj\left[|x|^{2q+2}\nabla^2P^{(k-2q)}\right] \label{eq:hatHq2}
\end{align}
for some $P^{(k-2q)}\in \mch_{k-2q}$, where $\nabla^2P^{(k-2q)}$ denotes the Hessian of $P^{(k-2q)}$ and
$$
A_{k,q} := (k-2q-1)\left[2-\frac{n-2}{n-1}(n+k-2q-1)\right]-(q+1)(n+2k-2q-4).
$$
\end{lemma}

\begin{lemma}[Lemma A.7 in \cite{KMS}]\label{D hat}
Let $\bar{D} \in \mcd_k$. Then there exist $\hat{D}_q\in \mcd_k$ for $q=0,\ldots,\lfloor\frac{k-2}{2}\rfloor$ such that
$$
\bar{D} = \sum_{q=0}^{\lfloor\frac{k-2}{2}\rfloor}\hat{D}_q,
$$
where $\langle \hat{D}_{q_1},\hat{D}_{q_2} \rangle=0$ for $q_1\neq q_2$. Furthermore, it holds that
\begin{align}
\mcl_k\hat{D}_q&=-q(n-2q+2k-2)\hat{D}_q; \label{eq:hatDq1} \\
\hat{D}_q&=|x|^{2q} M^{(k-2q)}, \label{eq:hatDq2}
\end{align}
for some $M^{(k-2q)} \in\{M \in \mcd_{k-2q}\mid \Delta M_{ij}=0 \text{ for each } i, j = 1,\ldots,n\}$.
\end{lemma}

\begin{lemma}[cf. Lemma \ref{lemma:W_hat0}]\label{W hat}
Let $\bar{W} \in \mcw_k$. Then there exist $\bar{D} \in \mcd_k$ and $\hat{W}_q\in \mcw_k/\mcd_k$ for $q=1,\ldots,\lfloor\frac{k-1}{2}\rfloor$ such that
\begin{equation}\label{eq:hatWq0}
\bar{W} = \bar{D}+\sum_{q=1}^{\lfloor\frac{k-1}{2}\rfloor}\hat{W}_q,
\end{equation}
where $\langle \hat{W}_{q_1},\hat{W}_{q_2} \rangle=0$ for $q_1\neq q_2$. Furthermore, if $\msd$ is the conformal Killing operator defined in \eqref{eq:cko}, then it holds that
\begin{align}
\mcl_k\hat{W}_q &= -\frac{(n+k-2)k}{2}\hat{W}_q; \label{eq:hatWq1} \\
\hat{W}_q &= \proj\left[|x|^{2q}\msd V^{(k-2q+1)}\right] \label{eq:hatWq2}
\end{align}
for some vector field $V^{(k-2q+1)}$ on $\R^n$ such that
\[V^{(k-2q+1)} \in \{V = (V_1,\ldots,V_n) \mid V_i \in \mch_{k-2q+1} \text{ for } i = 1,\ldots,n,\, \delta V=0,\, x_iV_i=0\}.\]
\end{lemma}
\begin{proof}
For each $j = 1,\ldots,k$, we consider the spherical harmonic decomposition of $\delta_j\bar{W}$: There exist $\bar{V}_j^{(k-2q-1)}\in \mch_{k-2q-1}$ for $q=0,\ldots,\lfloor\frac{k-1}{2}\rfloor$ such that
$$
\delta_j\bar{W}=\sum_{q=0}^{\lfloor\frac{k-1}{2}\rfloor}|x|^{2q}\bar{V}_j^{(k-2q-1)}.
$$

Let us define a property for a vector field $V$:
\begin{equation}\label{vector field property}
\delta V=V_{i,i}=0 \quad \text{and} \quad x_iV_i=0.
\end{equation}
From $\bar{W}\in \mcw_k$, we see that the vector field $\delta W$ satisfies property (\ref{vector field property}).
We assert that $\bar{V}^{(k-2q-1)}$ satisfies (\ref{vector field property}) for all $q=0,\ldots, \lfloor\frac{k-1}{2}\rfloor$.

First, since
$$
x_j\Delta \delta_j\bar{W} = \Delta(x_j\delta_j\bar{W})-2\pa_j\delta_j\bar{W}=0,
$$
$\Delta \delta \bar{W}$ satisfies (\ref{vector field property}). An induction argument shows that $\Delta^q \delta \bar{W}$ satisfies (\ref{vector field property}) for all $q \in \N \cup \{0\}$.
Picking $q = \lfloor \frac{k-1}{2} \rfloor$, we observe that $\bar{V}^{(0)}$ and $\bar{V}^{(1)}$ satisfy (\ref{vector field property}).

Second, by applying the identity
$$
\pa_j\(|x|^{2q'}\bar{V}_j^{(k-2q-1)}\) = 2q'|x|^{2q'-2}x_j\bar{V}_j^{(k-2q-1)}+|x|^{2q'}\pa_j\bar{V}_j^{(k-2q-1)}=0,
$$
and an induction argument, we can deduce that $\bar{V}^{(k-2q-1)}$ satisfies (\ref{vector field property}) for all $q=0,\ldots, \lfloor\frac{k-1}{2}\rfloor$. This proves the assertion.

As a result, $\bar{V}^{(0)} = 0$ when $k$ is odd. Moreover, the following claim holds, which will be proved at the end of the proof.
\begin{claim}\label{claim V}
Assume that $\bar{W}^* \in\mcw_k/\mcd_k$, $k=2m+2 \in \N$ even, and $\delta_j\bar{W}^* = |x|^{2m}\bar{V}_j^{(1)}$. If $\bar{V}^{(1)}$ satisfies (\ref{vector field property}), then $\bar{V}^{(1)}=0$.
\end{claim}

For $q=0,\ldots,\lfloor\frac{k-3}{2}\rfloor$, we define $V^{(k-2q-1)} = (V_1^{(k-2q-1)},\ldots,V_n^{(k-2q-1)})$, where
\[V_i^{(k-2q-1)}:=-\tfrac{1}{(n+k-2q-2)(k-2q-2)}\bar{V}_i^{(k-2q-1)} \in \mch_{k-2q-1}.\]
We also set $\hat{W}_{q+1}=\proj[|x|^{2q+2}\msd V^{(k-2q-1)}]$.
It follows from Lemma \ref{lemma:W_hat0} and Remark \ref{rmk:W_hat0} that $\hat{W}_{q+1} \in \mcw_k/\mcd_k$, $\delta_j\hat{W}_{q+1}=|x|^{2q}\bar{V}_j^{(k-2q-1)}$, and \eqref{eq:hatWq1} holds. Let
\[\bar{W}^{\sharp} = \bar{W}-\sum_{q=1}^{\lfloor\frac{k-1}{2}\rfloor}\hat{W}_q \in \mcw_k.\]
Then, $\bar{W}^{\sharp}$ can be decomposed as $\bar{W}^{\sharp} = \hat{W}^{\sharp} + \hat{D}^{\sharp}$ with $\hat{W}^{\sharp} \in \mcw_k/\mcd_k$ and $\hat{D}^{\sharp} \in \mcd_k$.
By applying Claim \ref{claim V} for $\hat{W}^{\sharp}$, we find that $\delta_j\hat{W}^{\sharp}=0$, i.e., $\hat{W}^{\sharp} \in \mcd_k$.
Thus, $\hat{W}^{\sharp} = 0$ and $\bar{W} = \sum_{q=1}^{\lfloor\frac{k-1}{2}\rfloor}\hat{W}_q + \hat{D}^{\sharp}$. Setting $\bar{D} = \hat{D}^{\sharp} \in \mcd_k$, we obtain \eqref{eq:hatWq0} as desired.

Also, a straight calculation by means of Lemma \ref{IBP} shows that $\langle \hat{W}_{q_1},\hat{W}_{q_2} \rangle=0$ for $q_1\neq q_2$.

\noindent \medskip \textsc{Proof of Claim \ref{claim V}.}
Let $\bar{V}_i^{(1)}(x)=a_{ij}x_j$ for $a_{ij} \in \R$. From property (\ref{vector field property}), we know that $a_{ij}+a_{ji}=0$. Combining this with the definition \eqref{eq:mclk} of $\mcl_{k}$, we get
\begin{align*}
(\mcl_{2m+2}\bar{W}^*)_{ij}&=\frac{|x|^2}{2}\left[2m|x|^{2m-2}\(x_i\bar{V}_j^{(1)}+x_j\bar{V}_i^{(1)}\)-\Delta \bar{W}^*_{ij}\right] \\
&\ -(m+1)|x|^{2m}\(x_i\bar{V}_j^{(1)}+x_j\bar{V}_i^{(1)}\).
\end{align*}
On the other hand, every element in $\mcw_k/\mcd_k$ is an eigenvector of $\mcl_k$ with the eigenvalue $-\frac{(n+k-2)k}{2}$, as shown in \cite[Page 184]{KMS}. Particularly, $(\mcl_{2m+2}\bar{W}^*)_{ij}=-(n+2m)(m+1) \bar{W}^*_{ij}$. Therefore,
\begin{equation}\label{eq:Wijid}
2|x|^{2m}\(x_i\bar{V}_j^{(1)}+x_j\bar{V}_i^{(1)}\)+|x|^2\Delta \bar{W}^*_{ij}=2(n+2m)(m+1)\bar{W}^*_{ij}.
\end{equation}
Now, by exploiting the spherical harmonic decomposition of $\bar{W}^*_{ij}$, Lemma \ref{harmonic poly}, and \eqref{eq:Wijid}, we can assume
$$
\bar{W}^*_{ij} =b_{ij}|x|^{2m+2}+\frac{1}{n}|x|^{2m}\(x_i\bar{V}_j^{(1)}+x_j\bar{V}_i^{(1)}\),
$$
where $b_{ij}=b_{ji} \in \R$. Then, given any $j=1,\ldots,n$,
$$
x_i\bar{W}^*_{ij} = \(b_{ij}+\frac{1}{n}a_{ij}\)x_i|x|^{2m+2}=0 \quad \text{for all } x\in \R^n,
$$
so $b_{ij}=-\frac{1}{n}a_{ij}$. Consequently, we conclude that $a_{ij} = b_{ij} = 0$, thereby confirming Claim \ref{claim V}.
\end{proof}

\subsection{Computational properties of eigenvectors}
We provide additional information about eigenvectors of $\mcl_k$, which will be useful for evaluating the Pohozaev quadratic forms in Section \ref{sec:tech} and Subsection \ref{sec:three cases Q6}.
\begin{lemma}\label{lemma:addition info of eigenvector}
We write $k=2q+s$ and $m=2q'+s$. Also, let $V_q^{(s+1)}=V^{(k-2q+1)}$ and $\hat{W}_q^{(k)}=\proj[|x|^{2q}\msd V^{(k-2q+1)}]$ be from Lemma \ref{W hat}. Then
\begin{equation}\label{eq:ddhatW}
\begin{aligned}
\delta_j\hat{W}_q^{(k)} &= -s(n+s)|x|^{2q-2}(V^{(s+1)}_q)_j; \\
\pa_i\delta_j\hat{W}_q^{(k)} &= -s(n+s)|x|^{2q-4}\left[2(q-1)x_i(V^{(s+1)}_q)_j+|x|^2(V^{(s+1)}_q)_{j,i}\right].
\end{aligned}
\end{equation}
Let $P_q^{(s)}=P^{(k-2q)}$ and $\hat{H}_q^{(k)}=\proj[|x|^{2q+2}\nabla^2P^{(k-2q)}]$ be from Lemma \ref{H hat}. Then
\begin{equation}\label{eq:ddhatH}
\begin{aligned}
\delta_j\hat{H}_q^{(k)} &= -\frac{n-2}{n-1}(s-1)(n+s-1)|x|^{2q-2} \left[|x|^{2}\pa_jP_q^{(s)}-sx_jP_q^{(s)}\right]; \\
\pa_i\delta_j\hat{H}_q^{(k)} &= -\frac{n-2}{n-1}(s-1)(n+s-1)|x|^{2q-4} \left[|x|^{4}\pa^2_{ij}P_q^{(s)}+2q|x|^{2}x_i\pa_{j}P_q^{(s)}\right.\\
&\hspace{70pt} \left.-s|x|^{2}x_j\pa_iP_q^{(s)}-2(q-1)sx_ix_jP_q^{(s)} -s|x|^{2}\delta_{ij}P_q^{(s)}\right]; \\
\delta^2\hat{H}_q^{(k)} &= \ka_s|x|^{2q-2}P_q^{(s)},
\end{aligned}
\end{equation}
where $\ka_s=\frac{n-2}{n-1}s(s-1)(n+s-1)(n+s-2)$.

Moreover, if $\hat{W}_{q'}^{(m)}=\proj[|x|^{2q'}\msd V^{(m-2q'+1)}]$ and $\hat{H}_{q'}^{(m)}=\proj[|x|^{2q'+2}\nabla^2P^{(m-2q')}]$ are from Lemmas \ref{W hat} and \ref{H hat}, respectively, then we have
\begin{equation}\label{eq:ddhatW2}
\begin{aligned}
\bla \hat{W}_q^{(k)}, \hat{W}_{q'}^{(m)} \bra &= 2s(n+s)\bla V_q^{(s+1)},V_{q'}^{(s+1)} \bra; \\
\bla \pa\delta\hat{W}_q^{(k)}, \pa \delta\hat{W}_{q'}^{(m)} \bra &= s^2(n+s)^2\left[4(q-1)(q'-1)+2(s+1)(q+q'-2)\right. \\
&\hspace{150pt} \left.+(s+1)(n+2s)\right]\bla V_q^{(s+1)},V_{q'}^{(s+1)} \bra,
\end{aligned}
\end{equation}
and
\begin{equation}\label{eq:ddhatH2}
\begin{aligned}
\bla \hat{H}_q^{(k)},\hat{H}_{q'}^{(m)} \bra &= \ka_s\bla P^{(s)}_q,P^{(s)}_{q'} \bra; \\
\bla \delta\hat{H}_q^{(k)},\delta\hat{H}_{q'}^{(m)} \bra &= \frac{\ka_s^2}{s(n+s-2)}\bla P^{(s)}_q,P^{(s)}_{q'} \bra; \\
\bla \pa \delta\hat{H}_q^{(k)},\pa \delta\hat{H}_{q'}^{(m)} \bra &= \frac{\ka_s^2}{s^2(n+s-2)^2}\left[s(s-1)(n+2s-2)(n+2s-4)\right. \\
&\hspace{25pt} +s(n+2s-2)(4qq'-2(q+q')-s^2+2s(q+q'+1)) \\
&\hspace{25pt} \left.+s^2(2s(1-(q+q'))-4qq'+2(q+q')+n-4)\right]\bla P^{(s)}_q,P^{(s)}_{q'} \bra; \\
\bla \pa \delta^2\hat{H}_q^{(k)},\pa \delta^2\hat{H}_{q'}^{(m)} \bra &= \ka_s^2[4(q-1)(q'-1)+2s(q+q'-2)+s(n+2s-2)]\bla P^{(s)}_q,P^{(s)}_{q'} \bra.
\end{aligned}
\end{equation}
\end{lemma}
\begin{proof}
Equations \eqref{eq:ddhatW} and \eqref{eq:ddhatH} follow from direct calculations. Readers may also refer to Lemma \ref{lemma:W_hat0} and \cite[(A.5)--(A.6)]{KMS}.

On the other hand, by the harmonicity of $P_q^{(s)}$ and $(V_q^{(s+1)})_1,\ldots,(V_q^{(s+1)})_n$, together with Lemmas \ref{Euler's homo thm} and \ref{IBP}, we have
\begin{align*}
\int_{\S^{n-1}}\big(V_q^{(s+1)}\big)_{j,i}\big(V_{q'}^{(s+1)}\big)_{j,i} &= \int_{\S^{n-1}}\pa_i\left[\big(V_q^{(s+1)}\big)_{j,i}\big(V_{q'}^{(s+1)}\big)_{j}\right] =(s+1)(n+2s)\bla V_q^{(s+1)},V_{q'}^{(s+1)}\bra;\\
\int_{\S^{n-1}}\pa_jP_q^{(s)} \pa_jP_{q'}^{(s)} &= \int_{\S^{n-1}}\pa_j\left[\pa_jP_q^{(s)} P_{q'}^{(s)}\right]=s(n+2s-2)\bla P_q^{(s)},P_{q'}^{(s)}\bra;\\
\int_{\S^{n-1}}\pa^2_{ij}P_q^{(s)} \pa^2_{ij}P_{q'}^{(s)} &= \int_{\S^{n-1}}\pa_i\left[\pa^2_{ij}P_q^{(s)} \pa_jP_{q'}^{(s)}\right] =(s-1)(n+2s-4)\int_{\S^{n-1}} \pa_jP_q^{(s)}\pa_jP_{q'}^{(s)}.
\end{align*}
Using these identities, we can prove \eqref{eq:ddhatW2} and \eqref{eq:ddhatH2}. See \cite[Page 190]{KMS} for the derivation of the first equality of \eqref{eq:ddhatH2}.
\end{proof}

\subsection{Orthogonality}
We prove orthogonality properties of the bilinear forms $I_{1,\ep}$ in \eqref{eq:I_1} and \eqref{eq:I_1 Q6}.
\begin{lemma}\label{lemma:Hortho}
Given $k = 2,\ldots,d$, we write $H^{(k)} = \hat{H}^{(k)} + \hat{W}^{(k)} + \hat{D}^{(k)}$, where $\hat{H}^{(k)} \in \mcv_k/\mcw_k$, $\hat{W}^{(k)} \in \mcw_k/\mcd_k$, and $\hat{D}^{(k)} \in \mcd_k$. Then it holds that
\[I_{1,\ep}[H^{(k)},H^{(m)}\big] = I_{1,\ep}\big[\hat{H}^{(k)},\hat{H}^{(m)}\big] + I_{1,\ep}\big[\hat{W}^{(k)},\hat{W}^{(m)}\big] + I_{1,\ep}\big[\hat{D}^{(k)},\hat{D}^{(m)}\big]\]
for $k, m \in \{2,\ldots,d\}$.
\end{lemma}
\begin{proof}
We must prove that
\begin{equation}\label{eq:I_1decom1}
I_{1,\ep}\big[\hat{H}^{(k)},\hat{W}^{(m)}\big] = I_{1,\ep}\big[\hat{H}^{(k)},\hat{D}^{(m)}\big] = I_{1,\ep}\big[\hat{W}^{(k)},\hat{D}^{(m)}\big] = 0.
\end{equation}
We first observe that, whether the indices $k_1$, $k_2$, and $k_3$ are the same or not, the inner products satisfy
\begin{equation}\label{eq:I_1decom2}
\bla \hat{H},\hat{W} \bra = \bla \hat{H},\hat{D} \bra = \bla \hat{W},\hat{D} \bra = 0
\end{equation}
for any $\hat{H} \in \mcv_{k_1}/\mcw_{k_1}$, $\hat{W} \in \mcw_{k_2}/\mcd_{k_2}$, and $\hat{D} \in \mcd_{k_3}$.
Indeed, the proof that the first two inner products are $0$ is provided in \cite[Page 186]{KMS}. The third inner product is also $0$, because of Lemma \ref{W hat} and
\begin{align*}
\int_{\S^{n-1}} \proj\left[|x|^{2q} \msd V^{(k_2-2q+1)}\right]_{ij} \hat{D}_{ij} &= \int_{\S^{n-1}} \(\pa_i V^{(k_2-2q+1)}_j + \pa_j V^{(k_2-2q+1)}_i\) \hat{D}_{ij} \\
&= (n+k_2-2q+k_3) \int_{B_1} \(\pa_i V^{(k_2-2q+1)}_j + \pa_j V^{(k_2-2q+1)}_i\) \hat{D}_{ij} \\
&= -2(n+k_2-2q+k_3) \int_{B_1} V^{(k_2-2q+1)}_j \delta_j\hat{D} = 0
\end{align*}
for $q=1,\ldots,\lfloor\frac{k_2-1}{2}\rfloor$, where we used integration by parts and $\hat{D} \in \mcd_{k_3}$.

\medskip
In the remainder of the proof, we will show that $I_{1,\ep}[\hat{H}^{(k)},\hat{W}^{(m)}] = 0$ only, since one can deal with the other terms in \eqref{eq:I_1decom1} analogously once \eqref{eq:I_1decom2} is known.
Owing to \eqref{eq:I_1exp}, Lemma \ref{ricci ricci lemma},
\eqref{eq:Q6 polar proof 4}, Lemmas \ref{lemma:Q6 sphere}, \ref{lemma:T2 sphere}, \ref{lemma:T4 sphere}, and \ref{Delta ricci lemma},
\eqref{eq:hatHq1}, and \eqref{eq:hatWq1}, it suffices to check that
\begin{equation}\label{eq:HW}
\begin{aligned}
\int_{\S^{n-1}} \hat{H}^{(k)}_{ij}\hat{W}^{(m)}_{ij} &= \int_{\S^{n-1}} \delta_i\hat{H}^{(k)} \delta_i\hat{W}^{(m)} = \int_{\S^{n-1}} \delta^2\hat{H}^{(k)} \delta^2\hat{W}^{(m)} \\
&= \int_{\S^{n-1}} \pa_j\delta_i\hat{H}^{(k)} \pa_j\delta_i\hat{W}^{(m)} = \int_{\S^{n-1}} \pa_j\delta^2\hat{H}^{(k)} \pa_j\delta^2\hat{W}^{(m)} = 0.
\end{aligned}
\end{equation}
By \eqref{eq:I_1decom2}, it holds that $\int_{\S^{n-1}} \hat{H}^{(k)}_{ij}\hat{W}^{(m)}_{ij} = \bla \hat{H}^{(k)},\hat{W}^{(m)} \bra = 0$.
Because $\delta^2\hat{W}^{(m)} = 0$, we also have that $\int_{\S^{n-1}}\delta^2\hat{H}^{(k)} \delta^2\hat{W}^{(m)} = \int_{\S^{n-1}} \pa_j\delta^2\hat{H}^{(k)} \pa_j\delta^2\hat{W}^{(m)} = 0$.

For the rest two integrals in \eqref{eq:HW}, we use \eqref{eq:hatHq2} and \eqref{eq:hatWq2} to write
\[\hat{H}^{(k)} = \sum_{q=1}^{\lfloor\frac{k-2}{2}\rfloor} \proj\left[|x|^{2q+2}\nabla^2P^{(k-2q)}\right] \quad \text{and} \quad \hat{W}^{(m)} = \sum_{q'=1}^{\lfloor\frac{m-1}{2}\rfloor} \proj\left[|x|^{2q'}\msd V^{(m-2q'+1)}\right].\]
A direct computation yields
\begin{align*}
\delta_i\hat{H}^{(k)} &= \sum_{q=1}^{\lfloor\frac{k-2}{2}\rfloor} \left[(c_1)_{k,q}|x|^{2q}\pa_iP^{(k-2q)} + (c_2)_{k,q}|x|^{2q-2}P^{(k-2q)}x_i\right]; \\
\pa_j\delta_i\hat{H}^{(k)} &= \sum_{q=1}^{\lfloor\frac{k-2}{2}\rfloor} \left[(c_1)_{k,q}|x|^{2q}\pa^2_{ij}P^{(k-2q)} + 2q(c_1)_{k,q}|x|^{2q-2}\pa_iP^{(k-2q)}x_j+(c_2)_{k,q}|x|^{2q-2}\pa_jP^{(k-2q)}x_i\right.\\
&\hspace{40pt} \left.+2(q-1)(c_2)_{k,q}|x|^{2q-4}P^{(k-2q)}x_ix_j +(c_2)_{k,q}|x|^{2q-2}P^{(k-2q)}\delta_{ij}\right]
\end{align*}
and
\begin{align*}
\delta_i\hat{W}^{(m)} &= \sum_{q'=1}^{\lfloor\frac{m-1}{2}\rfloor} c_{m,q'}|x|^{2q'-2}V^{(m-2q'+1)}_i;\\
\pa_j\delta_i\hat{W}^{(m)} &= \sum_{q'=1}^{\lfloor\frac{m-1}{2}\rfloor} c_{m,q'}|x|^{2q'-2}\pa_jV^{(m-2q'+1)}_i+2(q'-1)c_{m,q'}|x|^{2q'-4}V^{(m-2q'+1)}_ix_j,
\end{align*}
for some $(c_1)_{k,q},\, (c_2)_{k,q},\, c_{m,q'} \in \R$, the exact values of which can be seen in Lemma \ref{lemma:addition info of eigenvector}. Employing Lemma \ref{ricci ricci lemma} and $x_iV^{(m-2q'+1)}_i = \delta V^{(m-2q'+1)} = 0$, we deduce that
\begin{align*}
\int_{\S^{n-1}} \delta_i\hat{H}^{(k)} \delta_i\hat{W}^{(m)} &= \sum_{q=1}^{\lfloor\frac{k-2}{2}\rfloor} \sum_{q'=1}^{\lfloor\frac{m-1}{2}\rfloor} (c_1)_{k,q}c_{m,q'} \int_{\S^{n-1}} \pa_iP^{(k-2q)}V^{(m-2q'+1)}_i.
\end{align*}
Similarly, using Lemmas \ref{Euler's homo thm} and \ref{IBP}, $x_iV^{(m-2q'+1)}_i = \delta V^{(m-2q'+1)} = 0$, and $\Delta V^{(m-2q'+1)}_i=0$, we can also reduce $\int_{\S^{n-1}} \pa_j\delta_i\hat{H}^{(k)} \pa_j\delta_i\hat{W}^{(m)}$ to $\int_{\S^{n-1}} \pa_iP^{(k-2q)}V^{(m-2q'+1)}_i$. In the end, by combining with
$$
\int_{\S^{n-1}} \pa_iP^{(k-2q)}V^{(m-2q'+1)}_i=-\int_{\S^{n-1}} P^{(k-2q)} \delta V^{(m-2q'+1)} = 0,
$$
we prove that $\int_{\S^{n-1}}\delta_i \hat{H}^{(k)} \delta_i\hat{W}^{(m)} = \int_{\S^{n-1}} \pa_j\delta_i\hat{H}^{(k)} \pa_j\delta_i\hat{W}^{(m)} = 0$.

This completes the proof of $I_{1,\ep}[\hat{H}^{(k)},\hat{W}^{(m)}] = 0$.
\end{proof}

Following the approach from the proof of Lemma \ref{lemma:Hortho} and using \eqref{eq:L_1Hw}, \eqref{eq:L_1HZ0}--\eqref{eq:L_1HZ01}, \eqref{eq:L_1Hw Q6}, \eqref{eq:L_1HZ0 Q6}--\eqref{eq:L_1HZ01 Q6} above, along with the orthogonality of harmonic polynomials on $\S^{n-1}$ with different degrees such as
\[\int_{\S^{n-1}} (M_q^{(s_1)})_{ij}(M_{q'}^{(s_2)})_{ij} = \int_{\S^{n-1}} (V_q^{(s_1)})_i (V_{q'}^{(s_2)})_i = \int_{\S^{n-1}} P_q^{(s_1)}P_{q'}^{(s_2)} = 0 \quad \text{for } s_1 \neq s_2,\]
we obtain the following result. We omit the details.
\begin{lemma}\label{lemma:I_1ortho}
Let $\wte^D_s$, $\wte^W_s$ and $\wte^H_s$ be the matrices in \eqref{eq:MqsEDs}, \eqref{eq:VqsEWs} and \eqref{eq:PqsEHs}, respectively.
Then $I_{1,\ep}[\wte^D_s,\wte^D_{s'}] = 0$ for any $2 \le s \ne s' \le d$, $I_{1,\ep}[\wte^W_s,\wte^W_{s'}] = 0$ for any $1 \le s \ne s' \le d-2$, and $I_{i,\ep}[\wte^H_s,\wte^H_{s'}] = 0$ for any $i = 1,2,3$ and $2 \le s \ne s' \le d-2$.
\end{lemma}

\section{Geometric Explanation for the Cancellation Phenomenon}\label{sec:geoexp}
We observe the following two propositions inspired by \cite[Proposition 5]{Br}.

\medskip \noindent \textbf{The fourth-order case.} As before, let $w=(1+r^2)^{-\frac{n-4}{2}}$,
$$
\mfc_4(n) = (n-4)(n-2)n(n+2), \quad \text{and} \quad \tmfc_4(n) = (n-2)n(n+2)(n+4).
$$

\begin{prop}\label{prop:geoexp}
Let $S=\msd V$ where $\msd$ is the conformal Killing operator defined in \eqref{eq:cko} and
$$
\Psi=V_i\pa_iw+\frac{n-4}{2n}\delta V w.
$$
Then,
\begin{equation}\label{eq:geoexp}
\Delta^2 \Psi-\tmfc_4(n) w^{\frac{8}{n-4}} \Psi=-L_1[S]w+\Delta(\pa_i(S_{ij}\pa_j w))+\pa_i(S_{ij}\pa_j \Delta w),
\end{equation}
where $L_1[S]$ is defined by \eqref{eq:L1}.
\end{prop}
\begin{proof}
We have
\begin{align*}
\delta_jS&=\Delta V_j+\frac{n-2}{n}\pa_j\delta V;\\
\delta^2S&=\frac{2(n-1)}{n}\Delta \delta V;\\
(\Dot{\ricci}[S])_{ij}&=\frac{n-2}{n}\pa^2_{ij}\delta V+\frac{1}{n}\Delta \delta V\delta_{ij}.
\end{align*}
Then, there holds
\begin{align*}
&\ -L_1[S]w+\Delta(\pa_i(S_{ij}\pa_j w))+\pa_i(S_{ij}\pa_j \Delta w)\\
&= 4\pa_iV_j\pa^2_{ij}\Delta w-\frac{4}{n}\delta V\Delta^2w+2\Delta V_j\pa_j \Delta
w+\frac{2(n-4)}{n}\pa_j\delta V\pa_j\Delta w+4\pa^2_{jk}V_i\pa^3_{ijk}w\\
&\ +\frac{2(n-4)}{n}\pa^2_{ij}\delta V \pa^2_{ij}w+\frac{n-4}{n}\Delta \delta V\Delta w+4 \pa_i \Delta V_j \pa^2_{ij}w\\
&\ +\Delta^2V_j\pa_j w+\frac{2(n-4)}{n}\pa_j\Delta\delta V \pa_j w+\frac{n-4}{2n}\Delta^2\delta V w.
\end{align*}
Continuing the computation and plugging in $\Delta^2 w =\mfc_4(n)w^{\frac{n+4}{n-4}}$, we deduce
\begin{align*}
&\ -L_1[S]w+\Delta(\pa_i(S_{ij}\pa_j w))+\pa_i(S_{ij}\pa_j \Delta w)\\
&= \Delta^2\(V_i\pa_iw+\frac{n-4}{2n}\delta V w\)-V_i\pa_i\Delta^2w-\frac{n+4}{2n}\delta V\Delta^2 w\\
&= \Delta^2\(V_i\pa_iw+\frac{n-4}{2n}\delta V w\)-\tmfc_4(n)w^{\frac{8}{n-4}}\(V_i\pa_iw+\frac{n-4}{2n}\delta V w\),
\end{align*}
which is \eqref{eq:geoexp}.
\end{proof}

\noindent \textbf{The sixth-order case.} Let $w=(1+r^2)^{-\frac{n-6}{2}}$ and
\begin{align*}
\mfc_6(n)&:=(n-6)(n-4)(n-2)n(n+2)(n+4);\\
\tmfc_6(n)&:=(n-4)(n-2)n(n+2)(n+4)(n+6).
\end{align*}

\begin{prop}\label{prop:geoexpQ6}
Let $S=\mathscr{D} V$ where $\msd$ is the conformal Killing operator defined in \eqref{eq:cko} and
$$
\Psi=V_i\pa_iw+\frac{n-6}{2n}\delta V w.
$$
Then
\begin{equation}\label{eq:geoexpQ6}
\Delta^3 \Psi+\tmfc_6(n) w^{\frac{12}{n-6}} \Psi=-L_1[S]w+\Delta^2\pa_i(S_{ij}\pa_j w)+\Delta\pa_i(S_{ij}\pa_j \Delta w)+\pa_i(S_{ij}\pa_j \Delta^2 w),
\end{equation}
where $L_1[S]$ is defined by \eqref{Q6 L1 formula}.
\end{prop}
\begin{proof}
Because $S=\mathscr{D} V$ and \eqref{Q6 L1 formula}, we have
\begin{align*}
L_1[S] &= \frac{16}{n}\pa^2_{ij}\delta V\pa^2_{ij}\Delta-\frac{3n-6}{2n}\Delta\delta V\Delta^2 + \frac{16}{n}\pa^3_{ijk}\delta V\pa^3_{ijk}-\frac{3n-22}{n}\pa_j \Delta \delta V\pa_j \Delta \\
&\ - \frac{2(n-14)}{n}\pa^2_{ij}\Delta\delta V\pa^2_{ij}-\frac{3n-14}{2n}\Delta \Dot{R}\Delta - \frac{2(n-8)}{n}\pa_j\Delta^2\delta V\pa_j-\frac{n-6}{2n}\Delta^3 \delta V.
\end{align*}
Then, following the proof of \eqref{eq:geoexp}, we can obtain \eqref{eq:geoexpQ6}.
\end{proof}

\begin{rmk}
\

\noindent 1. The roles of Propositions \ref{prop:Psi} and \ref{prop:PsiQ6} are indispensable and cannot be replaced by Propositions \ref{prop:geoexp}--\ref{prop:geoexpQ6}:
In our proofs of Theorems \ref{thm:main4} and \ref{thm:main6}, we rely not only on the estimates for $\Psi$, but also on its explicit expression to compute the Pohozaev quadratic form.

\noindent 2. Nonetheless, these two propositions are valuable in themselves, because they illuminate the special relationship between the linearized equations at the standard bubble and the expansions of both the Paneitz operator and the sixth-order GJMS operator.

\noindent 3. The idea of using the conformal Killing operator to solve the linearized equation and to study the energy expansion is robust.
Analogous propositions may emerge in the higher-order $Q^{(2\msfk)}$-curvature problems, where $\msfk \ge 4$. \hfill $\diamond$
\end{rmk}

\medskip \noindent \textbf{Acknowledgement.}
This paper supersedes our earlier work \textit{Compactness of the Q-curvature problem} (arXiv:2502.14237) that treated only the compactness of the fourth-order constant $Q$-curvature problem.

S. Kim is supported by Basic Science Research Program through the National Research Foundation of Korea (NRF) funded by the Ministry of Science and ICT (2020R1C1C1A01010133, RS-2025-00558417).
Also, he expresses gratitude to Korea Institute for Advanced Study (KIAS) for its support through the visiting faculty program during his sabbatical year.
Part of this work was carried out during his visit to the Chinese University of Hong Kong, and he acknowledges the institute's hospitality.

The research of J. Wei is partially supported by GRF grant entitled ``New frontiers in singular limits of elliptic and parabolic equations".

L. Gong thanks YanYan Li for his encouragement and interesting discussions related to this topic. Part of this work was carried out during Gong's visit to Hanyang University; he acknowledges their hospitality with gratitude. The authors are grateful Mingxiang Li for numerous valuable discussions and also bringing the paper \cite{ALL} to our attention.

\end{document}